\documentclass{article}

\usepackage[english]{babel}
\usepackage[utf8]{inputenc}

\usepackage{amsmath}
\usepackage{mathrsfs}
\usepackage{amsfonts}
\usepackage{graphicx}
\usepackage{bm}
\usepackage[colorinlistoftodos]{todonotes}
\usepackage{blindtext}
\usepackage{amssymb}
\usepackage{bbm}
\usepackage{authblk}

\usepackage{amsthm}
\numberwithin{equation}{section}
\newtheorem{theorem}{Theorem}[section]
\newtheorem{corollary}{Corollary}[section]
\newtheorem{lemma}[theorem]{Lemma}
\newtheorem{proposition}{Proposition}[section]

\theoremstyle{definition}

\newtheorem{definition}{Definition}[section]

\DeclareMathOperator\supp{supp}

\title{Large deviation principle for the Airy point process}

\author{Chenyang Zhong}
\affil{Department of Statistics, Columbia University}

\date{\today}

\begin{document}
\maketitle
\begin{abstract}
The Airy point process is a determinantal point process that arises from the spectral edge of the Gaussian Unitary Ensemble. In this paper, we establish a large deviation principle for the Airy point process, resolving a conjecture in \cite{CG}. Our result also extends to point processes arising from the spectrum of the stochastic Airy operator. 
\end{abstract}

\section{Introduction}\label{Sect.1}

The Airy point process, arising from rescaled eigenvalues near the spectral edge of the Gaussian Unitary Ensemble, is a determinantal point processes with correlation kernel
\begin{equation*}
    K^{\text{Ai}}(x,y)=\frac{\text{Ai}(x)\text{Ai}'(y)-\text{Ai}(y)\text{Ai}'(x)}{x-y},
\end{equation*}
where $\text{Ai}(\cdot)$ is the Airy function. We refer the reader to \cite[Section 4.2]{AGZ} for a detailed introduction to the Airy point process. Since its introduction by Tracy and Widom \cite{TW}, the Airy point process has been extensively investigated in the probability and mathematical physics literature. For example, \cite{Soshnikov} established a central limit theorem for the number of particles in a large interval. The rigidity property of the Airy point process was established in \cite{Bufetov}, and global rigidity upper bounds for the Airy point process were obtained in \cite{CC2}. In \cite{CD,CC}, the generating function for the Airy point process was investigated, which is further related to a system of coupled Painlev\'e II equations and leads to large gap asymptotics for the Airy point process. Further results on the Airy point process are in \cite{BC,LP,Ahn}. The Airy point process also appears in the limit of various interacting particle systems; we refer to \cite{Johansson} for a review.

In this paper, we focus on large deviations of the Airy point process. Namely, denoting by $a_1>a_2>\cdots$ the ordered points from the Airy point process, we aim to investigate large deviations of the scaled and space-reversed Airy point process empirical measure 
\begin{equation}\label{EM}
    \frac{1}{k}\sum_{i=1}^{\infty} \delta_{-k^{-2\slash 3} a_i}
\end{equation}
as the scaling parameter $k\in \mathbb{N}^{*}$ approaches infinity. Through a moment-matching formula by Borodin and Gorin \cite{BG}, such large deviations are also related to the lower tail of the Kardar-Parisi-Zhang equation (see \cite{CGKLT,CG,CG2,Kim,Tsai,CC3}). 

In \cite{CGKLT}, a physical and non-rigorous derivation of large deviations of (\ref{EM}) was proposed. A conjecture was made in \cite{CG} asking for a rigorous derivation. In this paper, we resolve this conjecture by establishing a full large deviation principle for (\ref{EM}). Our result extends an earlier paper \cite{Zho} on large deviation bounds for this empirical measure.

The Airy point process is also related to the stochastic Airy operator introduced in \cite{RRV}. The stochastic Airy operator depends on a parameter $\beta>0$, and is a random Schr\"odinger operator that describes the rescaled eigenvalue configuration near the spectral edge of a class of random matrix ensembles called $\beta$ ensembles \cite{DE}. When $\beta=2$, the eigenvalue configuration of the stochastic Airy operator is equal in distribution to the space reversal of the Airy point process ($-a_1,-a_2,\cdots$). The reader is referred to \cite{KRV,GS,LS,DL,GGL} for recent works on the stochastic Airy operator.

In the rest of this paper, we refer to the point process formed by the eigenvalues of the stochastic Airy operator with parameter $\beta$ as the ``$\beta$-Airy point process''. Our large deviation result also extends to the $\beta$-Airy point process for $\beta\in\mathbb{N}^{*}$.

The rest of this Introduction is organized as follows. In Section \ref{Sect.1.1}, we present our main result on large deviations of the $\beta$-Airy point process. Then we provide a brief review of stochastic Airy operator and Gaussian $\beta$ ensemble in Section \ref{Sect.1.3}. 

\subsection{Main result}\label{Sect.1.1}

In this subsection, we present the large deviation principle for the $\beta$-Airy point process. Throughout the paper, we fix an arbitrary $\beta\in\mathbb{N}^{*}$.

We start with the following definitions. We denote by $\mathcal{B}_{\mathbb{R}}$ the Borel $\sigma$-algebra on $\mathbb{R}$. 

\begin{definition}\label{Defn1.1}
For any $R>0$ and finite signed Borel measures $\mu,\nu$ on $[-R,R]$, the Kantorovich-Rubinshtein distance between $\mu$ and $\nu$ is defined as
\begin{equation*}
d_R(\mu,\nu):=\sup_{\substack{f:[-R,R]\rightarrow \mathbb{R},\\
    \|f\|_{BL}\leq 1}}\Big|\int_{[-R,R]} fd\mu-\int_{[-R,R]} fd\nu\Big|,
\end{equation*}
where $\|f\|_{BL}:=\max\{\|f\|_{\infty},\|f\|_{Lip}\}$ is the bounded Lipschitz norm.
\end{definition}

\begin{definition}\label{Defn1.2}
Let $\nu_0$ be the Borel measure on $\mathbb{R}$ such that
\begin{equation*}
 \nu_0(A)=\frac{1}{\pi}\int_A\sqrt{x}\mathbbm{1}_{[0,\infty)}(x)dx, \quad \forall A\in \mathcal{B}_{\mathbb{R}}.
\end{equation*}
For any $R>0$, we let $\nu_{0;R}$ be the restriction of $\nu_0$ to $[-R,R]$, and let $\mathcal{X}_R$ be the set of finite signed Borel measures $\mu$ on $[-R,R]$ such that $\mu+\nu_{0;R}$ is a positive measure. We also let $\mathscr{X}$ be the set of compactly supported finite signed Borel measures $\mu$ on $\mathbb{R}$ such that $\mu+\nu_0$ is a positive measure and $\mu(\mathbb{R})=0$.
\end{definition}

\begin{definition}\label{Defn1.3}
For any $\mu\in\mathscr{X}$, we define
\begin{equation*}
    \mathscr{I}(\mu):=-\int \log(|x-y|)d\mu(x)d\mu(y)+\frac{4}{3}\int |x|^{3\slash 2}\mathbbm{1}_{(-\infty,0]}(x)d\mu(x).
\end{equation*}
For any $R>0$ and $\mu\in\mathscr{X}$, we define $\mu_R$ to be the restriction of $\mu$ to $[-R,R]$. For any $R>0$, $\mu\in\mathcal{X}_R$, and $\delta>0$, we define
\begin{equation*}
    I_R(\mu,\delta):=\inf_{\substack{\mu'\in \mathscr{X}:\\d_R(\mu'_R,\mu)\leq \delta}}\mathscr{I}(\mu').
\end{equation*}
Note that for fixed $R>0$ and $\mu\in\mathcal{X}_R$, $I_R(\mu,\delta)$ viewed as a function of $\delta>0$ is monotone non-increasing. For any $R>0$ and $\mu\in\mathcal{X}_R$, we define
\begin{equation}\label{Eq1.2.1}
    I_R(\mu):=\lim_{\delta\rightarrow 0^{+}} I_R(\mu,\delta).
\end{equation}
\end{definition}

Throughout the rest of the paper, we fix an arbitrary $R\geq 10$. We denote by $\lambda_1<\lambda_2<\cdots$ the ordered eigenvalues of the stochastic Airy operator with parameter $\beta$ (see Section \ref{Sect.1.3} for details). For any $k\in \mathbb{N}^{*}$, we define the scaled $\beta$-Airy point process empirical measure as 
\begin{equation*}
    \bar{\nu}_k:=\frac{1}{k}\sum_{i=1}^{\infty} \delta_{k^{-2\slash 3} \lambda_i}.
\end{equation*}
We also let $\bar{\nu}_{k;R}$ be the restriction of $\bar{\nu}_k$ to $[-R,R]$, and define
\begin{equation}\label{Eq1.1}
    \nu_{k;R}:=\bar{\nu}_{k;R}-\nu_{0;R}. 
\end{equation}
Note that $\nu_{k;R}\in\mathcal{X}_R$ almost surely.

Now we state our main result. We take the topological space to be $\mathcal{X}_R$ equipped with the Kantorovich-Rubinshtein distance $d_R(\cdot,\cdot)$. The following theorem establishes a full large deviation principle for the $\beta$-Airy point process.

\begin{theorem}\label{Thm1}
Assuming the preceding setup, $\{\nu_{k;R}\}_{k=1}^{\infty}$ satisfies a large deviation principle with speed $k^2$ and good rate function $\frac{\beta}{2} I_R$. 
\end{theorem}

\subsection{Stochastic Airy operator and Gaussian $\beta$ ensemble}\label{Sect.1.3}

In this subsection, we review background materials on stochastic Airy operator and Gaussian $\beta$ ensemble following \cite{RRV} and \cite[Section 4.5]{AGZ}. We also set up some notations that will be used in later parts of the paper.

Let $D$ be the space of Schwartz distributions, and let $H_{\text{loc}}^1$ be the space of real-valued functions $f$ on $[0,\infty)$ such that $f'\mathbbm{1}_I\in L^2$ for any compact set $I$. Let $(B_x,x\geq 0)$ be the standard Brownian motion. The stochastic Airy operator, denoted by $\mathcal{H}_{\beta}$, is a random linear map from $H_{\text{loc}}^1$ to $D$ that sends $f\in H_{\text{loc}}^1$ to
\begin{equation*}
    \mathcal{H}_{\beta} f=-f''+xf+\frac{2}{\sqrt{\beta}}f B',
\end{equation*}
where $f B'$ is understood as the distributional derivative of the continuous function $-\int_{0}^x B_t f'(t)dt+f(x)B_x$ for $x\in [0,\infty)$. We also define the Airy operator $\mathcal{A}$ as the nonrandom part of $\mathcal{H}_{\beta}$: for any $f\in H_{\text{loc}}^1$, $\mathcal{A} f =-f''+xf$.

Let $C_0^{\infty}$ be the space of compactly supported infinitely differentiable functions on $(0,\infty)$. Let 
\begin{equation*}
    L^{*}:=\Big\{f\in H^1[0,\infty): f(0)=0, \int_{0}^{\infty} ((f'(x))^2+(1+x)f(x)^2)dx<\infty\Big\}. 
\end{equation*}
The eigenvalues and eigenfunctions of $\mathcal{H}_{\beta}$ are defined as the pairs $(f,\lambda)\in L^{*}\times \mathbb{R}$ such that $\|f\|_2=1$ and $\mathcal{H}_{\beta}f=\lambda f$ in the sense of distributions. It was shown in \cite{RRV} that the collection of eigenvalues of $\mathcal{H}_{\beta}$ has a well-defined $i$th smallest element $\lambda_i$ for each $i\in\mathbb{N}^{*}$. We can similarly define the eigenvalues and eigenfunctions of $\mathcal{A}$, and denote the ordered eigenvalues of $\mathcal{A}$ by $\gamma_1<\gamma_2<\cdots$. 

The following result on $\{\gamma_i\}_{i=1}^{\infty}$ follows from \cite{MT,Titch}; see also \cite[Theorem 3.1]{Hua}).

\begin{proposition}\label{AiryOperator}
For any $i\in\mathbb{N}^{*}$, we have
\begin{equation*}
    \gamma_i=\Big(\frac{3\pi}{2}\Big(i-\frac{1}{4}+R(i)  \Big)\Big)^{2\slash 3},
\end{equation*}
where $|R(i)|\leq C\slash i$ for some absolute constant $C>0$.
\end{proposition}

The following tail bound on $\lambda_1$ was shown in \cite{DV}.

\begin{proposition}\label{TailBound}
When $t\rightarrow\infty$, we have
\begin{equation*}
    \mathbb{P}(\lambda_1<-t)=t^{-3\beta\slash 4}\exp\Big( -\frac{2}{3}\beta t^{3\slash 2} +O(\sqrt{\log{t}})\Big).
\end{equation*}
In particular, there exists a positive constant $C$ that only depends on $\beta$, such that for any $t>0$,
\begin{equation*}
    \mathbb{P}(\lambda_1<-t)\leq C\exp\Big(-\frac{2}{3}\beta t^{3\slash 2}\Big).
\end{equation*}
\end{proposition}

Throughout the rest of this paper, for any finite set $A$, we denote by $|A|$ the cardinality of $A$. For any $x\in\mathbb{R}$, we let 
\begin{equation}\label{Eq2.2}
   N(x):=|\{i\in\mathbb{N}^{*}: \lambda_i\leq x\}|,\quad N_0(x):=|\{i\in \mathbb{N}^{*}: \gamma_i\leq x\}|.
\end{equation}
For any $\lambda\in \mathbb{R}$, we define $p_{\lambda}(x), x\geq 0$ as the following diffusion: 
\begin{equation}\label{DefinitionP}
  dp_{\lambda}(x)=(x-\lambda-p_{\lambda}(x)^2)dx+\frac{2}{\sqrt{\beta}} dB_x, \quad p_{\lambda}(0)=\infty,
\end{equation}
and define $q_{\lambda}(x), x\geq 0$ as the solution to the following differential equation:
\begin{equation}\label{DefinitionQ}
   q_{\lambda}'(x)=x-\lambda-q_{\lambda}(x)^2, \quad q_{\lambda}(0)=\infty. 
\end{equation}
The diffusion $p_{\lambda}(x)$ or the function $q_{\lambda}(x)$ may blow up to $-\infty$ at a finite time; immediately after this blow-up time, we restart $p_{\lambda}(x)$ or $q_{\lambda}(x)$ at $+\infty$. For any $\lambda,\lambda_1\in\mathbb{R}$ and $\lambda_2\in\mathbb{R}\cup\{\infty\}$ such that $\lambda_1\leq \lambda_2$, we let $N(\lambda;\lambda_1,\lambda_2)$ be the number of blow-ups of $p_{\lambda}(x)$ in the interval $(\lambda_1,\lambda_2]$, and let $N_0(\lambda;\lambda_1,\lambda_2)$ be the number of blow-ups of $q_{\lambda}(x)$ in the interval $(\lambda_1,\lambda_2]$. 

We have the following proposition.

\begin{proposition}[\cite{RRV}]\label{diffu}
For any fixed nonrandom $\lambda\in\mathbb{R}$, $N(\lambda)$ is equal to the number of blow-ups of $p_{\lambda}(x)$ in $[0,\infty)$, and $N_0(\lambda)$ is equal to the number of blow-ups of $q_{\lambda}(x)$ in $[0,\infty)$.
\end{proposition}

The following result follows from Propositions \ref{TailBound} and \ref{diffu}. 

\begin{corollary}\label{tail}
There exists a positive constant $C$ that only depends on $\beta$, such that for any $t>0$, we have
\begin{equation*}
    \mathbb{P}(p_{-t}(x) \text{ blows up in finite time})\leq C\exp\Big(-\frac{2}{3}\beta t^{3\slash 2}\Big).
\end{equation*}
\end{corollary}

For any $\beta>0$, the \emph{Gaussian $\beta$ ensemble} of size $n$ is the probability distribution on ordered points $\lambda_1^{(n)}> \lambda_2^{(n)} > \cdots > \lambda_n^{(n)}$ with joint density
\begin{equation}\label{Eqn10}
 Z_{\beta,n}^{-1} \exp\Big(-\frac{\beta}{4}\sum\limits_{i=1}^n (\lambda_i^{(n)})^2\Big)\prod_{1\leq i<j\leq n} \big|\lambda_i^{(n)}-\lambda_j^{(n)}\big|^{\beta},
\end{equation}
where $Z_{\beta,n}$ is a normalizing constant. Dumitriu and Edelman \cite{DE} introduced a class of tridiagonal random matrix models (which we denote by $H_{\beta,n}$ in this paper) and showed that the joint density of the eigenvalues of such random matrices is given by (\ref{Eqn10}). Below we describe this random matrix $H_{\beta,n}$. We first recall the definition of $\chi$ random variables:

\begin{definition}
For any $t>0$, the $\chi$ distribution with $t$ degrees of freedom, denoted by $\chi_t$, is defined as the probability distribution on $(0,\infty)$ with density 
\begin{equation*}
    f_t(x)=\Gamma(t\slash 2)^{-1} 2^{1-t\slash 2}x^{t-1}e^{-x^2\slash 2}, \quad \forall x>0.
\end{equation*}
\end{definition} 

We denote $[0]:=\emptyset$ and $[m]:=\{1,2,\cdots,m\}$ for any $m\in\mathbb{N}^{*}$. Let $\xi_1,\cdots,\xi_n$ be i.i.d. $N(0,1)$ random variables. Let $Y_1,\cdots,Y_{n-1}$ be mutually independent random variables that are independent of $\{\xi_i\}_{i=1}^n$, such that $Y_i\sim \chi_{i\beta}$ for every $i\in [n-1]$. Then $H_{\beta,n}$ is an $n$ by $n$ random matrix whose entries are given as follows: for any $(i,j)\in [n]^2$ such that $|i-j|>1$, $H_{\beta,n}(i,j)=0$; for any $i\in [n]$, $H_{\beta,n}(i,i)=\sqrt{2\slash\beta} \xi_i$; for any $i\in [n-1]$, $H_{\beta,n}(i,i+1)=H_{\beta,n}(i+1,i)=Y_{n-i}\slash\sqrt{\beta}$.

In the following, we fix $k,n\in\mathbb{N}^{*}$ such that $k\leq n$. For any $i\in [n]$, we define $\tilde{\lambda}^{(n)}_i:=n^{1\slash 6} (\lambda^{(n)}_i-2\sqrt{n})$ and $b_i:=-k^{-2\slash 3}  \tilde{\lambda}^{(n)}_i$. Let $\mu_0$ be the Borel measure on $\mathbb{R}$ such that for any $A\in \mathcal{B}_{\mathbb{R}}$,
\begin{equation*}
  \mu_0(A)=\frac{1}{\pi}\int_A \sqrt{x}\Big(1-\frac{1}{4}(k\slash n)^{2\slash 3}x\Big)^{1\slash 2}\mathbbm{1}_{[0,4(n\slash k)^{2\slash 3}]}(x)dx.
\end{equation*} 
We define 
\begin{equation}
  \mu_{n,k}:=\frac{1}{k}\sum_{i=1}^n \delta_{b_i}-\mu_0,
\end{equation}
and let $\mu_{n,k;R}$ be the restriction of $\mu_{n,k}$ to $[-R,R]$. For any $x\in\mathbb{R}$, we let 
\begin{eqnarray}\label{Eq2.21}
  && \tilde{N}(x):=|\{i\in [n]: b_i\leq x\}|, \quad\tilde{N}_0(x):=\begin{cases}
   k \mu_0([0,x]),& \text{ if }x\geq 0\\
   0, & \text{ if }x<0
  \end{cases},
  \nonumber\\
  && \Psi(x):=\tilde{N}(x)-\tilde{N}_0(x).
\end{eqnarray}

\bigskip

Throughout the rest of this paper, we fix a probability space $(\Omega,\mathcal{F},\mathbb{P})$, on which a standard Brownian motion $(B_x,x\geq 0)$ is constructed with filtration $\mathcal{F}_t=\sigma(B_x,x\in [0,t])$, $t\geq 0$. For any $t\geq s\geq 0$, we let $\mathcal{F}_{s,t}:=\sigma(B_x-B_s,x\in [s,t])$. We denote by $Leb$ the Lebesgue measure on $\mathbb{R}$. Unless otherwise specified, we denote by $C,c$ positive constants that only depend on $\beta$; the values of these constants may change from line to line.

The rest of this paper is organized as follows. In Section \ref{Sect.3}, we establish some tail bounds on the number of blow-ups of the diffusion $p_{\lambda}(x)$. Then we introduce some basic setups in Section \ref{Sect.2} that will be used in the proof of Theorem \ref{Thm1}. Bounds on several quantities introduced in Section \ref{Sect.2} are provided in Section \ref{Sect.4}. In Section \ref{Sect.5}, we finish the proof of Theorem \ref{Thm1}. 

\subsection{Acknowledgements}

The author wishes to thank Amir Dembo for suggesting the problem and Ivan Corwin for proposing the problem. The author also thanks Ivan Corwin, Amir Dembo, and Li-Cheng Tsai for helpful conversations.

\section{Tail bounds on the number of blow-ups of the diffusion $p_{\lambda}(x)$}\label{Sect.3}

In this section, we establish tail bounds on the number of blow-ups of the diffusion $p_{\lambda}(x)$ as defined in (\ref{DefinitionP}) on various types of intervals. We recall the definitions of $N(x),N_0(x),N(\lambda;\lambda_1,\lambda_2),N_0(\lambda;\lambda_1,\lambda_2)$ from Section \ref{Sect.1.3} (see (\ref{Eq2.2}) and (\ref{Eq2.21})). We also introduce the following definition.

\begin{definition}\label{Def3.1}
For any $\lambda,\lambda_1\in \mathbb{R}$, we let $\overline{p}_{\lambda,\lambda_1}(x),x\geq \lambda_1$ be the diffusion 
\begin{equation}
  d\overline{p}_{\lambda,\lambda_1}(x)=(x-\lambda-\overline{p}_{\lambda,\lambda_1}(x)^2)dx+\frac{2}{\sqrt{\beta}}dB_x, \quad \overline{p}_{\lambda,\lambda_1}(\lambda_1)=\infty,
\end{equation}
and let $\overline{q}_{\lambda,\lambda_1}(x),x\geq \lambda_1$ be the solution to the differential equation
\begin{equation}
  \overline{q}_{\lambda,\lambda_1}'(x)=x-\lambda-\overline{q}_{\lambda,\lambda_1}(x)^2, \quad \overline{q}_{\lambda,\lambda_1}(\lambda_1)=\infty.
\end{equation}
For any $\lambda,\lambda_1\in\mathbb{R}$ and $\lambda_2\in\mathbb{R}\cup\{\infty\}$ such that $\lambda_1\leq \lambda_2$, we let $\overline{N}(\lambda;\lambda_1,\lambda_2)$ be the number of blow-ups of $\overline{p}_{\lambda,\lambda_1}(x)$ in the interval $(\lambda_1,\lambda_2]$, and let $\overline{N}_0(\lambda;\lambda_1,\lambda_2)$ be the number of blow-ups of $\overline{q}_{\lambda,\lambda_1}(x)$ in the interval $(\lambda_1,\lambda_2]$.
\end{definition}

The main results of this section are given in Propositions \ref{P3.1}-\ref{P3.5} below. 

\begin{proposition}\label{P3.1}
There exist positive constants $K\geq 10,C,c$ that only depend on $\beta$, such that the following holds. For any $\lambda,\lambda_1,\lambda_2,x\in\mathbb{R}$ that satisfy
\begin{equation*}
0\leq \lambda_1<\lambda_2<\lambda, \quad \min\{\lambda_2-\lambda_1,\lambda-\lambda_2\}\geq K,
\end{equation*}
\begin{equation*}
 x\geq K\Big(1+\log\Big(\frac{\lambda-\lambda_1}{\lambda-\lambda_2}\Big)\Big)\max\Big\{1,\log\log\Big(\frac{\lambda-\lambda_1}{\lambda-\lambda_2}\Big)\Big\},
\end{equation*}
there exists an event $\mathcal{U}\in\mathcal{F}_{\lambda_1,\lambda_2}$, such that 
\begin{equation}
\{|N(\lambda;\lambda_1,\lambda_2)-N_0(\lambda;\lambda_1,\lambda_2)|\geq x\}\subseteq \mathcal{U},
\end{equation}
\begin{eqnarray}
&& \mathbb{P}(|N(\lambda;\lambda_1,\lambda_2)-N_0(\lambda;\lambda_1,\lambda_2)|\geq x)\leq \mathbb{P}(\mathcal{U}) \nonumber \\
&\leq& \exp(C (\lambda-\lambda_1)^{1\slash 2}(\lambda_2-\lambda_1)\log\log(\lambda-\lambda_1))\nonumber\\
&& \times  \exp\Big(-c x^2\Big/\log\Big(\frac{\lambda-\lambda_1}{\lambda-\lambda_2}\Big)\Big).
\end{eqnarray}
\end{proposition}

\begin{proposition}\label{P3.2}
There exist positive constants $K,C,c$ that only depend on $\beta$, such that the following holds. For any $\lambda,\lambda_1,x\in\mathbb{R}$ and any $\lambda_2\in \mathbb{R}\cup \{\infty\}$ that satisfy $\lambda_1-\lambda\geq K$, $\lambda_2\geq \lambda_1$, and $x\geq 2$, there exists an event $\mathcal{U}\in \mathcal{F}_{\lambda_1,\lambda_2}$, such that 
\begin{equation}
\{N(\lambda;\lambda_1,\lambda_2)\geq x\}\subseteq \mathcal{U},
\end{equation}
\begin{equation}
  \mathbb{P}(N(\lambda;\lambda_1,\lambda_2)\geq x)\leq \mathbb{P}(\mathcal{U})\leq C\exp(-c(\lambda_1-\lambda)^{3\slash 2} x).
\end{equation}
\end{proposition}

\begin{proposition}\label{P3.3}
For any $L\geq 1$, there exist positive constants $K\geq 10, C, c$ that only depend on $\beta$ and $L$, such that the following holds. For any $\lambda,\lambda_1,\lambda_2,x\in\mathbb{R}$ that satisfy
\begin{equation*}
 0\leq \lambda_1<\lambda<\lambda_2, \quad \lambda-\lambda_1\geq K, \quad \lambda_2-\lambda\leq L(\lambda-\lambda_1),
\end{equation*}
\begin{equation*}
x>0,\quad x\geq K(\log(2+(\lambda-\lambda_1)x^{-2\slash 3}))^2,
\end{equation*}
there exists an event $\mathcal{U}\in\mathcal{F}_{\lambda_1,\lambda_2}$, such that
\begin{equation}
\{|N(\lambda;\lambda_1,\lambda_2)-N_0(\lambda;\lambda_1,\lambda_2)|\geq x\}\subseteq \mathcal{U},
\end{equation}
\begin{eqnarray}
&& \mathbb{P}(|N(\lambda;\lambda_1,\lambda_2)-N_0(\lambda;\lambda_1,\lambda_2)|\geq x)\leq \mathbb{P}(\mathcal{U})\nonumber\\
&\leq& \exp(C(\lambda-\lambda_1)^{3\slash 2} \log\log(\lambda-\lambda_1))\nonumber\\
&&\times \exp(-cx^2\slash\log(2+(\lambda-\lambda_1)x^{-2\slash 3})).
\end{eqnarray}
\end{proposition}

\begin{proposition}\label{P3.4}
There exist positive constants $K\geq 10,C,c$ that only depend on $\beta$, such that for any $\lambda,x\in\mathbb{R}$ that satisfy $\lambda\geq 0$ and $x\geq K\max\{1,\lambda^{3\slash 2}\}$, we have
\begin{equation}
   \mathbb{P}(N(\lambda)\geq x)\leq C\exp(-c x^2).
\end{equation}
\end{proposition}

\begin{proposition}\label{P3.5}
There exist positive constants $K,C,c$ that only depend on $\beta$, such that for any $\lambda,x\in\mathbb{R}$ that satisfy $\lambda\leq -K$ and $x>0$, we have
\begin{equation}
  \mathbb{P}(N(\lambda)\geq x)\leq C\exp(-c(-\lambda)^{3\slash 2}x).
\end{equation} 
\end{proposition}

The rest of this section is devoted to the proofs of Propositions \ref{P3.1}-\ref{P3.5}. We first give some preliminary lemmas in Section \ref{Sect.3.1}. The proofs of Propositions \ref{P3.1}-\ref{P3.5} are presented in Sections \ref{Sect.3.2}-\ref{Sect.3.6}, respectively.

\subsection{Preliminary lemmas}\label{Sect.3.1}

In this subsection, we state and prove several preliminary lemmas.

\begin{lemma}\label{Lem3.1}
For any $\lambda>0$, let $t_{\lambda}$ be the first blow-up time of the diffusion $p_{\lambda}(x)$. There exist positive constants $C,c$ that only depend on $\beta$, such that for any $s\geq 1$, we have
\begin{equation}\label{Eq3.4}
   \mathbb{P}\Big(t_{\lambda}< \frac{\pi}{2s\sqrt{\lambda}}\Big)\leq C\exp(-c \lambda^{3\slash 2} s^2).
\end{equation}
\end{lemma}
\begin{proof}

Let $r(x):=p_{\lambda}(x)-(2\slash \sqrt{\beta})B_x$. Note that $r(x)$ satisfies the following differential equation:
\begin{equation}
  r'(x) = x-\lambda-\Big(r(x)+\frac{2}{\sqrt{\beta}}B_x\Big)^2,\quad r(0)=\infty.
\end{equation}

Let $\mathcal{C}$ be the event that 
\begin{equation}
  \sup_{x\in [0, 2\pi\slash\sqrt{\lambda}]}|B_x|\leq s\sqrt{\beta\lambda\slash 8}.
\end{equation}
We recall the following standard estimate for Brownian motion: for any $t\geq 0$,
\begin{equation}\label{Eq3.1}
    \mathbb{P}(\sup_{x\in [0,1]}|B_x|>t)\leq  4 e^{-t^2\slash 2}.
\end{equation}
From (\ref{Eq3.1}), we can deduce that
\begin{equation}\label{Eq3.3}
  \mathbb{P}(\mathcal{C}^c)\leq 4\exp\Big(-\frac{\beta \lambda^{3\slash 2} s^2}{32\pi}\Big).
\end{equation}

Now we assume that the event $\mathcal{C}$ holds. By the AM-GM inequality, for any $x\in [0,\min\{2\pi\slash \sqrt{\lambda},t_{\lambda}\}]$, 
\begin{eqnarray*}
&& x-\lambda-\Big(r(x)+\frac{2}{\sqrt{\beta}}B_x\Big)^2\geq -\lambda-2r(x)^2-8\beta^{-1}B_x^2\nonumber\\
&\geq& -\lambda(1+s^2)-2r(x)^2\geq -2\lambda s^2-2r(x)^2.
\end{eqnarray*}
Let $w(x)$ be 
defined by
\begin{equation}\label{Eq3.2}
   w'(x)=-2\lambda s^2-2w(x)^2, \quad w(0)=\infty.
\end{equation}
For $x\in [0, \min\{2\pi\slash\sqrt{\lambda},t_{\lambda}\}]$, before $w(x)$ blows up, $r(x)$ is lower bounded by $w(x)$. Solving (\ref{Eq3.2}), we obtain that
\begin{equation*}
 w(x)=s\sqrt{\lambda}\tan\Big(\frac{\pi}{2}-2s\sqrt{\lambda}x\Big),
\end{equation*}
which blows up at $x=\pi\slash (2s\sqrt{\lambda})<2\pi\slash \sqrt{\lambda}$.
Therefore, we have $t_{\lambda}\geq \pi\slash (2s\sqrt{\lambda})$. Combining this with (\ref{Eq3.3}), we obtain (\ref{Eq3.4}) with $C=4, c=\beta\slash (32\pi)$. 

\end{proof}

\begin{lemma}\label{Lem3.3}
Consider any $a,\lambda\in\mathbb{R}$ such that $a>0$ and $\lambda\leq a$. Let $t_{\lambda}$ be the first blow-up time of the diffusion $p_{\lambda}(x)$. Then there exist positive constants $C,c$ that only depend on $\beta$, such that for any $s\geq 1$, we have
\begin{equation}
   \mathbb{P}\Big(t_{\lambda}< \frac{\pi}{2s\sqrt{a}}\Big)\leq C\exp(-c a^{3\slash 2} s^2).
\end{equation}
\end{lemma}

\begin{proof}
Let $t_{a}$ be the first blow-up time of the diffusion $p_{a}(x)$. As $\lambda\leq a$, $p_{\lambda}(x)$ is lower bounded by $p_a(x)$ (before $p_a(x)$ blows up), hence $t_{\lambda}\geq t_a$. By Lemma \ref{Lem3.1}, for any $s\geq 1$, we have
\begin{equation}
  \mathbb{P}\Big(t_a<\frac{\pi}{2s\sqrt{a}}\Big)\leq C\exp(-c a^{3\slash 2} s^2).
\end{equation}
Hence
\begin{equation}
  \mathbb{P}\Big(t_{\lambda}<\frac{\pi}{2s\sqrt{a}}\Big)\leq C\exp(-c a^{3\slash 2} s^2).
\end{equation}
\end{proof}

\begin{lemma}\label{Lem3.2}
Recall Definition \ref{Def3.1}. There exist positive absolute constants $K,C,c$, such that for any $\lambda,\lambda_1,\lambda_2\geq 0$ that satisfy $\lambda_1<\lambda_2<\lambda$, $\lambda-\lambda_1\leq 2(\lambda-\lambda_2)$, and $\lambda_2-\lambda_1\geq K$, we have
\begin{equation}\label{Eq3.8n}
   c (\lambda-\lambda_1)^{1\slash 2}(\lambda_2-\lambda_1)\leq \overline{N}_0(\lambda;\lambda_1,\lambda_2)\leq C(\lambda-\lambda_1)^{1\slash 2}(\lambda_2-\lambda_1).
\end{equation}
\end{lemma}

\begin{proof}
Let $\tau_0:=\lambda_1$. For any $j\in\mathbb{N}^{*}$, if there exist at least $j$ blow-ups of $\overline{q}_{\lambda,\lambda_1}(x)$ in the interval $(\lambda_1,\infty)$, we let $\tau_j$ be the $j$th blow-up time of $\overline{q}_{\lambda,\lambda_1}(x)$ in $(\lambda_1,\infty)$; otherwise we let $\tau_j:=\infty$. 

Consider any $j\in \{0,1,\cdots,\overline{N}_0(\lambda;\lambda_1,\lambda_2)\}$. We define $r_j(x):=\overline{q}_{\lambda,\lambda_1}(x+\tau_j)$ for any $x\geq 0$. Then $r_j(x)$ is the solution to the following differential equation:
\begin{equation*}
  r_j'(x)=x-(\lambda-\tau_j)-r_j(x)^2, \quad r_j(0)=\infty.
\end{equation*}
As $\lambda_1\leq \tau_j\leq \lambda_2$ and $\lambda-\lambda_1\leq 2(\lambda-\lambda_2)$, we have
\begin{equation}\label{Eq3.7}
  (\lambda-\lambda_1)\slash 2\leq \lambda-\tau_j\leq \lambda-\lambda_1.
\end{equation}
We take $K=2(4\pi)^{2\slash 3}$ and assume that $\lambda_2-\lambda_1\geq K$ in the following. Note that $\lambda-\tau_j\geq (\lambda-\lambda_1)\slash 2\geq (\lambda_2-\lambda_1)\slash 2\geq (4\pi)^{2\slash 3}$. By \cite[Lemma 2.3]{Zho}, we have
\begin{equation}\label{Eq3.6}
    \frac{\pi}{\sqrt{\lambda-\tau_j}} \leq \tau_{j+1}-\tau_j \leq \frac{\pi}{\sqrt{\lambda-\tau_j-2\pi\slash\sqrt{\lambda-\tau_j}}}\leq \frac{\sqrt{2}\pi}{\sqrt{\lambda-\tau_j}}.
\end{equation}
Combining (\ref{Eq3.7}) and (\ref{Eq3.6}), we obtain that
\begin{equation}
   \frac{\pi}{\sqrt{\lambda-\lambda_1}} \leq \tau_{j+1}-\tau_j\leq \frac{2\pi}{\sqrt{\lambda-\lambda_1}}.
\end{equation}

Therefore, noting that $\tau_{\overline{N}_0(\lambda;\lambda_1,\lambda_2)}\leq \lambda_2<\tau_{\overline{N}_0(\lambda;\lambda_1,\lambda_2)+1}$, we have
\begin{equation*}
   \lambda_2-\lambda_1 \geq \tau_{\overline{N}_0(\lambda;\lambda_1,\lambda_2)}-\lambda_1=\sum_{j=0}^{\overline{N}_0(\lambda;\lambda_1,\lambda_2)-1}(\tau_{j+1}-\tau_j)\geq \frac{\pi\overline{N}_0(\lambda;\lambda_1,\lambda_2)}{\sqrt{\lambda-\lambda_1}},
\end{equation*}
\begin{equation*}
 \lambda_2-\lambda_1\leq \tau_{\overline{N}_0(\lambda;\lambda_1,\lambda_2)+1}-\lambda_1\leq \sum_{j=0}^{\overline{N}_0(\lambda;\lambda_1,\lambda_2)}(\tau_{j+1}-\tau_j)\leq\frac{2\pi (\overline{N}_0(\lambda;\lambda_1,\lambda_2)+1)}{\sqrt{\lambda-\lambda_1}}.
\end{equation*}
Hence we have
\begin{equation*}
  (2\pi)^{-1}(\lambda-\lambda_1)^{1\slash 2}(\lambda_2-\lambda_1)-1\leq \overline{N}_0(\lambda;\lambda_1,\lambda_2)\leq  \pi^{-1}(\lambda-\lambda_1)^{1\slash 2}(\lambda_2-\lambda_1).
\end{equation*}
As $(\lambda-\lambda_1)^{1\slash 2}(\lambda_2-\lambda_1)\geq (\lambda_2-\lambda_1)^{3\slash 2}\geq 4\pi$, we have
\begin{equation}
 (4\pi)^{-1}(\lambda-\lambda_1)^{1\slash 2}(\lambda_2-\lambda_1)\leq \overline{N}_0(\lambda;\lambda_1,\lambda_2)\leq  \pi^{-1}(\lambda-\lambda_1)^{1\slash 2}(\lambda_2-\lambda_1).
\end{equation}
Thus we obtain (\ref{Eq3.8n}) by setting $C=\pi^{-1}$ and $c=(4\pi)^{-1}$.

\end{proof}

\subsection{Proof of Proposition \ref{P3.1}}\label{Sect.3.2}

In this subsection, we present the proof of Proposition \ref{P3.1}. We start with the case where $\lambda-\lambda_1\leq 2(\lambda-\lambda_2)$. Namely, we establish the following lemma.

\begin{lemma}\label{Lem3.3s}
There exist positive constants $K\geq 10,C,c$ that only depend on $\beta$, such that the following holds. For any $\lambda,\lambda_1,\lambda_2,x\in\mathbb{R}$ satisfying
\begin{equation*}
0\leq \lambda_1<\lambda_2<\lambda, \quad \min\{\lambda_2-\lambda_1,\lambda-\lambda_2,x\}\geq K, \quad\lambda-\lambda_1\leq 2(\lambda-\lambda_2),
\end{equation*}
there exists an event $\mathcal{U}\in\mathcal{F}_{\lambda_1,\lambda_2}$, such that 
\begin{equation}
 \{|N(\lambda;\lambda_1,\lambda_2)-N_0(\lambda;\lambda_1,\lambda_2)|\geq x\}\subseteq \mathcal{U},
\end{equation}
\begin{eqnarray}\label{Eq3.69}
&& \max\{\mathbb{P}(|N(\lambda;\lambda_1,\lambda_2)-N_0(\lambda;\lambda_1,\lambda_2)|\geq x),\nonumber\\
&& \quad\quad \mathbb{P}(|\overline{N}(\lambda;\lambda_1,\lambda_2)-\overline{N}_0(\lambda;\lambda_1,\lambda_2)|\geq x) \}\nonumber\\
&\leq& \mathbb{P}(\mathcal{U})\nonumber\\
&\leq&   \exp(C (\lambda-\lambda_1)^{1\slash 2}(\lambda_2-\lambda_1)\log\log(\lambda-\lambda_1))\nonumber\\
&& \times  \exp\Big(-c x^2\Big/\log\Big(\frac{\lambda-\lambda_1}{\lambda-\lambda_2}\Big)\Big).
\end{eqnarray}
\end{lemma}

\begin{proof}[Proof of Lemma \ref{Lem3.3s}]

We take $K\geq 1000$ sufficiently large (depending only on $\beta$), so that all the following estimates in this proof hold. We fix $\lambda,\lambda_1,\lambda_2\in\mathbb{R}$ such that $0\leq\lambda_1<\lambda_2<\lambda$, $\min\{\lambda_2-\lambda_1,\lambda-\lambda_2\}\geq K$, and $\lambda-\lambda_1\leq 2(\lambda-\lambda_2)$. By monotonicity, we have 
\begin{equation}\label{Eq3.5}
 |N(\lambda;\lambda_1,\lambda_2)-\overline{N}(\lambda;\lambda_1,\lambda_2)|\leq 1,\quad |N_0(\lambda;\lambda_1,\lambda_2)-\overline{N}_0(\lambda;\lambda_1,\lambda_2)|\leq 1.
\end{equation}

For any $x\geq 0$, we let $\mathcal{C}_1(x)$ be the event that $\overline{N}(\lambda;\lambda_1,\lambda_2)\geq \overline{N}_0(\lambda;\lambda_1,\lambda_2)+x$ and $\mathcal{C}_2(x)$ the event that $\overline{N}(\lambda;\lambda_1,\lambda_2)\leq \overline{N}_0(\lambda;\lambda_1,\lambda_2)-x$. By (\ref{Eq3.5}), for any $x\geq 2$, letting $\mathcal{U}=\{|\overline{N}(\lambda;\lambda_1,\lambda_2)-\overline{N}_0(\lambda;\lambda_1,\lambda_2)|\geq x-2\}\in\mathcal{F}_{\lambda_1,\lambda_2}$, we have
\begin{equation}
  \{|N(\lambda;\lambda_1,\lambda_2)-N_0(\lambda;\lambda_1,\lambda_2)|\geq x\}\subseteq \mathcal{U},
\end{equation}
\begin{eqnarray}\label{Eq3.68}
 && \max\{\mathbb{P}(|N(\lambda;\lambda_1,\lambda_2)-N_0(\lambda;\lambda_1,\lambda_2)|\geq x),\nonumber\\
&&  \quad\quad\mathbb{P}(|\overline{N}(\lambda;\lambda_1,\lambda_2)-\overline{N}_0(\lambda;\lambda_1,\lambda_2)|\geq x) \}\nonumber\\
 &\leq&  \mathbb{P}(\mathcal{U}) \leq \mathbb{P}(\mathcal{C}_1(x-2))+\mathbb{P}(\mathcal{C}_2(x-2)).
\end{eqnarray}


In the following, we bound $\mathbb{P}(\mathcal{C}_1(x))$ and $\mathbb{P}(\mathcal{C}_2(x))$ for any $x\geq K-2$. Let $\tau_0:=\lambda_1$ and $\tau_0':=\lambda_1$. For any $j\in\mathbb{N}^{*}$, if there exist at least $j$ blow-ups of $\overline{p}_{\lambda,\lambda_1}(x)$ in the interval $(\lambda_1,\infty)$, we let $\tau_j$ be the $j$th blow-up time of $\overline{p}_{\lambda,\lambda_1}(x)$ in $(\lambda_1,\infty)$; otherwise we set $\tau_j:=\infty$. For any $j\in \mathbb{N}^{*}$, if there exist at least $j$ blow-ups of $\overline{q}_{\lambda,\lambda_1}(x)$ in the interval $(\lambda_1,\infty)$, we let $\tau_j'$ be the $j$th blow-up time of $\overline{q}_{\lambda,\lambda_1}(x)$ in $(\lambda_1,\infty)$; otherwise we set $\tau_j':=\infty$. 

\bigskip

\paragraph{Step 1: Bounding $\mathbb{P}(\mathcal{C}_1(x))$ for $x\geq K-2$}

We consider the following two cases: $x\geq 10(\lambda-\lambda_1)^{1\slash 2} (\lambda_2-\lambda_1)$ or $K-2\leq x<10 (\lambda-\lambda_1)^{1\slash 2}(\lambda_2-\lambda_1)$.

\medskip

\subparagraph{Case 1: $x\geq 10(\lambda-\lambda_1)^{1\slash 2}(\lambda_2-\lambda_1)$}

We assume that the event $\mathcal{C}_1(x)$ holds in the following. Then $\overline{N}(\lambda;\lambda_1,\lambda_2)\geq \lceil x\rceil\geq 10(\lambda-\lambda_1)^{1\slash 2}(\lambda_2-\lambda_1)$ and $\tau_{\lceil x\rceil}\leq \lambda_2$.

Now we show that there exist at least $\lceil x\slash 3\rceil$ elements (denoted by $j$) from $\{0,1,\cdots,\lceil x\rceil-1\}$, such that $\tau_{j+1}-\tau_j\leq 2(\lambda_2-\lambda_1)\slash x$. Otherwise, there are at least $2x\slash 3$ of $j$'s from $\{0,1,\cdots,\lceil x\rceil-1\}$, such that $\tau_{j+1}-\tau_j\geq 2(\lambda_2-\lambda_1)\slash x$. Hence 
\begin{equation*}
 \lambda_2-\lambda_1\geq \sum_{j=0}^{\lceil x\rceil -1}(\tau_{j+1}-\tau_j)\geq \frac{2x}{3}\cdot\frac{2(\lambda_2-\lambda_1)}{x}= \frac{4}{3}(\lambda_2-\lambda_1),
\end{equation*} 
which leads to a contradiction.

For any $j\in \{0,1,\cdots,\lceil x\rceil-1\}$, we let
\begin{equation}
\mathcal{E}_j:=\{\tau_j\leq \lambda_2, \tau_{j+1}-\tau_j\leq 2(\lambda_2-\lambda_1)\slash x\}\in\mathcal{F}_{\tau_{j+1}}.
\end{equation}
For any $i_1<\cdots<i_{\lceil x\slash 3\rceil}$ such that $i_1,\cdots,i_{\lceil x\slash 3\rceil}\in \{0,1,\cdots,\lceil x\rceil-1\}$, we let
\begin{equation}
    \mathcal{D}_{i_1,\cdots,i_{\lceil x\slash 3\rceil}}:=\bigcap_{j\in\{i_1,\cdots,i_{\lceil x\slash 3\rceil}\}}\mathcal{E}_j.
\end{equation}
By the preceding discussion, we have
\begin{equation}\label{Eq3.8}
\mathcal{C}_1(x)\subseteq \bigcup_{\substack{i_1<\cdots<i_{\lceil x\slash 3\rceil}:\\i_1,\cdots, i_{\lceil x\slash 3\rceil}\in \{0,1,\cdots,\lceil x\rceil-1\}}}\mathcal{D}_{i_1,\cdots,i_{\lceil x\slash 3\rceil}}.
\end{equation}

Consider any $i_1<\cdots<i_{\lceil x\slash 3\rceil}$ such that $i_1,\cdots,i_{\lceil x\slash 3\rceil}\in \{0,1,\cdots,\lceil x\rceil-1\}$. We bound $\mathbb{P}(\mathcal{D}_{i_1,\cdots,i_{\lceil x\slash 3\rceil}})$ as follows. For any $j\in \{i_1,\cdots,i_{\lceil x\slash 3\rceil}\}$, by Lemma \ref{Lem3.1}, the strong Markov property, and the fact that $\lambda-\lambda_1\leq 2(\lambda-\lambda_2)$, we have
\begin{eqnarray*}
&& \mathbb{P}(\mathcal{E}_j|\mathcal{F}_{\tau_j})=
  \mathbbm{1}_{\tau_j\leq \lambda_2} \mathbb{P}(\tau_{j+1}-\tau_j\leq 2(\lambda_2-\lambda_1)\slash x|\mathcal{F}_{\tau_j} )\nonumber\\
  &\leq& \mathbbm{1}_{\tau_j\leq \lambda_2} C\exp(-c x^2(\lambda-\tau_j)^{1\slash 2} (\lambda_2-\lambda_1)^{-2})\nonumber\\
  &\leq& C\exp(-cx^2(\lambda-\lambda_2)^{1\slash 2} (\lambda_2-\lambda_1)^{-2}) 
  \leq C\exp(-cx^2(\lambda-\lambda_1)^{1\slash 2} (\lambda_2-\lambda_1)^{-2}),
\end{eqnarray*}
where in Lemma \ref{Lem3.1}, we replace $\lambda$ by $\lambda-\tau_j$ and take
\begin{equation*}
  s=\frac{\pi x}{5(\lambda_2-\lambda_1)\sqrt{\lambda-\tau_j}}\geq  \frac{\pi \cdot 10\sqrt{\lambda-\lambda_1}(\lambda_2-\lambda_1)}{5(\lambda_2-\lambda_1)\sqrt{\lambda-\lambda_1}}=  2\pi.
\end{equation*}
Therefore, we have
\begin{eqnarray}\label{Eq3.9}
&& \mathbb{P}(\mathcal{D}_{i_1,\cdots,i_{\lceil x\slash 3\rceil}}) =\mathbb{E}[\prod_{l=1}^{\lceil x\slash 3\rceil} \mathbbm{1}_{\mathcal{E}_{i_l}}]=\mathbb{E}[\prod_{l=1}^{\lceil x\slash 3\rceil-1}\mathbbm{1}_{\mathcal{E}_{i_l}}\mathbb{P}(\mathcal{E}_{i_{\lceil x\slash 3\rceil}}|\mathcal{F}_{\tau_{i_{\lceil x\slash 3\rceil}}})] \nonumber\\
&\leq& C\exp(-cx^2(\lambda-\lambda_1)^{1\slash 2} (\lambda_2-\lambda_1)^{-2})\mathbb{E}[\prod_{l=1}^{\lceil x\slash 3\rceil-1}\mathbbm{1}_{\mathcal{E}_{i_l}}]\leq\cdots\nonumber\\
&\leq& C^{\lceil x\slash 3\rceil} \exp(-c\lceil x\slash 3\rceil x^2 (\lambda-\lambda_1)^{1\slash 2} (\lambda_2-\lambda_1)^{-2})\nonumber\\
&\leq& C^{x+1} \exp(-c x^3(\lambda-\lambda_1)^{1\slash 2}(\lambda_2-\lambda_1)^{-2}).
\end{eqnarray}

By (\ref{Eq3.8}), (\ref{Eq3.9}), and the union bound, we have
\begin{eqnarray}\label{Eq3.10}
   \mathbb{P}(\mathcal{C}_1(x)) &\leq&   2^{\lceil x\rceil} C^{x+1} \exp(-c x^3(\lambda-\lambda_1)^{1\slash 2}(\lambda_2-\lambda_1)^{-2})\nonumber\\
   &\leq& C^{x+1} \exp(-c x^3(\lambda-\lambda_1)^{1\slash 2}(\lambda_2-\lambda_1)^{-2}).
\end{eqnarray}
As $x\geq 10(\lambda-\lambda_1)^{1\slash 2}(\lambda_2-\lambda_1)$, we have 
\begin{equation}\label{Eq3.11}
    x^2(\lambda-\lambda_1)^{1\slash 2}(\lambda_2-\lambda_1)^{-2}\geq  100 (\lambda-\lambda_1)^{3\slash 2}\geq 100 K^{3\slash 2}. 
\end{equation}
As we have taken $K$ to be sufficiently large, by (\ref{Eq3.10}) and (\ref{Eq3.11}), we have
\begin{equation}\label{Eq3.12}
   \mathbb{P}(\mathcal{C}_1(x))\leq C\exp(-c x^3(\lambda-\lambda_1)^{1\slash 2}(\lambda_2-\lambda_1)^{-2}).
\end{equation}
As $\lambda-\lambda_1\leq 2(\lambda-\lambda_2)$, we have $(\lambda_2-\lambda_1)\slash (\lambda-\lambda_2)\in [0,1]$. As $2\log(1+t)\geq  t  $ for any $t\in [0,1]$, we have
\begin{equation}\label{Eq3.21}
2\log\Big(\frac{\lambda-\lambda_1}{\lambda-\lambda_2}\Big)=2\log\Big(1+\frac{\lambda_2-\lambda_1}{\lambda-\lambda_2}\Big)\geq \frac{\lambda_2-\lambda_1}{\lambda-\lambda_2}\geq \frac{\lambda_2-\lambda_1}{\lambda-\lambda_1}.
\end{equation}
Now as $x\geq 10(\lambda-\lambda_1)^{1\slash 2}(\lambda_2-\lambda_1)$, we have 
\begin{equation}\label{Eq3.13}
  x^3(\lambda-\lambda_1)^{1\slash 2} (\lambda_2-\lambda_1)^{-2}\geq  10x^2\frac{\lambda-\lambda_1}{\lambda_2-\lambda_1}\geq 5x^2\Big/\log\Big(\frac{\lambda-\lambda_1}{\lambda-\lambda_2}\Big).
\end{equation}
By (\ref{Eq3.12}) and (\ref{Eq3.13}), we conclude that for any $x\geq 10(\lambda-\lambda_1)^{1\slash 2}(\lambda_2-\lambda_1)$,
\begin{equation}\label{Eq3.65}
 \mathbb{P}(\mathcal{C}_1(x))\leq C\exp\Big(-c x^2\Big/\log\Big(\frac{\lambda-\lambda_1}{\lambda-\lambda_2}\Big)\Big).
\end{equation}

\medskip

\subparagraph{Case 2: $K-2\leq x<10(\lambda-\lambda_1)^{1\slash 2}(\lambda_2-\lambda_1)$} We bound $\mathbb{P}(\mathcal{C}_1(x))$ through \textbf{Parts 1-5} as detailed below.

\textbf{Part 1}

Consider any $j\in\mathbb{N}$. By \cite[Lemma 2.1]{Zho} and the strong Markov property (taking $\delta=\epsilon$ and replacing $a$ by $\lambda-\tau_j$ in \cite[Lemma 2.1]{Zho}; note that we have taken $K$ to be sufficiently large), using the fact that $\lambda-\lambda_1\leq 2(\lambda-\lambda_2)$, we obtain that for any $\epsilon\in (0,1\slash 3)$,
\begin{eqnarray}\label{Eq3.14}
  && \mathbbm{1}_{\tau_j\leq \lambda_2}\mathbb{P}\Big(\Big\{\frac{\pi}{(1+\epsilon)\sqrt{\lambda-\tau_j}}\leq \tau_{j+1}-\tau_j\leq \frac{\pi}{\sqrt{(1-\epsilon)^2(\lambda-\tau_j)-2\pi\slash\sqrt{\lambda-\tau_j}}}\Big\}^c\Big|\mathcal{F}_{\tau_j}\Big)\nonumber\\
  &&\leq C_1\exp(-c_1' \epsilon^2 (\lambda-\tau_j)^{3\slash 2}) \mathbbm{1}_{\tau_j\leq \lambda_2}\leq C_1\exp(-c_1\epsilon^2 (\lambda-\lambda_1)^{3\slash 2}) ,
\end{eqnarray}
where $C_1,c_1,c_1'$ are positive constants that only depend on $\beta$. When $\tau_j\leq \lambda_2$, as $\lambda-\tau_j\geq \lambda-\lambda_2\geq K$ is sufficiently large, for any $\epsilon\in (0,1\slash 3)$,
\begin{eqnarray}\label{Eq3.15}
   && ((1-\epsilon)^2(\lambda-\tau_j)-2\pi\slash \sqrt{\lambda-\tau_j})^{-1\slash 2}\nonumber\\
   &=& (1-\epsilon)^{-1}(\lambda-\tau_j)^{-1\slash 2}(1-2\pi(1-\epsilon)^{-2}(\lambda-\tau_j)^{-3\slash 2})^{-1\slash 2}\nonumber\\
   &\leq& (1-\epsilon)^{-1}(\lambda-\tau_j)^{-1\slash 2} (1+C(\lambda-\tau_j)^{-3\slash 2})\nonumber\\
   &\leq& (1-\epsilon)^{-1}(\lambda-\tau_j)^{-1\slash 2} +C(\lambda-\lambda_1)^{-2},
\end{eqnarray}
where we use $(1-x)^{-1\slash 2}\leq 1+2x, \forall x\in [0,1\slash 2]$ for the first inequality and $\lambda-\lambda_1\leq 2(\lambda-\lambda_2)$ for the second inequality. By (\ref{Eq3.14}) and (\ref{Eq3.15}), for any $\epsilon\in (0,1\slash 3)$, we have
\begin{eqnarray}
&&  \mathbbm{1}_{\tau_j\leq \lambda_2} \mathbb{P}\Big(\Big\{\frac{\pi}{(1+\epsilon)\sqrt{\lambda-\tau_j}}\leq \tau_{j+1}-\tau_j\leq \frac{\pi}{(1-\epsilon)\sqrt{\lambda-\tau_j}}+\frac{C}{(\lambda-\lambda_1)^2}\Big\}^c\Big|\mathcal{F}_{\tau_j} \Big) \nonumber\\
&& \leq C_1\exp(-c_1 \epsilon^2 (\lambda-\lambda_1)^{3\slash 2}). 
\end{eqnarray}
Now note that for any $\epsilon\in (0,1\slash 3)$, when $\tau_j\leq\lambda_2$,
\begin{eqnarray*}
  && \max\Big\{\Big|\frac{\pi}{\sqrt{\lambda-\tau_j}}-\frac{\pi}{(1+\epsilon)\sqrt{\lambda-\tau_j}}\Big|,\Big|\frac{\pi}{\sqrt{\lambda-\tau_j}}-\frac{\pi}{(1-\epsilon)\sqrt{\lambda-\tau_j}}\Big|\Big\}\nonumber\\
  &&\leq \frac{2\epsilon\pi}{\sqrt{\lambda-\tau_j}}\leq C\epsilon(\lambda-\lambda_1)^{-1\slash 2}.
\end{eqnarray*}
Hence for any $\epsilon\in (0,1\slash 3)$, we have
\begin{eqnarray}\label{Eq3.16}
 && \mathbbm{1}_{\tau_j\leq \lambda_2}\mathbb{P}\Big(\Big|\tau_{j+1}-\tau_j-\frac{\pi}{\sqrt{\lambda-\tau_j}} \Big|> \frac{C_0\epsilon}{\sqrt{\lambda-\lambda_1}}+\frac{C_0}{(\lambda-\lambda_1)^2}\Big|\mathcal{F}_{\tau_j}\Big)\nonumber\\
 && \leq C_1\exp(-c_1  \epsilon^2 (\lambda-\lambda_1)^{3\slash 2}), 
\end{eqnarray}
where $C_0\geq 100$ is a positive constant that only depends on $\beta$. 

Again, consider an arbitrary $j\in\mathbb{N}$. By \cite[Lemma 2.2]{Zho} and the strong Markov property (replacing $a$ by $\lambda-\tau_j$ and taking $M=s(\lambda-\tau_j)\slash 8$), using the fact that $\lambda-\lambda_1\leq 2(\lambda-\lambda_2)$, we obtain that for any $s>0$ such that $50(\lambda-\lambda_1)^{-1\slash 2}\leq s \leq (\lambda-\lambda_1)\slash 50$, 
\begin{eqnarray}
  \mathbbm{1}_{\tau_j\leq \lambda_2}\mathbb{P}(\tau_{j+1}-\tau_j>s|\mathcal{F}_{\tau_j})&\leq& C_2\exp(-c_2' s(\lambda-\tau_j)^2) 1_{\tau_j\leq \lambda_2}\nonumber\\
  & \leq& C_2\exp(-c_2 s(\lambda-\lambda_1)^2),
\end{eqnarray}
where $C_2,c_2,c_2'$ are positive constants that only depend on $\beta$. When $\tau_j\leq\lambda_2$, we have $\pi\slash \sqrt{\lambda-\tau_j}\leq \pi\slash \sqrt{\lambda-\lambda_2}\leq \sqrt{2}\pi\slash \sqrt{\lambda-\lambda_1}<50 \slash \sqrt{\lambda-\lambda_1}$. Hence for any $s>0$ such that $50(\lambda-\lambda_1)^{-1\slash 2}\leq s \leq (\lambda-\lambda_1)\slash 50$, we have
\begin{equation}\label{Eq3.17}
  \mathbbm{1}_{\tau_j\leq \lambda_2}\mathbb{P}\Big(\Big|\tau_{j+1}-\tau_j-\frac{\pi}{\sqrt{\lambda-\tau_j}}\Big|\geq s\Big|\mathcal{F}_{\tau_j}\Big)\leq C_2\exp(-c_2 s(\lambda-\lambda_1)^2).
\end{equation}

As $\lambda-\lambda_1\geq K$ is sufficiently large, we assume that $C_0\leq (\lambda-\lambda_1)^3\slash 1000$ in the following. Let 
\begin{equation*}
  \mathcal{S}:=\{0\}\cup \{2^l C_0(\lambda-\lambda_1)^{-2}: l\in\mathbb{N}^{*}, 2^l C_0(\lambda-\lambda_1)^{-2}\leq (\lambda-\lambda_1)\slash 50\}. 
\end{equation*} 
For any $j\in\mathbb{N}$ and $s\in\mathcal{S}\backslash \{0\}$, consider the following two cases:
\begin{itemize}
  \item[(a)] If $2C_0(\lambda-\lambda_1)^{-2}\leq s\leq C_0(\lambda-\lambda_1)^{-1\slash 2}\slash 2$, we take $\epsilon=s\sqrt{\lambda-\lambda_1}\slash (4C_0)$ in (\ref{Eq3.16}) and obtain that
  \begin{equation}
     \mathbbm{1}_{\tau_j\leq \lambda_2}\mathbb{P}\Big(\Big|\tau_{j+1}-\tau_j-\frac{\pi}{\sqrt{\lambda-\tau_j}}\Big|\geq s\Big|\mathcal{F}_{\tau_j}\Big)\leq C_1\exp(-c_1 s^2 (\lambda-\lambda_1)^{5\slash 2}\slash (16C_0^2)).
  \end{equation}
  \item[(b)] If $C_0(\lambda-\lambda_1)^{-1\slash 2}\slash 2\leq s \leq (\lambda-\lambda_1)\slash 50$, by (\ref{Eq3.17}), we have
  \begin{equation}
     \mathbbm{1}_{\tau_j\leq \lambda_2}\mathbb{P}\Big(\Big|\tau_{j+1}-\tau_j-\frac{\pi}{\sqrt{\lambda-\tau_j}}\Big|\geq s\Big|\mathcal{F}_{\tau_j}\Big)\leq C_2\exp(-c_2 s(\lambda-\lambda_1)^2).
  \end{equation}
\end{itemize}
Therefore, for any $j\in\mathbb{N}$ and $s\in\mathcal{S}$,
\begin{equation}\label{Eq3.18}
  \mathbbm{1}_{\tau_j\leq \lambda_2} \mathbb{P}\Big(\Big|\tau_{j+1}-\tau_j-\frac{\pi}{\sqrt{\lambda-\tau_j}}\Big|\geq s\Big|\mathcal{F}_{\tau_j}\Big)\leq C\exp(-c \psi_{\lambda,\lambda_1}(s)),
\end{equation} 
where for any $s\geq 0$,
\begin{equation}\label{Eq3.40}
 \psi_{\lambda,\lambda_1}(s):=  \begin{cases}
                            s^2 (\lambda-\lambda_1)^{5\slash 2} & s\leq C_0(\lambda-\lambda_1)^{-1\slash 2}\slash 2 \\
                            s (\lambda-\lambda_1)^2 & s>C_0(\lambda-\lambda_1)^{-1\slash 2}\slash 2  \\
                            \end{cases}.
\end{equation}
For any $j\in \mathbb{N}$ and $s\in\mathcal{S}$, we let
\begin{equation}\label{Eq3.47}
 \overline{\mathcal{E}}_{j,s} := \begin{cases}
   \Big\{\tau_j\leq \lambda_2,\Big|\tau_{j+1}-\tau_j-\frac{\pi}{\sqrt{\lambda-\tau_j}}\Big|\leq  2C_0 (\lambda-\lambda_1)^{-2}\Big\} & s=0\\
       \Big \{  \tau_j\leq \lambda_2,s\leq \Big|\tau_{j+1}-\tau_j-\frac{\pi}{\sqrt{\lambda-\tau_j}}\Big|\leq 2s  \Big\} & s  \neq        0\\
 \end{cases}.
\end{equation}
Note that $\overline{\mathcal{E}}_{j,s}\in\mathcal{F}_{\tau_{j+1}}$. By (\ref{Eq3.18}), we have
\begin{equation}
  \mathbb{P}(\overline{\mathcal{E}}_{j,s}|\mathcal{F}_{\tau_j})\leq C\exp(-c \psi_{\lambda,\lambda_1}(s)).
\end{equation}
For any $m\in\mathbb{N}^{*}$, any $i_1,\cdots,i_m\in\mathbb{N}$ such that $i_1<\cdots<i_m$, and any $s_1,\cdots,s_m\in \mathcal{S}$, we let
\begin{equation}\label{Eq3.46}
   \overline{\mathcal{D}}_{i_1,\cdots, i_m;  s_1,\cdots, s_m} :=  \bigcap_{j=1}^m  \overline{\mathcal{E}}_{i_j,s_j}.
\end{equation}
Then we have
\begin{eqnarray}\label{Eq3.37}
\mathbb{P}(\overline{\mathcal{D}}_{i_1,\cdots,i_m;s_1\cdots,s_m})&=&\mathbb{E}\Big[\prod_{j=1}^m \mathbbm{1}_{\overline{\mathcal{E}}_{i_j,s_j}}\Big]=\mathbb{E}\Big[\prod_{j=1}^{m-1}\mathbbm{1}_{\overline{\mathcal{E}}_{i_j,s_j}} \mathbb{P}(\overline{\mathcal{E}}_{i_m,s_m}|\mathcal{F}_{\tau_{i_m}})\Big]\nonumber\\
&\leq &  C\exp(-c \psi_{\lambda,\lambda_1}(s_m)) \mathbb{E}\Big[\prod_{j=1}^{m-1}\mathbbm{1}_{\overline{\mathcal{E}}_{i_j,s_j}}\Big]\leq \cdots\nonumber\\
&\leq& C^m \exp\Big(-c\sum_{j=1}^m\psi_{\lambda,\lambda_1}(s_j)\Big).
\end{eqnarray}

\textbf{Part 2}

Now for any $j\in\mathbb{N}$, we let
\begin{equation}\label{Eq3.48}
 \mathcal{W}_j:=\Big\{\tau_j\leq \lambda_2, \Big|\tau_{j+1}-\tau_j-\frac{\pi}{\sqrt{\lambda-\tau_j}}\Big|\geq (\lambda-\lambda_1)\slash 50\Big\}.
\end{equation}
Note that by (\ref{Eq3.17}), we have 
\begin{equation}\label{Eq3.19}
  \mathbb{P}(\mathcal{W}_j)=\mathbb{E}[\mathbb{P}(\mathcal{W}_j|\mathcal{F}_{\tau_j})]\leq C\exp(-c (\lambda-\lambda_1)^3). 
\end{equation}
Let
\begin{equation}\label{Eq3.49}
  \mathcal{W}:= \bigcup_{j=0}^{\overline{N}_0(\lambda;\lambda_1,\lambda_2)}\mathcal{W}_j.
\end{equation}
By (\ref{Eq3.19}), the union bound, and Lemma \ref{Lem3.2}, 
\begin{eqnarray}
  && \mathbb{P}(\mathcal{W})\leq C(\overline{N}_0(\lambda;\lambda_1,\lambda_2)+1)\exp(-c (\lambda-\lambda_1)^3)\nonumber\\
   &\leq& C(\lambda-\lambda_1)^{3\slash 2}\exp(-c (\lambda-\lambda_1)^3)\leq C\exp(-c (\lambda-\lambda_1)^3).
\end{eqnarray}
As $x< 10(\lambda-\lambda_1)^{1\slash 2} (\lambda_2-\lambda_1)$, by (\ref{Eq3.21}), we have
\begin{equation*}
   (\lambda-\lambda_1)^3\geq \frac{x^2(\lambda-\lambda_1)^2}{100(\lambda_2-\lambda_1)^2}\geq \frac{x^2(\lambda-\lambda_1)}{100(\lambda_2-\lambda_1)}\geq \frac{1}{200}x^2\Big/ \log\Big(\frac{\lambda-\lambda_1}{\lambda-\lambda_2}\Big).
\end{equation*}
Hence we have
\begin{equation}\label{Eq3.38}
  \mathbb{P}(\mathcal{W}) \leq C\exp\Big(-cx^2\Big/\log\Big(\frac{\lambda-\lambda_1}{\lambda-\lambda_2}\Big)\Big).
\end{equation}

\textbf{Part 3}

For any $j\in\mathbb{N}$, we let
\begin{equation}
   \mathcal{W}'_j:=\Big\{\tau_j\leq \lambda_2, \tau_{j+1}-\tau_j<\frac{\pi}{4\sqrt{\lambda-\lambda_1}}\Big\}\in\mathcal{F}_{\tau_{j+1}}.
\end{equation}
By Lemma \ref{Lem3.1} and the strong Markov property (replacing $\lambda$ by $\lambda-\tau_j$ and taking $s=2$ in Lemma \ref{Lem3.1}), we have
\begin{eqnarray}
&& \mathbb{P}(\mathcal{W}'_j|\mathcal{F}_{\tau_j}) \leq \mathbbm{1}_{\tau_j\leq \lambda_2} \mathbb{P}\Big(\tau_{j+1}-\tau_j< \frac{\pi}{4\sqrt{\lambda-\tau_j}}\Big|\mathcal{F}_{\tau_j}\Big)\nonumber\\
&\leq & C\exp(-c(\lambda-\tau_j)^{3\slash 2})\mathbbm{1}_{\tau_j\leq \lambda_2}\leq C\exp(-c (\lambda-\lambda_1)^{3\slash 2}).
\end{eqnarray}
Hence for any $i_1,\cdots,i_{\lceil x\slash 2\rceil}\in \{\overline{N}_0(\lambda;\lambda_1,\lambda_2)+1,\cdots,\overline{N}_0(\lambda;\lambda_1,\lambda_2)+\lceil x\rceil -1\}$ such that $i_1<\cdots<i_{\lceil x\slash 2\rceil}$, we have
\begin{eqnarray}\label{Eq3.20}
  &&\mathbb{P}\Big(\bigcap_{j=1}^{\lceil x\slash 2\rceil} \mathcal{W}'_{i_j}   \Big)=\mathbb{E}\Big[\prod_{j=1}^{\lceil x\slash 2\rceil} \mathbbm{1}_{\mathcal{W}'_{i_j}}\Big]=\mathbb{E}\Big[\prod_{j=1}^{\lceil x\slash 2\rceil-1} \mathbbm{1}_{\mathcal{W}'_{i_j}}\mathbb{P}(\mathcal{W}'_{i_{\lceil x\slash 2\rceil}}|\mathcal{F}_{\tau_{i_{\lceil x\slash 2\rceil}}})\Big]\nonumber\\
 &\leq& C\exp(-c(\lambda-\lambda_1)^{3\slash 2}) \mathbb{E}\Big[\prod_{j=1}^{\lceil x\slash 2\rceil-1} \mathbbm{1}_{W'_{i_j}}\Big] \leq \cdots\nonumber\\
  &\leq& C^{\lceil x\slash 2\rceil} \exp(-c\lceil x\slash 2\rceil (\lambda-\lambda_1)^{3\slash 2})\leq C^x\exp(-c x(\lambda-\lambda_1)^{3\slash 2}).
\end{eqnarray}
Let
\begin{equation}
   \mathcal{W}':=\bigcup_{\substack{i_1<\cdots<i_{\lceil x\slash 2\rceil}: \\i_1,\cdots,i_{\lceil x\slash 2\rceil}\in \{\overline{N}_0(\lambda;\lambda_1,\lambda_2)+1,\cdots,\overline{N}_0(\lambda;\lambda_1,\lambda_2)+\lceil x\rceil -1\}}} \bigcap_{j=1}^{\lceil x\slash 2\rceil}\mathcal{W}'_{i_j}.
\end{equation}
By (\ref{Eq3.20}) and the union bound, as $\lambda-\lambda_1\geq K$ is sufficiently large,
\begin{equation}
 \mathbb{P}(\mathcal{W}')\leq 2^{\lceil x\rceil-1}  C^x\exp(-cx(\lambda-\lambda_1)^{3\slash 2})\leq C\exp(-cx(\lambda-\lambda_1)^{3\slash 2}).
\end{equation}
As $x< 10(\lambda-\lambda_1)^{1\slash 2} (\lambda_2-\lambda_1)$, by (\ref{Eq3.21}), we have
\begin{equation*}
   x(\lambda-\lambda_1)^{3\slash 2}\geq  \frac{x^2(\lambda-\lambda_1)}{10(\lambda_2-\lambda_1)}\geq \frac{1}{20}  x^2\Big/ \log\Big(\frac{\lambda-\lambda_1}{\lambda-\lambda_2}\Big).
\end{equation*}
Hence we have
\begin{equation}\label{Eq3.39}
\mathbb{P}(\mathcal{W}')\leq C\exp\Big(-cx^2\Big/\log\Big(\frac{\lambda-\lambda_1}{\lambda-\lambda_2}\Big)\Big).
\end{equation}

\textbf{Part 4}

Now consider an arbitrary choice of $\{s_j\}_{j\in\{0,1,\cdots,\overline{N}_0(\lambda;\lambda_1,\lambda_2)\}}$ such that $s_j\in \mathcal{S}$ for any $j\in\{0,1,\cdots,\overline{N}_0(\lambda;\lambda_1,\lambda_2)\}$. In the following, we assume that the event
\begin{equation*}
 \mathcal{W}^c \cap(\mathcal{W}')^c \cap \overline{\mathcal{D}}_{0,1,\cdots,\overline{N}_0(\lambda;\lambda_1,\lambda_2);s_0,s_1,\cdots,s_{\overline{N}_0(\lambda;\lambda_1,\lambda_2)}}\cap \mathcal{C}_1(x)
\end{equation*}
holds. As the event $\mathcal{C}_1(x)$ holds, for any $j\in \{0,1,\cdots,\overline{N}_0(\lambda;\lambda_1,\lambda_2)+\lceil x\rceil\}$, 
\begin{equation}\label{Eq3.22}
  \tau_j\leq \lambda_2.
\end{equation}

For any $j\in\{0,1,\cdots,\overline{N}_0(\lambda;\lambda_1,\lambda_2)+\lceil x\rceil-1\}$, we let $d_{j,1}:=\tau_{j+1}-\tau_j$ (note that this is well-defined due to (\ref{Eq3.22})). For any $j\in\{0,1,\cdots,\overline{N}_0(\lambda;\lambda_1,\lambda_2)\}$, we let $d_{j,2}:=\tau'_{j+1}-\tau_j'$. For any $j\in\{0,1,\cdots,\overline{N}_0(\lambda;\lambda_1,\lambda_2)\}$, we let $\tau_j''$ be the first blow-up time of $\tilde{q}_j(x),x\geq \tau_j$, which is defined by
\begin{equation}
  \tilde{q}_j'(x)=x-\lambda-\tilde{q}_j(x)^2, \quad \tilde{q}_j(\tau_j)=\infty,
\end{equation}
and let $d_{j,3}:=\tau_j''-\tau_j$.

In the following, we bound $\tau_{\overline{N}_0(\lambda;\lambda_1,\lambda_2)}-\tau'_{\overline{N}_0(\lambda;\lambda_1,\lambda_2)}$. Consider an arbitrary $j\in \{0,1,\cdots,\overline{N}_0(\lambda;\lambda_1,\lambda_2)\}$. By \cite[Lemma 2.3]{Zho} (replacing $a$ by $\lambda-\tau_j'$ in \cite[Lemma 2.3]{Zho}; note that $\lambda-\tau_j'\geq \lambda-\lambda_2\geq K$ is sufficiently large),
\begin{equation}
   \frac{\pi}{\sqrt{\lambda-\tau_j'}}\leq d_{j,2}\leq \frac{\pi}{\sqrt{\lambda-\tau_j'-2\pi\slash \sqrt{\lambda-\tau_j'}}}.
\end{equation}
By the facts that $(1-x)^{-1\slash 2}\leq 1+2x,\forall x\in [0,1\slash 2]$ and $\lambda-\lambda_1\leq 2(\lambda-\lambda_2)$, 
\begin{eqnarray*}
   && (\lambda-\tau_j'-2\pi\slash \sqrt{\lambda-\tau_j'})^{-1\slash 2}=(\lambda-\tau_j')^{-1\slash 2} (1-2\pi(\lambda-\tau_j')^{-3\slash 2})^{-1 \slash 2}     \nonumber  \\
   &\leq&    (\lambda-\tau_j')^{-1\slash 2} (1+4\pi (\lambda-\tau_j')^{-3\slash 2})\leq (\lambda-\tau_j')^{-1\slash 2}+C(\lambda-\lambda_1)^{-2}.
\end{eqnarray*}
Hence we have
\begin{equation}\label{Eq3.23}
   \frac{\pi}{\sqrt{\lambda-\tau_j'}}\leq d_{j,2}\leq \frac{\pi}{\sqrt{\lambda-\tau_j'}}+\frac{C}{(\lambda-\lambda_1)^2}.
\end{equation}
Similarly, by \cite[Lemma 2.3]{Zho}, we have
\begin{equation}\label{Eq3.25}
  \frac{\pi}{\sqrt{\lambda-\tau_j}}\leq d_{j,3}\leq \frac{\pi}{\sqrt{\lambda-\tau_j}}+\frac{C}{(\lambda-\lambda_1)^2}. 
\end{equation}
As the event $\overline{\mathcal{E}}_{j,s_j}$ holds, we have
\begin{equation}\label{Eq3.24}
\Big|d_{j,1}-\frac{\pi}{\sqrt{\lambda-\tau_j}}\Big|\leq 2s_j+\frac{2C_0}{(\lambda-\lambda_1)^2}. 
\end{equation}
By (\ref{Eq3.25}) and (\ref{Eq3.24}), we have
\begin{equation}\label{Eq3.26}
   |d_{j,1}-d_{j,3}|\leq 2s_j+C(\lambda-\lambda_1)^{-2}.
\end{equation}
As $\min\{\lambda-\tau_j,\lambda-\tau_j'\}\geq \lambda-\lambda_2\geq (\lambda-\lambda_1)\slash 2$, we have
\begin{eqnarray}
 \Big|\frac{1}{\sqrt{\lambda-\tau_j}}-\frac{1}{\sqrt{\lambda-\tau_j'}}\Big|&=&\frac{|\tau_j-\tau_j'|}{\sqrt{(\lambda-\tau_j)(\lambda-\tau_j')}(\sqrt{\lambda-\tau_j}+\sqrt{\lambda-\tau_j'})}\nonumber\\
 &\leq & \sqrt{2}(\lambda-\lambda_1)^{-3\slash 2}|\tau_j-\tau_j'|.
\end{eqnarray}
Hence by (\ref{Eq3.23}) and (\ref{Eq3.25}),
\begin{eqnarray}\label{Eq3.27}
   |d_{j,2}-d_{j,3}|&\leq& \Big|d_{j,2}-\frac{\pi}{\sqrt{\lambda-\tau_j'}}\Big|+\Big|\frac{\pi}{\sqrt{\lambda-\tau_j}}-\frac{\pi}{\sqrt{\lambda-\tau_j'}}\Big|+\Big|d_{j,3}-\frac{\pi}{\sqrt{\lambda-\tau_j}}\Big|\nonumber\\
   &\leq&   C(\lambda-\lambda_1)^{-3\slash 2}|\tau_j-\tau'_j|+C(\lambda-\lambda_1)^{-2}.
\end{eqnarray}

Consider any $j\in \{0,1,\cdots,\overline{N}_0(\lambda;\lambda_1,\lambda_2)+1\}$. Note that $\tau_j=\lambda_1+\sum_{l=0}^{j-1} d_{l,1}$ and $\tau_j'=\lambda_1+\sum_{l=0}^{j-1} d_{l,2}$. Hence by (\ref{Eq3.26}), (\ref{Eq3.27}), and Lemma \ref{Lem3.2}, we have
\begin{eqnarray}\label{Eq3.28}
 &&  |\tau_j-\tau_j'|   \leq     \sum_{l=0}^{j-1}|d_{l,1}-d_{l,2}|\leq \sum_{l=0}^{j-1}|d_{l,1}-d_{l,3}|+\sum_{l=0}^{j-1}|d_{l,2}-d_{l,3}|\nonumber\\
 &\leq& 2\sum_{l=0}^{j-1}s_l+C(\overline{N}_0(\lambda;\lambda_1,\lambda_2)+1)(\lambda-\lambda_1)^{-2}+\sum_{l=0}^{j-1}|d_{l,2}-d_{l,3}|\nonumber\\
 &\leq& C\Big(\sum_{l=0}^{j-1}s_l+\sum_{l=0}^{j-1}|d_{l,2}-d_{l,3}|+(\lambda-\lambda_1)^{-1\slash 2}\Big).
\end{eqnarray}
By (\ref{Eq3.27}) and (\ref{Eq3.28}), for any $j\in \{0,1,\cdots,\overline{N}_0(\lambda;\lambda_1,\lambda_2)\}$, we have
\begin{equation}\label{Eq3.29}
  |d_{j,2}-d_{j,3}|\leq C(\lambda-\lambda_1)^{-3\slash 2}\Big(\sum_{l=0}^{j-1}s_l+\sum_{l=0}^{j-1}|d_{l,2}-d_{l,3}|+(\lambda-\lambda_1)^{-1\slash 2}\Big).
\end{equation}

For any $j\in \{0,1,\cdots,\overline{N}_0(\lambda;\lambda_1,\lambda_2)\}$, let $D_j:=\sum_{l=0}^j|d_{l,2}-d_{l,3}|$. Let $D_{-1}:=0$. By (\ref{Eq3.29}), for any $j\in \{0,1,\cdots,\overline{N}_0(\lambda;\lambda_1,\lambda_2)\}$, we have
\begin{eqnarray}
  D_j-D_{j-1}\leq C(\lambda-\lambda_1)^{-3\slash 2}\Big(D_{j-1}+\sum_{l=0}^{\overline{N}_0(\lambda;\lambda_1,\lambda_2)-1}s_l+(\lambda-\lambda_1)^{-1\slash 2}\Big),
\end{eqnarray}
which leads to 
\begin{eqnarray}
&& D_j+\sum_{l=0}^{\overline{N}_0(\lambda;\lambda_1,\lambda_2)-1}s_l+(\lambda-\lambda_1)^{-1\slash 2}\nonumber\\
&\leq& (1+C(\lambda-\lambda_1)^{-3\slash 2})\Big(D_{j-1}+\sum_{l=0}^{\overline{N}_0(\lambda;\lambda_1,\lambda_2)-1}s_l+(\lambda-\lambda_1)^{-1\slash 2}\Big).
\end{eqnarray}
Hence
\begin{eqnarray}\label{Eq3.30}
 &&  D_{\overline{N}_0(\lambda;\lambda_1,\lambda_2)}\leq D_{\overline{N}_0(\lambda;\lambda_1,\lambda_2)}+\sum_{l=0}^{\overline{N}_0(\lambda;\lambda_1,\lambda_2)-1}s_l+(\lambda-\lambda_1)^{-1\slash 2}\nonumber\\
  &\leq& (1+C(\lambda-\lambda_1)^{-3\slash 2})^{\overline{N}_0(\lambda;\lambda_1,\lambda_2)+1}\Big(\sum_{l=0}^{\overline{N}_0(\lambda;\lambda_1,\lambda_2)-1}s_l+(\lambda-\lambda_1)^{-1\slash 2}\Big)\nonumber\\
  &\leq& C \Big(\sum_{l=0}^{\overline{N}_0(\lambda;\lambda_1,\lambda_2)-1}s_l+(\lambda-\lambda_1)^{-1\slash 2}\Big),
\end{eqnarray}
where for the third inequality, we note that by Lemma \ref{Lem3.2},
\begin{equation*}
 (1+C(\lambda-\lambda_1)^{-3\slash 2})^{\overline{N}_0(\lambda;\lambda_1,\lambda_2)+1}\leq  (1+C(\lambda-\lambda_1)^{-3\slash 2})^{C(\lambda-\lambda_1)^{3\slash 2}}\leq C.
\end{equation*}
By (\ref{Eq3.28}) and (\ref{Eq3.30}), we have
\begin{equation}
 |\tau_{\overline{N}_0(\lambda;\lambda_1,\lambda_2)+1}-\tau'_{\overline{N}_0(\lambda;\lambda_1,\lambda_2)+1}|\leq C \Big(\sum_{l=0}^{\overline{N}_0(\lambda;\lambda_1,\lambda_2)}s_l+(\lambda-\lambda_1)^{-1\slash 2}\Big).
\end{equation}
As $\tau'_{\overline{N}_0(\lambda;\lambda_1,\lambda_2)+1}\geq \lambda_2$, we have
\begin{equation}\label{Eq3.31}
   \tau_{\overline{N}_0(\lambda;\lambda_1,\lambda_2)+1}\geq \lambda_2-C \Big(\sum_{l=0}^{\overline{N}_0(\lambda;\lambda_1,\lambda_2)}s_l+(\lambda-\lambda_1)^{-1\slash 2}\Big).
\end{equation}

As the event $(\mathcal{W}')^c\cap\mathcal{C}_1(x)$ holds, at most $\lceil x\slash 2\rceil-1$ elements (denoted by $j$) from $\{\overline{N}_0(\lambda;\lambda_1,\lambda_2)+1,\cdots,\overline{N}_0(\lambda;\lambda_1,\lambda_2)+\lceil x\rceil-1\}$ satisfy $d_{j,1}<\pi\slash (4\sqrt{\lambda-\lambda_1})$. As $x\geq K-2$ is sufficiently large, at least $(\lceil x\rceil-1)-(\lceil x\slash 2\rceil-1)\geq x\slash 4$ elements (denoted by $j$) from $\{\overline{N}_0(\lambda;\lambda_1,\lambda_2)+1,\cdots,\overline{N}_0(\lambda;\lambda_1,\lambda_2)+\lceil x\rceil-1\}$ satisfy $d_{j,1}\geq \pi\slash (4\sqrt{\lambda-\lambda_1})$. Hence
\begin{eqnarray}\label{Eq3.32}
 && \tau_{\overline{N}_0(\lambda;\lambda_1,\lambda_2)+\lceil x\rceil}-\tau_{\overline{N}_0(\lambda;\lambda_1,\lambda_2)+1}=\sum_{j=\overline{N}_0(\lambda;\lambda_1,\lambda_2)+1}^{\overline{N}_0(\lambda;\lambda_1,\lambda_2)+\lceil x\rceil-1}d_{j,1}\nonumber\\
 &\geq& \frac{x}{4}\cdot \frac{\pi}{4\sqrt{\lambda-\lambda_1}}=\frac{\pi x}{16\sqrt{\lambda-\lambda_1}}.
\end{eqnarray}
By (\ref{Eq3.22}), (\ref{Eq3.31}), and (\ref{Eq3.32}), we have
\begin{equation}
  \lambda_2\geq \tau_{\overline{N}_0(\lambda;\lambda_1,\lambda_2)+\lceil x\rceil}\geq \lambda_2-C \Big(\sum_{l=0}^{\overline{N}_0(\lambda;\lambda_1,\lambda_2)}s_l+(\lambda-\lambda_1)^{-1\slash 2}\Big)+cx(\lambda-\lambda_1)^{-1\slash 2},
\end{equation}
which leads to 
\begin{equation}\label{Eq3.34}
\sum_{l=0}^{\overline{N}_0(\lambda;\lambda_1,\lambda_2)}s_l\geq (c_0 x-1)(\lambda-\lambda_1)^{-1\slash 2},
\end{equation}
where $c_0$ is a positive constant that only depends on $\beta$.
 
Therefore, we conclude that for any $\{s_j\}_{j\in\{0,1,\cdots,\overline{N}_0(\lambda;\lambda_1,\lambda_2)\}}$ such that $s_j\in \mathcal{S}$ for any $j\in\{0,1,\cdots,\overline{N}_0(\lambda;\lambda_1,\lambda_2)\}$ and 
\begin{equation}\label{Eq3.33}
\sum_{l=0}^{\overline{N}_0(\lambda;\lambda_1,\lambda_2)}s_l< (c_0 x-1)(\lambda-\lambda_1)^{-1\slash 2},
\end{equation}
we have
\begin{equation}\label{Eq3.36}
 \mathcal{W}^c \cap(\mathcal{W}')^c \cap \overline{\mathcal{D}}_{0,1,\cdots,\overline{N}_0(\lambda;\lambda_1,\lambda_2);s_0,s_1,\cdots,s_{\overline{N}_0(\lambda;\lambda_1,\lambda_2)}}\cap \mathcal{C}_1(x)=\emptyset.
\end{equation}

\textbf{Part 5}

Now consider any $j\in \{0,1,\cdots,\overline{N}_0(\lambda;\lambda_1,\lambda_2)\}$. Note that 
\begin{equation}\label{Eq3.35}
   \{\tau_j\leq \lambda_2\}\subseteq \mathcal{W}_j\bigcup \Big(\bigcup_{s_j\in\mathcal{S}}\overline{\mathcal{E}}_{j,s_j}  \Big)\subseteq \mathcal{W}\bigcup \Big(\bigcup_{s_j\in\mathcal{S}}\overline{\mathcal{E}}_{j,s_j}  \Big). 
\end{equation}
In the following, we denote $\mathbf{s}=(s_0,s_1,\cdots,s_{\overline{N}_0(\lambda;\lambda_1,\lambda_2)})$. Let $\mathcal{S}'$ be the set of $\mathbf{s}$ such that $s_j\in \mathcal{S}$ for any $j\in \{0,1,\cdots,\overline{N}_0(\lambda;\lambda_1,\lambda_2)\}$ and $\mathcal{S}''$ the set of $\mathbf{s}\in \mathcal{S}'$ such that (\ref{Eq3.34}) is satisfied. By (\ref{Eq3.35}), we have
\begin{eqnarray}
\mathcal{C}_1(x) &\subseteq  &\bigcap_{j=0}^{\overline{N}_0(\lambda;\lambda_1,\lambda_2)}\{\tau_j\leq \lambda_2\}\subseteq \Big(\bigcap_{j=0}^{\overline{N}_0(\lambda;\lambda_1,\lambda_2)}\bigcup_{s_j\in\mathcal{S}}\overline{\mathcal{E}}_{j,s_j}\Big)\bigcup \mathcal{W} \nonumber\\
   &\subseteq & \Big(\bigcup_{\mathbf{s}\in\mathcal{S}'}\overline{\mathcal{D}}_{0,1,\cdots,\overline{N}_0(\lambda;\lambda_1,\lambda_2);s_0,s_1,\cdots,s_{\overline{N}_0(\lambda;\lambda_1,\lambda_2)}}\Big)\bigcup\mathcal{W}.
\end{eqnarray}
Hence by (\ref{Eq3.36}),
\begin{eqnarray}
\mathcal{C}_1(x)&\subseteq& \Big(\bigcup_{\mathbf{s}\in\mathcal{S}'}\big(\mathcal{W}^c\cap(\mathcal{W}')^c\cap\overline{\mathcal{D}}_{0,1,\cdots,\overline{N}_0(\lambda;\lambda_1,\lambda_2);s_0,s_1,\cdots,s_{\overline{N}_0(\lambda;\lambda_1,\lambda_2)}}\cap\mathcal{C}_1(x)\big)\Big)\nonumber\\
&&\bigcup\mathcal{W} \bigcup \mathcal{W}'\nonumber\\
&\subseteq & \Big(\bigcup_{\mathbf{s}\in\mathcal{S}''}\big(\mathcal{W}^c\cap(\mathcal{W}')^c\cap\overline{\mathcal{D}}_{0,1,\cdots,\overline{N}_0(\lambda;\lambda_1,\lambda_2);s_0,s_1,\cdots,s_{\overline{N}_0(\lambda;\lambda_1,\lambda_2)}}\cap\mathcal{C}_1(x)\big)\Big)\nonumber\\
&&\bigcup\mathcal{W} \bigcup \mathcal{W}'\nonumber\\
&\subseteq& \Big(\bigcup_{\mathbf{s}\in\mathcal{S}''}\overline{\mathcal{D}}_{0,1,\cdots,\overline{N}_0(\lambda;\lambda_1,\lambda_2);s_0,s_1,\cdots,s_{\overline{N}_0(\lambda;\lambda_1,\lambda_2)}}\Big)\bigcup\mathcal{W} \bigcup \mathcal{W}'.
\end{eqnarray}
By (\ref{Eq3.37}), (\ref{Eq3.38}), (\ref{Eq3.39}), and the union bound, we have
\begin{eqnarray}\label{Eq3.43}
\mathbb{P}(\mathcal{C}_1(x)) &\leq& \sum_{\mathbf{s}\in \mathcal{S}''}\mathbb{P}(\overline{\mathcal{D}}_{0,1,\cdots,\overline{N}_0(\lambda;\lambda_1,\lambda_2);s_0,s_1,\cdots,s_{\overline{N}_0(\lambda;\lambda_1,\lambda_2)}})+\mathbb{P}(\mathcal{W})+\mathbb{P}(\mathcal{W}')\nonumber\\
&\leq&C^{\overline{N}_0(\lambda;\lambda_1,\lambda_2)+1}\sum_{\mathbf{s}\in\mathcal{S}''} \exp\Big(-c\sum_{j=0}^{\overline{N}_0(\lambda;\lambda_1,\lambda_2)}\psi_{\lambda,\lambda_1}(s_j)\Big)\nonumber\\
&&+C\exp\Big(-cx^2\Big/\log\Big(\frac{\lambda-\lambda_1}{\lambda-\lambda_2}\Big)\Big).
\end{eqnarray}

Note that 
\begin{equation}
 |\mathcal{S}|\leq \frac{\log((\lambda-\lambda_1)^3\slash (50C_0))}{\log(2)}+1\leq C\log(\lambda-\lambda_1).
\end{equation}
Hence by Lemma \ref{Lem3.2}, we have
\begin{eqnarray}\label{Eq3.59}
   |\mathcal{S}''|&\leq& |\mathcal{S}'|\leq |\mathcal{S}|^{\overline{N}_0(\lambda;\lambda_1,\lambda_2)+1}\leq (C\log(\lambda-\lambda_1))^{C(\lambda-\lambda_1)^{1\slash 2}(\lambda_2-\lambda_1)}\nonumber\\
 &\leq& \exp(C (\lambda-\lambda_1)^{1\slash 2}(\lambda_2-\lambda_1) \log\log(\lambda-\lambda_1)).
\end{eqnarray}
Now consider any $\mathbf{s}=(s_0,s_1,\cdots,s_{\overline{N}_0(\lambda;\lambda_1,\lambda_2)})\in\mathcal{S}''$. Recalling the definition of $\psi_{\lambda,\lambda_1}(s)$ from (\ref{Eq3.40}), we have
\begin{eqnarray}\label{Eq3.41}
\sum_{j=0}^{\overline{N}_0(\lambda;\lambda_1,\lambda_2)}\psi_{\lambda,\lambda_1}(s_j)&=&(\lambda-\lambda_1)^{5\slash 2}\sum_{\substack{j\in\{0,1,\cdots,\overline{N}_0(\lambda;\lambda_1,\lambda_2)\}:\\ s_j\leq C_0(\lambda-\lambda_1)^{-1\slash 2}\slash 2 }} s_j^2\nonumber\\
&&+(\lambda-\lambda_1)^2\sum_{\substack{j\in\{0,1,\cdots,\overline{N}_0(\lambda;\lambda_1,\lambda_2)\}:\\ s_j> C_0(\lambda-\lambda_1)^{-1\slash 2}\slash 2 }}s_j.
\end{eqnarray}
Consider the following two cases (noting (\ref{Eq3.34})):
\begin{itemize}
  \item[(a)] If $\sum\limits_{\substack{j\in\{0,1,\cdots,\overline{N}_0(\lambda;\lambda_1,\lambda_2)\}:\\ s_j\leq C_0(\lambda-\lambda_1)^{-1\slash 2}\slash 2 }} s_j\geq (c_0x-1)(\lambda-\lambda_1)^{-1\slash 2}\slash 2$, by the Cauchy-Schwarz inequality and Lemma \ref{Lem3.2}, noting that $x\geq K-2$ is sufficiently large, we obtain that
  \begin{eqnarray}\label{Eq3.42}
     \sum_{\substack{j\in\{0,1,\cdots,\overline{N}_0(\lambda;\lambda_1,\lambda_2)\}:\\ s_j\leq C_0(\lambda-\lambda_1)^{-1\slash 2}\slash 2 }} s_j^2&\geq& \Big(\sum\limits_{\substack{j\in\{0,1,\cdots,\overline{N}_0(\lambda;\lambda_1,\lambda_2)\}:\\ s_j\leq C_0(\lambda-\lambda_1)^{-1\slash 2}\slash 2 }} s_j  \Big)^2\Big\slash(\overline{N}_0(\lambda;\lambda_1,\lambda_2)+1)\nonumber\\
     &\geq&  \frac{cx^2}{(\lambda-\lambda_1)^{3\slash 2}(\lambda_2-\lambda_1)}.
  \end{eqnarray}
  By (\ref{Eq3.21}), (\ref{Eq3.41}) and (\ref{Eq3.42}), we have
  \begin{equation}\label{Eq3.44}
     \sum_{j=0}^{\overline{N}_0(\lambda;\lambda_1,\lambda_2)}\psi_{\lambda,\lambda_1}(s_j)\geq \frac{cx^2(\lambda-\lambda_1)}{\lambda_2-\lambda_1}\geq cx^2\Big/\log\Big(\frac{\lambda-\lambda_1}{\lambda-\lambda_2}\Big).
  \end{equation}
  \item[(b)] If $\sum\limits_{\substack{j\in\{0,1,\cdots,\overline{N}_0(\lambda;\lambda_1,\lambda_2)\}:\\ s_j> C_0(\lambda-\lambda_1)^{-1\slash 2}\slash 2 }} s_j\geq (c_0x-1)(\lambda-\lambda_1)^{-1\slash 2}\slash 2$, by (\ref{Eq3.21}), (\ref{Eq3.41}), and our assumption that $x<10(\lambda-\lambda_1)^{1\slash 2}(\lambda_2-\lambda_1)$, we have
  \begin{equation}\label{Eq3.45}
     \sum_{j=0}^{\overline{N}_0(\lambda;\lambda_1,\lambda_2)}\psi_{\lambda,\lambda_1}(s_j)\geq cx(\lambda-\lambda_1)^{3\slash 2}\geq \frac{cx^2(\lambda-\lambda_1)}{\lambda_2-\lambda_1}\geq cx^2\Big/\log\Big(\frac{\lambda-\lambda_1}{\lambda-\lambda_2}\Big).
  \end{equation}
\end{itemize}

By (\ref{Eq3.43}), (\ref{Eq3.44}), (\ref{Eq3.45}), and Lemma \ref{Lem3.2}, we conclude that for any $x\geq 0$ such that $K-2\leq x<10(\lambda-\lambda_1)^{1\slash 2}(\lambda_2-\lambda_1)$,
\begin{eqnarray}\label{Eq3.66}
 \mathbb{P}(\mathcal{C}_1(x))&\leq& \exp(C(\lambda-\lambda_1)^{1\slash 2}(\lambda_2-\lambda_1)\log\log(\lambda-\lambda_1))\nonumber\\
 &&\times \exp\Big(-cx^2\Big/\log\Big(\frac{\lambda-\lambda_1}{\lambda-\lambda_2}\Big)\Big).
\end{eqnarray}

\bigskip

\paragraph{Step 2: Bounding $\mathbb{P}(\mathcal{C}_2(x))$ for $x\geq K-2$}

If $x>\overline{N}_0(\lambda;\lambda_1,\lambda_2)$, then $\mathbb{P}(\mathcal{C}_2(x))=0$. Below we assume that $x\leq \overline{N}_0(\lambda;\lambda_1,\lambda_2)$. Note that by Lemma \ref{Lem3.2}, we have
\begin{equation}\label{Eq3.51}
  x\leq \overline{N}_0(\lambda;\lambda_1,\lambda_2)\leq C_0'(\lambda-\lambda_1)^{1\slash 2}(\lambda_2-\lambda_1),
\end{equation}
where $C_0'$ is a positive absolute constant. 

For any $j\in \{\lceil x\rceil, \cdots, \overline{N}_0(\lambda;\lambda_1,\lambda_2)\}$, we let 
\begin{equation}
  \mathcal{V}_j:=\{\overline{N}(\lambda;\lambda_1,\lambda_2)=\overline{N}_0(\lambda;\lambda_1,\lambda_2)-j\}.
\end{equation}
By the union bound, we have
\begin{equation}\label{Eq3.63}
   \mathbb{P}(\mathcal{C}_2(x))\leq \sum_{j=\lceil x\rceil}^{\overline{N}_0(\lambda;\lambda_1,\lambda_2)}\mathbb{P}(\mathcal{V}_j). 
\end{equation}

In the following, we fix an arbitrary $j\in \{\lceil x\rceil, \cdots, \overline{N}_0(\lambda;\lambda_1,\lambda_2)\}$. Recall the definitions of $\overline{\mathcal{E}}_{j,s}$, $\overline{\mathcal{D}}_{i_1,\cdots,i_m;s_1,\cdots,s_m}$, $\mathcal{W}_j$, and $\mathcal{W}$ from (\ref{Eq3.47}), (\ref{Eq3.46}), (\ref{Eq3.48}), and (\ref{Eq3.49}). Note that for any $l\in \{0,1,\cdots,\overline{N}_0(\lambda;\lambda_1,\lambda_2)-j\}$, 
\begin{equation}\label{Eq3.50}
   \{\tau_l\leq\lambda_2\}\subseteq \mathcal{W}_l\bigcup \Big(\bigcup_{s_l\in\mathcal{S}}\overline{\mathcal{E}}_{l,s_l}\Big)\subseteq \mathcal{W}\bigcup \Big(\bigcup_{s_l\in\mathcal{S}}\overline{\mathcal{E}}_{l,s_l}\Big). 
\end{equation}
In the following, we denote $\tilde{\mathbf{s}}=(s_0,s_1,\cdots,s_{\overline{N}_0(\lambda;\lambda_1,\lambda_2)-j})$. We let $\tilde{\mathcal{S}}'$ be the set of $\tilde{\mathbf{s}}$ such that $s_l\in\mathcal{S}$ for any $l\in\{0,1,\cdots,\overline{N}_0(\lambda;\lambda_1,\lambda_2)-j\}$. By (\ref{Eq3.50}), we have
\begin{eqnarray}\label{Eq3.56}
&& \mathcal{V}_j\subseteq \{\tau_{\overline{N}_0(\lambda;\lambda_1,\lambda_2)-j+1}>\lambda_2\}\bigcap\Big(\bigcap_{l=0}^{\overline{N}_0(\lambda;\lambda_1,\lambda_2)-j}\{\tau_l\leq \lambda_2\}\Big)\nonumber\\
&\subseteq& \{\tau_{\overline{N}_0(\lambda;\lambda_1,\lambda_2)-j+1}>\lambda_2\}\bigcap\Big(\mathcal{W}\bigcup\Big(\bigcap_{l=0}^{\overline{N}_0(\lambda;\lambda_1,\lambda_2)-j}\bigcup_{s_l\in\mathcal{S}}\overline{\mathcal{E}}_{l,s_l}\Big)\Big)\nonumber\\
&\subseteq&  \Big(\bigcup_{\tilde{\mathbf{s}}\in\tilde{\mathcal{S}}'}\big(\overline{\mathcal{D}}_{0,1,\cdots,\overline{N}_0(\lambda;\lambda_1,\lambda_2)-j;s_0,s_1,\cdots,s_{\overline{N}_0(\lambda;\lambda_1,\lambda_2)-j}}\cap \{\tau_{\overline{N}_0(\lambda;\lambda_1,\lambda_2)-j+1}>\lambda_2\}\big)  \Big)\nonumber\\
&& \bigcup  \mathcal{W}.
\end{eqnarray}

In the following, we consider any $\tilde{\mathbf{s}}=(s_0,s_1,\cdots,s_{\overline{N}_0(\lambda;\lambda_1,\lambda_2)-j})\in\tilde{\mathcal{S}}'$, and assume that the event 
\begin{equation*}
\overline{\mathcal{D}}_{0,1,\cdots,\overline{N}_0(\lambda;\lambda_1,\lambda_2)-j;s_0,s_1,\cdots,s_{\overline{N}_0(\lambda;\lambda_1,\lambda_2)-j}}\cap \{\tau_{\overline{N}_0(\lambda;\lambda_1,\lambda_2)-j+1}>\lambda_2\}
\end{equation*}
holds. Arguing as in \textbf{Part 4} of \textbf{Step 1, Case 2} (see below (\ref{Eq3.22})), we can deduce that for any $l\in \{0,1,\cdots,\overline{N}_0(\lambda;\lambda_1,\lambda_2)-j+1\}$,
\begin{equation}\label{Eq3.52}
  |\tau_l-\tau_l'|\leq C\Big(\sum_{i=0}^{\overline{N}_0(\lambda;\lambda_1,\lambda_2)-j}s_i+(\lambda-\lambda_1)^{-1\slash 2}\Big).
\end{equation}
For any $l\in \{0,1,\cdots,\overline{N}_0(\lambda;\lambda_1,\lambda_2)\}$, define $d_{l,2}:=\tau_{l+1}'-\tau_l'$. Recall from (\ref{Eq3.23}) that
\begin{equation}
   \frac{\pi}{\sqrt{\lambda-\tau_l'}}\leq d_{l,2}\leq \frac{\pi}{\sqrt{\lambda-\tau_l'}}+\frac{C}{(\lambda-\lambda_1)^2}.
\end{equation}
As $\tau_l'\geq \lambda_1$, we have
\begin{equation}
    d_{l,2}\geq \frac{\pi}{\sqrt{\lambda-\lambda_1}}.
\end{equation}
As $\tau'_{\overline{N}_0(\lambda;\lambda_1,\lambda_2)}\leq \lambda_2$, we have
\begin{eqnarray}
  && \tau'_{\overline{N}_0(\lambda;\lambda_1,\lambda_2)-j+1}= \tau'_{\overline{N}_0(\lambda;\lambda_1,\lambda_2)}-\sum_{l=\overline{N}_0(\lambda;\lambda_1,\lambda_2)-j+1}^{\overline{N}_0(\lambda;\lambda_1,\lambda_2)-1}d_{l,2}\nonumber\\
  &\leq& \lambda_2-(j-1) \frac{\pi}{\sqrt{\lambda-\lambda_1}}\leq \lambda_2-cx(\lambda-\lambda_1)^{-1\slash 2},
\end{eqnarray}
where we use the fact that $j\geq \lceil x\rceil\geq K-2$ is sufficiently large. Now note that $\tau_{\overline{N}_0(\lambda;\lambda_1,\lambda_2)-j+1}> \lambda_2$. Hence
\begin{eqnarray}\label{Eq3.53}
   |\tau_{\overline{N}_0(\lambda;\lambda_1,\lambda_2)-j+1}-\tau'_{\overline{N}_0(\lambda;\lambda_1,\lambda_2)-j+1}|&\geq& \tau_{\overline{N}_0(\lambda;\lambda_1,\lambda_2)-j+1}-\tau'_{\overline{N}_0(\lambda;\lambda_1,\lambda_2)-j+1}\nonumber\\
 &\geq& cx(\lambda-\lambda_1)^{-1\slash 2}.
\end{eqnarray}
Combining (\ref{Eq3.52}) and (\ref{Eq3.53}) and using the fact that $x\geq K-2$ is sufficiently large, we obtain that
\begin{equation}\label{Eq3.54}
   \sum_{l=0}^{\overline{N}_0(\lambda;\lambda_1,\lambda_2)-j}s_l\geq (cx-1)(\lambda-\lambda_1)^{-1\slash 2}\geq c_0' x(\lambda-\lambda_1)^{-1\slash 2},
\end{equation}
where $c_0'$ is a positive constant that only depends on $\beta$.

Let $\tilde{\mathcal{S}}''$ be the set of $\tilde{\mathbf{s}}=(s_0,s_1,\cdots,s_{\overline{N}_0(\lambda;\lambda_1,\lambda_2)-j})\in \tilde{\mathcal{S}}'$ such that (\ref{Eq3.54}) holds. By the preceding argument, for any $\tilde{\mathbf{s}}\in\tilde{\mathcal{S}}'\backslash \tilde{\mathcal{S}}''$, we have
\begin{equation}\label{Eq3.55}
 \overline{\mathcal{D}}_{0,1,\cdots,\overline{N}_0(\lambda;\lambda_1,\lambda_2)-j;s_0,s_1,\cdots,s_{\overline{N}_0(\lambda;\lambda_1,\lambda_2)-j}}\cap \{\tau_{\overline{N}_0(\lambda;\lambda_1,\lambda_2)-j+1}>\lambda_2\}=\emptyset.
\end{equation}
By (\ref{Eq3.56}) and (\ref{Eq3.55}), we have
\begin{eqnarray}\label{Eq3.56n}
 \mathcal{V}_j&\subseteq&  \Big(\bigcup_{\tilde{\mathbf{s}}\in\tilde{\mathcal{S}}''}\big(\overline{\mathcal{D}}_{0,1,\cdots,\overline{N}_0(\lambda;\lambda_1,\lambda_2)-j;s_0,s_1,\cdots,s_{\overline{N}_0(\lambda;\lambda_1,\lambda_2)-j}}\cap \{\tau_{\overline{N}_0(\lambda;\lambda_1,\lambda_2)-j+1}>\lambda_2\}\big)  \Big)\nonumber\\
&& \bigcup  \mathcal{W}\nonumber\\
&\subseteq& \Big(\bigcup_{\tilde{\mathbf{s}}\in\tilde{\mathcal{S}}''}\overline{\mathcal{D}}_{0,1,\cdots,\overline{N}_0(\lambda;\lambda_1,\lambda_2)-j;s_0,s_1,\cdots,s_{\overline{N}_0(\lambda;\lambda_1,\lambda_2)-j}}\Big)\bigcup  \mathcal{W}.
\end{eqnarray}
Hence by the union bound, we have
\begin{equation}\label{Eq3.57}
   \mathbb{P}(\mathcal{V}_j)\leq \mathbb{P}(\mathcal{W})+\sum_{\tilde{\mathbf{s}}\in\tilde{\mathcal{S}}''}\mathbb{P}(\overline{\mathcal{D}}_{0,1,\cdots,\overline{N}_0(\lambda;\lambda_1,\lambda_2)-j;s_0,s_1,\cdots,s_{\overline{N}_0(\lambda;\lambda_1,\lambda_2)-j}}).
\end{equation}
By (\ref{Eq3.51}), using an argument that is similar to that used in the proof of (\ref{Eq3.38}), we obtain that 
\begin{equation}\label{Eq3.58}
   \mathbb{P}(\mathcal{W})\leq  C\exp\Big(-cx^2\Big/\log\Big(\frac{\lambda-\lambda_1}{\lambda-\lambda_2}\Big)\Big).
\end{equation} 
By (\ref{Eq3.37}), (\ref{Eq3.57}), and (\ref{Eq3.58}), we obtain that 
\begin{eqnarray}\label{Eq3.60}
   \mathbb{P}(\mathcal{V}_j)&\leq& C^{\overline{N}_0(\lambda;\lambda_1,\lambda_2)}\sum_{\tilde{\mathbf{s}}\in\tilde{\mathcal{S}}''}\exp\Big(-c\sum_{l=0}^{\overline{N}_0(\lambda;\lambda_1,\lambda_2)-j}\psi_{\lambda,\lambda_1}(s_l)\Big)\nonumber\\
   && + C\exp\Big(-cx^2\Big/\log\Big(\frac{\lambda-\lambda_1}{\lambda-\lambda_2}\Big)\Big).
\end{eqnarray}
Using an argument that is similar to that used in the proof of (\ref{Eq3.59}), we can deduce that
\begin{equation}\label{Eq3.61}
 |\tilde{\mathcal{S}}''|\leq  \exp(C(\lambda-\lambda_1)^{1\slash 2} (\lambda_2-\lambda_1)\log\log(\lambda-\lambda_1)).
\end{equation}
Now consider any $\tilde{\mathbf{s}}=(s_0,s_1,\cdots,s_{\overline{N}_0(\lambda;\lambda_1,\lambda_2)-j})\in\tilde{\mathcal{S}}''$. Recalling the definition of $\psi_{\lambda,\lambda_1}(s)$ from (\ref{Eq3.40}), following the arguments used in the proofs of (\ref{Eq3.44}) and (\ref{Eq3.45}) (noting (\ref{Eq3.51}) and (\ref{Eq3.54})), we obtain that
\begin{equation}\label{Eq3.62}
   \sum_{l=0}^{\overline{N}_0(\lambda;\lambda_1,\lambda_2)-j}\psi_{\lambda,\lambda_1}(s_l)\geq  cx^2\Big/\log\Big(\frac{\lambda-\lambda_1}{\lambda-\lambda_2}\Big).
\end{equation}
By (\ref{Eq3.60}), (\ref{Eq3.61}), (\ref{Eq3.62}), and Lemma \ref{Lem3.2}, we have
\begin{eqnarray}\label{Eq3.64}
 \mathbb{P}(\mathcal{V}_j)&\leq& \exp(C(\lambda-\lambda_1)^{1\slash 2}(\lambda_2-\lambda_1)\log\log(\lambda-\lambda_1))\nonumber\\
 &&\times \exp\Big(-cx^2\Big/\log\Big(\frac{\lambda-\lambda_1}{\lambda-\lambda_2}\Big)\Big).
\end{eqnarray}
By (\ref{Eq3.63}), (\ref{Eq3.64}), and Lemma \ref{Lem3.2}, we conclude that for any $x\geq K-2$,
\begin{eqnarray}\label{Eq3.67}
 \mathbb{P}(\mathcal{C}_2(x))&\leq& \exp(C(\lambda-\lambda_1)^{1\slash 2}(\lambda_2-\lambda_1)\log\log(\lambda-\lambda_1))\nonumber\\
 &&\times \exp\Big(-cx^2\Big/\log\Big(\frac{\lambda-\lambda_1}{\lambda-\lambda_2}\Big)\Big).
\end{eqnarray}

\bigskip

By (\ref{Eq3.68}), (\ref{Eq3.65}), (\ref{Eq3.66}), and (\ref{Eq3.67}), we conclude that for any $x\geq K$, (\ref{Eq3.69}) holds. This finishes the proof of the proposition.




\end{proof}

Based on Lemma \ref{Lem3.3s}, we give the proof of Proposition \ref{P3.1} as follows.

\begin{proof}[Proof of Proposition \ref{P3.1}]
We take $K\geq 1000$ sufficiently large (depending on $\beta$), so that all the following estimates in this proof hold. We fix $\lambda,\lambda_1,\lambda_2\in\mathbb{R}$ such that $0\leq \lambda_1<\lambda_2<\lambda$, $\min\{\lambda_2-\lambda_1,\lambda-\lambda_2\}\geq K$. If $\lambda-\lambda_1\leq 2(\lambda-\lambda_2)$, the conclusion of the proposition follows from Lemma \ref{Lem3.3s}. In the rest of the proof, we assume that $\lambda-\lambda_1>2(\lambda-\lambda_2)$. 

Let 
\begin{equation}
 T_0:=\Big\lfloor 2\log\Big(\frac{\lambda-\lambda_1}{\lambda-\lambda_2}\Big)\Big/\log(2)\Big\rfloor\geq 2.
\end{equation}
For every $j\in \{0\}\cup [T_0-1]$, let $Z_j:=\lambda-2^{j\slash 2}(\lambda-\lambda_2)$. We also let $Z_{T_0}:=\lambda_1$. Note that $\lambda>Z_0>Z_1>\cdots>Z_{T_0}\geq 0$. For any $j\in \{0,1,\cdots,T_0-2\}$, we have $(\lambda-Z_{j+1})\slash (\lambda-Z_j)=\sqrt{2}$. Moreover, we have
\begin{equation*}
   \lambda-Z_{T_0-1}=2^{(T_0-1)\slash 2}(\lambda-\lambda_2)\geq (\lambda-\lambda_1)\slash 2=(\lambda-Z_{T_0})\slash 2.
\end{equation*}
Hence for any $j\in \{0\}\cup [T_0-1]$, we have
\begin{equation}\label{Eq3.71}
  \lambda-Z_{j+1}\leq 2(\lambda-Z_j), \quad 0\leq Z_{j+1}<Z_j<\lambda.
\end{equation}
We also note that for any $j\in \{0\}\cup [T_0-1]$, 
\begin{equation*}
  \lambda-Z_j\geq \lambda-\lambda_2;
\end{equation*}
for any $j\in \{0,1,\cdots,T_0-2\}$,
\begin{equation*}
    Z_j-Z_{j+1}=(\sqrt{2}-1)2^{j\slash 2}(\lambda-\lambda_2)\geq (\lambda-\lambda_2)\slash 4;
\end{equation*}
\begin{eqnarray*}
   Z_{T_0-1}-Z_{T_0}&=&\lambda-\lambda_1-2^{(T_0-1)\slash 2}(\lambda-\lambda_2)\geq (1-2^{-1\slash 2})(\lambda-\lambda_1)\nonumber\\
   &\geq& (2-\sqrt{2})(\lambda-\lambda_2).
\end{eqnarray*}
Hence for any $j\in \{0\}\cup [T_0-1]$,
\begin{equation}\label{Eq3.72}
   \min\{\lambda-Z_j,Z_j-Z_{j+1}\}\geq (\lambda-\lambda_2)\slash 4 \geq K\slash 4. 
\end{equation}
We also note that
\begin{equation}\label{Eq3.77}
  N(\lambda;\lambda_1,\lambda_2)=\sum_{j=0}^{T_0-1}N(\lambda;Z_{j+1},Z_j),\quad N_0(\lambda;\lambda_1,\lambda_2)=\sum_{j=0}^{T_0-1}N_0(\lambda;Z_{j+1},Z_j).
\end{equation}
By monotonicity, for any $j\in \{0\}\cup [T_0-1]$,
\begin{equation}\label{Eq3.78}
  |N(\lambda;Z_{j+1},Z_j)-\overline{N}(\lambda;Z_{j+1},Z_j)|\leq 1, \quad |N_0(\lambda;Z_{j+1},Z_j)-\overline{N}_0(\lambda;Z_{j+1},Z_j)|\leq 1.
\end{equation}

Let $\mathcal{S}:=\{0\}\cup \{2^t K\slash 40: t\in \mathbb{N}\}$. For any $j\in \{0\}\cup [T_0-1]$ and $s\in\mathcal{S}$, we let 
\begin{equation}\label{Eq3.47n}
 \mathcal{E}_{j,s} := \begin{cases}
   \{|\overline{N}(\lambda;Z_{j+1},Z_j)-\overline{N}_0(\lambda;Z_{j+1},Z_j)|\leq K\slash 40\} & s=0\\
    \{s\leq |\overline{N}(\lambda;Z_{j+1},Z_j)-\overline{N}_0(\lambda;Z_{j+1},Z_j)|\leq 2s\}    & s  \neq        0\\
 \end{cases}.
\end{equation} 
Note that
\begin{equation}\label{Eq3.70}
  \mathcal{E}_{j,s}\in \mathcal{F}_{Z_{j+1},Z_j}\subseteq \mathcal{F}_{\lambda_1,\lambda_2}. 
\end{equation}
By Lemma \ref{Lem3.3s} (noting (\ref{Eq3.71}), (\ref{Eq3.72}), and the fact that $K$ is sufficiently large), for any $j\in \{0\}\cup [T_0-1]$ and $s\in \mathcal{S}\backslash \{0\}$, 
\begin{eqnarray}\label{Eq3.74}
 &&  \mathbb{P}(\mathcal{E}_{j,s})\leq \mathbb{P}(|\overline{N}(\lambda;Z_{j+1},Z_j)-\overline{N}_0(\lambda;Z_{j+1},Z_j)|\geq s)\nonumber\\
 &\leq& \exp(C(\lambda-Z_{j+1})^{1\slash 2}(Z_j-Z_{j+1})\log\log(\lambda-Z_{j+1}))\nonumber\\
 && \times \exp\Big(-cs^2\Big/\log\Big(\frac{\lambda-Z_{j+1}}{\lambda-Z_j}\Big)\Big)\nonumber\\
 &\leq& \exp(C(\lambda-\lambda_1)^{1\slash 2}\log\log(\lambda-\lambda_1)(Z_j-Z_{j+1}))\nonumber\\
 && \times \exp\Big(-cs^2\Big/\log\Big(\frac{\lambda-Z_{j+1}}{\lambda-Z_j}\Big)\Big).
\end{eqnarray}

For any $s_0,s_1,\cdots, s_{T_0-1}\in \mathcal{S}$, let
\begin{equation}
    \mathcal{D}_{s_0,s_1,\cdots,s_{T_0-1}}:=\bigcap_{j=0}^{T_0-1} \mathcal{E}_{j,s_j} \in\mathcal{F}_{\lambda_1,\lambda_2}. 
\end{equation}
Let $\mathcal{S}'$ be the set of $(s_0,s_1,\cdots,s_{T_0-1})$ such that $s_j\in \mathcal{S}$ for every $j\in\{0\}\cup [T_0-1]$. Note that for any $j\in \{0\}\cup [T_0-1]$, $\bigcup\limits_{s\in\mathcal{S}}\mathcal{E}_{j,s}=\Omega$. Hence
\begin{equation}\label{Eq3.80}
   \bigcup_{(s_0,s_1,\cdots,s_{T_0-1})\in\mathcal{S}'} \mathcal{D}_{s_0,s_1,\cdots,s_{T_0-1}}=\Omega.
\end{equation}
By (\ref{Eq3.70}), the events $\{\mathcal{E}_{j,s_j}\}_{j=0}^{T_0-1}$ are mutually independent. Hence we have
\begin{equation}\label{Eq3.73}
  \mathbb{P}(\mathcal{D}_{s_0,s_1,\cdots,s_{T_0-1}})=\prod_{j=0}^{T_0-1} \mathbb{P}(\mathcal{E}_{j,s_j}). 
\end{equation}
By (\ref{Eq3.74}) and (\ref{Eq3.73}), we have
\begin{eqnarray}\label{Eq3.75}
  \mathbb{P}(\mathcal{D}_{s_0,s_1,\cdots,s_{T_0-1}})&\leq &\exp(C(\lambda-\lambda_1)^{1\slash 2}(\lambda_2-\lambda_1)\log\log(\lambda-\lambda_1))\nonumber\\
  && \times \exp\Big(-c\sum_{j=0}^{T_0-1} \Big(s_j^2\Big/\log\Big(\frac{\lambda-Z_{j+1}}{\lambda-Z_j}\Big)\Big)\Big).
\end{eqnarray}
By the Cauchy-Schwarz inequality, we have
\begin{equation*}
  \Big(\sum_{j=0}^{T_0-1} \Big(s_j^2\Big/\log\Big(\frac{\lambda-Z_{j+1}}{\lambda-Z_j}\Big)\Big)\Big)\Big(\sum_{j=0}^{T_0-1}\log\Big(\frac{\lambda-Z_{j+1}}{\lambda-Z_j}\Big)\Big)\geq \Big(\sum_{j=0}^{T_0-1}s_j\Big)^2.
\end{equation*}
As
\begin{equation*}
    \sum_{j=0}^{T_0-1}\log\Big(\frac{\lambda-Z_{j+1}}{\lambda-Z_j}\Big)=\log\Big(\frac{\lambda-Z_{T_0}}{\lambda-Z_0}\Big)=\log\Big(\frac{\lambda-\lambda_1}{\lambda-\lambda_2}\Big),
\end{equation*}
we have
\begin{equation}\label{Eq3.76}
   \sum_{j=0}^{T_0-1} \Big(s_j^2\Big/\log\Big(\frac{\lambda-Z_{j+1}}{\lambda-Z_j}\Big)\Big)\geq \Big(\sum_{j=0}^{T_0-1}s_j\Big)^2\Big/ \log\Big(\frac{\lambda-\lambda_1}{\lambda-\lambda_2}\Big).
\end{equation}
By (\ref{Eq3.75}) and (\ref{Eq3.76}), 
\begin{eqnarray}\label{Eq3.82}
  \mathbb{P}(\mathcal{D}_{s_0,s_1,\cdots,s_{T_0-1}})&\leq &\exp(C(\lambda-\lambda_1)^{1\slash 2}(\lambda_2-\lambda_1)\log\log(\lambda-\lambda_1))\nonumber\\
  && \times \exp\Big(-c\Big(\sum_{j=0}^{T_0-1}s_j\Big)^2\Big/ \log\Big(\frac{\lambda-\lambda_1}{\lambda-\lambda_2}\Big)\Big).
\end{eqnarray}

Consider any $s_0,s_1,\cdots, s_{T_0-1}\in \mathcal{S}$. In the following, we assume that the event $\{|N(\lambda;\lambda_1,\lambda_2)-N_0(\lambda;\lambda_1,\lambda_2)|\geq x\}\cap\mathcal{D}_{s_0,s_1,\cdots,s_{T_0-1}}$ holds. By (\ref{Eq3.77}) and (\ref{Eq3.78}), we have
\begin{eqnarray}
x&\leq& |N(\lambda;\lambda_1,\lambda_2)-N_0(\lambda;\lambda_1,\lambda_2)|\nonumber\\
&\leq& \sum_{j=0}^{T_0-1} |\overline{N}(\lambda;Z_{j+1},Z_j)-\overline{N}_0(\lambda;Z_{j+1},Z_j)|+2 T_0\nonumber\\
&\leq& \sum_{j=0}^{T_0-1}(2s_j+K\slash 40)+2T_0\leq 2\sum_{j=0}^{T_0-1}s_j+KT_0\slash 20. 
\end{eqnarray}
As
\begin{equation*}
   x\geq K\Big(1+\log\Big(\frac{\lambda-\lambda_1}{\lambda-\lambda_2}\Big)\Big)\geq KT_0\slash 4,
\end{equation*}
we have 
\begin{equation}\label{Eq3.79}
  \sum_{j=0}^{T_0-1}s_j\geq x\slash 4.
\end{equation}
Let $\mathcal{S}''$ be the set of $(s_0,s_1,\cdots,s_{T_0-1})\in\mathcal{S}'$ such that (\ref{Eq3.79}) is satisfied. By the preceding argument, for any $(s_0,s_1,\cdots,s_{T_0-1})\in\mathcal{S}'\backslash \mathcal{S}''$, we have
\begin{equation}\label{Eq3.81}
  \{|N(\lambda;\lambda_1,\lambda_2)-N_0(\lambda;\lambda_1,\lambda_2)|\geq x\}\cap\mathcal{D}_{s_0,s_1,\cdots,s_{T_0-1}}=\emptyset.
\end{equation}
By (\ref{Eq3.80}) and (\ref{Eq3.81}), we have
\begin{equation}\label{Eq3.83}
  \{|N(\lambda;\lambda_1,\lambda_2)-N_0(\lambda;\lambda_1,\lambda_2)|\geq x\}\subseteq \bigcup_{(s_0,s_1,\cdots,s_{T_0-1})\in\mathcal{S}''} \mathcal{D}_{s_0,s_1,\cdots,s_{T_0-1}}.
\end{equation}
For any $l\in \mathbb{N}$, let $\mathcal{S}''_l$ be the set of $(s_0,s_1,\cdots,s_{T_0-1})\in\mathcal{S}''$ such that $\sum_{j=0}^{T_0-1}s_j=Kl\slash 40$. Note that $\mathcal{S}''=\bigcup\limits_{l=0}^{\infty} \mathcal{S}_l''$. For any $l\in \mathbb{N}$ such that $l<10 x\slash K$, we have $\sum_{j=0}^{T_0-1}s_j<x\slash 4$, hence $\mathcal{S}''_l=\emptyset$. Hence letting
\begin{equation}
 \mathcal{U}=\bigcup\limits_{l=\lceil 10x\slash K  \rceil}^{\infty}\bigcup\limits_{(s_0,s_1,\cdots,s_{T_0-1})\in\mathcal{S}''_l} \mathcal{D}_{s_0,s_1,\cdots,s_{T_0-1}}\in\mathcal{F}_{\lambda_1,\lambda_2},
\end{equation}
by (\ref{Eq3.83}), we have
\begin{equation}\label{Eq3.84}
  \{|N(\lambda;\lambda_1,\lambda_2)-N_0(\lambda;\lambda_1,\lambda_2)|\geq x\}\subseteq \mathcal{U}.
\end{equation}

Now consider any $l\in\mathbb{N}$ such that $l\geq \lceil   10 x\slash K\rceil$ (as $x\geq K$, this implies $l\geq 10$). For any $(s_0,s_1,\cdots,s_{T_0-1})\in\mathcal{S}''_l$ and any $j\in \{0\}\cup [T_0-1]$, we have $s_j\leq Kl\slash 40$, hence $s_j\in\{0\} \cup \{2^t K\slash 40: t\in\mathbb{N}, t\leq \log_2(l)\}$. Hence 
\begin{equation}\label{Eq3.85}
  |\mathcal{S}''_l|  \leq (2+\log_2(l))^{T_0}\leq \exp\Big(C\log\log(l)\log\Big(\frac{\lambda-\lambda_1}{\lambda-\lambda_2}\Big)\Big).
\end{equation}
For any $(s_0,s_1,\cdots,s_{T_0-1})\in\mathcal{S}''_l$, by (\ref{Eq3.82}), we have
\begin{eqnarray}\label{Eq3.86}
  \mathbb{P}(\mathcal{D}_{s_0,s_1,\cdots,s_{T_0-1}})&\leq &\exp(C(\lambda-\lambda_1)^{1\slash 2}(\lambda_2-\lambda_1)\log\log(\lambda-\lambda_1))\nonumber\\
  && \times \exp\Big(-cK^2l^2\Big/ \log\Big(\frac{\lambda-\lambda_1}{\lambda-\lambda_2}\Big)\Big).
\end{eqnarray}
By (\ref{Eq3.84}), (\ref{Eq3.85}), (\ref{Eq3.86}), and the union bound, we have
\begin{eqnarray}\label{Eq3.87}
  &&\mathbb{P}(|N(\lambda;\lambda_1,\lambda_2)-N_0(\lambda;\lambda_1,\lambda_2)|\geq x)\leq \mathbb{P}(\mathcal{U})\nonumber\\
  &\leq& \sum_{l=\lceil 10x\slash K\rceil}^{\infty} \sum_{(s_0,s_1,\cdots,s_{T_0-1})\in\mathcal{S}''_l}  \mathbb{P}(\mathcal{D}_{s_0,s_1,\cdots,s_{T_0-1}})\nonumber\\
  &\leq& \exp(C(\lambda-\lambda_1)^{1\slash 2}(\lambda_2-\lambda_1)\log\log(\lambda-\lambda_1))\nonumber\\
  &\times& \sum_{l=\lceil 10x\slash K\rceil}^{\infty}   \exp\Big(-cK^2l^2\Big/ \log\Big(\frac{\lambda-\lambda_1}{\lambda-\lambda_2}\Big)+C\log\log(l)\log\Big(\frac{\lambda-\lambda_1}{\lambda-\lambda_2}\Big)\Big).\nonumber\\
  &&
\end{eqnarray}

We define $h(x):=x^2\slash \log\log(x)$ for any $x\geq 10$. For any $x\geq 10$,
\begin{equation*}
   h'(x)=\frac{2x\log\log(x)-x\slash \log(x)}{(\log\log(x))^2}\geq 0.
\end{equation*}
Hence $h(x)$ is non-decreasing on $[10,\infty)$. Below we consider any $l\in\mathbb{N}$ such that $l\geq \lceil 10x\slash K\rceil$. As
\begin{equation}\label{Eq3.92}
 x\geq K\Big(1+\log\Big(\frac{\lambda-\lambda_1}{\lambda-\lambda_2}\Big)\Big)\max\Big\{1,\log\log\Big(\frac{\lambda-\lambda_1}{\lambda-\lambda_2}\Big)\Big\},
\end{equation}
we have
\begin{equation}
  l\geq 10\Big(1+\log\Big(\frac{\lambda-\lambda_1}{\lambda-\lambda_2}\Big)\Big)\max\Big\{1,\log\log\Big(\frac{\lambda-\lambda_1}{\lambda-\lambda_2}\Big)\Big\}=:L_0.
\end{equation}
As $L_0\geq 10$, we have $h(l)\geq h(L_0)$. Note that
\begin{equation}\label{Eq3.88}
 L_0^2\geq 100 \Big(\log\Big(\frac{\lambda-\lambda_1}{\lambda-\lambda_2}\Big)\Big)^2\Big(\max\Big\{1,\log\log\Big(\frac{\lambda-\lambda_1}{\lambda-\lambda_2}\Big)\Big\}\Big)^2.
\end{equation}
Moreover, 
\begin{equation*}
  L_0\leq 20\Big(\max\Big\{1,\log\Big(\frac{\lambda-\lambda_1}{\lambda-\lambda_2}\Big)\Big\}\Big)^2,
\end{equation*}
\begin{equation}\label{Eq3.89}
   \log(L_0)\leq \log(20)+2\log\Big(\max\Big\{1,\log\Big(\frac{\lambda-\lambda_1}{\lambda-\lambda_2}\Big)\Big\}\Big).
\end{equation}
If $(\lambda-\lambda_1)\slash (\lambda-\lambda_2)\leq e^e$, by (\ref{Eq3.88}) and (\ref{Eq3.89}), we have 
\begin{equation*}
   L_0^2\geq 100 \Big(\log\Big(\frac{\lambda-\lambda_1}{\lambda-\lambda_2}\Big)\Big)^2, \quad \log\log(L_0)\leq 2,
\end{equation*}
hence 
\begin{equation}\label{Eq3.90}
  h(l)   \geq h(L_0)\geq 50\Big(\log\Big(\frac{\lambda-\lambda_1}{\lambda-\lambda_2}\Big)\Big)^2.
\end{equation}
If $(\lambda-\lambda_1)\slash (\lambda-\lambda_2)> e^e$, by (\ref{Eq3.88}) and (\ref{Eq3.89}), we have 
\begin{equation*}
   L_0^2\geq 100 \Big(\log\Big(\frac{\lambda-\lambda_1}{\lambda-\lambda_2}\Big)\Big)^2\Big(\log\log\Big(\frac{\lambda-\lambda_1}{\lambda-\lambda_2}\Big)\Big)^2,
\end{equation*}
\begin{equation*}
 \log(L_0)\leq \log(20)+2\log\log\Big(\frac{\lambda-\lambda_1}{\lambda-\lambda_2}\Big)\leq 6\log\log\Big(\frac{\lambda-\lambda_1}{\lambda-\lambda_2}\Big),
\end{equation*}
\begin{equation*}
 \log\log(L_0)\leq \log(6)+\log\log\log\Big(\frac{\lambda-\lambda_1}{\lambda-\lambda_2}\Big)\leq 3\log\log\Big(\frac{\lambda-\lambda_1}{\lambda-\lambda_2}\Big),
\end{equation*}
hence
\begin{equation}\label{Eq3.91}
  h(l)\geq h(L_0)\geq 20  \Big(\log\Big(\frac{\lambda-\lambda_1}{\lambda-\lambda_2}\Big)\Big)^2.
\end{equation}
By (\ref{Eq3.90}) and (\ref{Eq3.91}), as $K$ is sufficiently large, for any $l\in\mathbb{N}$ such that $l\geq \lceil 10x\slash K\rceil$, we have 
\begin{equation}
  c K^2l^2 \Big/\Big(2\log\Big(\frac{\lambda-\lambda_1}{\lambda-\lambda_2}\Big)\Big)\geq C\log\log(l)\log\Big(\frac{\lambda-\lambda_1}{\lambda-\lambda_2}\Big),
\end{equation}
where $C,c$ are the constants appearing in (\ref{Eq3.87}). Hence by (\ref{Eq3.87}),
\begin{eqnarray}\label{Eq3.94}
  &&\mathbb{P}(|N(\lambda;\lambda_1,\lambda_2)-N_0(\lambda;\lambda_1,\lambda_2)|\geq x)\leq\mathbb{P}(\mathcal{U})\nonumber\\
  &\leq&  \exp(C(\lambda-\lambda_1)^{1\slash 2}(\lambda_2-\lambda_1)\log\log(\lambda-\lambda_1))\nonumber\\
  &\times& \sum_{l=\lceil 10x\slash K\rceil}^{\infty}   \exp\Big(-cK^2l^2\Big/\Big(2\log\Big(\frac{\lambda-\lambda_1}{\lambda-\lambda_2}\Big)\Big)\Big).
\end{eqnarray}
Now note that
\begin{eqnarray}\label{Eq3.93}
&& \sum_{l=\lceil 10x\slash K\rceil}^{\infty}   \exp\Big(-cK^2l^2\Big/\Big(2\log\Big(\frac{\lambda-\lambda_1}{\lambda-\lambda_2}\Big)\Big)\Big)\nonumber\\
&\leq& \sum_{l=\lceil 10x\slash K \rceil  }^{\infty}\exp\Big(-cx K l\Big/\log\Big(\frac{\lambda-\lambda_1}{\lambda-\lambda_2}\Big)\Big)\nonumber\\
&\leq & \exp\Big(-cx^2\Big/\log\Big(\frac{\lambda-\lambda_1}{\lambda-\lambda_2}\Big)\Big)\Big/\Big(1-\exp\Big(-cxK\Big/\log\Big(\frac{\lambda-\lambda_1}{\lambda-\lambda_2}\Big)\Big)\Big).\nonumber\\
&&
\end{eqnarray}
By (\ref{Eq3.92}), as $K$ is sufficiently large, we have
\begin{equation}\label{Eq3.95}
 \exp\Big(-cxK\Big/\log\Big(\frac{\lambda-\lambda_1}{\lambda-\lambda_2}\Big)\Big)\leq \exp(-cK^2)\leq \frac{1}{2}.
\end{equation}
By (\ref{Eq3.94}), (\ref{Eq3.93}), and (\ref{Eq3.95}), we conclude that
\begin{eqnarray}
&& \mathbb{P}(|N(\lambda;\lambda_1,\lambda_2)-N_0(\lambda;\lambda_1,\lambda_2)|\geq x)\leq\mathbb{P}(\mathcal{U})\nonumber\\
&\leq& \exp(C (\lambda-\lambda_1)^{1\slash 2}(\lambda_2-\lambda_1)\log\log(\lambda-\lambda_1))\nonumber\\
&& \times  \exp\Big(-c x^2\Big/\log\Big(\frac{\lambda-\lambda_1}{\lambda-\lambda_2}\Big)\Big).
\end{eqnarray}

\end{proof}

\subsection{Proof of Proposition \ref{P3.2}}\label{Sect.3.3}

In this subsection, we give the proof of Proposition \ref{P3.2}.

\begin{proof}[Proof of Proposition \ref{P3.2}]
We take $K$ sufficiently large (depending only on $\beta$), so that all the following estimates in this proof hold. By monotonicity, we have
\begin{equation}
  |N(\lambda;\lambda_1,\lambda_2)-\overline{N}(\lambda;\lambda_1,\lambda_2)|\leq 1. 
\end{equation}
Let $\mathcal{U}=\{\overline{N}(\lambda;\lambda_1,\lambda_2)\geq x-1\}\in\mathcal{F}_{\lambda_1,\lambda_2}$. Then $\{N(\lambda;\lambda_1,\lambda_2)\geq x\}\subseteq \mathcal{U}$. As $x\geq 2$, we have
\begin{equation}\label{Eq3.119}
 \mathbb{P}(N(\lambda;\lambda_1,\lambda_2)\geq x)\leq \mathbb{P}(\mathcal{U})\leq \mathbb{P}(\overline{N}(\lambda;\lambda_1,\lambda_2)\geq x\slash 2).
\end{equation}

For any $j\in\mathbb{N}^{*}$, if there are at least $j$ blow-ups of $\overline{p}_{\lambda,\lambda_1}(x)$ in the interval $(\lambda_1,\infty)$, we let $\tau_j$ be the $j$th blow-up time of $\overline{p}_{\lambda,\lambda_1}(x)$ in $(\lambda_1,\infty)$; otherwise we let $\tau_j:=\infty$. We also let $\tau_0:=\lambda_1$.

For any $j\in\mathbb{N}$, by Corollary \ref{tail} and the strong Markov property (replacing $t$ by $-(\lambda-\tau_j)$ in Corollary \ref{tail}; note that $-(\lambda-\tau_j)\geq \lambda_1-\lambda\geq K$), we have
\begin{equation}
  \mathbbm{1}_{\tau_j<\infty} \mathbb{P}(\tau_{j+1}<\infty|\mathcal{F}_{\tau_j})\leq C\exp(-c (\lambda_1-\lambda)^{3\slash 2}). 
\end{equation}
Therefore, as $K$ is sufficiently large, noting (\ref{Eq3.119}), we have
\begin{eqnarray}\label{Eq3.132}
&& \mathbb{P}(N(\lambda;\lambda_1,\lambda_2)\geq x)\leq \mathbb{P}(\mathcal{U})\nonumber\\
&\leq& \mathbb{P}(\overline{N}(\lambda;\lambda_1,\lambda_2)\geq x\slash 2)
\leq\mathbb{P}(\tau_{\lceil x\slash 2\rceil }<\infty)\nonumber\\
&=& \mathbb{E}\Big[\prod_{j=1}^{\lceil x\slash 2 \rceil} \mathbbm{1}_{\tau_j<\infty}\Big]=\mathbb{E}\Big[\prod_{j=1}^{\lceil x\slash 2\rceil -1}\mathbbm{1}_{\tau_j<\infty}\mathbb{P}(\tau_{\lceil x\slash 2\rceil }<\infty|\mathcal{F}_{\tau_{\lceil x\slash 2\rceil -1}})\Big]\nonumber\\
&\leq& C\exp(-c(\lambda_1-\lambda)^{3\slash 2}) \mathbb{E}\Big[\prod_{j=1}^{\lceil x\slash 2 \rceil -1}\mathbbm{1}_{\tau_j<\infty}\Big]\leq \cdots\nonumber\\
&\leq& C^{\lceil x\slash 2\rceil}\exp(-c(\lambda_1-\lambda)^{3\slash 2}\lceil x\slash 2\rceil)\leq C \exp(-c(\lambda_1-\lambda)^{3\slash 2} x).\nonumber\\
&& 
\end{eqnarray}

\end{proof}

\subsection{Proof of Proposition \ref{P3.3}}\label{Sect.3.4}

In this subsection, we give the proof of Proposition \ref{P3.3}.

\begin{proof}[Proof of Proposition \ref{P3.3}]

Throughout this proof, we use $C,c$ to denote positive constants that only depend on $\beta$ and $L$. The values of these constants may change from line to line.

We take $K\geq 10$ sufficiently large (depending only on $\beta$ and $L$), so that all the following estimates in this proof hold. We also fix any $\lambda,\lambda_1,\lambda_2,x\in\mathbb{R}$ such that the conditions stated in the theorem hold. By monotonicity, we have
\begin{equation*}
 |N(\lambda;\lambda_1,\lambda_2)-\overline{N}(\lambda;\lambda_1,\lambda_2)|\leq 1,\quad |N_0(\lambda;\lambda_1,\lambda_2)-\overline{N}_0(\lambda;\lambda_1,\lambda_2)|\leq 1.
\end{equation*}
Hence we have
\begin{equation}\label{Eq3.120}
\{|N(\lambda;\lambda_1,\lambda_2)-N_0(\lambda;\lambda_1,\lambda_2)|\geq x\} \subseteq \{|\overline{N}(\lambda;\lambda_1,\lambda_2)-\overline{N}_0(\lambda;\lambda_1,\lambda_2)|\geq x-2\},
\end{equation}
\begin{eqnarray}\label{Eq3.98}
 && \mathbb{P}(|N(\lambda;\lambda_1,\lambda_2)-N_0(\lambda;\lambda_1,\lambda_2)|\geq x) \nonumber\\
 &\leq & \mathbb{P}(|\overline{N}(\lambda;\lambda_1,\lambda_2)-\overline{N}_0(\lambda;\lambda_1,\lambda_2)|\geq x-2).
\end{eqnarray}

Let $t_0$ be the number of blow-ups of $\overline{q}_{\lambda,\lambda_1}(x)$ in the interval $(\lambda_1,\lambda]$. Below we show that
\begin{equation}\label{Eq3.96}
   t_0\leq \overline{N}_0(\lambda;\lambda_1,\lambda_2) \leq t_0+1. 
\end{equation}
The first inequality in (\ref{Eq3.96}) is obvious, and we derive the second inequality in (\ref{Eq3.96}) as follows. Suppose that $\overline{N}_0(\lambda;\lambda_1,\lambda_2)\geq t_0+2$. Then there exists some $x_0> \lambda$, such that the function $r(x), x\geq x_0$ defined by
\begin{equation*}
    r'(x)=x-\lambda-r(x)^2, \quad r(x_0)=\infty
\end{equation*}
blows up in finite time. Note that for $x\geq x_0$, $r(x)$ is lower bounded by $w(x)$, defined by
\begin{equation*}
 w'(x)=x_0-\lambda-w(x)^2, \quad w(x_0)=\infty. 
\end{equation*}
Solving the above equation, we obtain that for any $x\geq x_0$,
\begin{equation*}
  w(x)=\frac{\sqrt{x_0-\lambda}(1+e^{-2\sqrt{x_0-\lambda}(x-x_0)})}{1-e^{-2\sqrt{x_0-\lambda}(x-x_0)}},
\end{equation*}
which does not blow up in finite time. This leads to a contradiction, and we conclude that $\overline{N}_0(\lambda;\lambda_1,\lambda_2) \leq t_0+1$.

We bound $t_0$ as follows. Let $\tau_0':=\lambda_1$. For any $j\in\{1,\cdots,t_0\}$, let $\tau_j'$ be the $j$th blow-up time of $\overline{q}_{\lambda,\lambda_1}(x)$ in the interval $(\lambda_1,\lambda]$. For any $j\in\{0,1,\cdots,t_0-1\}$, we let $r_j(x)=\overline{q}_{\lambda,\lambda_1}(x+\tau_j')$ for any $x\geq 0$. Note that   
\begin{equation}
  r_j'(x)=x-(\lambda-\tau_j')-r_j(x)^2, \quad  r_j(0)=\infty.
\end{equation}
As $\tau_j'\in [\lambda_1,\lambda)$, we have $\lambda-\tau_j'\in (0,\lambda-\lambda_1]$. By \cite[Lemma 2.3]{Zho}, 
\begin{equation*}
  \tau_{j+1}'-\tau_j'\geq   \frac{\pi}{\sqrt{\lambda-\tau_j'}} \geq \frac{\pi}{\sqrt{\lambda-\lambda_1}}.
\end{equation*}
Hence 
\begin{equation*}
  \lambda-\lambda_1\geq \tau_{t_0}'-\tau_0'=\sum_{j=0}^{t_0-1}(\tau_{j+1}'-\tau_j')\geq \frac{\pi t_0}{\sqrt{\lambda-\lambda_1}},
\end{equation*}
which leads to
\begin{equation}\label{Eq3.97}
   t_0\leq \pi^{-1} (\lambda-\lambda_1)^{3\slash 2}. 
\end{equation}

Let $\tau_0:=\lambda_1$. For any $j\in\mathbb{N}^{*}$, if there exist at least $j$ blow-ups of $\overline{p}_{\lambda,\lambda_1}(x)$ in the interval $(\lambda_1,\infty)$, we let $\tau_j$ be the $j$th blow-up time of $\overline{p}_{\lambda,\lambda_1}(x)$ in $(\lambda_1,\infty)$; otherwise we set $\tau_j:=\infty$. 

We consider the two cases $x\geq 100 L^2 (\lambda-\lambda_1)^{3\slash 2}$ and $x< 100 L^2 (\lambda-\lambda_1)^{3\slash 2}$ separately in the following. 

\paragraph{Case 1: $x\geq 100 L^2 (\lambda-\lambda_1)^{3\slash 2}$} 

For this case, we let
\begin{equation}
\mathcal{U}=\{|\overline{N}(\lambda;\lambda_1,\lambda_2)-\overline{N}_0(\lambda;\lambda_1,\lambda_2)|\geq x-2\}\in\mathcal{F}_{\lambda_1,\lambda_2}.
\end{equation}
Note that by (\ref{Eq3.120}), we have
\begin{equation}\label{Eq3.125}
\{|N(\lambda;\lambda_1,\lambda_2)-N_0(\lambda;\lambda_1,\lambda_2)|\geq x\} \subseteq \mathcal{U}.
\end{equation}
By (\ref{Eq3.96}) and (\ref{Eq3.97}), as $\lambda-\lambda_1\geq K$ is sufficiently large, we have
\begin{eqnarray*}
  && \overline{N}_0(\lambda;\lambda_1,\lambda_2)-(x-2)\leq t_0+1-(x-2)\nonumber\\
  &\leq& \pi^{-1}(\lambda-\lambda_1)^{3\slash 2}+3-100 (\lambda-\lambda_1)^{3\slash 2} <0.
\end{eqnarray*}
Hence $\overline{N}(\lambda;\lambda_1,\lambda_2)\geq 0>\overline{N}_0(\lambda;\lambda_1,\lambda_2)-(x-2)$. Combining this with (\ref{Eq3.98}), as $x\geq K\log(2)^2$ is sufficiently large, we obtain that
\begin{eqnarray}\label{Eq3.106}
 && \mathbb{P}(|N(\lambda;\lambda_1,\lambda_2)-N_0(\lambda;\lambda_1,\lambda_2)|\geq x)\leq \mathbb{P}(\mathcal{U})\nonumber\\
 &\leq& \mathbb{P}(\overline{N}(\lambda;\lambda_1,\lambda_2)\geq \overline{N}_0(\lambda;\lambda_1,\lambda_2)+x-2)\nonumber\\
&\leq& \mathbb{P}(\overline{N}(\lambda;\lambda_1,\lambda_2)\geq x\slash 2).
\end{eqnarray}

In the following, we assume that $\overline{N}(\lambda;\lambda_1,\lambda_2)\geq x\slash 2$. This implies that $\tau_{\lceil x\slash 2\rceil}\leq \lambda_2$. Suppose that there exist at least $\lceil x\slash 6 \rceil$ elements (denoted by $j$) from $\{0,1,\cdots,\lceil x\slash 2\rceil-1\}$, such that $\tau_{j+1}-\tau_j\geq 8(\lambda_2-\lambda_1)\slash x$. Then we have
\begin{equation*}
   \lambda_2-\lambda_1\geq \tau_{\lceil x\slash 2\rceil}-\tau_0=\sum_{j=0}^{\lceil x\slash 2\rceil-1} (\tau_{j+1}-\tau_j)\geq \Big\lceil\frac{x}{6}\Big\rceil \frac{8(\lambda_2-\lambda_1)}{x} \geq \frac{4}{3}(\lambda_2-\lambda_1),
\end{equation*}
which leads to a contradiction. Hence there are at most $\lceil x\slash 6\rceil -1$ elements (denoted by $j$) from $\{0,1,\cdots,\lceil x\slash 2\rceil-1\}$, such that $\tau_{j+1}-\tau_j\geq 8(\lambda_2-\lambda_1)\slash x$. From this, we obtain that there are at least $\lceil x\slash 2\rceil-(\lceil x\slash 6\rceil-1)\geq x\slash 3$ elements (denoted by $j$) from $\{0,1,\cdots,\lceil x\slash 2\rceil-1\}$, such that 
\begin{equation*}
 \tau_{j+1}-\tau_j< \frac{8(\lambda_2-\lambda_1)}{x} \leq \frac{8(L+1)(\lambda-\lambda_1)}{x},
\end{equation*}
where we use the assumption that $\lambda_2-\lambda\leq L(\lambda-\lambda_1)$.

For any $j\in\{0,1,\cdots,\lceil x\slash 2\rceil-1\}$, we let
\begin{equation}
   \mathcal{E}_j:=\{\tau_j\leq \lambda_2, \tau_{j+1}-\tau_j\leq 8(L+1)(\lambda-\lambda_1)\slash x\}\in\mathcal{F}_{\tau_{j+1}}.
\end{equation}
For any $i_1<\cdots<i_{\lceil x\slash 3\rceil}$ such that $i_1,\cdots,i_{\lceil x\slash 3\rceil}\in\{0,1,\cdots,\lceil x\slash 2\rceil-1\}$, let
\begin{equation}
  \mathcal{D}_{i_1,\cdots,i_{\lceil x\slash 3\rceil}} :=\bigcap_{j\in \{i_1,\cdots,i_{\lceil x\slash 3\rceil}\}}\mathcal{E}_j.
\end{equation}
By the preceding discussion, we have
\begin{equation}\label{Eq3.99}
 \{\overline{N}(\lambda;\lambda_1,\lambda_2)\geq x\slash 2\}  \subseteq \bigcup_{\substack{i_1<\cdots<i_{\lceil x\slash 3\rceil}:\\ i_1,\cdots,i_{\lceil x\slash 3\rceil}\in \{0,1,\cdots,\lceil x\slash 2\rceil-1\} }}\mathcal{D}_{i_1,\cdots,i_{\lceil x\slash 3\rceil}}.
\end{equation}

Consider any $i_1<\cdots<i_{\lceil x\slash 3\rceil}$ such that $i_1,\cdots,i_{\lceil x\slash 3\rceil}\in\{0,1,\cdots,\lceil x\slash 2\rceil -1\}$. For any $j\in\{i_1,\cdots,i_{\lceil x\slash 3\rceil}\}$, if $\tau_j\leq \lambda_2$, then $\lambda-\tau_j\leq \lambda-\lambda_1$. By Lemma \ref{Lem3.3} and the strong Markov property (replacing $\lambda$ by $\lambda-\tau_j$ and taking $a=\lambda-\lambda_1$ and $s=(\pi\slash 20) (L+1)^{-1}(\lambda-\lambda_1)^{-3\slash 2}x$ in Lemma \ref{Lem3.3}; note that we have $s\geq (\pi\slash 40)L^{-1}(\lambda-\lambda_1)^{-3\slash 2}x\geq 1$), we have
\begin{eqnarray}
  \mathbb{P}(\mathcal{E}_j|\mathcal{F}_{\tau_j})&=&\mathbbm{1}_{\tau_j\leq \lambda_2} \mathbb{P}(\tau_{j+1}-\tau_j\leq 8(L+1)(\lambda-\lambda_1)\slash x|\mathcal{F}_{\tau_j}) \nonumber\\
 &\leq& C\exp(-c L^{-2}(\lambda-\lambda_1)^{-3\slash 2} x^2).
\end{eqnarray}
Therefore, we have
\begin{eqnarray}\label{Eq3.100}
 &&  \mathbb{P}(\mathcal{D}_{i_1,\cdots,i_{\lceil x\slash 3\rceil}})=\mathbb{E}\Big[\prod_{l=1}^{\lceil x\slash 3\rceil} \mathbbm{1}_{\mathcal{E}_{i_l}}\Big]=\mathbb{E}\Big[\prod_{l=1}^{\lceil x\slash 3\rceil-1}\mathbbm{1}_{\mathcal{E}_{i_l}}\mathbb{P}(\mathcal{E}_{i_{\lceil x\slash 3\rceil}}|\mathcal{F}_{\tau_{i_{\lceil x\slash 3\rceil}}})\Big]\nonumber\\
 &\leq& C\exp(-c L^{-2}(\lambda-\lambda_1)^{-3\slash 2} x^2) \mathbb{E}\Big[\prod_{l=1}^{\lceil x\slash 3\rceil-1}\mathbbm{1}_{\mathcal{E}_{i_l}}\Big] \leq \cdots\nonumber\\
 &\leq& C^{\lceil x\slash 3\rceil}\exp(-c\lceil x\slash 3\rceil L^{-2}(\lambda-\lambda_1)^{-3\slash 2} x^2)\nonumber\\
 &\leq& C^{x+1}\exp(-c L^{-2}(\lambda-\lambda_1)^{-3\slash 2} x^3).
\end{eqnarray}
By (\ref{Eq3.99}), (\ref{Eq3.100}), and the union bound, we have
\begin{eqnarray}\label{Eq3.101}
  \mathbb{P}(\overline{N}(\lambda;\lambda_1,\lambda_2)\geq x\slash 2)&\leq& 2^{\lceil x\slash 2\rceil}C^{x+1}\exp(-c L^{-2}(\lambda-\lambda_1)^{-3\slash 2} x^3) \nonumber\\
  &\leq& C^{x+1} \exp(-c L^{-2}(\lambda-\lambda_1)^{-3\slash 2} x^3).
\end{eqnarray}
As $x\geq K\log(2)^2\geq K\slash 4$ and $x\geq 100 L^2(\lambda-\lambda_1)^{3\slash 2}$, we have
\begin{equation}\label{Eq3.102}
 L^{-2}(\lambda-\lambda_1)^{-3\slash 2} x^3\geq 25 K x.
\end{equation}
As we have taken $K$ to be sufficiently large (depending only on $\beta$ and $L$), by (\ref{Eq3.101}), (\ref{Eq3.102}), and the fact that $x\geq 100 L^2  (\lambda-\lambda_1)^{3\slash 2}$ we have
\begin{equation}\label{Eq3.107}
  \mathbb{P}(\overline{N}(\lambda;\lambda_1,\lambda_2)\geq x\slash 2) \leq C\exp(-c L^{-2}(\lambda-\lambda_1)^{-3\slash 2} x^3) \leq C\exp(-c x^2).                            
\end{equation}
By (\ref{Eq3.106}) and (\ref{Eq3.107}), we conclude that when $x\geq 100 L^2 (\lambda-\lambda_1)^{3\slash 2}$, we have
\begin{equation}\label{Eq3.117}
\mathbb{P}(|N(\lambda;\lambda_1,\lambda_2)-N_0(\lambda;\lambda_1,\lambda_2)|\geq x)\leq \mathbb{P}(\mathcal{U})\leq C\exp(-c x^2).
\end{equation}

\paragraph{Case 2: $x< 100L^2 (\lambda-\lambda_1)^{3\slash 2}$}

We take 
\begin{equation}
    \lambda_1'=\lambda-\frac{1}{50}x^{2  \slash  3} L^{-4\slash 3}, \quad \lambda_2'=\min\Big\{\lambda+\frac{1}{50}x^{2\slash 3} L^{-4\slash 3},   \lambda_2\Big\}.
\end{equation}
Note that 
\begin{equation}\label{Eq3.103}
  0\leq  \lambda_1<\lambda_1'<\lambda <\lambda_2'\leq \lambda_2, \quad \lambda_2'-\lambda\leq \lambda-\lambda_1',
\end{equation}
\begin{equation}\label{Eq3.104}
  \lambda-\lambda_1'=\frac{1}{50}x^{2\slash 3} L^{-4\slash 3}\geq \frac{1}{200}K^{2\slash 3}L^{-4\slash 3}
\end{equation}
\begin{equation}\label{Eq3.105}
   x= 50^{3\slash 2} L^2 (\lambda-\lambda_1')^{3\slash 2}\geq 300 L^2(\lambda-\lambda_1')^{3\slash 2},
\end{equation}
where we use the fact that $x\geq K\slash 4$ in (\ref{Eq3.104}). As we have taken $K$ to be sufficiently large (depending only on $\beta$ and $L$), by the result of \textbf{Case 1} (replacing $\lambda_1,\lambda_2$ by $\lambda_1',\lambda_2'$, respectively) and (\ref{Eq3.103})-(\ref{Eq3.105}), there exists an event
\begin{equation*}
\mathcal{U}_1\in\mathcal{F}_{\lambda_1',\lambda_2'}\subseteq \mathcal{F}_{\lambda_1,\lambda_2},
\end{equation*}
such that
\begin{equation}\label{Eq3.121}
   \{|N(\lambda;\lambda_1',\lambda_2')-N_0(\lambda;\lambda_1',\lambda_2')|\geq x\slash 2\} \subseteq \mathcal{U}_1,
\end{equation}
\begin{equation}\label{Eq3.114}
   \mathbb{P}(|N(\lambda;\lambda_1',\lambda_2')-N_0(\lambda;\lambda_1',\lambda_2')|\geq x\slash 2)\leq \mathbb{P}(\mathcal{U}_1)\leq C\exp(-c x^2).
\end{equation}

As $x<100 L^2(\lambda-\lambda_1)^{3\slash 2}$ and $\lambda-\lambda_1\geq K$, we have
\begin{eqnarray}\label{Eq3.112}
&& \lambda_1'-\lambda_1=\lambda-\lambda_1-\frac{1}{50}x^{2\slash 3} L^{-4\slash 3}\nonumber\\
&>& (1-100^{2\slash 3}\slash 50)(\lambda-\lambda_1) >  (\lambda-\lambda_1)\slash 2 \geq K\slash 2.
\end{eqnarray}
Now note that
\begin{equation*}
  \frac{\lambda-\lambda_1}{\lambda-\lambda_1'}=50 L^{4\slash 3} (\lambda-\lambda_1) x^{-2\slash 3}, 
\end{equation*}
which leads to
\begin{eqnarray}\label{Eq3.108}
  && 1+\log\Big(\frac{\lambda-\lambda_1}{\lambda-\lambda_1'}\Big)  = 1+\log(50 L^{4\slash 3})+\log((\lambda-\lambda_1)x^{-2\slash 3})\nonumber\\
  &\leq& C \log(2+(\lambda-\lambda_1)x^{-2\slash 3}).
\end{eqnarray}
Using the fact that $\log(1+x)\leq x$ for any $x>-1$, we obtain that
\begin{eqnarray}\label{Eq3.109}
  \log\log\Big(\frac{\lambda-\lambda_1}{\lambda-\lambda_1'}\Big)&\leq& \log\Big(1+\log\Big(\frac{\lambda-\lambda_1}{\lambda-\lambda_1'}\Big)\Big)\leq \log\Big(\frac{\lambda-\lambda_1}{\lambda-\lambda_1'}\Big).
\end{eqnarray}
By (\ref{Eq3.108}) and (\ref{Eq3.109}),
\begin{equation}\label{Eq3.110}
  \max\Big\{1,\log\log\Big(\frac{\lambda-\lambda_1}{\lambda-\lambda_1'}\Big)\Big\}\leq C\log(2+(\lambda-\lambda_1)x^{-2\slash 3}).
\end{equation}
By (\ref{Eq3.108}), (\ref{Eq3.110}), and the fact that $x\geq K(\log(2+(\lambda-\lambda_1)x^{-2\slash 3}))^2$, we have
\begin{eqnarray}\label{Eq3.111}
 && \Big(1+\log\Big(\frac{\lambda-\lambda_1}{\lambda-\lambda_1'}\Big)\Big) \max\Big\{1,\log\log\Big(\frac{\lambda-\lambda_1}{\lambda-\lambda_1'}\Big)\Big\}\nonumber\\
 &\leq& C (\log(2+(\lambda-\lambda_1)x^{-2\slash 3}))^2\leq C K^{-1} x.
\end{eqnarray}

By Proposition \ref{P3.1} (replacing $\lambda_2$ by $\lambda_1'$), noting (\ref{Eq3.104}), (\ref{Eq3.112}), (\ref{Eq3.111}), and the fact that $K$ is taken to be sufficiently large (depending only on $\beta$ and $L$), we conclude that there exists an event $\mathcal{U}_2\in \mathcal{F}_{\lambda_1,\lambda_1'}\subseteq \mathcal{F}_{\lambda_1,\lambda_2}$, such that
\begin{equation}\label{Eq3.122}
\{|N(\lambda;\lambda_1,\lambda_1')-N_0(\lambda;\lambda_1,\lambda_1')|\geq x\slash 4\}\subseteq \mathcal{U}_2,
\end{equation}
\begin{eqnarray}\label{Eq3.115}
  &&  \mathbb{P}(|N(\lambda;\lambda_1,\lambda_1')-N_0(\lambda;\lambda_1,\lambda_1')|\geq x\slash 4)\leq \mathbb{P}(\mathcal{U}_2)\nonumber\\
  &\leq& \exp(C(\lambda-\lambda_1)^{3\slash 2}\log\log(\lambda-\lambda_1)) \exp\Big(-c x^2\Big/\log\Big(\frac{\lambda-\lambda_1}{\lambda-\lambda_1'}\Big)\Big)\nonumber\\
  &\leq&  \exp(C(\lambda-\lambda_1)^{3\slash 2}\log\log(\lambda-\lambda_1))\exp(-c x^2\slash \log(2+(\lambda-\lambda_1)x^{-2\slash 3})), \nonumber\\
  &&   
\end{eqnarray}
where we use (\ref{Eq3.108}) for the last inequality. 

If $\lambda_2'=\lambda_2$, then $N(\lambda;\lambda_2',\lambda_2)=N_0(\lambda;\lambda_2',\lambda_2)=0$. Below we assume that $\lambda_2'<\lambda_2$. Following an argument that is similar to that below (\ref{Eq3.96}), we obtain that $N_0(\lambda;\lambda_2',\lambda_2)\leq 1$. As $\lambda_2'<\lambda_2$ and $x\geq K\slash 4$, we have
\begin{equation}\label{Eq3.113}
 \lambda_2'-\lambda= \frac{1}{50}x^{2\slash 3}L^{-4\slash 3}\geq \frac{1}{200} K^{2\slash 3} L^{-4\slash 3}.
\end{equation}
By Proposition \ref{P3.2} and (\ref{Eq3.113}), as $x\geq \log(2)^2 K$ and $K$ is sufficiently large (depending only on $\beta$ and $L$), we have $N_0(\lambda;\lambda_2',\lambda_2)\leq 1\leq x\slash 8$, and there exists an event $\mathcal{U}_3\in \mathcal{F}_{\lambda_2',\lambda_2}\subseteq \mathcal{F}_{\lambda_1,\lambda_2}$, such that
\begin{equation}\label{Eq3.123}
 \{|N(\lambda;\lambda_2',\lambda_2)-N_0(\lambda;\lambda_2',\lambda_2)|\geq x\slash 4\}\subseteq \{N(\lambda;\lambda_2',\lambda_2)\geq x\slash 8\}\subseteq \mathcal{U}_3,
\end{equation}
\begin{eqnarray}\label{Eq3.116}
 && \mathbb{P}(|N(\lambda;\lambda_2',\lambda_2)-N_0(\lambda;\lambda_2',\lambda_2)|\geq x\slash 4)\leq \mathbb{P}(N(\lambda;\lambda_2',\lambda_2)\geq x\slash 8)\nonumber\\
 && \leq \mathbb{P}(\mathcal{U}_3)\leq C\exp(-c (\lambda_2'-\lambda)^{3\slash 2}x)\leq C\exp(-c x^2).\nonumber\\
 &&
\end{eqnarray}

By (\ref{Eq3.103}), we have
\begin{equation*}
 N(\lambda;\lambda_1,\lambda_2)=N(\lambda;\lambda_1,\lambda_1')+N(\lambda;\lambda_1',\lambda_2')+N(\lambda;\lambda_2',\lambda_2),
\end{equation*}
\begin{equation*}
 N_0(\lambda;\lambda_1,\lambda_2)=N_0(\lambda;\lambda_1,\lambda_1')+N_0(\lambda;\lambda_1',\lambda_2')+N_0(\lambda;\lambda_2',\lambda_2).
\end{equation*}
Hence when $x<100 L^2(\lambda-\lambda_1)^{3\slash 2}$, letting $\mathcal{U}=\mathcal{U}_1\cup\mathcal{U}_2\cup\mathcal{U}_3\in\mathcal{F}_{\lambda_1,\lambda_2}$, by (\ref{Eq3.121}), (\ref{Eq3.122}), and (\ref{Eq3.123}), we have
\begin{eqnarray}\label{Eq3.124}
&& \{|N(\lambda;\lambda_1,\lambda_2)-N_0(\lambda;\lambda_1,\lambda_2)|\geq x\}\nonumber\\
&\subseteq&  \{|N(\lambda;\lambda_1',\lambda_2')-N_0(\lambda;\lambda_1',\lambda_2')|\geq x\slash 2\}\nonumber\\
&&\cup \{|N(\lambda;\lambda_1,\lambda_1')-N_0(\lambda;\lambda_1,\lambda_1')|\geq x\slash 4\}\nonumber\\
&& \cup \{|N(\lambda;\lambda_2',\lambda_2)-N_0(\lambda;\lambda_2',\lambda_2)|\geq x\slash 4\}\nonumber\\
&\subseteq& \mathcal{U}_1\cup\mathcal{U}_2\cup\mathcal{U}_3=\mathcal{U};
\end{eqnarray}
moreover, by (\ref{Eq3.114}), (\ref{Eq3.115}), (\ref{Eq3.116}), (\ref{Eq3.124}), and the union bound, we have
\begin{eqnarray}\label{Eq3.118}
&& \mathbb{P}(|N(\lambda;\lambda_1,\lambda_2)-N_0(\lambda;\lambda_1,\lambda_2)|\geq x)\nonumber\\
&\leq& \mathbb{P}(\mathcal{U})\leq \mathbb{P}(\mathcal{U}_1)+\mathbb{P}(\mathcal{U}_2)+\mathbb{P}(\mathcal{U}_3)\nonumber\\
&\leq& \exp(C(\lambda-\lambda_1)^{3\slash 2}\log\log(\lambda-\lambda_1))\nonumber\\
&&\times\exp(-c x^2\slash \log(2+(\lambda-\lambda_1)x^{-2\slash 3})).
\end{eqnarray}

\bigskip

By (\ref{Eq3.125}), (\ref{Eq3.117}), (\ref{Eq3.124}), and (\ref{Eq3.118}), we conclude that there exists an event $\mathcal{U}\in \mathcal{F}_{\lambda_1,\lambda_2}$, such that
\begin{equation}
  \{|N(\lambda;\lambda_1,\lambda_2)-N_0(\lambda;\lambda_1,\lambda_2)|\geq x\}\subseteq \mathcal{U},
\end{equation}
\begin{eqnarray}
&& \mathbb{P}(|N(\lambda;\lambda_1,\lambda_2)-N_0(\lambda;\lambda_1,\lambda_2)|\geq x)\leq\mathbb{P}(\mathcal{U})\nonumber\\
&\leq& \exp(C(\lambda-\lambda_1)^{3\slash 2} \log\log(\lambda-\lambda_1))\nonumber\\
&&\times \exp(-cx^2\slash\log(2+(\lambda-\lambda_1)x^{-2\slash 3})).
\end{eqnarray}

\end{proof}

\subsection{Proof of Proposition \ref{P3.4}}\label{Sect.3.5}

In this subsection, we give the proof of Proposition \ref{P3.4}.

\begin{proof}[Proof of Proposition \ref{P3.4}]

We take $K\geq 10$ sufficiently large (depending only on $\beta$), so that all the following estimates in this proof hold. For any $j\in\mathbb{N}^{*}$, if there are at least $j$ blow-ups of $p_{\lambda}(x)$ in $[0,\infty)$, we let $\tau_j$ be the $j$th blow-up time of $p_{\lambda}(x)$; otherwise we let $\tau_j:=\infty$. We also let $\tau_0:=0$.

Note that by Proposition \ref{diffu}, $N(\lambda)=N(\lambda;0,x^{2\slash 3}\slash 32)+N(\lambda;x^{2\slash 3}\slash 32,\infty)$. Hence
\begin{equation}\label{Eq3.129}
  \mathbb{P}(N(\lambda)\geq x)\leq \mathbb{P}(N(\lambda;0,x^{2\slash 3}\slash 32)\geq x\slash 2)+\mathbb{P}(N(\lambda;x^{2\slash 3}\slash 32,\infty)\geq x\slash 2).
\end{equation}

We first bound $\mathbb{P}(N(\lambda;0,x^{2\slash 3}\slash 32)\geq x\slash 2)$. In the following, we assume that $N(\lambda;0,x^{2\slash 3}\slash 32)\geq x\slash 2$. Then we have $\tau_{\lceil x\slash 2\rceil }\leq x^{2\slash 3}\slash 32$. Suppose that there exist at least $\lceil x\slash 4\rceil$ elements (denoted by $j$) of $\{0,1,\cdots,\lceil x\slash 2\rceil-1\}$, such that $\tau_{j+1}-\tau_j\geq 1\slash (4 x^{1\slash 3})$. Then we have
\begin{equation*}
  \frac{x^{2\slash 3}}{32}\geq \tau_{\lceil x\slash 2\rceil}=\sum_{j=0}^{\lceil x\slash 2\rceil-1} (\tau_{j+1}-\tau_j)\geq \frac{1}{4x^{1\slash 3}}\cdot \frac{x}{4}= \frac{x^{2\slash 3}}{16},
\end{equation*}
which leads to a contradiction. Hence there exist at least
\begin{equation*}
\lceil x\slash 2\rceil-(\lceil x\slash 4\rceil-1)\geq x\slash 4
\end{equation*}
elements (denoted by $j$) of $\{0,1,\cdots,\lceil x\slash 2\rceil-1\}$, such that $\tau_{j+1}-\tau_j\leq 1\slash (4x^{1\slash 3})$. For any $j\in\{0,1,\cdots,\lceil x\slash 2\rceil-1\}$, we let 
\begin{equation}
 \mathcal{E}_j:=\Big\{\tau_j\leq \frac{x^{2\slash 3}}{32}, \tau_{j+1}-\tau_j\leq \frac{1}{4x^{1\slash 3}}\Big\}\in\mathcal{F}_{\tau_{j+1}}.
\end{equation}
For any $i_1<\cdots<i_{\lceil x\slash 4\rceil}$ such that $i_1,\cdots,i_{\lceil x\slash  4 \rceil}\in\{0,1,\cdots,\lceil x\slash 2\rceil-1\}$, let
\begin{equation}
\mathcal{D}_{i_1,\cdots,i_{\lceil x\slash 4\rceil}}:=\bigcap_{j\in\{i_1,\cdots,i_{\lceil x\slash 4\rceil}\}} \mathcal{E}_j.
\end{equation}
By the preceding argument, we have 
\begin{equation}\label{Eq3.126}
\{N(\lambda;0,x^{2\slash 3}\slash 32 )\geq x\slash 2 \}\subseteq \bigcup_{\substack{i_1<\cdots<i_{\lceil x\slash 4\rceil}:\\i_1,\cdots,i_{\lceil x\slash 4\rceil\in \{0,1,\cdots,\lceil x\slash 2\rceil-1\}}}}\mathcal{D}_{i_1,\cdots,i_{\lceil x\slash 4\rceil}}.
\end{equation}

Consider any $i_1<\cdots<i_{\lceil x\slash 4\rceil}$ such that $i_1,\cdots,i_{\lceil x\slash 4\rceil}\in\{0,1,\cdots,\lceil x\slash 2\rceil-1\}$. For any $j\in\{i_1,\cdots,i_{\lceil x\slash 4\rceil}\}$, by Lemma \ref{Lem3.3} and the strong Markov property (replacing $\lambda$ by $\lambda-\tau_j$ and taking $a=x^{2\slash 3}, s=\pi\geq 1$; as $x\geq K\lambda^{3\slash 2}\geq \lambda^{3\slash 2}$, we have $\lambda-\tau_j\leq\lambda\leq x^{2\slash 3}=a$ and $\pi\slash (2s\sqrt{a})>\pi\slash (4s\sqrt{a})=1\slash (4 x^{1\slash 3})$), we have
\begin{equation}
  \mathbb{P}(\mathcal{E}_j|\mathcal{F}_{\tau_j})= \mathbbm{1}_{\tau_j\leq x^{2\slash 3}\slash 32} \mathbb{P}(\tau_{j+1}-\tau_j\leq 1\slash (4x^{1\slash 3})|\mathcal{F}_{\tau_j})
  \leq C\exp(-c x).
\end{equation}
Hence as $x\geq K$ is sufficiently large, we have
\begin{eqnarray}\label{Eq3.127}
&& \mathbb{P}(\mathcal{D}_{i_1,\cdots,i_{\lceil x\slash 4\rceil}})=\mathbb{E}\Big[\prod_{l=1}^{\lceil x\slash 4\rceil} \mathbbm{1}_{\mathcal{E}_{i_l}}\Big]=\mathbb{E}\Big[\prod_{l=1}^{\lceil x\slash 4\rceil-1}\mathbbm{1}_{\mathcal{E}_{i_l}}\mathbb{P}(\mathcal{E}_{i_{\lceil x\slash 4\rceil}}|\mathcal{F}_{\tau_{i_{\lceil x\slash 4\rceil}}})\Big] \nonumber\\
&\leq& C\exp(-c x)\mathbb{E}\Big[\prod_{l=1}^{\lceil x\slash 4\rceil  -1} \mathbbm{1}_{\mathcal{E}_{i_l}}\Big]\leq \cdots\leq C^{\lceil x\slash 4\rceil}\exp(-cx\lceil x\slash 4\rceil)\nonumber\\
&\leq& C^{x+1}\exp(-c x^2)\leq C\exp(-c x^2).
\end{eqnarray}

By (\ref{Eq3.126}), (\ref{Eq3.127}), and the union bound, as $x\geq K$ is sufficiently large, we have
\begin{eqnarray}\label{Eq3.130}
\mathbb{P}(N(\lambda;0,x^{2\slash 3}\slash 32)\geq x\slash 2)\leq  2^{\lceil x\slash 2\rceil}C\exp(-c x^2)\leq C\exp(-c x^2).
\end{eqnarray}

Now we bound $\mathbb{P}(N(\lambda;x^{2\slash 3}\slash 32,\infty)\geq x\slash 2)$. As we have taken $K$ to be sufficiently large and $x\geq K\max\{1,\lambda^{3\slash 2}\}$, we have 
\begin{equation}\label{Eq3.128}
 \frac{x^{2\slash 3}}{32}-\lambda\geq \frac{x^{2\slash 3}}{32}-\frac{x^{2\slash 3}}{K^{2\slash 3}}\geq \frac{x^{2\slash 3}}{64}\geq \frac{K^{2\slash 3}}{64}.
\end{equation}
By Proposition \ref{P3.2} (replacing $x$ by $x\slash 2$ and taking $\lambda_1=x^{2\slash 3}\slash 32,\lambda_2=\infty$) and (\ref{Eq3.128}), as $x\geq K$ and we have taken $K$ to be sufficiently large, we obtain that
\begin{equation}\label{Eq3.131}
  \mathbb{P}(N(\lambda;x^{2\slash 3}\slash 32,\infty)\geq x\slash 2)\leq C\exp(-c x^2).
\end{equation}

By (\ref{Eq3.129}), (\ref{Eq3.130}), and (\ref{Eq3.131}), we conclude that
\begin{equation}
 \mathbb{P}(N(\lambda)\geq x) \leq C\exp(-c x^2). 
\end{equation}

\end{proof}

\subsection{Proof of Proposition \ref{P3.5}}\label{Sect.3.6}

In this subsection, we give the proof of Proposition \ref{P3.5}.

\begin{proof}[Proof of Proposition \ref{P3.5}]

We take $K$ sufficiently large (depending only on $\beta$), so that all the following estimates in this proof hold. For any $j\in\mathbb{N}^{*}$, if there are at least $j$ blow-ups of $p_{\lambda}(x)$ in $[0,\infty)$, we let $\tau_j$ be the $j$th blow-up time of $p_{\lambda}(x)$ in $[0,\infty)$; otherwise we let $\tau_j:=\infty$. We also let $\tau_0:=0$. 

For any $j\in \mathbb{N}$, by Corollary \ref{tail} and the strong Markov property (replacing $t$ by $-(\lambda-\tau_j)$ in Corollary \ref{tail}; note that $-(\lambda-\tau_j)\geq -\lambda\geq K$), we have
\begin{equation}
  \mathbbm{1}_{\tau_j<\infty}\mathbb{P}(\tau_{j+1}<\infty|\mathcal{F}_{\tau_j})\leq C\exp(-c(-\lambda)^{3\slash 2}).
\end{equation}
Arguing as in (\ref{Eq3.132}), we obtain that
\begin{equation}
  \mathbb{P}(N(\lambda)\geq x)\leq\mathbb{E}\Big[\prod_{j=1}^{\lceil x\rceil}\mathbbm{1}_{\tau_j<\infty}\Big]\leq C\exp(-c(-\lambda)^{3\slash 2}x).
\end{equation}

\end{proof}

\section{Basic setups}\label{Sect.2}

In this section, we set up some basic notations and definitions that will be used in later sections.

Throughout the rest of this paper, we fix $M>0$, $\epsilon,\delta\in (0,1)$, and $n,k\in\mathbb{N}^{*}$. We take $M_0\geq 10$ sufficiently large (depending on $\beta$), $\epsilon_0\in (0,1\slash 10)$ sufficiently small (depending on $\beta$), $\delta_{0}(M,\epsilon)\in (0,1\slash 10)$ sufficiently small (depending on $\beta,M,\epsilon$), and $K_{0}(M,\epsilon,\delta)\geq 10$ sufficiently large (depending on $\beta,M,\epsilon,\delta$), such that when
\begin{equation}\label{Eq2.1}
   M\geq M_0, \quad \epsilon\in (0,\epsilon_0), \quad \delta\in (0,\delta_0(M,\epsilon)), \quad k\geq K_0(M,\epsilon,\delta),
\end{equation}
all the estimates in the rest of this paper hold. We assume (\ref{Eq2.1}) throughout the rest of this paper.

Let $\mathcal{I}_0:=2\lceil M^2\slash \epsilon\rceil \in\mathbb{N}^{*}$. We define $\{W_i\}_{i\in [\mathcal{I}_0+1] }$ as follows. Let $W_1:=10$. For any $i\in [\mathcal{I}_0]$, having defined $W_i$, we let $\overline{W}_{i+1}:=\max\{\lceil e^{M\slash \epsilon} (W_i+1)\rceil, 10 W_i\}$ and $W_{i+1}:=40 \overline{W}_{i+1}$. Note that by construction, for any $i\in [\mathcal{I}_0+1]$, $W_i\in\mathbb{N}^{*}$ and $W_i\geq 10^i$; for any $i\in [\mathcal{I}_0]$, $W_{i+1}>\overline{W}_{i+1}>W_i$. We also let $T:=2^{W_{\mathcal{I}_0+1}}+1$. Note that $\mathcal{I}_0$, $\{W_i\}_{i\in [\mathcal{I}_0+1]}$, and $T$ only depend on $M$ and $\epsilon$.

We define $I_{0}:=[0,2]$, $I_{0,1}:=[0,1]$, and $I_{0,2}:=(1,2]$. For any $l\in \mathbb{N}^{*}$, we let $I_l:=(2^l, 2^{l+1}]$. For any $l\in\mathbb{N}^{*}$ and $j\in [2l]$, we let 
\begin{equation*}
    I_{l,j}:=(2^l (1+(j-1)\slash (2l)), 2^l  (1+j\slash (2l))].
\end{equation*}
For any $t\in \mathbb{N}$, we let
\begin{equation*}
 \mathcal{R}:=\{(0,1),(0,2)\}\cup\{(l,j):l\in\mathbb{N}^{*}, j\in [2l]\}, \quad \mathcal{R}_t:=\{(l,j)\in\mathcal{R}: l\leq t\},
\end{equation*}
\begin{equation*}
\mathcal{R}^{EVEN}:=\{(l,j)\in\mathcal{R}: j\text{ is even}\}, \quad \mathcal{R}^{EVEN}_t:=\{(l,j)\in\mathcal{R}^{EVEN}: l\leq t\},
\end{equation*}
\begin{equation*}
\mathcal{R}^{ODD}:=\{(l,j)\in\mathcal{R}: j\text{ is odd}\}, \quad \mathcal{R}^{ODD}_t:=\{(l,j)\in\mathcal{R}^{ODD}: l\leq t\}.
\end{equation*}
For any $(l,j)\in\mathcal{R}$, let $a_{l,j}$ and $b_{l,j}$ be the left and right endpoints of $I_{l,j}$, respectively. We introduce the following total order on $\mathcal{R}$: For any $(l_1,j_1),(l_2,j_2)\in \mathcal{R}$, we let $(l_1,j_1)\preceq (l_2,j_2)$, if for any $x\in I_{l_1,j_1}, y\in I_{l_2,j_2}$, we have $x\leq y$. We let
\begin{equation*}
\overline{a}_{0,1}:=0, \overline{b}_{0,1}:=3\slash 2; \quad \overline{a}_{0,2}:=0,\overline{b}_{0,2}:=5\slash 2.
\end{equation*}
For any $(l,j)\in\mathcal{R}$ such that $l\geq 1$, we let $P_{l,j}$ be the largest element in $\mathcal{R}$ such that $P_{l,j}\prec (l,j)$ and $S_{l,j}$ be the smallest element in $\mathcal{R}$ such that $(l,j)\prec S_{l,j}$. We also take
\begin{equation*}
\overline{a}_{l,j}:=(a_{P_{l,j}}+b_{P_{l,j}})\slash 2,\quad \overline{b}_{l,j}:=(a_{S_{l,j}}+b_{S_{l,j}})\slash 2.
\end{equation*}

Now we let $\mathcal{Q}:=\{l\cdot 2^{-10(T+1)}:l\in\mathbb{N}\}$. In the following, we consider an arbitrary $(l,j)\in \mathcal{R}_T$. Let $\mathcal{Q}_{l,j}:=\mathcal{Q}\cap I_{l,j}$. We let
\begin{equation}
  \Delta_{l,j}:=\max_{x\in \mathcal{Q}_{l,j}}\{|N(x k^{2\slash 3})-N_0(x k^{2\slash 3})|\},
\end{equation}
and define $J_{l,j}$ as the smallest element in $\mathcal{Q}_{l,j}$ such that
\begin{equation*}
   |N(J_{l,j} k^{2\slash 3})-N_0(J_{l,j} k^{2\slash 3})|=\Delta_{l,j}.
\end{equation*}
We let 
\begin{equation}
\Delta_{l,j;1}:=|N(J_{l,j}k^{2\slash 3};0,\overline{b}_{l,j}k^{2\slash 3})-N_0(J_{l,j}k^{2\slash 3};0,\overline{b}_{l,j}k^{2\slash 3})|,
\end{equation}
\begin{equation}\label{Eq2.15}
\Delta_{l,j;2}:=|N(J_{l,j}k^{2\slash 3};\overline{b}_{l,j}k^{2\slash 3},\infty)-N_0(J_{l,j}k^{2\slash 3};\overline{b}_{l,j}k^{2\slash 3},\infty)|.
\end{equation}
Note that
\begin{equation}\label{Eq2.19}
 \Delta_{l,j}\leq \Delta_{l,j,1}+\Delta_{l,j,2}.
\end{equation}
We also let (recall (\ref{Eq2.21}))
\begin{equation}\label{Eq2.22}
\overline{\Delta}_{l,j}:=\sup_{x\in I_{l,j}}|N(x k^{2\slash 3})-N_0(x k^{2\slash 3})|,
\end{equation}
\begin{equation}
\tilde{\Delta}_{l,j}:=\sup_{x\in I_{l,j}}|\tilde{N}(x)-\tilde{N}_0(x)|=\sup_{x\in I_{l,j}}|\Psi(x)|.
\end{equation}

For every $(l,j)\in \mathcal{R}_T$, if $(l,j)\in\mathcal{R}^{EVEN}_T$, we let
\begin{equation}\label{Eq2.3}
    \mathcal{P}_{l,j}:=\{(l',j')\in\mathcal{R}^{EVEN}_T: (l',j')\preceq (l,j)\},
\end{equation}
\begin{equation}\label{Eq2.5}
 \mathcal{P}_{l,j}^{\circ}:=\{(l',j')\in\mathcal{R}^{EVEN}_T: (l',j')\prec (l,j)\};
\end{equation}
if $(l,j)\in\mathcal{R}^{ODD}_T$, we let
\begin{equation}\label{Eq2.4}
    \mathcal{P}_{l,j}:=\{(l',j')\in\mathcal{R}^{ODD}_T: (l',j')\preceq (l,j)\},
\end{equation}
\begin{equation}\label{Eq2.6}
    \mathcal{P}_{l,j}^{\circ}:=\{(l',j')\in\mathcal{R}^{ODD}_T: (l',j')\prec (l,j)\}.
\end{equation}
For any $(l,j)\in\mathcal{R}_T$ and $(l',j')\in \mathcal{P}_{l,j}$, we let  
\begin{equation}\label{Eq2.7}
    \Delta_{l,j;l',j';1}:=|N(J_{l,j}k^{2\slash 3};\overline{a}_{l',j'}k^{2\slash 3}, \overline{b}_{l',j'}k^{2\slash 3})-N_0(J_{l,j}k^{2\slash 3};\overline{a}_{l',j'}k^{2\slash 3}, \overline{b}_{l',j'}k^{2\slash 3})|.
\end{equation}
Note that for any $(l,j)\in\mathcal{R}_T$, we have
\begin{equation}\label{Eq2.20}
    \Delta_{l,j;1}\leq \sum_{(l',j')\in\mathcal{P}_{l,j}} \Delta_{l,j;l',j';1}.
\end{equation}

For every $l\in \{0\}\cup [T]$, we let
\begin{equation}\label{Eq2.11}
  \mathcal{K}_l:=\{0\}\cup \{(l+1)^{-10}\cdot 2^t: t\in\mathbb{N}, (l+1)^{-10}\cdot 2^t\leq M(l+1)^{10}\}.
\end{equation}
Now we define
\begin{equation}\label{Eq2.17}
   \mathcal{J}_T:=\{(z_{l,j})_{(l,j)\in\mathcal{R}_T}: z_{l,j}\in \mathcal{Q}_{l,j}\text{ for every }(l,j)\in\mathcal{R}_T\},
\end{equation}
\begin{eqnarray}\label{Eq2.18}
   \Theta_T &:=&\{(\theta_{l,j;l',j'})_{(l,j)\in\mathcal{R}_T, (l',j')\in \mathcal{P}_{l,j}}: \theta_{l,j;l',j'}\in\mathcal{K}_l\text{ for any }\nonumber\\
   &&\quad\quad (l,j)\in\mathcal{R}_T, (l',j')\in \mathcal{P}_{l,j}\}.
\end{eqnarray}
\begin{eqnarray}\label{Eq2.16}
   \Theta_T' &:=&\{(\theta_{l,j})_{(l,j)\in\mathcal{R}_T}:\theta_{l,j}\in\mathcal{K}_l\text{ for every }(l,j)\in\mathcal{R}_T\}.
\end{eqnarray}
For any $(l,j)\in\mathcal{R}_T$ and $z_{l,j}\in\mathcal{Q}_{l,j}$, we let $\mathcal{E}_0((l,j),z_{l,j})$ be the event that $J_{l,j}=z_{l,j}$. For any $(l,j)\in\mathcal{R}_T$ and $\theta'_{l,j}\in \mathcal{K}_l$, we let $\mathcal{E}'((l,j),\theta'_{l,j})$ be the event that the following holds:
\begin{itemize}
   \item[(a)] If $\theta'_{l,j}=0$, then $\Delta_{l,j;2}\leq (l+1)^{-10}k$;
   \item[(b)] If $\theta'_{l,j}\neq 0$, then $\theta'_{l,j}k \leq \Delta_{l,j;2}\leq 2 \theta'_{l,j}k$.
\end{itemize}
For any $(l,j)\in\mathcal{R}_T$, $(l',j')\in\mathcal{P}_{l,j}$, and any $\theta_{l,j;l',j'}\in\mathcal{K}_l$, we let $\mathcal{E}((l,j),(l',j'),\theta_{l,j;l',j'})$ be the event that the following holds:
\begin{itemize}
   \item[(a)] If $\theta_{l,j;l',j'}=0$, then $\Delta_{l,j;l',j';1}\leq (l+1)^{-10}k$;
   \item[(b)] If $\theta_{l,j;l',j'}\neq 0$, then $\theta_{l,j;l',j'}k \leq \Delta_{l,j;l',j';1}\leq 2 \theta_{l,j;l',j'}k$.
\end{itemize}
Now for any $\mathbf{z}=(z_{l,j})_{(l,j)\in\mathcal{R}_T}\in\mathcal{J}_T$, $\bm{\theta}=(\theta_{l,j;l',j'})_{(l,j)\in\mathcal{R}_T,(l',j')\in\mathcal{P}_{l,j}}\in \Theta_T$, and $\bm{\theta}'=(\theta_{l,j}')_{(l,j)\in\mathcal{R}_T} \in  \Theta'_T$, we define  
\begin{eqnarray}\label{DefinitionH}
\mathcal{H}(\mathbf{z},\bm{\theta},\bm{\theta}')&:=&\big(\bigcap_{(l,j)\in\mathcal{R}_T}\mathcal{E}_0((l,j),z_{l,j})\big)\bigcap\big(\bigcap_{(l,j)\in\mathcal{R}_T} \mathcal{E}'((l,j),\theta'_{l,j})\big)\nonumber\\
&& \bigcap\big(\bigcap_{(l,j)\in\mathcal{R}_T}\bigcap_{(l',j')\in\mathcal{P}_{l,j}}\mathcal{E}((l,j),(l',j'),\theta_{l,j;l',j'})\big).
\end{eqnarray}
We let $\mathcal{H}'_1$ be the event that there exist $(l,j)\in\mathcal{R}_T$, $(l',j')\in\mathcal{P}_{l,j}$, and $x\in \mathcal{Q}_{l,j}$, such that 
\begin{equation}\label{Eq2.8}
    |N(x k^{2\slash 3}; \overline{a}_{l',j'}k^{2\slash 3}, \overline{b}_{l',j'}k^{2\slash 3})-N_0(x k^{2\slash 3}; \overline{a}_{l',j'}k^{2\slash 3}, \overline{b}_{l',j'}k^{2\slash 3})|\geq M(l+1)^{10} k.
\end{equation}
We also let $\mathcal{H}'_2$ be the event that there exist $(l,j)\in\mathcal{R}_T$ and $x \in \mathcal{Q}_{l,j}$, such that
\begin{equation}\label{Eq2.9}
    |N( x k^{2\slash 3};\overline{b}_{l,j}k^{2\slash 3},\infty)-N_0(xk^{2\slash 3};\overline{b}_{l,j}k^{2\slash 3},\infty)|\geq M (l+1)^{10} k.
\end{equation}
Finally, we let 
\begin{equation}\label{Eq4.123}
    \mathcal{H}':=\mathcal{H}'_1\cup\mathcal{H}'_2.
\end{equation}

For any $\Lambda\geq 1$, we define $\mathcal{A}_0(\Lambda)$ as the event that
\begin{equation}\label{Eq2.10}
   \sum_{(l,j)\in \mathcal{R}_T} \frac{(\overline{\Delta}_{l,j})^2}{l+1}\leq \Lambda M^2 k^2,
\end{equation}
\begin{equation}\label{Eq2.12}
  N(-j k^{2\slash 3})\leq \frac{\Lambda M k}{(j+2)^{3\slash 2}} \text{ for any }j\in \{-1,0\}\cup [k^{10}],
\end{equation}
and define $\mathcal{B}_0(\Lambda)$ as the event that
\begin{equation}\label{Eq2.13}
  \sum_{(l,j)\in\mathcal{R}_{T-1}}\frac{(\tilde{\Delta}_{l,j})^2}{l+1}\leq\Lambda M^2 k^2,
\end{equation}
\begin{equation}\label{Eq2.14}
  \tilde{N}(-j)\leq \frac{\Lambda M k}{(j+1)^{3\slash 2}} \text{ for any } j\in \{0\}\cup [k^{10}].
\end{equation}
Let $\Phi$ be the set of $\bm{\phi}=(i_0,l_0, j_0,l_0', j_0')\in\mathbb{N}^5$ such that
\begin{itemize}
 \item[(a)] $\mathcal{I}_0\slash 2\leq i_0\leq \mathcal{I}_0$;
 \item[(b)] $2^{W_{i_0}}+1\leq l_0\leq 2^{\overline{W}_{i_0+1}}$, $l_0\slash 2  \leq j_0\leq 3 l_0 \slash 2$;
 \item[(c)] $10(l_0+1)\leq l_0'\leq 11(l_0+1)$, $2^{W_{i_0}}+1\leq l_0'\leq 2^{W_{i_0+1}}$, $1\leq j_0'\leq 2l_0'$.
\end{itemize}
For any $\bm{\phi}=(i_0,l_0,j_0,l_0',j_0')\in\Phi$ and any $\Lambda\geq 0$, we define $\mathcal{B}(\bm{\phi};\Lambda)$ as the event that the following conditions hold:
\begin{equation}\label{Eq2.1.1}
 \sum_{(l,j)\in \mathcal{R}_{T-1}}\frac{(\tilde{\Delta}_{l,j})^2}{l+1}\leq \Lambda M^2 k^2;
\end{equation}
\begin{equation}\label{Eq2.1.2}
   \tilde{N}(-j)\leq \frac{\Lambda M k}{(j+1)^{3\slash 2}}\text{ for any }j\in \{0\}\cup [k^{10}];
\end{equation}
\begin{equation}\label{Eq2.1.3}
\sum_{l=2^{W_{i_0}}+1}^{2^{W_{i_0+1}}}\sum_{j=1}^{2l}\frac{(\tilde{\Delta}_{l,j})^2}{l+1} \leq \Lambda \epsilon k^2;
\end{equation}
\begin{equation}\label{Eq2.1.4}
  \sum_{j=1}^{2l_0}\frac{(\tilde{\Delta}_{l_0,j})^2}{l_0+1}\leq \frac{\epsilon k^2}{(l_0+1)\log(l_0+1)};
\end{equation}
\begin{equation}\label{Eq2.1.5}
  \tilde{\Delta}_{l_0,j_0}\leq 2\sqrt{\frac{\epsilon}{(l_0+1)\log(l_0+1)}} k;
\end{equation}
\begin{equation}\label{Eq2.1.6}
   \tilde{\Delta}_{l_0',j_0'}\leq \sqrt{\frac{12\Lambda \epsilon}{l_0'+1}} k.
\end{equation}
For any $\bm{\phi}=(i_0,l_0, j_0,l_0', j_0')\in \Phi$, we further define
\begin{equation}\label{Eq2.1.11}
  R_0(\bm{\phi}):=2^{l_0}\Big(1+\frac{2j_0-1}{4l_0}\Big)\in I_{l_0,j_0}, \quad R_0'(\bm{\phi}):=2^{l_0'}\Big(1+\frac{2j_0'-1}{4l_0'}\Big)\in I_{l_0',j_0'}.
\end{equation}
Note that
\begin{equation}\label{Eq2.1.12}
  2^{l_0}\leq R_0(\bm{\phi})\leq 2^{l_0+1}. 
\end{equation}
As $i_0\geq \mathcal{I}_0\slash 2\geq M^2\slash \epsilon$, we have $W_{i_0}\geq 10^{i_0}\geq 10^{M^2\slash \epsilon}$. As $2^i\geq i$ for any $i\in\mathbb{N}$, we have
\begin{equation}\label{Eq2.1.10}
 l_0\geq 2^{W_{i_0}}+1\geq W_{i_0}\geq 10^{M^2\slash \epsilon}\geq 2000,\quad R_0(\bm{\phi})\geq 2^{l_0}\geq l_0\geq 10^{M^2\slash \epsilon}\geq 2000.
\end{equation}
where we use the fact that $M\geq M_0$ is sufficiently large and $\epsilon\in (0,  \epsilon_0)$ is sufficiently small. From $10(l_0+1)\leq l_0'\leq 11(l_0+1)$, as $M_0$ is sufficiently large and $\epsilon_0$ is sufficiently small, we have $10 (l_0+1)\leq l_0'\leq 20l_0-1$. As $2^{l_0'}\leq R_0'(\bm{\phi})\leq 2^{l_0'+1}$, we have
\begin{equation}\label{Eq2.2.1}
   R_0(\bm{\phi})^{10}\leq R_0'(\bm{\phi})\leq R_0(\bm{\phi})^{20}.
\end{equation}
We also note that
\begin{equation}\label{Eq2.1.15}
 l_0'\leq 2^{W_{i_0+1}}\leq 2^{W_{\mathcal{I}_0+1}}=T-1, \quad R_0'(\bm{\phi})\leq 2^{l_0'+1}\leq 2^T.
\end{equation}

\section{Bounding various quantities}\label{Sect.4}

In this section, we bound various quantities defined in Section \ref{Sect.2}.

\subsection{Bounding $\mathbb{P}(\mathcal{H}(\mathbf{z}, \bm{\theta}, \bm{\theta}'))$}

Recall the definition of $\mathcal{H}(\mathbf{z}, \bm{\theta}, \bm{\theta}')$ in (\ref{DefinitionH}) and the related quantities above (\ref{DefinitionH}).  We also introduce the following definition.

\begin{definition}\label{Def4.1}
For any $\bm{\theta}=(\theta_{l,j;l',j'})_{(l,j)\in\mathcal{R}_T,(l',j')\in\mathcal{P}_{l,j}}\in \Theta_T$ and any $(l,j)\in\mathcal{R}_T$, we let 
\begin{equation}
 \mathcal{S}_{l,j}(\bm{\theta}):=\sum_{(l',j')\in\mathcal{P}_{l,j}} \theta_{l,j;l',j'}.
\end{equation}
\end{definition}

In this subsection, we give an upper bound on $\mathbb{P}(\mathcal{H}(\mathbf{z}, \bm{\theta}, \bm{\theta}'))$ for any choice of $\mathbf{z}=(z_{l,j})_{(l,j)\in\mathcal{R}_T}\in\mathcal{J}_T$, $\bm{\theta}=(\theta_{l,j;l',j'})_{(l,j)\in\mathcal{R}_T,(l',j')\in\mathcal{P}_{l,j}}\in \Theta_T$, and $\bm{\theta}'=(\theta_{l,j}')_{(l,j)\in\mathcal{R}_T} \in  \Theta'_T$. The main result is given in the following proposition.

\begin{proposition}\label{P4.1}
There exist positive constants $C,c$ that only depend on $\beta$, such that the following holds. For any
\begin{equation*}
\mathbf{z}=(z_{l,j})_{(l,j)\in\mathcal{R}_T}\in\mathcal{J}_T,
\end{equation*}
\begin{equation*}
\bm{\theta}=(\theta_{l,j;l',j'})_{(l,j)\in\mathcal{R}_T,(l',j')\in\mathcal{P}_{l,j}}\in \Theta_T, \quad \bm{\theta}'=(\theta_{l,j}')_{(l,j)\in\mathcal{R}_T} \in  \Theta'_T,
\end{equation*}
we have
\begin{eqnarray}
\mathbb{P}(\mathcal{H}(\mathbf{z}, \bm{\theta}, \bm{\theta}'))&\leq& C  \min\Big\{\exp\Big(-c k^2 \Big(\sum_{(l,j)\in\mathcal{R}_T}(l+1)^{-1}\mathcal{S}_{l,j}(\bm{\theta})^2 \Big)\Big),\nonumber\\
&& \quad\quad\exp\Big(-c k^2 \max_{(l,j)\in\mathcal{R}_T}\Big\{2^{3l\slash 2}(l+1)^{-3\slash 2}\theta'_{l,j}\Big\} \Big)\Big\}.
\end{eqnarray}
\end{proposition}

\begin{proof}

We fix any $\mathbf{z}=(z_{l,j})_{(l,j)\in\mathcal{R}_T}\in\mathcal{J}_T$, $\bm{\theta}=(\theta_{l,j;l',j'})_{(l,j)\in\mathcal{R}_T,(l',j')\in\mathcal{P}_{l,j}}\in \Theta_T$, and $\bm{\theta}'=(\theta_{l,j}')_{(l,j)\in\mathcal{R}_T} \in  \Theta'_T$. For each $(l,j)\in \mathcal{R}^{EVEN}_T$ such that $l\geq 1$, we let $P^{EVEN}_{l,j}$ be the largest element in $\mathcal{R}^{EVEN}_T$ such that $P^{EVEN}_{l,j} \prec (l,j)$. 

\paragraph{Step 1}

Recall the definition of $\mathcal{P}_{l,j}$ and $\mathcal{P}_{l,j}^{\circ}$ from (\ref{Eq2.3})-(\ref{Eq2.6}). Sequentially for each $(l,j)\in \mathcal{R}^{EVEN}_T$ (from the smallest element to the largest element in terms of the order $\preceq$ defined in Section \ref{Sect.2}), we define $Index(l,j;l',j')$ and $Level(l,j;l',j')$ for every $(l',j')\in \mathcal{P}_{l,j}$ as follows. We take
\begin{equation}\label{Eq4.7}
  Index(l,j;l,j):=(l,j), \quad Level(l,j;l,j):=\frac{(\theta_{l,j;l,j})^2}{\log(\overline{b}_{l,j}-\overline{a}_{l,j})}.
\end{equation}

If $l=0$, we call $(l,j)$ a ``type I index''.

If $l\geq 1$, we denote $(\tilde{l},\tilde{j})=P^{EVEN}_{l,j}$, and assume that $Index(\tilde{l},\tilde{j};l',j')$ and $Level(\tilde{l},\tilde{j};l',j')$ have been determined for every $(l',j')\in \mathcal{P}_{\tilde{l},\tilde{j}}$ (due to our sequential construction). We consider the following two cases.
\begin{itemize}
  \item[(a)] If 
  \begin{equation}\label{Eq4.16}
     \sum_{(l',j')\in\mathcal{P}_{\tilde{l},\tilde{j}}}\sqrt{Level(\tilde{l},\tilde{j};l',j')}\sqrt{\log\Big(\frac{\overline{b}_{l,j}-\overline{a}_{l',j'}}{\overline{b}_{l,j}-\overline{b}_{l',j'}}\Big)} \leq \frac{1}{100}\sum_{(l',j')\in\mathcal{P}_{l,j}}\theta_{l,j;l',j'},
  \end{equation}
  we call $(l,j)$ a ``type I index''. For every $(l',j')\in\mathcal{P}_{\tilde{l},\tilde{j}}$, if 
  \begin{equation}\label{Eq4.2}
      (\theta_{l,j;l',j'})^2\geq Level(\tilde{l},\tilde{j};l',j') \log\Big(\frac{\overline{b}_{l,j}-\overline{a}_{l',j'}}{\overline{b}_{l,j}-\overline{b}_{l',j'}}\Big),
  \end{equation}
  we let
  \begin{equation}\label{Eq4.3}
      Index(l,j;l',j'):=(l,j), \quad Level(l,j;l',j'):=(\theta_{l,j;l',j'})^2 \Big/\log\Big(\frac{\overline{b}_{l,j}-\overline{a}_{l',j'}}{\overline{b}_{l,j}-\overline{b}_{l',j'}}\Big);
  \end{equation}
  otherwise, we let
  \begin{equation}\label{Eq4.4}
      Index(l,j;l',j'):=Index(\tilde{l},\tilde{j};l',j'), \quad Level(l,j;l',j'):=Level(\tilde{l},\tilde{j};l',j').
  \end{equation}
  \item[(b)] If 
  \begin{equation}
     \sum_{(l',j')\in\mathcal{P}_{\tilde{l},\tilde{j}}}\sqrt{Level(\tilde{l},\tilde{j};l',j')}\sqrt{\log\Big(\frac{\overline{b}_{l,j}-\overline{a}_{l',j'}}{\overline{b}_{l,j}-\overline{b}_{l',j'}}\Big)} > \frac{1}{100}\sum_{(l',j')\in\mathcal{P}_{l,j}}\theta_{l,j;l',j'},
  \end{equation}
  we call $(l,j)$ a ``type II index''. For every $(l',j')\in\mathcal{P}_{\tilde{l},\tilde{j}}$, we let
  \begin{equation}\label{Eq4.5}
      Index(l,j;l',j'):=Index(\tilde{l},\tilde{j};l',j'), \quad Level(l,j;l',j'):=Level(\tilde{l},\tilde{j};l',j').
  \end{equation}
\end{itemize}

Note that $\overline{b}_{0,1}-\overline{a}_{0,1}=3\slash 2>1$ and $\overline{b}_{0,2}-\overline{a}_{0,2}=5\slash 2>1$. For any $(l,j)\in\mathcal{R}_T$ such that $l\geq 1$, $\overline{b}_{l,j}-\overline{a}_{l,j}>b_{l,j}-a_{l,j}=2^l\slash (2l)\geq 1$. Hence for any $(l,j)\in\mathcal{R}_T$, we have $\log(\overline{b}_{l,j}-\overline{a}_{l,j})>0$. Moreover, for any $(l,j)\in\mathcal{R}_T$ and $(l',j')\in\mathcal{P}_{l,j}^{\circ}$, as $\overline{a}_{l',j'}<\overline{b}_{l',j'}$, we have 
\begin{equation*}
    \log\Big(\frac{\overline{b}_{l,j}-\overline{a}_{l',j'}}{\overline{b}_{l,j}-\overline{b}_{l',j'}}\Big)>0.
\end{equation*}
Therefore, by the above construction, for any $(l,j)\in \mathcal{R}^{EVEN}_T$ and $(l',j')\in\mathcal{P}_{l,j}$, we have $Level(l,j;l',j')\geq 0$.

Note that for any $(l,j)\in\mathcal{R}^{EVEN}_T$ and $(l',j')\in\mathcal{P}_{l,j}$ such that $(l',j')\prec (l,j)$ (which implies $l\geq 1$), we have 
\begin{equation}\label{Eq4.56}
  \overline{b}_{l,j}-\overline{b}_{l',j'}\geq b_{l,j}-a_{l,j}=\frac{2^l}{2l}\geq 1.
\end{equation}
By (\ref{Eq4.56}) and the above construction, for any $(l,j)\in \mathcal{R}^{EVEN}_T$ and $(l',j')\in\mathcal{P}_{l,j}$, we have
\begin{equation}\label{Eq4.57}
  (l',j')\preceq Index(l,j;l',j')\preceq (l,j);
\end{equation}
moreover, if we denote $(l'',j'')=Index(l,j;l',j')$, then
\begin{equation}\label{Eq4.101}
Level(l,j;l',j')=(\theta_{l'',j'';l',j'})^2\Big/\log\Big(\frac{\overline{b}_{l'',j''}-\overline{a}_{l',j'}}{\max\{\overline{b}_{l'',j''}-\overline{b}_{l',j'},1\}}\Big).
\end{equation}

By (\ref{Eq4.2})-(\ref{Eq4.4}) and (\ref{Eq4.5}), for any $(l,j)\in\mathcal{R}^{EVEN}_T$ such that $l\geq 1$ and any $(l',j')\in \mathcal{P}_{\tilde{l},\tilde{j}}$, we have
\begin{equation}\label{Eq4.6}
    Level(l,j;l',j')\geq Level(\tilde{l},\tilde{j};l',j').
\end{equation}
Therefore, for any $(l_1,j_1),(l_2,j_2)\in\mathcal{R}^{EVEN}_T$ such that $(l_1,j_1)\succeq (l_2,j_2)$ and any $(l',j')\in\mathcal{P}_{l_2,j_2}$, we have
\begin{equation}\label{Eq4.21}
  Level(l_1,j_1;l',j')\geq Level(l_2,j_2;l',j').
\end{equation}

\paragraph{Step 2}

For any $(l,j)\in\mathcal{R}_T$, we denote
\begin{equation}\label{Eq4.14}
    \theta_{l,j}:=\sum_{(l',j')\in\mathcal{P}_{l,j}} \theta_{l,j;l',j'}=\mathcal{S}_{l,j}(\bm{\theta}).
\end{equation}
In the following, we show by induction that for any $(l,j)\in\mathcal{R}^{EVEN}_T$,
\begin{equation}\label{Eq4.1}
    \sum_{(l',j')\in\mathcal{P}_{l,j}} Level(l,j;l',j')\geq \frac{1}{20}\sum_{\substack{(l',j')\in\mathcal{P}_{l,j}:\\ (l',j')\text{ is a type I index}}} \frac{(\theta_{l',j'})^2}{\log(\overline{b}_{l',j'})}.
\end{equation}

For $(l,j)=(0,2)$, the left-hand side of (\ref{Eq4.1}) is $(\theta_{0,2;0,2})^2\slash \log(\overline{b}_{0,2})$, and the right-hand side of (\ref{Eq4.1}) is at most $(\theta_{0,2;0,2})^2\slash( 20 \log(\overline{b}_{0,2}))$, hence (\ref{Eq4.1}) holds.

Below we assume that $l\geq 1$ and that (\ref{Eq4.1}) holds for any $(l',j')\in\mathcal{P}_{\tilde{l},\tilde{j}}$. If $(l,j)$ is a type II index, then we have
\begin{eqnarray}\label{Eq4.19}
  &&  \sum_{(l',j')\in\mathcal{P}_{l,j}} Level(l,j;l',j') \geq \sum_{(l',j')\in\mathcal{P}_{\tilde{l},\tilde{j}}} Level(l,j;l',j')\nonumber\\
  &\geq& \sum_{(l',j')\in\mathcal{P}_{\tilde{l},\tilde{j}}}Level(\tilde{l},\tilde{j};l',j') \geq \frac{1}{20}\sum_{\substack{(l',j')\in\mathcal{P}_{\tilde{l},\tilde{j}}:\\ (l',j')\text{ is a type I index}}} \frac{(\theta_{l',j'})^2}{\log(\overline{b}_{l',j'})}\nonumber\\
  &=& \frac{1}{20}\sum_{\substack{(l',j')\in\mathcal{P}_{l,j}:\\ (l',j')\text{ is a type I index}}} \frac{(\theta_{l',j'})^2}{\log(\overline{b}_{l',j'})},
\end{eqnarray}
where we use (\ref{Eq4.6}) for the second inequality and the induction hypothesis for the third inequality.

Below we consider the case where $(l,j)$ is a type I index. We have
\begin{eqnarray}\label{Eq4.8}
&& \sum_{(l',j')\in\mathcal{P}_{l,j}} Level(l,j;l',j')-\sum_{(l',j')\in\mathcal{P}_{\tilde{l},\tilde{j}}} Level(\tilde{l},\tilde{j};l',j') \nonumber\\
&=& \sum_{(l',j')\in\mathcal{P}_{\tilde{l},\tilde{j}}}(Level(l,j;l',j')-Level(\tilde{l},\tilde{j};l',j'))+Level(l,j;l,j)\nonumber\\
&\geq& \sum_{\substack{(l',j')\in\mathcal{P}_{\tilde{l},\tilde{j}}:\\ Level(l,j;l',j')\geq 2 Level(\tilde{l},\tilde{j};l',j')}}(Level(l,j;l',j')-Level(\tilde{l},\tilde{j};l',j'))\nonumber\\
&& +\frac{(\theta_{l,j;l,j})^2}{\log(\overline{b}_{l,j}-\overline{a}_{l,j})}\nonumber\\
&\geq& \frac{1}{2}\sum_{\substack{(l',j')\in\mathcal{P}_{\tilde{l},\tilde{j}}:\\ Level(l,j;l',j')\geq 2 Level(\tilde{l},\tilde{j};l',j')}} Level(l,j;l',j')+\frac{(\theta_{l,j;l,j})^2}{\log(\overline{b}_{l,j}-\overline{a}_{l,j})},
\end{eqnarray}
where we use (\ref{Eq4.6}) in the third line and (\ref{Eq4.7}) in the fourth line. 

For any choice of $(l',j')\in\mathcal{P}_{\tilde{l},\tilde{j}}$ such that $Level(l,j;l',j')\geq 2 Level(\tilde{l},\tilde{j};l',j')$, we have the following two cases:
\begin{itemize}
    \item[(a)] If $Level(\tilde{l},\tilde{j};l',j')=0$, then (\ref{Eq4.2}) holds. Hence
    \begin{equation}\label{Eq4.9}
        Level(l,j;l',j')=(\theta_{l,j;l',j'})^2\Big/\log\Big(\frac{\overline{b}_{l,j}-\overline{a}_{l',j'}}{\overline{b}_{l,j}-\overline{b}_{l',j'}}\Big).
    \end{equation}
    \item[(b)] If $Level(\tilde{l},\tilde{j};l',j')>0$, then
    \begin{equation*}
        Level(l,j;l',j')\geq 2 Level(\tilde{l},\tilde{j};l',j')>Level(\tilde{l},\tilde{j};l',j').
    \end{equation*}
    By the construction in \textbf{Step 1}, we have
    \begin{equation}\label{Eq4.10}
        Level(l,j;l',j')=(\theta_{l,j;l',j'})^2\Big/\log\Big(\frac{\overline{b}_{l,j}-\overline{a}_{l',j'}}{\overline{b}_{l,j}-\overline{b}_{l',j'}}\Big).
    \end{equation}
\end{itemize}
By the Cauchy-Schwarz inequality, we have
\begin{eqnarray}\label{Eq4.11}
&& \Big(\sum_{\substack{(l',j')\in\mathcal{P}_{\tilde{l},\tilde{j}}:\\ Level(l,j;l',j')\geq 2 Level(\tilde{l},\tilde{j};l',j')}}(\theta_{l,j;l',j'})^2\Big/\log\Big(\frac{\overline{b}_{l,j}-\overline{a}_{l',j'}}{\overline{b}_{l,j}-\overline{b}_{l',j'}}\Big)+\frac{(\theta_{l,j;l,j})^2}{\log(\overline{b}_{l,j}-\overline{a}_{l,j})}\Big) \nonumber\\
&\times& \Big(\sum_{\substack{(l',j')\in\mathcal{P}_{\tilde{l},\tilde{j}}:\\ Level(l,j;l',j')\geq 2 Level(\tilde{l},\tilde{j};l',j')}}\log\Big(\frac{\overline{b}_{l,j}-\overline{a}_{l',j'}}{\overline{b}_{l,j}-\overline{b}_{l',j'}}\Big)+\log(\overline{b}_{l,j}-\overline{a}_{l,j})\Big)\nonumber\\
&\geq& \Big(\sum_{\substack{(l',j')\in\mathcal{P}_{\tilde{l},\tilde{j}}:\\ Level(l,j;l',j')\geq 2 Level(\tilde{l},\tilde{j};l',j')}}\theta_{l,j;l',j'}+\theta_{l,j;l,j}\Big)^2.
\end{eqnarray}
Moreover, 
\begin{eqnarray}\label{Eq4.12}
&& \sum_{\substack{(l',j')\in\mathcal{P}_{\tilde{l},\tilde{j}}:\\ Level(l,j;l',j')\geq 2 Level(\tilde{l},\tilde{j};l',j')}}\log\Big(\frac{\overline{b}_{l,j}-\overline{a}_{l',j'}}{\overline{b}_{l,j}-\overline{b}_{l',j'}}\Big)+\log(\overline{b}_{l,j}-\overline{a}_{l,j}) \nonumber\\
&&\leq \sum_{(l',j')\in\mathcal{P}_{\tilde{l},\tilde{j}}} \log\Big(\frac{\overline{b}_{l,j}-\overline{a}_{l',j'}}{\overline{b}_{l,j}-\overline{b}_{l',j'}}\Big) + \log(\overline{b}_{l,j}-\overline{a}_{l,j})=\log(\overline{b}_{l,j}).
\end{eqnarray}
By (\ref{Eq4.11}) and (\ref{Eq4.12}), we have
\begin{eqnarray}\label{Eq4.13}
&& \sum_{\substack{(l',j')\in\mathcal{P}_{\tilde{l},\tilde{j}}:\\ Level(l,j;l',j')\geq 2 Level(\tilde{l},\tilde{j};l',j')}}(\theta_{l,j;l',j'})^2\Big/\log\Big(\frac{\overline{b}_{l,j}-\overline{a}_{l',j'}}{\overline{b}_{l,j}-\overline{b}_{l',j'}}\Big)+\frac{(\theta_{l,j;l,j})^2}{\log(\overline{b}_{l,j}-\overline{a}_{l,j})} \nonumber\\
&&\geq\Big(\sum_{\substack{(l',j')\in\mathcal{P}_{\tilde{l},\tilde{j}}:\\ Level(l,j;l',j')\geq 2 Level(\tilde{l},\tilde{j};l',j')}}\theta_{l,j;l',j'}+\theta_{l,j;l,j}\Big)^2 \Big/ \log(\overline{b}_{l,j}).
\end{eqnarray}
By (\ref{Eq4.8})-(\ref{Eq4.10}) and (\ref{Eq4.13}), we have
\begin{eqnarray}\label{Eq4.18}
&& \sum_{(l',j')\in\mathcal{P}_{l,j}} Level(l,j;l',j')-\sum_{(l',j')\in\mathcal{P}_{\tilde{l},\tilde{j}}} Level(\tilde{l},\tilde{j};l',j') \nonumber\\
&\geq& \frac{1}{2}\sum_{\substack{(l',j')\in\mathcal{P}_{\tilde{l},\tilde{j}}:\\ Level(l,j;l',j')\geq 2 Level(\tilde{l},\tilde{j};l',j')}}(\theta_{l,j;l',j'})^2\Big/\log\Big(\frac{\overline{b}_{l,j}-\overline{a}_{l',j'}}{\overline{b}_{l,j}-\overline{b}_{l',j'}}\Big)\nonumber\\
&& +\frac{(\theta_{l,j;l,j})^2}{\log(\overline{b}_{l,j}-\overline{a}_{l,j})}\nonumber\\
&\geq& \Big(\sum_{\substack{(l',j')\in\mathcal{P}_{\tilde{l},\tilde{j}}:\\ Level(l,j;l',j')\geq 2 Level(\tilde{l},\tilde{j};l',j')}}\theta_{l,j;l',j'}+\theta_{l,j;l,j}\Big)^2 \Big/ (2\log(\overline{b}_{l,j})).
\end{eqnarray}

Now consider any $(l',j')\in\mathcal{P}_{\tilde{l},\tilde{j}}$ such that $Level(l,j;l',j')<2 Level(\tilde{l},\tilde{j};l',j')$. Then either
\begin{equation*}
    \theta_{l,j;l',j'}<\sqrt{Level(\tilde{l},\tilde{j};l',j')}\sqrt{\log\Big(\frac{\overline{b}_{l,j}-\overline{a}_{l',j'}}{\overline{b}_{l,j}-\overline{b}_{l',j'}}\Big)},
\end{equation*}
or
\begin{equation*}
    \theta_{l,j;l',j'}\geq \sqrt{Level(\tilde{l},\tilde{j};l',j')}\sqrt{\log\Big(\frac{\overline{b}_{l,j}-\overline{a}_{l',j'}}{\overline{b}_{l,j}-\overline{b}_{l',j'}}\Big)}.
\end{equation*}
For the latter case, (\ref{Eq4.2}) holds, hence by (\ref{Eq4.3}), we obtain that
\begin{eqnarray*}
    \theta_{l,j;l',j'}&=&\sqrt{Level(l,j;l',j')}\sqrt{\log\Big(\frac{\overline{b}_{l,j}-\overline{a}_{l',j'}}{\overline{b}_{l,j}-\overline{b}_{l',j'}}\Big)}  \nonumber\\
    &<& \sqrt{2} \sqrt{Level(\tilde{l},\tilde{j};l',j')}\sqrt{\log\Big(\frac{\overline{b}_{l,j}-\overline{a}_{l',j'}}{\overline{b}_{l,j}-\overline{b}_{l',j'}}\Big)}.
\end{eqnarray*}
Hence in both cases, we have
\begin{equation}\label{Eq4.15}
    \theta_{l,j;l',j'}<\sqrt{2} \sqrt{Level(\tilde{l},\tilde{j};l',j')}\sqrt{\log\Big(\frac{\overline{b}_{l,j}-\overline{a}_{l',j'}}{\overline{b}_{l,j}-\overline{b}_{l',j'}}\Big)}.
\end{equation}

By (\ref{Eq4.14}) and (\ref{Eq4.15}), we have
\begin{eqnarray}\label{Eq4.17}
&& \sum_{\substack{(l',j')\in\mathcal{P}_{\tilde{l},\tilde{j}}:\\ Level(l,j;l',j')\geq 2 Level(\tilde{l},\tilde{j};l',j')}}\theta_{l,j;l',j'}+\theta_{l,j;l,j}\nonumber\\
&=&\theta_{l,j}-\sum_{\substack{(l',j')\in\mathcal{P}_{\tilde{l},\tilde{j}}:\\ Level(l,j;l',j')< 2 Level(\tilde{l},\tilde{j};l',j')}}\theta_{l,j;l',j'}\nonumber\\
&\geq&  \theta_{l,j}-\sqrt{2}\sum_{\substack{(l',j')\in\mathcal{P}_{\tilde{l},\tilde{j}}:\\ Level(l,j;l',j')< 2 Level(\tilde{l},\tilde{j};l',j')}}\sqrt{Level(\tilde{l},\tilde{j};l',j')}\sqrt{\log\Big(\frac{\overline{b}_{l,j}-\overline{a}_{l',j'}}{\overline{b}_{l,j}-\overline{b}_{l',j'}}\Big)}\nonumber\\
&\geq& \theta_{l,j}-\sqrt{2}\sum_{(l',j')\in\mathcal{P}_{\tilde{l},\tilde{j}}}\sqrt{Level(\tilde{l},\tilde{j};l',j')}\sqrt{\log\Big(\frac{\overline{b}_{l,j}-\overline{a}_{l',j'}}{\overline{b}_{l,j}-\overline{b}_{l',j'}}\Big)}\nonumber\\
&\geq&\theta_{l,j}-\frac{1}{50}\sum_{(l',j')\in\mathcal{P}_{l,j}} \theta_{l,j;l',j'}=\frac{49}{50}\theta_{l,j},
\end{eqnarray}
where we use (\ref{Eq4.16}) and the fact that $(l,j)$ is a type I index in the fifth line. By (\ref{Eq4.18}) and (\ref{Eq4.17}), we have
\begin{eqnarray}
&& \sum_{(l',j')\in\mathcal{P}_{l,j}} Level(l,j;l',j')-\sum_{(l',j')\in\mathcal{P}_{\tilde{l},\tilde{j}}} Level(\tilde{l},\tilde{j};l',j') \nonumber\\
&&\geq \frac{1}{2}\Big(\frac{49}{50}\Big)^2 \frac{(\theta_{l,j})^2}{\log(\overline{b}_{l,j})}\geq \frac{1}{10} \frac{(\theta_{l,j})^2}{\log(\overline{b}_{l,j})}.
\end{eqnarray}
Hence by the induction hypothesis, we have
\begin{eqnarray}\label{Eq4.20}
 && \sum_{(l',j')\in\mathcal{P}_{l,j}} Level(l,j;l',j')\geq \sum_{(l',j')\in\mathcal{P}_{\tilde{l},\tilde{j}}} Level(\tilde{l},\tilde{j};l',j')+\frac{1}{10} \frac{(\theta_{l,j})^2}{\log(\overline{b}_{l,j})} \nonumber\\
 &\geq& \frac{1}{20}\sum_{\substack{(l',j')\in\mathcal{P}_{\tilde{l},\tilde{j}}:\\ (l',j')\text{ is a type I index}}} \frac{(\theta_{l',j'})^2}{\log(\overline{b}_{l',j'})}+\frac{1}{10} \frac{(\theta_{l,j})^2}{\log(\overline{b}_{l,j})} \nonumber\\
 &\geq & \frac{1}{20}\sum_{\substack{(l',j')\in\mathcal{P}_{l,j}:\\(l',j')\text{ is a type I index}}}\frac{(\theta_{l',j'})^2}{\log(\overline{b}_{l',j'})}.
\end{eqnarray}

By (\ref{Eq4.19}) and (\ref{Eq4.20}), we conclude that (\ref{Eq4.1}) holds.

\paragraph{Step 3}

Consider any $(l,j)\in\mathcal{R}^{EVEN}_T$ and $(l',j')\in\mathcal{P}_{l,j}$ such that $(l',j')$ is a type II index (note that this implies $l'\geq 1 $). We denote $(\overline{l},\overline{j})=P^{EVEN}_{l',j'}$ (note that $(\overline{l},\overline{j})$ depends on $(l',j')$). As $(l',j')$ is a type II index, we have
\begin{eqnarray}
&& \theta_{l',j'}=\sum_{(l'',j'')\in\mathcal{P}_{l',j'}}\theta_{l',j';l'',j''}\nonumber\\
&<& 100 \sum_{(l'',j'')\in\mathcal{P}_{\overline{l},\overline{j}}}\sqrt{Level(\overline{l},\overline{j};l'',j'')}\sqrt{\log\Big(\frac{\overline{b}_{l',j'}-\overline{a}_{l'',j''}}{\overline{b}_{l',j'}-\overline{b}_{l'',j''}}\Big)}\nonumber\\
&\leq&  100 \sum_{(l'',j'')\in\mathcal{P}_{\overline{l},\overline{j}}}\sqrt{Level(l,j;l'',j'')}\sqrt{\log\Big(\frac{\overline{b}_{l',j'}-\overline{a}_{l'',j''}}{\overline{b}_{l',j'}-\overline{b}_{l'',j''}}\Big)},
\end{eqnarray}
where we use (\ref{Eq4.21}) for the second inequality. Hence we have
\begin{eqnarray}\label{Eq4.52n}
&&\sum_{\substack{(l',j')\in\mathcal{P}_{l,j}:\\(l',j')\text{ is a type II index}}} \frac{(\theta_{l',j'})^2}{\log(\overline{b}_{l',j'})}\nonumber\\
&\leq& 10^4 \sum_{(l',j')\in \mathcal{P}_{l,j} }\frac{1}{\log(\overline{b}_{l',j'})}\bigg(\sum_{(l'',j'')\in\mathcal{P}_{\overline{l},\overline{j}}}\sqrt{Level(l,j;l'',j'')}\sqrt{\log\Big(\frac{\overline{b}_{l',j'}-\overline{a}_{l'',j''}}{\overline{b}_{l',j'}-\overline{b}_{l'',j''}}\Big)}\bigg)^2\nonumber\\
&=& 10^4  \sum_{(l',j')\in \mathcal{P}_{l,j} }\frac{1}{\log(\overline{b}_{l',j'})}\bigg(\sum_{\substack{(l_1,j_1)\in\mathcal{P}_{\overline{l},\overline{j}},\\(l_2,j_2)\in\mathcal{P}_{\overline{l},\overline{j}}}}\sqrt{Level(l,j;l_1,j_1)}\sqrt{Level(l,j;l_2,j_2)}\nonumber\\
&& \quad\quad\quad\quad\quad\quad\quad\quad\quad\quad\quad \times \sqrt{\log\Big(\frac{\overline{b}_{l',j'}-\overline{a}_{l_1,j_1}}{\overline{b}_{l',j'}-\overline{b}_{l_1,j_1}}\Big)}\sqrt{\log\Big(\frac{\overline{b}_{l',j'}-\overline{a}_{l_2,j_2}}{\overline{b}_{l',j'}-\overline{b}_{l_2,j_2}}\Big)}\bigg)\nonumber\\
&\leq& \frac{10^4}{2} \sum_{\substack{(l',j')\in\mathcal{P}_{l,j},\\(l_1,j_1)\in\mathcal{P}_{\overline{l},\overline{j}},\\(l_2,j_2)\in\mathcal{P}_{\overline{l},\overline{j}}}} \frac{1}{\log(\overline{b}_{l',j'})}(Level(l,j;l_1,j_1)+Level(l,j;l_2,j_2))\nonumber\\
&& \quad\quad\quad\quad\quad\quad\quad\quad\times\sqrt{\log\Big(\frac{\overline{b}_{l',j'}-\overline{a}_{l_1,j_1}}{\overline{b}_{l',j'}-\overline{b}_{l_1,j_1}}\Big)}\sqrt{\log\Big(\frac{\overline{b}_{l',j'}-\overline{a}_{l_2,j_2}}{\overline{b}_{l',j'}-\overline{b}_{l_2,j_2}}\Big)}\nonumber\\
&=& 10^4 \sum_{\substack{(l',j')\in\mathcal{P}_{l,j},\\(l_1,j_1)\in\mathcal{P}_{\overline{l},\overline{j}},\\(l_2,j_2)\in\mathcal{P}_{\overline{l},\overline{j}}}} \frac{Level(l,j;l_1,j_1) }{\log(\overline{b}_{l',j'})}\sqrt{\log\Big(\frac{\overline{b}_{l',j'}-\overline{a}_{l_1,j_1}}{\overline{b}_{l',j'}-\overline{b}_{l_1,j_1}}\Big)}\sqrt{\log\Big(\frac{\overline{b}_{l',j'}-\overline{a}_{l_2,j_2}}{\overline{b}_{l',j'}-\overline{b}_{l_2,j_2}}\Big)}\nonumber\\
&=& 10^4 \sum_{(l_1,j_1)\in\mathcal{P}_{l,j}}Level(l,j;l_1,j_1)\bigg(\sum_{\substack{(l',j')\in \mathcal{P}_{l,j}:\\ (l',j')\succ (l_1,j_1)}}\sum_{(l_2,j_2)\in\mathcal{P}_{l',j'}^{\circ}}\frac{1}{\log(\overline{b}_{l',j'})}\nonumber\\
&&\quad\quad\quad\quad\quad\quad\quad\quad\quad\quad\quad\times\sqrt{\log\Big(\frac{\overline{b}_{l',j'}-\overline{a}_{l_1,j_1}}{\overline{b}_{l',j'}-\overline{b}_{l_1,j_1}}\Big)}\sqrt{\log\Big(\frac{\overline{b}_{l',j'}-\overline{a}_{l_2,j_2}}{\overline{b}_{l',j'}-\overline{b}_{l_2,j_2}}\Big)}\bigg),
\end{eqnarray}
where we use the AM-GM inequality in the fifth line.

Consider any $(l',j')\in\mathcal{R}^{EVEN}_T$ such that $l'\geq 1$. Let
\begin{equation}\label{Eq4.53}
\mathcal{T}_{l',j'}:=\sum_{\substack{(l_2,j_2)\in\mathcal{P}_{l',j'}^{\circ}}}\sqrt{\log\Big(\frac{\overline{b}_{l',j'}-\overline{a}_{l_2,j_2}}{\overline{b}_{l',j'}-\overline{b}_{l_2,j_2}}\Big)}.
\end{equation}
In the following, we bound $\mathcal{T}_{l',j'}$. We consider the following two cases separately: $1\leq j'\leq 2l'-1$ or $j'=2l'$. For any $t\in \mathbb{N}^{*}$ and $s\in \{0\}\cup [4t]$, we let
\begin{equation}
  \alpha_{t,s}:=2^t\Big(1+\frac{s}{4t}\Big). 
\end{equation}

\subparagraph{Case 1: $1\leq j'\leq 2l'-1$}

In this case, we have 
\begin{equation*}
 \overline{b}_{l',j'}=2^{l'}\Big(1+\frac{2j'+1}{4l'}\Big)=\alpha_{l',2j'+1}.
\end{equation*}
As $\sqrt{a+b}\leq \sqrt{a}+\sqrt{b}$ for any $a,b\geq 0$, we have
\begin{eqnarray}\label{Eq4.22}
 \mathcal{T}_{l',j'}
&\leq& \sqrt{\log\Big(\frac{\alpha_{l',2j'+1}}{\alpha_{l',2j'+1}-2}\Big)}+\sum_{t=1}^{l'-1}\sum_{s=1}^{4t}\sqrt{\log\Big(\frac{\alpha_{l',2j'+1}-\alpha_{t,s-1}}{\alpha_{l',2j'+1}-\alpha_{t,s}}\Big)}\nonumber\\
&& + \sum_{s=1}^{2j'} \sqrt{\log\Big(\frac{\alpha_{l',2j'+1}-\alpha_{l',s-1}}{\alpha_{l',2j'+1}-\alpha_{l',s}}\Big)}.
\end{eqnarray}

Note that $\alpha_{l',2j'+1}\geq 3$. Hence
\begin{equation}\label{Eq4.23}
   \sqrt{\log\Big(\frac{\alpha_{l',2j'+1}}{\alpha_{l',2j'+1}-2}\Big)}=\sqrt{\log\Big(1+\frac{2}{\alpha_{l',2j'+1}-2}\Big)}\leq \sqrt{\log(3)}\leq 2.
\end{equation}
Moreover, we have
\begin{eqnarray}\label{Eq4.24}
&& \sum_{s=1}^{2j'}\sqrt{\log\Big(\frac{\alpha_{l',2j'+1}-\alpha_{l',s-1}}{\alpha_{l',2j'+1}-\alpha_{l',s}}\Big)}= \sum_{s=1}^{2j'}\sqrt{\log\Big(1+\frac{1}{2j'+1-s}\Big)}\nonumber\\
&\leq& \sum_{s=1}^{2j'} \sqrt{\frac{1}{2j'+1-s}}=\sum_{s=1}^{2j'}s^{-1\slash 2}\leq \sum_{s=1}^{4l'} s^{-1\slash 2}\nonumber\\
&\leq& \int_{0}^{4l'} x^{-1\slash 2}dx=4\sqrt{l'},
\end{eqnarray}
where we use $\log(1+x)\leq x,\forall x>-1$ in the second line.

If $l'\geq 2$, we consider any $t\in [1,l'-1]\cap\mathbb{N}$ below. Note that
\begin{equation}
   \alpha_{l',2j'+1}\geq 2^{l'}\Big(1+\frac{1}{2l'}\Big).
\end{equation}
If $t=l'-1$, for any $s\in [4t]$, as $l'\geq 2$, we have 
\begin{eqnarray}
   \alpha_{l',2j'+1}-\alpha_{t,s}&\geq& 2^{l'}\Big(1+\frac{1}{2l'}\Big)-2^{l'-1}\Big(1+\frac{s}{4(l'-1)}\Big) \nonumber\\
   &\geq& 2^{l'-1} \Big(1-\frac{s-1}{4(l'-1)}\Big).
\end{eqnarray}
Hence
\begin{eqnarray}\label{Eq4.25}
&& \sum_{s=1}^{4t}\sqrt{\log\Big(\frac{\alpha_{l',2j'+1}-\alpha_{t,s-1}}{\alpha_{l',2j'+1}-\alpha_{t,s}}\Big)}=\sum_{s=1}^{4t}  \sqrt{\log\Big(1+\frac{\alpha_{t,s}-\alpha_{t,s-1}}{\alpha_{l',2j'+1}-\alpha_{t,s}}\Big)}\nonumber\\
&\leq& \sum_{s=1}^{4t}\sqrt{\log\Big(1+\frac{2^{l'-1}\slash (4(l'-1))}{2^{l'-1}(1-(s-1)\slash(4(l'-1)))}\Big)}\nonumber\\
&\leq& \sum_{s=1}^{4(l'-1)} \sqrt{\frac{1}{4(l'-1)-(s-1)}}=\sum_{s=1}^{4(l'-1)} s^{-1\slash 2}\nonumber\\
&\leq& \int_{0}^{4(l'-1)}x^{-1\slash 2}dx\leq 4\sqrt{l'-1},
\end{eqnarray}
where we use $\log(1+x)\leq x,\forall x>-1$ in the third line. If $l'\geq 3$ and $t\leq l'-2$, for any $s\in [4t]$, we have
\begin{equation}
 \alpha_{l',2j'+1}-\alpha_{t,s}\geq 2^{l'}\Big(1+\frac{1}{2l'}\Big)-2^{l'-1}\geq 2^{l'-1}. 
\end{equation}
Hence
\begin{eqnarray}\label{Eq4.26}
&& \sum_{s=1}^{4t}\sqrt{\log\Big(\frac{\alpha_{l',2j'+1}-\alpha_{t,s-1}}{\alpha_{l',2j'+1}-\alpha_{t,s}}\Big)}=\sum_{s=1}^{4t}  \sqrt{\log\Big(1+\frac{\alpha_{t,s}-\alpha_{t,s-1}}{\alpha_{l',2j'+1}-\alpha_{t,s}}\Big)}\nonumber\\
&\leq& \sum_{s=1}^{4t} \sqrt{\log\Big(1+\frac{2^t\slash (4t)}{2^{l'-1}}\Big)}\leq \sum_{s=1}^{4t} \sqrt{\frac{2^{t-l'}}{2t}}\leq 4 \cdot 2^{(t-l')\slash 2}\sqrt{t},
\end{eqnarray}
where we use $\log(1+x)\leq x,\forall x>-1$ in the second line.

By (\ref{Eq4.22})-(\ref{Eq4.24}), (\ref{Eq4.25}), and (\ref{Eq4.26}), as $l'\geq 1$, we have
\begin{eqnarray}\label{Eq4.31}
 && \mathcal{T}_{l',j'}\leq 2+4\sqrt{l'}+4\sqrt{l'-1}+4\sum_{t=1}^{l'-2} 2^{(t-l')\slash 2}\sqrt{t} \nonumber\\
&\leq& 10\sqrt{l'}+4\sqrt{l'}\sum_{t=1}^{l'-2}2^{(t-l')\slash 2} \leq 10\sqrt{l'}+4\Big(\sum_{s=0}^{\infty} 2^{-s\slash 2}\Big)\sqrt{l'}\leq 30\sqrt{l'}.
\end{eqnarray}

\subparagraph{Case 2: $j'=2l'$}

In this case, we have
\begin{equation}
 \overline{b}_{l',j'}=2^{l'+1}\Big(1+\frac{1}{4(l'+1)}\Big)=\alpha_{l'+1,1}.
\end{equation}
As $\sqrt{a+b}\leq \sqrt{a}+\sqrt{b}$ for any $a,b\geq 0$, we have
\begin{equation}\label{Eq4.27}
\mathcal{T}_{l',j'}
\leq\sqrt{\log\Big(\frac{\alpha_{l'+1,1}}{\alpha_{l'+1,1}-2}\Big)}+\sum_{t=1}^{l'}\sum_{s=1}^{4t}\sqrt{\log\Big(\frac{\alpha_{l'+1,1}-\alpha_{t,s-1}}{\alpha_{l'+1,1}-\alpha_{t,s}}\Big)}.
\end{equation}

Note that $\alpha_{l'+1,1}\geq 4$. Hence
\begin{equation}\label{Eq4.28}
   \sqrt{\log\Big(\frac{\alpha_{l'+1,1}}{\alpha_{l'+1,1}-2}\Big)}=\sqrt{\log\Big(1+\frac{2}{\alpha_{l'+1,1}-2}\Big)}\leq \sqrt{\log(2)}\leq 1.
\end{equation}

Consider the case where $t=l'$. For any $s\in [4l']$, as $l'\geq 1$, we have
\begin{eqnarray}
 && \alpha_{l'+1,1}-\alpha_{t,s} = 2^{l'+1}\Big(1+\frac{1}{4(l'+1)}\Big)-2^{l'}\Big(1+\frac{s}{4l'}\Big) \nonumber\\
 &=& 2^{l'}\Big(\frac{2}{4(l'+1)}+\frac{4l'-s}{4l'}\Big) \geq \frac{2^{l'}(4l'-s+1)}{4l'}.
\end{eqnarray}
Hence
\begin{eqnarray}\label{Eq4.29}
&& \sum_{s=1}^{4t}\sqrt{\log\Big(\frac{\alpha_{l'+1,1}-\alpha_{t,s-1}}{\alpha_{l'+1,1}-\alpha_{t,s}}\Big)}=\sum_{s=1}^{4t}  \sqrt{\log\Big(1+\frac{\alpha_{t,s}-\alpha_{t,s-1}}{\alpha_{l'+1,1}-\alpha_{t,s}}\Big)}\nonumber\\
&\leq& \sum_{s=1}^{4t}\sqrt{\log\Big(1+\frac{2^{l'}\slash (4l')}{2^{l'}(4l'-s+1)\slash (4l')}\Big)}\leq \sum_{s=1}^{4l'} \sqrt{\frac{1}{4l'-s+1}}\nonumber\\
&=&\sum_{s=1}^{4l'} s^{-1\slash 2}\leq \int_{0}^{4l'}x^{-1\slash 2}dx\leq 4\sqrt{l'},
\end{eqnarray}
where we use $\log(1+x)\leq x, \forall  x>-1$ in the second line.

Now if $l'\geq 2$, we consider any $t\in [1,l'-1]\cap\mathbb{N}$. For any $s\in [4t]$, we have
\begin{equation}
  \alpha_{l'+1,1}-\alpha_{t,s}\geq 2^{l'+1}-2^{l'}=2^{l'}.
\end{equation}
Hence we have
\begin{eqnarray}\label{Eq4.30}
&& \sum_{s=1}^{4t}\sqrt{\log\Big(\frac{\alpha_{l'+1,1}-\alpha_{t,s-1}}{\alpha_{l'+1,1}-\alpha_{t,s}}\Big)}=\sum_{s=1}^{4t}  \sqrt{\log\Big(1+\frac{\alpha_{t,s}-\alpha_{t,s-1}}{\alpha_{l'+1,1}-\alpha_{t,s}}\Big)}\nonumber\\
&\leq& \sum_{s=1}^{4t}\sqrt{\log\Big(1+\frac{2^t\slash (4t)}{2^{l'}}\Big)}\leq \sum_{s=1}^{4t} \sqrt{\frac{2^{t-l'}}{4t}}= 2\cdot 2^{(t-l')\slash 2}\sqrt{t},
\end{eqnarray}
where we use the inequality $\log(1+x)\leq x$ for any $x>-1$ in the second line.

By (\ref{Eq4.27}), (\ref{Eq4.28}), (\ref{Eq4.29}), and (\ref{Eq4.30}), as $l'\geq 1$, we have
\begin{eqnarray}\label{Eq4.32}
    \mathcal{T}_{l',j'}&\leq& 1+4\sqrt{l'}+2\sum_{t=1}^{l'-1} 2^{(t-l')\slash 2}\sqrt{t} \leq 5\sqrt{l'}+2\sqrt{l'} \sum_{t=1}^{l'-1}2^{(t-l')\slash 2} \nonumber\\
    &\leq& 5\sqrt{l'}+2\Big(\sum_{s=0}^{\infty}2^{-s\slash 2}\Big)\sqrt{l'}\leq 20\sqrt{l'}.
\end{eqnarray}

\bigskip

Consider any $(l,j)\in\mathcal{R}^{EVEN}_T$. For any $(l_1,j_1)\in\mathcal{P}_{l,j}$, let
\begin{equation}\label{Eq4.54}
 \overline{\mathcal{T}}_{l_1,j_1}:=\sum_{\substack{(l',j')\in \mathcal{P}_{l,j}:\\ (l',j')\succ (l_1,j_1)}}\frac{\mathcal{T}_{l',j'}}{\log(\overline{b}_{l',j'})}\sqrt{\log\Big(\frac{\overline{b}_{l',j'}-\overline{a}_{l_1,j_1}}{\overline{b}_{l',j'}-\overline{b}_{l_1,j_1}}\Big)}.
\end{equation}
Note that for any $(l',j')\in\mathcal{R}^{EVEN}_T$ such that $l'\geq 1$, we have
\begin{equation}\label{Eq4.33}
\log(\overline{b}_{l',j'})\geq \log(2^{l'})\geq \frac{3}{5}l'.
\end{equation}

By (\ref{Eq4.31}), (\ref{Eq4.32}), and (\ref{Eq4.33}), we have
\begin{eqnarray}\label{Eq4.38}
 \overline{\mathcal{T}}_{l_1,j_1}\leq  50\sum_{\substack{(l',j')\in \mathcal{P}_{l,j}:\\ (l',j')\succ (l_1,j_1)}} (l')^{-1\slash 2} \sqrt{\log\Big(\frac{\overline{b}_{l',j'}-\overline{a}_{l_1,j_1}}{\overline{b}_{l',j'}-\overline{b}_{l_1,j_1}}\Big)}.
\end{eqnarray}

For any $(l_1,j_1)\in \mathcal{R}^{EVEN}_T$, if $l_1\geq 1$, we have $\overline{b}_{l_1,j_1}-\overline{a}_{l_1,j_1}\leq 2^{l_1+1}\slash (l_1+1)$; if $l_1=0$, we have $\overline{b}_{l_1,j_1}-\overline{a}_{l_1,j_1}=5\slash 2$. Hence for any $(l_1,j_1)\in\mathcal{R}^{EVEN}_T$, we have 
\begin{equation}
 \overline{b}_{l_1,j_1}-\overline{a}_{l_1,j_1}\leq \frac{4\cdot 2^{l_1}}{l_1+1}.
\end{equation}
Hence for any $(l_1,j_1),(l',j')\in \mathcal{P}_{l,j}$ such that $(l',j')\succ (l_1,j_1)$, we have
\begin{eqnarray}\label{Eq4.44}
 && \sqrt{\log\Big(\frac{\overline{b}_{l',j'}-\overline{a}_{l_1,j_1}}{\overline{b}_{l',j'}-\overline{b}_{l_1,j_1}}\Big)}=\sqrt{\log\Big(1+\frac{\overline{b}_{l_1,j_1}-\overline{a}_{l_1,j_1}}{\overline{b}_{l',j'}-\overline{b}_{l_1,j_1}}\Big)} \nonumber\\
 &\leq& \sqrt{\frac{\overline{b}_{l_1,j_1}-\overline{a}_{l_1,j_1}}{\overline{b}_{l',j'}-\overline{b}_{l_1,j_1}}}\leq 2\sqrt{\frac{2^{l_1}}{l_1+1}}\sqrt{\frac{1}{\overline{b}_{l',j'}-\overline{b}_{l_1,j_1}}}.
\end{eqnarray}

Below we consider any $(l_1,j_1)\in\mathcal{P}_{l,j}$. We deal with the following three cases separately: $l_1\geq 1$ and $1\leq j_1\leq 2l_1-1$, or $l_1\geq 1$ and $j_1=2l_1$, or $(l_1,j_1)=(0,2)$.

\subparagraph{Case 1: $l_1\geq 1$ and $1\leq j_1\leq 2l_1-1$}

For this case, we have 
\begin{equation*}
\overline{b}_{l_1,j_1}=2^{l_1}\Big(1+\frac{2j_1+1}{4l_1}\Big)=\alpha_{l_1,2j_1+1}. 
\end{equation*}
Hence 
\begin{eqnarray}\label{Eq4.34}
  && \sum_{\substack{(l',j')\in \mathcal{P}_{l,j}:\\ (l',j')\succ (l_1,j_1)}} (l')^{-1\slash 2}\sqrt{\frac{1}{\overline{b}_{l',j'}-\overline{b}_{l_1,j_1}}}
   \leq  l_1^{-1\slash 2}\sum_{s=2j_1+2}^{4l_1}\sqrt{\frac{1}{\alpha_{l_1,s}-\alpha_{l_1,2j_1+1}}}\nonumber\\
   && \quad\quad\quad \quad\quad\quad\quad\quad +\sum_{t=l_1+1}^{\infty}(t-1)^{-1\slash 2} \bigg(\sum_{s=1}^{4t} \sqrt{\frac{1}{\alpha_{t,s}-\alpha_{l_1,2j_1+1}}}\bigg).
\end{eqnarray}

For any $s\in [2j_1+2,4l_1]\cap\mathbb{N}$, we have
\begin{equation}
  \alpha_{l_1,s}-\alpha_{l_1,2j_1+1}=\frac{2^{l_1}(s-(2j_1+1))}{4l_1}. 
\end{equation}
Hence
\begin{eqnarray}\label{Eq4.35}
&& \sum_{s=2j_1+2}^{4l_1}\sqrt{\frac{1}{\alpha_{l_1,s}-\alpha_{l_1,2j_1+1}}}=2^{1-l_1\slash 2}\sqrt{l_1}\sum_{s=2j_1+2}^{4 l_1}\frac{1}{\sqrt{s-(2j_1+1)}}\nonumber\\
&\leq& 2^{1-l_1\slash 2}\sqrt{l_1}\sum_{s=1}^{4l_1}s^{-1\slash 2}\leq 2^{1-l_1\slash 2}\sqrt{l_1}\int_{0}^{4l_1} x^{-1\slash 2}dx\leq 8\cdot 2^{-l_1\slash 2}l_1.
\end{eqnarray}

For any $s\in [4(l_1+1)]$, we have
\begin{equation}
  \alpha_{l_1+1,s}-\alpha_{l_1,2j_1+1}\geq 2^{l_1+1}\Big(1+\frac{s}{4(l_1+1)}\Big)-2^{l_1+1}=\frac{2^{l_1-1}s}{l_1+1}.
\end{equation}
Hence
\begin{eqnarray}\label{Eq4.36}
 && \sum_{s=1}^{4(l_1+1)}\sqrt{\frac{1}{\alpha_{l_1+1,s}-\alpha_{l_1,2j_1+1}}}\leq 2^{(1-l_1)\slash 2}\sqrt{l_1+1}\sum_{s=1}^{4(l_1+1)} s^{-1\slash 2} \nonumber\\
 &\leq& 2^{(1-l_1)\slash 2}\sqrt{l_1+1}\int_0^{4(l_1+1)}x^{-1\slash 2} dx \leq 8\cdot 2^{-(l_1+1)\slash 2} (l_1+1).
\end{eqnarray}

For any $t\in\mathbb{N}$ with $t\geq l_1+2$ and any $s\in [4t]$, we have
\begin{equation}
  \alpha_{t,s}-\alpha_{l_1,2j_1+1}\geq 2^t-2^{l_1+1}\geq 2^{t-1}. 
\end{equation}
Hence for any $t\in\mathbb{N}$ with $t\geq l_1+2$, we have
\begin{eqnarray}\label{Eq4.37}
 && \sum_{s=1}^{4t} \sqrt{\frac{1}{\alpha_{t,s}-\alpha_{l_1,2j_1+1}}}\leq \sum_{s=1}^{4t} 2^{-(t-1)\slash 2}\leq 8\cdot 2^{-t\slash 2} t.
\end{eqnarray}

By (\ref{Eq4.34}), (\ref{Eq4.35}), (\ref{Eq4.36}), and (\ref{Eq4.37}), we have
\begin{eqnarray}\label{Eq4.39}
&&\sum_{\substack{(l',j')\in \mathcal{P}_{l,j}:\\ (l',j')\succ (l_1,j_1)}} (l')^{-1\slash 2}\sqrt{\frac{1}{\overline{b}_{l',j'}-\overline{b}_{l_1,j_1}}}\nonumber\\
&\leq& 8\cdot 2^{-l_1\slash 2}\sqrt{l_1}+8\cdot 2^{-(l_1+1)\slash 2}(l_1+1) l_1^{-1\slash 2}+8\sum_{t=l_1+2}^{\infty}2^{-t\slash 2}t (t-1)^{-1\slash 2}\nonumber\\
&\leq& 8\cdot 2^{-l_1\slash 2}\sqrt{l_1}+16\sum_{t=l_1}^{\infty} 2^{-t\slash 2}\sqrt{t}\leq 24\sum_{t=l_1}^{\infty} 2^{-t\slash 2}\sqrt{t}.
\end{eqnarray}

By (\ref{Eq4.38}), (\ref{Eq4.44}), and (\ref{Eq4.39}), we have
\begin{eqnarray}\label{Eq4.45}
  &&\overline{\mathcal{T}}_{l_1,j_1}\leq 2400 \cdot2^{l_1\slash 2}(l_1+1)^{-1\slash 2} \sum_{t=l_1}^{\infty} 2^{-t\slash 2}\sqrt{t} \nonumber\\
&\leq& 2400\cdot 2^{l_1\slash 2} l_1^{-1}\sum_{t=l_1}^{\infty} 2^{-t\slash 2} t\nonumber\\
&=&  2400\cdot 2^{l_1\slash 2} l_1^{-1}\frac{2^{-1\slash 2}(l_1 2^{-(l_1-1)\slash 2}-(l_1-1)2^{-l_1\slash 2})}{(1-2^{-1\slash 2})^2}\leq 30000.
\end{eqnarray}

\subparagraph{Case 2: $l_1\geq 1$ and $j_1=2l_1$} 

For this case, we have
\begin{equation}
\overline{b}_{l_1,j_1}=2^{l_1+1}\Big(1+\frac{1}{4(l_1+1)}\Big)=\alpha_{l_1+1,1}.
\end{equation}
Hence 
\begin{eqnarray}\label{Eq4.40}
  && \sum_{\substack{(l',j')\in \mathcal{P}_{l,j}:\\ (l',j')\succ (l_1,j_1)}} (l')^{-1\slash 2}\sqrt{\frac{1}{\overline{b}_{l',j'}-\overline{b}_{l_1,j_1}}}
   \leq  l_1^{-1\slash 2}\sum_{s=2}^{4(l_1+1)}\sqrt{\frac{1}{\alpha_{l_1+1,s}-\alpha_{l_1+1,1}}}\nonumber\\
   && \quad\quad\quad \quad\quad\quad\quad\quad +\sum_{t=l_1+2}^{\infty}(t-1)^{-1\slash 2} \bigg(\sum_{s=1}^{4t} \sqrt{\frac{1}{\alpha_{t,s}-\alpha_{l_1+1,1}}}\bigg).
\end{eqnarray}

For any $s\in [2, 4(l_1+1)]\cap\mathbb{N}$, we have
\begin{equation}
   \alpha_{l_1+1,s}-\alpha_{l_1+1,1}=\frac{2^{l_1}(s-1)}{2(l_1+1)}.
\end{equation}
Hence
\begin{eqnarray}\label{Eq4.41}
&&\sum_{s=2}^{4(l_1+1)}\sqrt{\frac{1}{\alpha_{l_1+1,s}-\alpha_{l_1+1,1}}}=2^{-(l_1-1)\slash 2}\sqrt{l_1+1}\sum_{s=2}^{4(l_1+1)}\frac{1}{\sqrt{s-1}}\nonumber\\
&\leq& 2^{1-l_1\slash 2}\sqrt{l_1}\sum_{s=1}^{4(l_1+1)}s^{-1\slash 2}\leq 2^{1-l_1\slash 2}\sqrt{l_1}\int_0^{4(l_1+1)}x^{-1\slash 2} dx\nonumber\\
&=& 8\cdot 2^{-l_1\slash 2}\sqrt{l_1(l_1+1)}\leq 16\cdot 2^{-l_1\slash 2} l_1.
\end{eqnarray}

For any $t\in\mathbb{N}$ with $t\geq l_1+2$ and any $s\in [4t]$, we have
\begin{equation}
  \alpha_{t,s}-\alpha_{l_1+1,1}\geq 2^{t}-2^{l_1+1}\Big(1+\frac{1}{4(l_1+1)}\Big)\geq 2^{t-2}.
\end{equation}
Hence for any $t\in\mathbb{N}$ with $t\geq l_1+2$, we have
\begin{eqnarray}\label{Eq4.42}
  \sum_{s=1}^{4t} \sqrt{\frac{1}{\alpha_{t,s}-\alpha_{l_1+1,1}}}\leq \sum_{s=1}^{4t} 2^{1-t\slash 2}=8\cdot 2^{-t\slash 2} t.
\end{eqnarray}

By (\ref{Eq4.40}), (\ref{Eq4.41}), and (\ref{Eq4.42}), we have
\begin{eqnarray}\label{Eq4.43}
  && \sum_{\substack{(l',j')\in \mathcal{P}_{l,j}:\\ (l',j')\succ (l_1,j_1)}} (l')^{-1\slash 2}\sqrt{\frac{1}{\overline{b}_{l',j'}-\overline{b}_{l_1,j_1}}}\nonumber\\
  &\leq &  16\cdot 2^{-l_1\slash 2}\sqrt{l_1}+8\sum_{t=l_1+2}^{\infty} (t-1)^{-1\slash 2} 2^{-t\slash 2}t\nonumber\\
  &\leq& 16\cdot 2^{-l_1\slash 2}\sqrt{l_1}+16\sum_{t=l_1+1}^{\infty} 2^{-t\slash 2}\sqrt{t}= 16\sum_{t=l_1}^{\infty} 2^{-t\slash 2}\sqrt{t}.
\end{eqnarray}

By (\ref{Eq4.38}), (\ref{Eq4.44}), and (\ref{Eq4.43}), following the calculation in (\ref{Eq4.45}), we have
\begin{eqnarray}\label{Eq4.51}
\overline{\mathcal{T}}_{l_1,j_1}\leq 1600\cdot 2^{l_1\slash 2}(l_1+1)^{-1\slash 2}\sum_{t=l_1}^{\infty} 2^{-t\slash 2}\sqrt{t}\leq 30000.
\end{eqnarray}

\subparagraph{Case 3: $(l_1,j_1)=(0,2)$}

For this case, we have $\overline{b}_{l_1,j_1}=5\slash 2$. Hence
\begin{eqnarray}\label{Eq4.46}
  && \sum_{\substack{(l',j')\in \mathcal{P}_{l,j}:\\ (l',j')\succ (l_1,j_1)}} (l')^{-1\slash 2}\sqrt{\frac{1}{\overline{b}_{l',j'}-\overline{b}_{l_1,j_1}}}
   \leq  \sum_{s=2}^{4}\sqrt{\frac{1}{\alpha_{1,s}-5\slash 2}}\nonumber\\
   && \quad\quad\quad \quad\quad\quad\quad\quad +\sum_{t=2}^{\infty}(t-1)^{-1\slash 2} \bigg(\sum_{s=1}^{4t} \sqrt{\frac{1}{\alpha_{t,s}-5\slash 2}}\bigg).
\end{eqnarray}

For any $s\in\{2,3,4\}$, $\alpha_{1,s}-5\slash 2\geq 1\slash 2$, hence
\begin{equation}\label{Eq4.47}
 \sum_{s=2}^{4}\sqrt{\frac{1}{\alpha_{1,s}-5\slash 2}}\leq 4\sqrt{2}\leq 6.
\end{equation}

For any $t\in\mathbb{N}$ with $t\geq 2$ and any $s\in [4t]$, we have
\begin{equation}
  \alpha_{t,s}-5\slash 2\geq 2^t-5\slash 2\geq 2^{t-2}.
\end{equation}
Hence for any $t\in\mathbb{N}$ with $t\geq 2$, we have
\begin{eqnarray}\label{Eq4.48}
\sum_{s=1}^{4t} \sqrt{\frac{1}{\alpha_{t,s}-5\slash 2}}\leq \sum_{s=1}^{4t} 2^{-t\slash 2+1}=8\cdot 2^{-t\slash 2} t.
\end{eqnarray}

By (\ref{Eq4.46}), (\ref{Eq4.47}), and (\ref{Eq4.48}), we have
\begin{eqnarray}\label{Eq4.49}
&& \sum_{\substack{(l',j')\in \mathcal{P}_{l,j}:\\ (l',j')\succ (l_1,j_1)}} (l')^{-1\slash 2}\sqrt{\frac{1}{\overline{b}_{l',j'}-\overline{b}_{l_1,j_1}}}\leq 6+8\sum_{t=2}^{\infty}(t-1)^{-1\slash 2}2^{-t\slash 2}t \nonumber\\
&\leq& 16\sum_{t=1}^{\infty} 2^{-t\slash 2}\sqrt{t}\leq 16\sum_{t=1}^{\infty} 2^{-t\slash 2}t=\frac{16\cdot 2^{-1\slash 2}}{(1-2^{-1\slash 2})^2}\leq 200. 
\end{eqnarray}

By (\ref{Eq4.38}), (\ref{Eq4.44}), and (\ref{Eq4.49}), we have
\begin{equation}\label{Eq4.50}
\overline{\mathcal{T}}_{l_1,j_1}\leq 20000.
\end{equation}

\bigskip

By (\ref{Eq4.45}), (\ref{Eq4.51}), and (\ref{Eq4.50}), we conclude that for any $(l_1,j_1)\in\mathcal{P}_{l,j}$,
\begin{equation}\label{Eq4.55}
 \overline{\mathcal{T}}_{l_1,j_1}\leq 30000.
\end{equation}

Therefore, for any $(l,j)\in\mathcal{R}^{EVEN}_T$, by (\ref{Eq4.52n}), (\ref{Eq4.53}), (\ref{Eq4.54}), and (\ref{Eq4.55}), 
\begin{eqnarray}\label{Eq4.52}
\sum_{\substack{(l',j')\in\mathcal{P}_{l,j}:\\(l',j')\text{ is a type II index}}} \frac{(\theta_{l',j'})^2}{\log(\overline{b}_{l',j'})}&\leq& 10^4\sum_{(l_1,j_1)\in \mathcal{P}_{l,j}} Level(l,j;l_1,j_1) \overline{\mathcal{T}}_{l_1,j_1}\nonumber\\
&\leq& 3\cdot 10^8 \sum_{(l_1,j_1)\in\mathcal{P}_{l,j}} Level(l,j;l_1,j_1) \nonumber\\
&=& 3\cdot 10^8 \sum_{(l',j')\in\mathcal{P}_{l,j}} Level(l,j;l',j').
\end{eqnarray}

By (\ref{Eq4.1}) and (\ref{Eq4.52}), for any $(l,j)\in\mathcal{R}^{EVEN}_T$, we have
\begin{equation}\label{Eq4.110}
 \sum_{\substack{(l',j')\in\mathcal{P}_{l,j}}}\frac{(\theta_{l',j'})^2}{\log(\overline{b}_{l',j'})} \leq 4\cdot 10^8 \sum_{(l',j')\in\mathcal{P}_{l,j}} Level(l,j;l',j').
\end{equation}

\paragraph{Step 4}

Recall that we have constructed $Level(l,j;l',j')$ and $Index(l,j;l',j')$ for any $(l,j)\in\mathcal{R}^T_{EVEN}$ and $(l',j')\in \mathcal{P}_{l,j}$ in \textbf{Step 1}. For any $(l,j)\in\mathcal{R}^{EVEN}_T$, we let
\begin{equation*}
 Index(l,j):=Index(T,2T;l,j), \quad Level(l,j):=Level(T,2T;l,j).
\end{equation*}

In the following, we assume that the event $\mathcal{H}(\mathbf{z}, \bm{\theta}, \bm{\theta}')$ holds. By (\ref{DefinitionH}), for any $(l,j)\in\mathcal{R}^{EVEN}_T$ and $(l',j')\in \mathcal{P}_{l,j}$, we have 
\begin{equation}
 J_{l,j}=z_{l,j}, \quad \theta_{l,j;l',j'} k \leq\Delta_{l,j;l',j';1}\leq 2\theta_{l,j;l',j'} k+ (l+1)^{-10} k.
\end{equation}
Noting (\ref{Eq2.7}), we obtain that
\begin{eqnarray}\label{Eq4.100}
&& \theta_{l,j;l',j'} k \nonumber\\
&\leq& |N(z_{l,j}k^{2\slash 3};\overline{a}_{l',j'}k^{2\slash 3}, \overline{b}_{l',j'}k^{2\slash 3})-N_0(z_{l,j}k^{2\slash 3};\overline{a}_{l',j'}k^{2\slash 3}, \overline{b}_{l',j'}k^{2\slash 3})|\nonumber\\
&\leq& 2\theta_{l,j;l',j'} k+ (l+1)^{-10} k.
\end{eqnarray}

For every $(l,j)\in\mathcal{R}^{EVEN}_{T}$, we denote $(x_{l,j},y_{l,j})=Index(l,j)$. Note that by (\ref{Eq4.57}), $(x_{l,j},y_{l,j})\in\mathcal{R}^{EVEN}_T$ and $(l,j)\in\mathcal{P}_{x_{l,j},y_{l,j}}$. Replacing $(l,j)$ and $(l',j')$ by $(x_{l,j},y_{l,j})$ and $(l,j)$ respectively in (\ref{Eq4.100}), we obtain that
\begin{eqnarray}\label{Eq4.102}
 && \mathcal{H}(\mathbf{z}, \bm{\theta}, \bm{\theta}')\nonumber\\
 &\subseteq& \{|N(z_{x_{l,j},y_{l,j}}k^{2\slash 3};\overline{a}_{l,j}k^{2\slash 3}, \overline{b}_{l,j}k^{2\slash 3})-N_0(z_{x_{l,j},y_{l,j}}k^{2\slash 3};\overline{a}_{l,j}k^{2\slash 3}, \overline{b}_{l,j}k^{2\slash 3})|\nonumber\\
 && \quad \geq \theta_{x_{l,j},y_{l,j};l,j} k\}. 
\end{eqnarray}

For any $(l,j)\in\mathcal{R}^{EVEN}_T$, we let $\mathcal{U}_{l,j}$ be the event that
\begin{eqnarray}
&& |N(z_{x_{l,j},y_{l,j}}k^{2\slash 3};\overline{a}_{l,j}k^{2\slash 3}, \overline{b}_{l,j}k^{2\slash 3})-N_0(z_{x_{l,j},y_{l,j}}k^{2\slash 3};\overline{a}_{l,j}k^{2\slash 3}, \overline{b}_{l,j}k^{2\slash 3})|\nonumber\\
&& \geq \theta_{x_{l,j},y_{l,j};l,j} k. 
\end{eqnarray}
Note that by (\ref{Eq4.102}), we have
\begin{equation}\label{Eq4.103}
 \mathcal{H}(\mathbf{z}, \bm{\theta}, \bm{\theta}')\subseteq \bigcap_{(l,j)\in\mathcal{R}^{EVEN}_T} \mathcal{U}_{l,j}.
\end{equation}

In the following, we consider an arbitrary $(l,j)\in\mathcal{R}^{EVEN}_T$. First consider the case where $(x_{l,j},y_{l,j})\succ (l,j)$. In Proposition \ref{P3.1}, take 
\begin{equation*}
\lambda_1=\overline{a}_{l,j}k^{2\slash 3}, \quad \lambda_2=\overline{b}_{l,j} k^{2\slash 3}, \quad \lambda=z_{x_{l,j},y_{l,j}} k^{2\slash 3}.
\end{equation*}
Note that $\lambda_2-\lambda_1\geq (b_{l,j}-a_{l,j})k^{2\slash 3}\geq k^{2\slash 3}$. As $z_{x_{l,j},y_{l,j}}\in \mathcal{Q}_{x_{l,j},y_{l,j}}\subseteq I_{x_{l,j},y_{l,j}}$, we have 
\begin{equation}\label{Eq4.64}
a_{x_{l,j},y_{l,j}}\leq z_{x_{l,j},y_{l,j}}\leq b_{x_{l,j},y_{l,j}}.
\end{equation}
As $(l,j)\prec (x_{l,j},y_{l,j})$ and $(l,j),(x_{l,j},y_{l,j})\in\mathcal{R}^{EVEN}_T$, we have 
\begin{equation*}
\lambda-\lambda_2\geq(a_{x_{l,j},y_{l,j}}-\overline{b}_{l,j})k^{2\slash 3}\geq k^{2\slash 3}\slash  2.
\end{equation*}
Hence
\begin{equation}\label{Eq4.61}
 0\leq \lambda_1<\lambda_2<\lambda, \quad \min\{\lambda_2-\lambda_1,\lambda-\lambda_2\}\geq k^{2\slash 3}\slash 2.
\end{equation}
Moreover, note that
\begin{equation}\label{Eq4.58}
 \frac{\lambda-\lambda_1}{\lambda-\lambda_2}=1+\frac{\overline{b}_{l,j}-\overline{a}_{l,j}}{z_{x_{l,j},y_{l,j}}-\overline{b}_{l,j}}\leq 1+\frac{\overline{b}_{l,j}-\overline{a}_{l,j}}{a_{x_{l,j},y_{l,j}}-\overline{b}_{l,j}}.
\end{equation}
We can check that 
\begin{equation*}
 a_{x_{l,j},y_{l,j}}-\overline{b}_{l,j}\geq \begin{cases}
   \frac{1}{2}, & x_{l,j}=1\\
   \frac{2^{x_{l,j}-1}}{4(x_{l,j}-1)}, & x_{l,j}\geq 2
 \end{cases},
\end{equation*}
hence 
\begin{equation}\label{Eq4.59}
a_{x_{l,j},y_{l,j}}-\overline{b}_{l,j}\geq \frac{2^{x_{l,j}}}{8x_{l,j}}\geq \frac{2^{x_{l,j}+1}}{16 (x_{l,j}+1)}.
\end{equation}
If $(l,j)=(0,2)$, we have $\overline{b}_{l,j}-\overline{a}_{l,j}=5\slash 2$; if $l\geq 1$, we have
\begin{equation*}
 \overline{b}_{l,j}-\overline{a}_{l,j}\leq \frac{2^l}{2l}+\frac{2^{l+1}}{2(l+1)}\leq \frac{2^{l+1}}{l+1}.
\end{equation*}
Hence
\begin{equation}\label{Eq4.60}
  \overline{b}_{l,j}-\overline{a}_{l,j}\leq \frac{2\cdot 2^{l+1}}{l+1}.
\end{equation}
As $l\leq x_{l,j}$, by (\ref{Eq4.58})-(\ref{Eq4.60}), we have
\begin{equation}\label{Eq4.62}
  \frac{\lambda-\lambda_1}{\lambda-\lambda_2}\leq 1+\frac{32\cdot 2^{l+1}\slash (l+1)}{2^{x_{l,j}+1}\slash (x_{l,j}+1)}\leq 33.
\end{equation}
If $\theta_{x_{l,j},y_{l,j};l,j}\neq 0$, as $\theta_{x_{l,j},y_{l,j};l,j}\in \mathcal{K}_{x_{l,j}}$, we have 
\begin{equation}\label{Eq4.63}
\theta_{x_{l,j},y_{l,j};l,j}k\geq (x_{l,j}+1)^{-10}k\geq (T+1)^{-10}k.
\end{equation}
As we have taken $K_0(M,\epsilon,\delta)$ to be sufficiently large (depending on $\beta, M,\epsilon,\delta$; recall (\ref{Eq2.1}) and the fact that $T$ only depends on $M,\epsilon$), by Proposition \ref{P3.1}, noting (\ref{Eq4.61}), (\ref{Eq4.62}), and (\ref{Eq4.63}), there exists an event $\overline{\mathcal{U}}_{l,j}\in \mathcal{F}_{\overline{a}_{l,j}k^{2\slash 3}, \overline{b}_{l,j}k^{2\slash 3}}$, such that
\begin{equation}\label{Eq4.104}
 \mathcal{U}_{l,j}\subseteq \overline{\mathcal{U}}_{l,j},
\end{equation}
\begin{eqnarray}\label{Eq4.70}
\mathbb{P}(\overline{\mathcal{U}}_{l,j})&\leq& \exp(C(\lambda-\lambda_1)^{3\slash 2}\log\log(\lambda-\lambda_1))\nonumber\\
&&\times \exp\Big(-c(\theta_{x_{l,j},y_{l,j};l,j})^2 k^2\Big/\log\Big(\frac{\lambda-\lambda_1}{\lambda-\lambda_2}\Big)\Big).
\end{eqnarray}
Note that by (\ref{Eq4.64}), we have
\begin{equation*}
  \lambda-\lambda_1\leq b_{x_{l,j},y_{l,j}} k^{2\slash 3}\leq 2^{T+1} k^{2\slash 3},
\end{equation*}
which leads to
\begin{equation}\label{Eq4.68}
  (\lambda-\lambda_1)^{3\slash 2}\log\log(\lambda-\lambda_1)\leq 2^{3(T+1)\slash 2} k \log\log(2^{T+1} k^{2\slash 3}).
\end{equation}
By (\ref{Eq4.62}) and (\ref{Eq4.63}), if $\theta_{x_{l,j},y_{l,j};l,j}\neq 0$, we have
\begin{equation}\label{Eq4.69}
(\theta_{x_{l,j},y_{l,j};l,j})^2 k^2\Big/\log\Big(\frac{\lambda-\lambda_1}{\lambda-\lambda_2}\Big)\geq c(T+1)^{-20} k^2.
\end{equation}
By (\ref{Eq4.70})-(\ref{Eq4.69}), as we have taken $K_0(M,\epsilon,\delta)$ to be sufficiently large (depending on $\beta, M,\epsilon,\delta$; recall (\ref{Eq2.1}) and the fact that $T$ only depends on $M,\epsilon$), we have
\begin{equation}\label{Eq4.71}
\mathbb{P}(\overline{\mathcal{U}}_{l,j})\leq \exp\Big(-c(\theta_{x_{l,j},y_{l,j};l,j})^2 k^2\Big/\log\Big(\frac{\lambda-\lambda_1}{\lambda-\lambda_2}\Big)\Big). 
\end{equation}
Now note that if $x_{l,j}=1$, we have $(l,j)=(0,2)$ and $(x_{l,j},y_{l,j})=(1,2)$, hence $a_{x_{l,j},y_{l,j}}-\overline{b}_{l,j}=1\slash 2$ and
\begin{equation}\label{Eq4.65}
  \overline{b}_{x_{l,j},y_{l,j}}-\overline{b}_{l,j}=2= 4(a_{x_{l,j},y_{l,j}}-\overline{b}_{l,j}).
\end{equation}
If $x_{l,j} \geq  2$, by (\ref{Eq4.59}), we have
\begin{equation*}
   \overline{b}_{x_{l,j},y_{l,j}}-a_{x_{l,j},y_{l,j}}\leq \frac{2^{x_{l,j}}}{2x_{l,j}}+\frac{2^{x_{l,j}+1}}{4(x_{l,j}+1)} \leq \frac{2^{x_{l,j}}}{x_{l,j}} \leq 8(a_{x_{l,j},y_{l,j}}-\overline{b}_{l,j}),
\end{equation*}
hence
\begin{equation}\label{Eq4.66}
  \overline{b}_{x_{l,j},y_{l,j}}-\overline{b}_{l,j}=(\overline{b}_{x_{l,j},y_{l,j}}-a_{x_{l,j},y_{l,j}})+(a_{x_{l,j},y_{l,j}}-\overline{b}_{l,j})\leq 9(a_{x_{l,j},y_{l,j}}-\overline{b}_{l,j}).
\end{equation}
By (\ref{Eq4.58}), (\ref{Eq4.65}), and (\ref{Eq4.66}), we have
\begin{equation}\label{Eq4.67}
 \frac{\lambda-\lambda_1}{\lambda-\lambda_2}\leq 1+\frac{9(\overline{b}_{l,j}-\overline{a}_{l,j})}{\overline{b}_{x_{l,j},y_{l,j}}-\overline{b}_{l,j}}.
\end{equation}
By (\ref{Eq4.67}) and the fact that $\log(1+9x)\leq 9\log(1+x),\forall x\geq 0$, we have 
\begin{equation}
  \log\Big(\frac{\lambda-\lambda_1}{\lambda-\lambda_2}\Big)\leq 9\log\Big(1+\frac{\overline{b}_{l,j}-\overline{a}_{l,j}}{\overline{b}_{x_{l,j},y_{l,j}}-\overline{b}_{l,j}}\Big)=9\log\Big(\frac{\overline{b}_{x_{l,j},y_{l,j}}-\overline{a}_{l,j}}{\overline{b}_{x_{l,j},y_{l,j}}-\overline{b}_{l,j}}\Big).
\end{equation}
Hence 
\begin{equation}\label{Eq4.72}
 (\theta_{x_{l,j},y_{l,j};l,j})^2 k^2\Big/\log\Big(\frac{\lambda-\lambda_1}{\lambda-\lambda_2}\Big)\geq \frac{1}{9} (\theta_{x_{l,j},y_{l,j};l,j})^2 k^2 \Big/ \log\Big(\frac{\overline{b}_{x_{l,j},y_{l,j}}-\overline{a}_{l,j}}{\overline{b}_{x_{l,j},y_{l,j}}-\overline{b}_{l,j}}\Big).
\end{equation}
By (\ref{Eq4.71}) and (\ref{Eq4.72}), we have
\begin{eqnarray}\label{Eq4.108}
 \mathbb{P}(\overline{\mathcal{U}}_{l,j})&\leq& \exp\Big(-c(\theta_{x_{l,j},y_{l,j};l,j})^2 k^2\Big/\log\Big(\frac{\overline{b}_{x_{l,j},y_{l,j}}-\overline{a}_{l,j}}{\overline{b}_{x_{l,j},y_{l,j}}-\overline{b}_{l,j}}\Big)\Big)\nonumber\\
 &=& \exp(-ck^2 Level(l,j)).
\end{eqnarray}
where we use (\ref{Eq4.56}) and (\ref{Eq4.101}) in the last equality.

Now consider the case where $(x_{l,j},y_{l,j})=(l,j)$ and $l\geq 1$. We take
\begin{equation*}
  \lambda_1=\overline{a}_{l,j} k^{2\slash 3}, \quad \lambda_2=\overline{b}_{l,j} k^{2\slash 3}, \quad \lambda=z_{l,j} k^{2\slash 3}
\end{equation*}
in Proposition \ref{P3.3}. Note that
\begin{equation*}
  a_{l,j}-\overline{a}_{l,j}\geq \begin{cases}
   \frac{1}{2}, & l=1\\
   \frac{2^{l-1}}{4(l-1)}\geq \frac{2^l}{8l}, & l\geq 2
  \end{cases},
\end{equation*}
which leads to $a_{l,j}-\overline{a}_{l,j}\geq 2^l\slash(8l)\geq 1\slash 4$. As $z_{l,j}\in\mathcal{Q}_{l,j}\subseteq I_{l,j}$, we have
\begin{equation}\label{Eq4.75}
 a_{l,j}\leq z_{l,j}\leq b_{l,j}.
\end{equation}
Hence
\begin{equation}\label{Eq4.73}
  0\leq \lambda_1<\lambda<\lambda_2, \quad \lambda-\lambda_1\geq (a_{l,j}-\overline{a}_{l,j}) k^{2\slash 3}\geq \frac{2^l k^{2\slash 3}}{8l}\geq \frac{1}{4}k^{2\slash 3},
\end{equation}
\begin{equation}\label{Eq4.74}
  \lambda_2-\lambda \leq (\overline{b}_{l,j}-a_{l,j}) k^{2\slash 3}\leq \Big(\frac{2^l}{2l}+\frac{2^{l+1}}{4(l+1)}\Big) k^{2\slash 3}\leq \frac{2^l k^{2\slash 3}}{l}\leq 8(\lambda-\lambda_1),
\end{equation}
where we use (\ref{Eq4.73}) in the last inequality of (\ref{Eq4.74}). Now if $\theta_{l,j;l,j}\neq 0$, as $\theta_{l,j;l,j}\in\mathcal{K}_l$, we have
\begin{equation}\label{Eq4.75.1}
  \theta_{l,j;l,j} k\geq (l+1)^{-10} k\geq (T+1)^{-10} k.
\end{equation}
Moreover, by (\ref{Eq4.75}), we have
\begin{equation}\label{Eq4.76}
  \lambda\leq b_{l,j}k^{2\slash 3} \leq 2^{T+1} k^{2\slash 3}.
\end{equation}
By (\ref{Eq4.75.1}) and (\ref{Eq4.76}), if $\theta_{l,j;l,j}\neq 0$, we have
\begin{equation}\label{Eq4.77}
 \log(2+(\lambda-\lambda_1)(\theta_{l,j;l,j} k)^{-2\slash 3})\leq \log(2^{T+2}(T+1)^{10}).
\end{equation}
By (\ref{Eq4.73})-(\ref{Eq4.75.1}), (\ref{Eq4.77}), and Proposition \ref{P3.3} (taking $L=8$), as we have taken $K_0(M,\epsilon,\delta)$ to be sufficiently large (depending on $\beta, M,\epsilon,\delta$; recall (\ref{Eq2.1}) and the fact that $T$ only depends on $M,\epsilon$), we conclude that if $\theta_{l,j;l,j} \neq 0$, then there exists an event $\overline{\mathcal{U}}_{l,j}\in\mathcal{F}_{\overline{a}_{l,j}k^{2\slash 3}, \overline{b}_{l,j} k^{2\slash 3}}$, such that
\begin{equation}\label{Eq4.84}
 \mathcal{U}_{l,j}\subseteq \overline{\mathcal{U}}_{l,j},
\end{equation}
\begin{eqnarray}\label{Eq4.80}
   \mathbb{P}(\overline{\mathcal{U}}_{l,j})&\leq& \exp(C(\lambda-\lambda_1)^{3\slash 2}\log\log(\lambda-\lambda_1))\nonumber\\
   &&\times\exp(-c(\theta_{l,j;l,j})^2 k^2 \slash \log(2+(\lambda-\lambda_1)k^{-2\slash 3} (\theta_{l,j;l,j})^{-2\slash 3})).\nonumber\\
   &&
\end{eqnarray}
By (\ref{Eq4.76}), we have
\begin{equation}\label{Eq4.78}
  (\lambda-\lambda_1)^{3\slash 2}\log\log(\lambda-\lambda_1)\leq \lambda^{3\slash 2}\log\log(\lambda)\leq 2^{3(T+1)\slash 2} k\log\log(2^{T+1}k^{2\slash 3}). 
\end{equation}
By (\ref{Eq4.75.1}) and (\ref{Eq4.77}), if $\theta_{l,j;l,j}\neq 0$, we have
\begin{equation}\label{Eq4.79}
(\theta_{l,j;l,j})^2 k^2 \slash \log(2+(\lambda-\lambda_1)k^{-2\slash 3} (\theta_{l,j;l,j})^{-2\slash 3})\geq (T+1)^{-20} k^2 \slash \log(2^{T+2}(T+1)^{10}).
\end{equation}
By (\ref{Eq4.80})-(\ref{Eq4.79}), as we have taken $K_0(M,\epsilon,\delta)$ to be sufficiently large (depending on $\beta, M,\epsilon,\delta$; recall (\ref{Eq2.1}) and the fact that $T$ only depends on $M,\epsilon$), if $\theta_{l,j;l,j}\neq 0$, we have
\begin{equation}\label{Eq4.81}
 \mathbb{P}(\overline{\mathcal{U}}_{l,j})\leq \exp(-c(\theta_{l,j;l,j})^2 k^2 \slash \log(2+(\lambda-\lambda_1)k^{-2\slash 3} (\theta_{l,j;l,j})^{-2\slash 3})).
\end{equation}
Now note that when $\theta_{l,j;l,j}\neq 0$, by (\ref{Eq4.75}) and (\ref{Eq4.75.1}), we have
\begin{equation*}
 (\lambda-\lambda_1) k^{-2\slash 3} (\theta_{l,j;l,j})^{-2\slash 3}\leq (b_{l,j}-\overline{a}_{l,j})(l+1)^{10} \leq 2^{l+1} (l+1)^{10}\leq 2^{12l},
\end{equation*}
where we use the inequality $l+1\leq 2^l$ for any $l\in\mathbb{N}^{*}$. Hence
\begin{equation}\label{Eq4.82}
 \log(2+(\lambda-\lambda_1) k^{-2\slash 3} (\theta_{l,j;l,j})^{-2\slash 3})\leq \log(2+2^{12l})\leq \log(2^{14 l})\leq 10 l.
\end{equation}
Note that $\overline{b}_{l,j}-\overline{a}_{l,j}\geq 2$ when $l\in [2]$. Moreover, 
\begin{equation}
  \overline{b}_{l,j}-\overline{a}_{l,j}\geq b_{l,j}-a_{l,j}=\frac{2^l}{2l}\geq 2^{l\slash 2-2},
\end{equation}
where we use the inequality $l\leq 2^{l\slash 2+1}$ for any $l\in\mathbb{N}^{*}$ (this inequality can be proved by induction: when $l\in [3]$, the inequality holds; when $l\geq 4$, assuming that the inequality holds with $l$ replaced by $l-1$, we have $2^{l\slash 2}=2^{1\slash 2} \cdot 2^{(l-1)\slash 2}  \geq 2^{-1\slash 2}(l-1)\geq l\slash 2 $). Hence for any $l\in [4]$, $\overline{b}_{l,j}-\overline{a}_{l,j}\geq 4\slash 3\geq 2^{l\slash 10}$; for any $l\geq 5$, $\overline{b}_{l,j}-\overline{a}_{l,j}\geq 2^{l\slash 2-2}\geq 2^{l\slash 10}$. Hence
\begin{equation}\label{Eq4.83}
\log(\overline{b}_{l,j}-\overline{a}_{l,j})\geq \log(2^{l\slash 10})\geq l\slash 20.
\end{equation}
By (\ref{Eq4.81}), (\ref{Eq4.82}), and (\ref{Eq4.83}), if $\theta_{l,j;l,j}\neq 0$, we have
\begin{equation}\label{Eq4.85}
  \mathbb{P}(\overline{\mathcal{U}}_{l,j})\leq \exp(-c (\theta_{l,j;l,j})^2 k^2\slash \log(\overline{b}_{l,j}-\overline{a}_{l,j}))=\exp(-ck^2 Level(l,j)), 
\end{equation}
where we use (\ref{Eq4.101}) in the last equality. If $\theta_{l,j;l,j}=0$, we take $\overline{\mathcal{U}}_{l,j}=\Omega$. Note that (\ref{Eq4.84}) and (\ref{Eq4.85}) still hold for this case.

Now consider the case where $(x_{l,j},y_{l,j})=(l,j)=(0,2)$. We take 
\begin{equation*}
  \lambda_1=0, \quad \lambda_2=\frac{5}{2}k^{2\slash 3}, \quad \lambda=z_{0,2} k^{2\slash 3}.
\end{equation*}
By Proposition \ref{AiryOperator}, there exists a positive absolute constant $C_0$, such that $\gamma_i\geq (\pi i)^{2\slash 3}$ for any $i\geq C_0$. As $z_{0,2}\leq 2$, $k\geq K_0(M,\epsilon,\delta)$, and $K_0(M,\epsilon,\delta)$ has been taken to be sufficiently large (depending on $\beta,M,\epsilon,\delta$), we have
$\gamma_{k}\geq  (\pi k)^{2\slash 3}>2k^{2\slash 3}\geq z_{0,2}k^{2\slash 3}=\lambda$. Hence $N_0(\lambda)<k$. As $\theta_{0,2;0,2}\in\mathcal{K}_0$, if $\theta_{0,2;0,2}\neq 0$, we have $\theta_{0,2;0,2}\geq 1$. In the following, we let $K\geq 10$ be the constant appearing in Proposition \ref{P3.4}.

First consider the sub-case where $z_{0,2}\leq K^{-2\slash 3}$ and $\theta_{0,2;0,2}\neq 0$. When the event $\mathcal{U}_{0,2}$ holds, $|N(\lambda;\lambda_1,\lambda_2)-N_0(\lambda;\lambda_1,\lambda_2)|\geq \theta_{0,2;0,2}k$. As $\theta_{0,2;0,2}\geq 1$ and $N_0(\lambda;\lambda_1,\lambda_2)\leq N_0(\lambda)$, we have 
\begin{equation*}
  N(\lambda;\lambda_1,\lambda_2)-N_0(\lambda;\lambda_1,\lambda_2)>-k\geq -\theta_{0,2;0,2} k.
\end{equation*}
Hence $N(\lambda)\geq N(\lambda;\lambda_1,\lambda_2)\geq N_0(\lambda;\lambda_1,\lambda_2)+\theta_{0,2;0,2}k\geq \theta_{0,2;0,2}k$. Therefore, by Proposition \ref{P3.4} (taking $x=\theta_{0,2;0,2}k$; note that $x\geq k\geq K_0(M,\epsilon,\delta)$ is sufficiently large and $K\lambda^{3\slash 2}=K z_{0,2}^{3\slash 2}k\leq k\leq \theta_{0,2;0,2} k=x$), we have
\begin{eqnarray}\label{Eq4.105}
  \mathbb{P}(\mathcal{U}_{0,2})&\leq& \mathbb{P}(N(\lambda)\geq \theta_{0,2;0,2}k ) \leq C\exp(-c (\theta_{0,2;0,2})^2 k^2)\nonumber\\
  &\leq&  \exp(-ck^2 Level(0,2)),
\end{eqnarray}
where we use (\ref{Eq4.101}) in the last inequality.

Now consider the sub-case where $K^{-2\slash 3}<z_{0,2}\leq 2$ and $\theta_{0,2;0,2}\neq 0$. Note that $0=\lambda_1<\lambda\leq 2k^{2\slash 3}< \lambda_2$, and 
\begin{equation*}
\lambda-\lambda_1=z_{0,2}k^{2\slash 3}\geq K^{-2\slash 3}k^{2\slash 3},\quad \lambda_2-\lambda\leq \frac{5}{2} k^{2\slash 3}\leq 3 K^{2\slash 3}(\lambda-\lambda_1).
\end{equation*}
Moreover, as $\theta_{0,2;0,2}\geq 1$, $\lambda-\lambda_1=\lambda\leq 2k^{2\slash 3}$, we have
\begin{equation*}
(\log(2+(\lambda-\lambda_1)(\theta_{0,2;0,2} k)^{-2\slash 3}))^2\leq \log(4)^2\leq 2.
\end{equation*}
Hence by Proposition \ref{P3.3} (taking $L=3K^{2\slash 3}$; note that $k\geq K_0(M,\epsilon,\delta)$ is sufficiently large), we have  
\begin{eqnarray}\label{Eq4.106}
 \mathbb{P}(\mathcal{U}_{0,2})&=&\mathbb{P}(|N(\lambda;\lambda_1,\lambda_2)-N_0(\lambda;\lambda_1,\lambda_2)|\geq \theta_{0,2;0,2}k) \nonumber\\
 &\leq& \exp(C(\lambda-\lambda_1)^{3\slash 2}\log\log(\lambda-\lambda_1))\nonumber\\
 && \times \exp(-c (\theta_{0,2;0,2})^2 k^2\Big/ \log(2+(\lambda-\lambda_1)(\theta_{0,2;0,2}k)^{-2\slash 3}))\nonumber\\
 &\leq& \exp(C k\log\log(2k^{2\slash 3})-c(\theta_{0,2;0,2})^2 k^2)\leq \exp(-c(\theta_{0,2;0,2})^2 k^2)\nonumber\\
 &\leq& \exp(-ck^2 Level(0,2)),
\end{eqnarray}
where we use (\ref{Eq4.101}) in the last inequality.

By (\ref{Eq4.105}) and (\ref{Eq4.106}), we conclude that when $(x_{l,j},y_{l,j})=(l,j)=(0,2)$,
\begin{equation}\label{Eq4.107}
  \mathbb{P}(\mathcal{U}_{0,2})\leq \exp(-ck^2 Level(0,2)). 
\end{equation}

\paragraph{Step 5}


By (\ref{Eq4.103}), (\ref{Eq4.104}), and (\ref{Eq4.84}), we have $\mathcal{H}(\mathbf{z}, \bm{\theta}, \bm{\theta}')\subseteq \mathcal{U}_{0,2}$ and
\begin{equation}
   \mathcal{H}(\mathbf{z}, \bm{\theta}, \bm{\theta}')\subseteq \bigcap_{\substack{(l,j)\in\mathcal{R}^{EVEN}_T:\\(l,j)\neq (0,2)} } \overline{\mathcal{U}}_{l,j}.
\end{equation}
Hence by (\ref{Eq4.104}), (\ref{Eq4.108}), and (\ref{Eq4.107}), we have
\begin{equation}\label{Eq4.113}
  \mathbb{P}(\mathcal{H}(\mathbf{z}, \bm{\theta}, \bm{\theta}'))\leq \mathbb{P}(\mathcal{U}_{0,2})\leq \exp(-c k^2 Level(0,2)).
\end{equation}
For any $(l,j)\in\mathcal{R}^{EVEN}_T$ such that $(l,j)\neq (0,2)$, we have $\overline{\mathcal{U}}_{l,j}\in\mathcal{F}_{\overline{a}_{l,j}k^{2\slash 3},\overline{b}_{l,j}k^{2\slash 3}}$. Hence we have
\begin{equation}\label{Eq4.109}
  \mathbb{P}(\mathcal{H}(\mathbf{z}, \bm{\theta}, \bm{\theta}'))\leq \mathbb{P}\Big(\bigcap_{\substack{(l,j)\in\mathcal{R}^{EVEN}_T:\\(l,j)\neq (0,2)} } \overline{\mathcal{U}}_{l,j}\Big)=\prod_{\substack{(l,j)\in\mathcal{R}^{EVEN}_T:\\(l,j)\neq (0,2)}} \mathbb{P}(\overline{\mathcal{U}}_{l,j}). 
\end{equation}
By (\ref{Eq4.108}), (\ref{Eq4.85}), and (\ref{Eq4.109}), we obtain that
\begin{eqnarray}\label{Eq4.111}
 \mathbb{P}(\mathcal{H}(\mathbf{z}, \bm{\theta}, \bm{\theta}'))&\leq& \prod_{\substack{(l,j)\in\mathcal{R}^{EVEN}_T:\\(l,j)\neq (0,2)}} \exp(-ck^2 Level(l,j)) \nonumber\\
 &=& \exp\Big(-c k^2 \Big(\sum_{\substack{(l,j)\in\mathcal{R}^{EVEN}_T:\\(l,j)\neq (0,2)}}Level(l,j)\Big)\Big).
\end{eqnarray}

Replacing $(l,j)$ and $(l',j')$ in (\ref{Eq4.110}) by $(T,2T)$ and $(l,j)$ respectively, we obtain that
\begin{equation}
 \sum_{(l,j)\in\mathcal{R}^{EVEN}_T}\frac{(\theta_{l,j})^2}{\log(\overline{b}_{l,j})} \leq 4\cdot 10^8 \sum_{(l,j)\in\mathcal{R}^{EVEN}_T} Level(l,j),
\end{equation}
which leads to
\begin{equation}\label{Eq4.112}
  \max\Big\{Level(0,2), \sum_{\substack{(l,j)\in\mathcal{R}^{EVEN}_T:\\(l,j)\neq (0,2)}} Level(l,j)\Big\}
  \geq 10^{-9} \sum_{(l,j)\in\mathcal{R}^{EVEN}_T}\frac{(\theta_{l,j})^2}{\log(\overline{b}_{l,j})}.
\end{equation}
By (\ref{Eq4.113}), (\ref{Eq4.111}), and (\ref{Eq4.112}), we have
\begin{equation}
  \mathbb{P}(\mathcal{H}(\mathbf{z}, \bm{\theta}, \bm{\theta}'))\leq \exp\Big(-c k^2 \Big(\sum_{(l,j)\in\mathcal{R}^{EVEN}_T}(\theta_{l,j})^2 \slash \log(\overline{b}_{l,j}) \Big) \Big).
\end{equation}
For any $(l,j)\in\mathcal{R}^{EVEN}_T$, we have $\overline{b}_{l,j}\leq 2^{l+2}\leq 2^{2(l+1)}$. Hence
\begin{equation}\label{Eq4.114}
  \mathbb{P}(\mathcal{H}(\mathbf{z}, \bm{\theta}, \bm{\theta}'))\leq \exp\Big(-c k^2 \Big(\sum_{(l,j)\in\mathcal{R}^{EVEN}_T}(\theta_{l,j})^2 \slash (l+1) \Big) \Big).
\end{equation}

Similarly, we can deduce that 
\begin{equation}\label{Eq4.115}
  \mathbb{P}(\mathcal{H}(\mathbf{z}, \bm{\theta}, \bm{\theta}'))\leq \exp\Big(-c k^2 \Big(\sum_{(l,j)\in\mathcal{R}^{ODD}_T}(\theta_{l,j})^2 \slash (l+1) \Big) \Big).
\end{equation}
Combining (\ref{Eq4.114}) and (\ref{Eq4.115}), we conclude that
\begin{eqnarray}\label{Eq4.121}
  &&\mathbb{P}(\mathcal{H}(\mathbf{z}, \bm{\theta}, \bm{\theta}'))\nonumber\\
  &\leq& \exp\Big(-c k^2 \max\Big\{\sum_{(l,j)\in\mathcal{R}^{EVEN}_T}(\theta_{l,j})^2 \slash (l+1), \sum_{(l,j)\in\mathcal{R}^{ODD}_T}(\theta_{l,j})^2 \slash (l+1)\Big\} \Big)\nonumber\\
  &\leq& \exp\Big(-c k^2 \Big(\sum_{(l,j)\in\mathcal{R}_T}(\theta_{l,j})^2 \slash (l+1) \Big)\Big).
\end{eqnarray}

\paragraph{Step 6}

In the following, we consider an arbitrary $(l,j)\in\mathcal{R}_T$. Let $\mathcal{V}_{l,j}$ be the event that 
\begin{equation}\label{Eq4.116}
  |N(z_{l,j}k^{2\slash 3};\overline{b}_{l,j} k^{2\slash 3},\infty)-N_0(z_{l,j}k^{2\slash 3};\overline{b}_{l,j} k^{2\slash 3},\infty)|\geq \theta'_{l,j} k.
\end{equation}
When the event $\mathcal{H}(\mathbf{z}, \bm{\theta}, \bm{\theta}')$ holds, by (\ref{DefinitionH}), we have
\begin{equation}
  J_{l,j}=z_{l,j}, \quad |N(J_{l,j}k^{2\slash 3};\overline{b}_{l,j} k^{2\slash 3},\infty)-N_0(J_{l,j}k^{2\slash 3};\overline{b}_{l,j} k^{2\slash 3},\infty)|\geq \theta'_{l,j} k,
\end{equation} 
hence (\ref{Eq4.116}) holds. Therefore, we have
\begin{equation}\label{Eq4.119}
\mathcal{H}(\mathbf{z}, \bm{\theta}, \bm{\theta}')\subseteq \mathcal{V}_{l,j}.
\end{equation} 

As $z_{l,j}\in \mathcal{Q}_{l,j}\subseteq I_{l,j}$, we have $z_{l,j}\leq b_{l,j}<\overline{b}_{l,j}$. Following the argument below (\ref{Eq3.96}), we can deduce that $N_0(z_{l,j}k^{2\slash 3};\overline{b}_{l,j} k^{2\slash 3},\infty)\leq 1$. Below we assume that $\theta'_{l,j}\neq 0$. As $\theta'_{l,j}\in\mathcal{K}_l$, we have $\theta'_{l,j}\geq (l+1)^{-10}\geq (T+1)^{-10}$; as $k\geq K_0(M,\epsilon,\delta)$ is sufficiently large (depending on $\beta,M,\epsilon,\delta$) and $T$ only depends on $M,\epsilon$, we have $\theta'_{l,j} k\geq (T+1)^{-10} k\geq 10$. Hence we have
\begin{equation*}
  N(z_{l,j}k^{2\slash 3};\overline{b}_{l,j} k^{2\slash 3},\infty)-N_0(z_{l,j}k^{2\slash 3};\overline{b}_{l,j} k^{2\slash 3},\infty)\geq -1>-\theta'_{l,j}k.
\end{equation*}
Hence if $\theta'_{l,j}\neq 0$ and the event $\mathcal{V}_{l,j}$ holds, we have
\begin{equation}\label{Eq4.117}
  N(z_{l,j}k^{2\slash 3};\overline{b}_{l,j} k^{2\slash 3},\infty)\geq N_0(z_{l,j}k^{2\slash 3};\overline{b}_{l,j} k^{2\slash 3},\infty)+\theta'_{l,j} k \geq \theta'_{l,j} k.
\end{equation}

In Proposition \ref{P3.2}, take
\begin{equation*}
   \lambda=z_{l,j} k^{2\slash 3}, \quad \lambda_1=\overline{b}_{l,j} k^{2\slash 3}, \quad \lambda_2=\infty.
\end{equation*}
Note that $\theta_{l,j}'k\geq 10$ when $\theta'_{l,j} \neq 0$, and
\begin{equation*}
  \lambda_1-\lambda\geq (\overline{b}_{l,j}-b_{l,j}) k^{2\slash 3}\geq \frac{2^l  k^{2\slash 3}}{4(l+1)} \geq \frac{1}{4} k^{2\slash 3},
\end{equation*}
where we use the fact that $\overline{b}_{0,2}-b_{0,2}=1\slash 2$ and $\overline{b}_{l,j}-b_{l,j}\geq 2^l\slash (4l)$ if $l\geq 1$. As $k\geq K_0(M,\epsilon,\delta)$ is sufficiently large (depending on $\beta, M, \epsilon, \delta$), by Proposition \ref{P3.2} and (\ref{Eq4.117}), we obtain that when $\theta'_{l,j}\neq 0$,
\begin{eqnarray}\label{Eq4.118}
 &&\mathbb{P}(\mathcal{V}_{l,j})\leq \mathbb{P}(N(z_{l,j}k^{2\slash 3};\overline{b}_{l,j} k^{2\slash 3},\infty)\geq \theta'_{l,j} k)\nonumber\\
 &\leq& C\exp(-c (\lambda_1-\lambda)^{3\slash 2} \theta'_{l,j} k)\leq C\exp(-c\cdot 2^{3l\slash 2}(l+1)^{-3\slash 2}\theta'_{l,j} k^2).\nonumber\\
 && 
\end{eqnarray}
Note that (\ref{Eq4.118}) also holds when $\theta'_{l,j}=0$. Hence by (\ref{Eq4.119}), we have
\begin{equation}\label{Eq4.120}
  \mathbb{P}(\mathcal{H}(\mathbf{z}, \bm{\theta}, \bm{\theta}'))\leq C\exp(-c\cdot 2^{3l\slash 2}(l+1)^{-3\slash 2}\theta'_{l,j} k^2).
\end{equation}

As (\ref{Eq4.120}) holds for any $(l,j)\in \mathcal{R}_T$, we have
\begin{equation}\label{Eq4.122}
  \mathbb{P}(\mathcal{H}(\mathbf{z}, \bm{\theta}, \bm{\theta}'))\leq C\exp\Big(-c k^2\max_{(l,j)\in\mathcal{R}_T}\Big\{2^{3l\slash 2}(l+1)^{-3\slash 2}\theta'_{l,j}\Big\} \Big).
\end{equation}
Combining (\ref{Eq4.121}) and (\ref{Eq4.122}), we conclude that
\begin{eqnarray}
\mathbb{P}(\mathcal{H}(\mathbf{z}, \bm{\theta}, \bm{\theta}'))&\leq& C  \min\Big\{\exp\Big(-c k^2 \Big(\sum_{(l,j)\in\mathcal{R}_T}(l+1)^{-1}(\theta_{l,j})^2 \Big)\Big),\nonumber\\
&& \exp\Big(-c k^2 \max_{(l,j)\in\mathcal{R}_T}\Big\{2^{3l\slash 2}(l+1)^{-3\slash 2}\theta'_{l,j}\Big\} \Big)\Big\}.
\end{eqnarray}



\end{proof}

\subsection{Bounding $\mathbb{P}(\mathcal{H}')$}

In this subsection, we bound $\mathbb{P}(\mathcal{H}')$, where $\mathcal{H}'$ is defined in (\ref{Eq4.123}). The main result is given in the following proposition.

\begin{proposition}\label{P4.2}
There exist positive constants $C,c$ that only depend on $\beta$, such that
\begin{equation}
   \mathbb{P}(\mathcal{H}')\leq C\exp(-c M k^2).
\end{equation}
\end{proposition}
 
\begin{proof}

Recall the definitions of $\mathcal{H}_1'$ and $\mathcal{H}_2'$ from (\ref{Eq2.8}) and (\ref{Eq2.9}). By (\ref{Eq4.123}) and the union bound, we have
\begin{equation}\label{Eq4.137}
    \mathbb{P}(\mathcal{H}')\leq \mathbb{P}(\mathcal{H}_1')+\mathbb{P}(\mathcal{H}_2').
\end{equation}
In the following, we bound $\mathbb{P}(\mathcal{H}_1')$ and $\mathbb{P}(\mathcal{H}_2')$ separately.

\paragraph{Step 1: Bounding $\mathbb{P}(\mathcal{H}_1')$}

For any $(l,j)\in\mathcal{R}_T$, $(l',j')\in\mathcal{P}_{l,j}$, and $x\in\mathcal{Q}_{l,j}$, we let $\mathcal{U}_{l,j;l',j';x}$ be the event that 
\begin{equation}\label{Eq4.124}
|N(xk^{2\slash 3};\overline{a}_{l',j'}k^{2\slash 3},\overline{b}_{l',j'}k^{2\slash 3})-N_0(xk^{2\slash 3};\overline{a}_{l',j'}k^{2\slash 3},\overline{b}_{l',j'}k^{2\slash 3})|\geq M(l+1)^{10} k.
\end{equation}
By (\ref{Eq2.8}), we have
\begin{equation}\label{Eq4.130}
\mathcal{H}_1'=\bigcup_{(l,j)\in\mathcal{R}_T}\bigcup_{\substack{(l',j')\in\mathcal{P}_{l,j},\\x\in\mathcal{Q}_{l,j} }} \mathcal{U}_{l,j;l',j';x}.
\end{equation}

Consider any $(l,j)\in \mathcal{R}_T$, $(l',j')\in\mathcal{P}_{l,j}$, and $x\in\mathcal{Q}_{l,j}$, such that $l\geq 1$ and $(l',j')\neq (l,j)$. In Proposition \ref{P3.1}, we take
\begin{equation*}
  \lambda=x k^{2\slash 3}, \quad \lambda_1=\overline{a}_{l',j'} k^{2\slash 3}, \quad \lambda_2=\overline{b}_{l',j'} k^{2\slash 3}. 
\end{equation*}
As $x\in\mathcal{Q}_{l,j}\subseteq I_{l,j}$, we have $a_{l,j}\leq x\leq b_{l,j}$, hence $0\leq \lambda_1<\lambda_2<\lambda$. Moreover, 
\begin{equation*}
   \lambda_2-\lambda_1\geq (b_{l,j}-a_{l,j})k^{2\slash 3}\geq k^{2\slash 3}, \quad \lambda-\lambda_2\geq (a_{l,j}-\overline{b}_{l',j'})k^{2\slash 3}\geq \frac{1}{2}k^{2\slash 3}. 
\end{equation*}
Note that 
\begin{equation*}
  \frac{\lambda-\lambda_1}{\lambda-\lambda_2}= 1+\frac{\overline{b}_{l',j'}-\overline{a}_{l',j'}}{x-\overline{b}_{l',j'}}\leq 1+\frac{\overline{b}_{l',j'}-\overline{a}_{l',j'}}{a_{l,j}-\overline{b}_{l',j'}}.
\end{equation*}
If $l'=0$, we have $\overline{b}_{l',j'}-\overline{a}_{l',j'}\leq 5\slash 2$; if $l'\geq 1$, we have
\begin{equation*}
  \overline{b}_{l',j'}-\overline{a}_{l',j'}\leq \frac{2^{l'}}{2l'}+\frac{2^{l'+1}}{2(l'+1)}\leq \frac{2^{l'+1}}{l'+1}.
\end{equation*}
Hence $\overline{b}_{l',j'}-\overline{a}_{l',j'}\leq 2\cdot 2^{l'+1}\slash (l'+1)$. Moreover, if $l'=0$, $a_{l,j}-\overline{b}_{l',j'}\geq 1\slash 2$; if $l'\geq 1$, $a_{l,j}-\overline{b}_{l',j'}\geq 2^{l'}\slash (4l')$. Hence $a_{l,j}-\overline{b}_{l',j'}\geq 2^{l'}\slash (4(l'+1))$. Therefore,
\begin{equation*}
\frac{\lambda-\lambda_1}{\lambda-\lambda_2}\leq 17.
\end{equation*}
As $k\geq K_0(M,\epsilon,\delta)$ is sufficiently large (depending on $\beta,M,\epsilon,\delta$) and $T$ only depends on $M,\epsilon$, by (\ref{Eq4.124}) and Proposition \ref{P3.1}, we have
\begin{eqnarray}\label{Eq4.127}
 \mathbb{P}(\mathcal{U}_{l,j;l',j';x})&\leq& \exp(C(\lambda-\lambda_1)^{3\slash 2}\log\log(\lambda-\lambda_1))\nonumber\\
 && \times \exp\Big(-cM^2(l+1)^{20}k^2\big/\log\Big(\frac{\lambda-\lambda_1}{\lambda-\lambda_2}\Big)\Big) \nonumber\\
 &\leq& \exp(C\cdot 2^{3(T+1)\slash 2}k\log\log(2^{T+1}k^{2\slash 3})-cM^2 k^2)\nonumber\\
 &\leq& \exp(-c M^2 k^2),
\end{eqnarray}
where we use the fact that $\lambda\leq b_{l,j}k^{2\slash 3}\leq 2^{T+1} k^{2\slash 3}$ in the second inequality.

Now consider any $(l,j)\in\mathcal{R}_T$ such that $l\geq 1$ and any $x\in\mathcal{Q}_{l,j}$. In Proposition \ref{P3.3}, we take
\begin{equation*}
  \lambda=xk^{2\slash 3}, \quad \lambda_1=\overline{a}_{l,j} k^{2\slash 3}, \quad \lambda_2=\overline{b}_{l,j} k^{2\slash 3}. 
\end{equation*}
As $x\in\mathcal{Q}_{l,j}\subseteq I_{l,j}$, we have $a_{l,j}\leq x\leq b_{l,j}$. Hence $0\leq \lambda_1<\lambda<\lambda_2$. Moreover,   
\begin{equation*}
  \lambda-\lambda_1\geq (a_{l,j}-\overline{a}_{l,j})k^{2\slash 3}\geq \begin{cases}
  \frac{2^{l-1}}{4(l-1)} k^{2\slash 3} & l\geq 2\\
    \frac{1}{2} k^{2\slash 3} & l=1
  \end{cases},
\end{equation*}
which leads to $\lambda-\lambda_1\geq 2^l k^{2\slash 3}\slash (8l)\geq k^{2\slash 3}\slash 4$. We also have
\begin{equation*}
  \lambda_2-\lambda\leq (\overline{b}_{l,j}-a_{l,j}) k^{2\slash 3} \leq \Big(\frac{2^l}{2l}+\frac{2^{l+1}}{4(l+1)}\Big)k^{2\slash 3} \leq \frac{2^l}{l} k^{2\slash 3} \leq 8(\lambda-\lambda_1).
\end{equation*}
As $M\geq M_0\geq 10$, we have
\begin{eqnarray}\label{Eq4.125}
  && \log(2+(\lambda-\lambda_1)(M(l+1)^{10}k)^{-2\slash 3})\leq \log(2+b_{l,j})\leq \log(2+2^{l+1})\nonumber\\
  &\leq& \log(2^{l+2})\leq \log(2^{2(l+1)})=2\log(2) (l+1)\leq 2(l+1)\leq 2(T+1).
\end{eqnarray}
As $k\geq K_0(M,\epsilon,\delta)$ is sufficiently large (depending on $\beta,M,\epsilon,\delta$) and $T$ only depends on $M,\epsilon$, by Proposition \ref{P3.3} (with $L=8$), we have
\begin{eqnarray}\label{Eq4.128}
 \mathbb{P}(\mathcal{U}_{l,j;l,j;x})&\leq& \exp(C(\lambda-\lambda_1)^{3\slash 2}\log\log(\lambda-\lambda_1))\nonumber\\
 && \times \exp(-c(M(l+1)^{10}k)^2\slash \log(2+(\lambda-\lambda_1)(M(l+1)^{10} k)^{-2\slash 3}))\nonumber\\
 &\leq& \exp(C\cdot 2^{3(T+1)\slash 2}k \log\log(2^{T+1} k^{2\slash 3})-cM^2 k^2)\nonumber\\
 &\leq& \exp(-c M^2 k^2),
\end{eqnarray}
where we use (\ref{Eq4.125}) and the fact that $\lambda\leq b_{l,j}k^{2\slash 3}\leq 2^{T+1} k^{2\slash 3}$ in the second inequality. 

Now consider any $(l,j)\in \mathcal{R}_T$ such that $l=0$ and any $x\in\mathcal{Q}_{l,j}$. Note that $x\in [0,2]$. By Proposition \ref{AiryOperator}, there exists a positive absolute constant $C_0$, such that for any $i\in\mathbb{N}^{*}$ with $i\geq C_0$, we have $\gamma_i\geq (\pi i)^{2\slash 3}$. As $k\geq K_0(M,\epsilon,\delta)$ is sufficiently large (depending on $\beta,M,\epsilon,\delta$), we have 
\begin{equation*}
 \gamma_k\geq (\pi k)^{2\slash 3}> 2k^{2\slash 3}\geq x k^{2\slash 3},
\end{equation*}
which leads to $N_0(x k^{2\slash 3})<k$. Hence 
\begin{eqnarray*}
&& N(xk^{2\slash 3};\overline{a}_{l,j}k^{2\slash 3},\overline{b}_{l,j}k^{2\slash 3})-N_0(xk^{2\slash 3};\overline{a}_{l,j}k^{2\slash 3},\overline{b}_{l,j}k^{2\slash 3})\nonumber\\
&& \geq -N_0(x k^{2\slash 3})> -k\geq -M(l+1)^{10}k.
\end{eqnarray*}
Hence when the event $\mathcal{U}_{l,j;l,j;x}$ holds, we have 
\begin{eqnarray}\label{Eq4.126}
N(xk^{2\slash 3};\overline{a}_{l,j}k^{2\slash 3},\overline{b}_{l,j}k^{2\slash 3})&\geq& N_0(xk^{2\slash 3};\overline{a}_{l,j}k^{2\slash 3},\overline{b}_{l,j}k^{2\slash 3})+M(l+1)^{10} k\nonumber\\
&\geq& M(l+1)^{10} k\geq Mk. 
\end{eqnarray}
In Proposition \ref{P3.4}, take $\lambda=x k^{2\slash 3}$, then $\max\{1,\lambda^{3\slash 2}\}\leq 2^{3\slash 2} k\leq 3k$. As $M\geq M_0$ is sufficiently large (depending on $\beta$), by Proposition \ref{P3.4} and (\ref{Eq4.126}), 
\begin{eqnarray}\label{Eq4.129}
\mathbb{P}(\mathcal{U}_{l,j;l,j;x})&\leq& \mathbb{P}(N(xk^{2\slash 3};\overline{a}_{l,j}k^{2\slash 3},\overline{b}_{l,j}k^{2\slash 3})\geq Mk) \nonumber\\
&\leq& \mathbb{P}(N(\lambda)\geq Mk)\leq C\exp(-c M^2 k^2).
\end{eqnarray}

By (\ref{Eq4.127}), (\ref{Eq4.128}), and (\ref{Eq4.129}), for any $(l,j)\in\mathcal{R}_T$, $(l',j')\in\mathcal{P}_{l,j}$, and $x\in\mathcal{Q}_{l,j}$, we have
\begin{equation}\label{Bdd}
   \mathbb{P}(\mathcal{U}_{l,j;l',j';x})\leq C\exp(-c M^2 k^2). 
\end{equation}

From the discussion at the beginning of Section \ref{Sect.2}, $T\geq 20$. Note that
\begin{equation}\label{Eq4.131}
   |\mathcal{R}_T|=2+2\sum_{l=1}^Tl =2+(T+1)T\leq 4T^2. 
\end{equation}
For any $(l,j)\in \mathcal{R}_T$, we have
\begin{equation}\label{Eq4.132}
  |\mathcal{P}_{l,j}|\leq |\mathcal{R}_T|\leq 4T^2; 
\end{equation}
moreover, as $\mathcal{Q}_{l,j}=\mathcal{Q}\cap I_{l,j}\subseteq \mathcal{Q}\cap [0,2^{T+1}]$, we have
\begin{equation}\label{Eq4.133}
    |\mathcal{Q}_{l,j}|\leq 2^{11(T+1)}+1\leq 2^{11(T+1)+1}\leq 2^{12T}. 
\end{equation}
By (\ref{Eq4.130}), (\ref{Bdd})-(\ref{Eq4.133}), and the union bound, we have
\begin{equation}
  \mathbb{P}(\mathcal{H}'_1)\leq C(4T^2)^2\cdot 2^{12T}\exp(-c M^2 k^2)\leq C T^4 2^{12T} \exp(-c M^2 k^2).
\end{equation}
As $M\geq M_0\geq 10$, $k\geq K_0(M,\epsilon,\delta)$ is sufficiently large (depending on $\beta, M,\epsilon,\delta$), and $T$ only depends on $M,\epsilon$, we have
\begin{equation}\label{Eq4.138}
  \mathbb{P}(\mathcal{H}'_1)\leq \exp(-c M^2 k^2).
\end{equation}

\paragraph{Step 2: Bounding $\mathbb{P}(\mathcal{H}_2')$}

For any $(l,j)\in\mathcal{R}_T$ and $x\in\mathcal{Q}_{l,j}$, we let $\mathcal{V}_{l,j;x}$ be the event that
\begin{equation}
 |N( x k^{2\slash 3};\overline{b}_{l,j}k^{2\slash 3},\infty)-N_0(xk^{2\slash 3};\overline{b}_{l,j}k^{2\slash 3},\infty)|\geq M (l+1)^{10} k.
\end{equation}
By (\ref{Eq2.9}), we have
\begin{equation}\label{Eq4.136}
\mathcal{H}_2'=\bigcup_{(l,j)\in\mathcal{R}_T}\bigcup_{x\in\mathcal{Q}_{l,j}} \mathcal{V}_{l,j;x}.
\end{equation}

Consider any $(l,j)\in\mathcal{R}_T$ and $x\in\mathcal{Q}_{l,j}$. As $x\in\mathcal{Q}_{l,j}\subseteq I_{l,j}$, we have that $a_{l,j}\leq x\leq b_{l,j}<\overline{b}_{l,j}$. Following the argument below (\ref{Eq3.96}), we can deduce that $N_0(x k^{2\slash 3};\overline{b}_{l,j} k^{2\slash 3},\infty)\leq 1$. As $M\geq M_0\geq 10$ and $k\geq K_0(M,\epsilon,\delta)\geq 10$, we have 
\begin{equation*}
  N( x k^{2\slash 3};\overline{b}_{l,j}k^{2\slash 3},\infty)-N_0(xk^{2\slash 3};\overline{b}_{l,j}k^{2\slash 3},\infty)\geq -1>-M(l+1)^{10} k.
\end{equation*}
Hence when the event $\mathcal{V}_{l,j;x}$ holds, we have
\begin{eqnarray}\label{Eq4.132n}
 N( x k^{2\slash 3};\overline{b}_{l,j}k^{2\slash 3},\infty)&\geq& N_0(xk^{2\slash 3};\overline{b}_{l,j}k^{2\slash 3},\infty)+M(l+1)^{10} k\nonumber\\
 &\geq& M(l+1)^{10} k.
\end{eqnarray}

If $l\geq 1$, in Proposition \ref{P3.2}, we take
\begin{equation*}
  \lambda=x k^{2\slash 3}, \quad \lambda_1=\overline{b}_{l,j} k^{2\slash 3}, \quad \lambda_2=\infty.
\end{equation*}
Note that $\lambda_1-\lambda\geq (\overline{b}_{l,j}-b_{l,j}) k^{2\slash 3}\geq 2^l k^{2\slash 3}\slash (4l)\geq k^{2\slash 3}\slash 2$ and $M(l+1)^{10} k\geq 100$. As $k\geq K_0(M,\epsilon,\delta)$ is sufficiently large (depending on $\beta,M,\epsilon,\delta$), by (\ref{Eq4.132n}) and Proposition \ref{P3.2}, we have
\begin{eqnarray}\label{Eq4.134}
 && \mathbb{P}(\mathcal{V}_{l,j;x})\leq \mathbb{P}(N(x k^{2\slash 3};\overline{b}_{l,j} k^{2\slash 3}, \infty)\geq M(l+1)^{10} k )\nonumber\\
 &\leq& C\exp(-c (\lambda_1-\lambda)^{3\slash 2} M(l+1)^{10} k)\leq C\exp(-c Mk^2). 
\end{eqnarray}

If $l=0$, we have $x\in [0,2]$. In Proposition \ref{P3.4}, we take $\lambda=x k^{2\slash 3}$. Note that $\max\{1,\lambda^{3\slash 2}\}\leq \max\{1, 2^{3\slash 2}k\}\leq 3k$. As $M\geq M_0$ is sufficiently large (depending on $\beta$), by (\ref{Eq4.132n}) and Proposition \ref{P3.4}, we have
\begin{eqnarray}\label{Eq4.135}
 && \mathbb{P}(\mathcal{V}_{l,j;x})\leq \mathbb{P}(N(x k^{2\slash 3};\overline{b}_{l,j} k^{2\slash 3}, \infty)\geq M(l+1)^{10} k )\nonumber\\
 &\leq& \mathbb{P}(N(\lambda)\geq M(l+1)^{10} k)\leq C\exp(-c M^2 k^2). 
\end{eqnarray}

By (\ref{Eq4.134}) and (\ref{Eq4.135}), for any $(l,j)\in\mathcal{R}_T$ and $x\in\mathcal{Q}_{l,j}$, we have
\begin{equation}
  \mathbb{P}(\mathcal{V}_{l,j;x})\leq C\exp(- c Mk^2). 
\end{equation}
Hence by (\ref{Eq4.131}), (\ref{Eq4.133}), (\ref{Eq4.136}), and the union bound, we have
\begin{equation}
  \mathbb{P}(\mathcal{H}_2')\leq C (4T^2)\cdot 2^{12 T} \exp(-c Mk^2)\leq C T^2 2^{12 T} \exp(-c Mk^2). 
\end{equation}
As $M\geq M_0\geq 10$, $k\geq K_0(M,\epsilon,\delta)$ is sufficiently large (depending on $\beta, M,\epsilon,\delta$), and $T$ only depends on $M,\epsilon$, we have
\begin{equation}\label{Eq4.139}
  \mathbb{P}(\mathcal{H}'_2)\leq \exp(-c M k^2).
\end{equation}

\bigskip

By (\ref{Eq4.137}), (\ref{Eq4.138}), and (\ref{Eq4.139}), as $M\geq M_0\geq 10$, we conclude that
\begin{equation}
 \mathbb{P}(\mathcal{H}')\leq C\exp(-c M k^2). 
\end{equation}

\end{proof}

\subsection{Bounding $\mathbb{P}(\mathcal{A}_0(\Lambda))$ and $\mathbb{P}(\mathcal{B}_0(\Lambda))$}

Recall the definitions of $\mathcal{A}_0(\Lambda)$ and $\mathcal{B}_0(\Lambda)$ from (\ref{Eq2.10})-(\ref{Eq2.14}). In this subsection, we bound $\mathbb{P}(\mathcal{A}_0(\Lambda))$ and $\mathbb{P}(\mathcal{B}_0(\Lambda))$. The main results are given in the following two propositions.

\begin{proposition}\label{P4.3}
There exist positive constants $\Lambda_0\geq 1,C,c$ that only depend on $\beta$, such that 
\begin{equation}
\mathbb{P}(\mathcal{A}_0(\Lambda_0)^c)\leq C\exp(-c M k^2). 
\end{equation}
\end{proposition}

\begin{proposition}\label{P4.4}
Assume that $\beta,n\in\mathbb{N}^{*}$ and $n^{10^{-7}}\leq k\leq n^{1\slash 20000}$. Then there exist positive constants $\Lambda_0\geq 1,C,c$ that only depend on $\beta$, such that 
\begin{equation}
\mathbb{P}(\mathcal{B}_0(\Lambda_0)^c)\leq C\exp(-c M k^2). 
\end{equation}
\end{proposition}

We give the proofs of Propositions \ref{P4.3} and \ref{P4.4} as follows. 

\begin{proof}[Proof of Proposition \ref{P4.3}]

The proof consists of \textbf{Steps 1-4} as given below.

\paragraph{Step 1} 

Recall the definition of $\mathcal{H}'$ from (\ref{Eq4.123}). Below we assume that the event $(\mathcal{H}')^c$ holds. Then for any $(l,j)\in\mathcal{R}_T$ and $x\in\mathcal{Q}_{l,j}$, we have 
\begin{equation}\label{Eq4.140}
  |N( x k^{2\slash 3};\overline{b}_{l,j}k^{2\slash 3},\infty)-N_0(xk^{2\slash 3};\overline{b}_{l,j}k^{2\slash 3},\infty)|< M (l+1)^{10} k.
\end{equation}
Moreover, for any $(l,j)\in\mathcal{R}_T$, $(l',j')\in\mathcal{P}_{l,j}$, and $x\in\mathcal{Q}_{l,j}$, we have 
\begin{equation}\label{Eq4.141}
    |N(x k^{2\slash 3}; \overline{a}_{l',j'}k^{2\slash 3}, \overline{b}_{l',j'}k^{2\slash 3})-N_0(x k^{2\slash 3}; \overline{a}_{l',j'}k^{2\slash 3}, \overline{b}_{l',j'}k^{2\slash 3})|< M(l+1)^{10} k.
\end{equation}
Recall the definition of $\mathcal{K}_l$ from (\ref{Eq2.11}). We also recall the definitions of 
\begin{equation*}
\mathcal{E}_0((l,j),z_{l,j}), \quad \mathcal{E}'((l,j),\theta'_{l,j}), \quad \mathcal{E}((l,j),(l',j'),\theta_{l,j;l',j'}) 
\end{equation*}
below (\ref{Eq2.16}). For any $(l,j)\in\mathcal{R}_T$, as $J_{l,j}\in \mathcal{Q}_{l,j}$, there exists some $z_{l,j}\in\mathcal{Q}_{l,j}$, such that $J_{l,j}=z_{l,j}$. Hence recalling the definition of $\mathcal{J}_T$ from (\ref{Eq2.17}), we have
\begin{equation}\label{Eq4.142}
  (\mathcal{H}')^c\subseteq \bigcap_{(l,j)\in\mathcal{R}_T}\bigcup_{z_{l,j}\in\mathcal{Q}_{l,j}} \mathcal{E}_0((l,j),z_{l,j})=\bigcup_{\mathbf{z}=(z_{l,j})_{(l,j)\in\mathcal{R}_T}\in\mathcal{J}_T}\bigcap_{(l,j)\in\mathcal{R}_T} \mathcal{E}_0((l,j),z_{l,j}).
\end{equation}
For any $(l,j)\in\mathcal{R}_T$, taking $x=J_{l,j}$ in (\ref{Eq4.140}), noting (\ref{Eq2.15}), we obtain that $\Delta_{l,j;2}<M(l+1)^{10}k$. Hence there exists $\theta'_{l,j}\in\mathcal{K}_l$, such that $\mathcal{E}'((l,j),\theta'_{l,j})$ holds. Hence recalling the definition of $\Theta'_T$ from (\ref{Eq2.16}), we have
\begin{equation}\label{Eq4.143}
   (\mathcal{H}')^c\subseteq \bigcap_{(l,j)\in\mathcal{R}_T} \bigcup_{\theta'_{l,j}\in\mathcal{K}_l}\mathcal{E}'((l,j),\theta'_{l,j})=\bigcup_{\bm{\theta}'=(\theta_{l,j}')_{(l,j)\in\mathcal{R}_T} \in  \Theta'_T}\bigcap_{(l,j)\in\mathcal{R}_T} \mathcal{E}'((l,j),\theta'_{l,j}) . 
\end{equation}
For any $(l,j)\in\mathcal{R}_T$ and $(l',j')\in\mathcal{P}_{l,j}$, taking $x=J_{l,j}$ in (\ref{Eq4.141}), noting (\ref{Eq2.7}), we have $\Delta_{l,j;l',j';1}<M(l+1)^{10}k$. Hence there exists $\theta_{l,j;l',j'}\in\mathcal{K}_l$, such that $\mathcal{E}((l,j),(l',j'),\theta_{l,j;l',j'})$ holds. Hence recalling the definition of $\Theta_T$ from (\ref{Eq2.18}), we have 
\begin{eqnarray}\label{Eq4.144}
  && (\mathcal{H}')^c\subseteq \bigcap_{(l,j)\in\mathcal{R}_T}\bigcap_{(l',j')\in\mathcal{P}_{l,j}} \bigcup_{\theta_{l,j;l',j'\in\mathcal{K}_l}}\mathcal{E}((l,j),(l',j'),\theta_{l,j;l',j'})\nonumber\\
&=& \bigcup_{\bm{\theta}=(\theta_{l,j;l',j'})_{(l,j)\in\mathcal{R}_T,(l',j')\in\mathcal{P}_{l,j}}\in \Theta_T}\bigcap_{(l,j)\in\mathcal{R}_T}\bigcap_{(l',j')\in\mathcal{P}_{l,j}} \mathcal{E}((l,j),(l',j'),\theta_{l,j;l',j'}).\nonumber\\
&& 
\end{eqnarray}
By (\ref{Eq4.142})-(\ref{Eq4.144}), recalling (\ref{DefinitionH}), we conclude that
\begin{equation}\label{Eq4.148}
 (\mathcal{H}')^c\subseteq  \bigcup_{\substack{\mathbf{z}=(z_{l,j})_{(l,j)\in\mathcal{R}_T}\in\mathcal{J}_T,\\ \bm{\theta}=(\theta_{l,j;l',j'})_{(l,j)\in\mathcal{R}_T,(l',j')\in\mathcal{P}_{l,j}}\in \Theta_T,\\ \bm{\theta}'=(\theta_{l,j}')_{(l,j)\in\mathcal{R}_T} \in  \Theta'_T}} \mathcal{H}(\mathbf{z},\bm{\theta},\bm{\theta}').
\end{equation}

\paragraph{Step 2}

Let $\mathcal{Z}$ be the set of $(\mathbf{z},\bm{\theta},\bm{\theta}')$ such that $\mathbf{z}=(z_{l,j})_{(l,j)\in\mathcal{R}_T}\in\mathcal{J}_T$, $\bm{\theta}=(\theta_{l,j;l',j'})_{(l,j)\in\mathcal{R}_T,(l',j')\in\mathcal{P}_{l,j}}\in \Theta_T$, and $\bm{\theta}'=(\theta_{l,j}')_{(l,j)\in\mathcal{R}_T} \in  \Theta'_T$. Recall Definition \ref{Def4.1}. For any $t\geq 1$, we let $\mathcal{Z}_t$ be the set of $(\mathbf{z},\bm{\theta},\bm{\theta}')\in\mathcal{Z}$ such that 
\begin{equation}\label{Eq4.155}
  \sum_{(l,j)\in\mathcal{R}_T}\frac{\mathcal{S}_{l,j}(\bm{\theta})^2}{l+1}\leq t M, \quad \max_{(l,j)\in\mathcal{R}_T} \Big\{\Big(\frac{2^{l+1}}{l+1}\Big)^{3\slash 2}  \theta'_{l,j} \Big\}\leq t M.
\end{equation}
For any $t\geq 1$, we let
\begin{equation}\label{Eq4.145}
 \mathcal{H}_t:=\bigcup_{(\mathbf{z},\bm{\theta},\bm{\theta}')\in\mathcal{Z}_t} \mathcal{H}(\textbf{z},\bm{\theta},\bm{\theta}'), \quad  \tilde{\mathcal{H}}_t:=\bigcup_{(\mathbf{z},\bm{\theta},\bm{\theta}')\in\mathcal{Z}\backslash \mathcal{Z}_t} \mathcal{H}(\textbf{z},\bm{\theta},\bm{\theta}').
\end{equation}
By (\ref{Eq4.148}), for any $t\geq 1$, we have $\mathcal{H}'\cup\mathcal{H}_t\cup \tilde{\mathcal{H}}_t=\Omega$, hence
\begin{equation}\label{Eq4.149}
  (\mathcal{H}_t)^c \subseteq \mathcal{H}'\cup \tilde{\mathcal{H}}_t. 
\end{equation}

By Proposition \ref{P4.1}, for any $t\geq 1$ and any $(\mathbf{z},\bm{\theta},\bm{\theta}')\in \mathcal{Z}\backslash \mathcal{Z}_t$, we have
\begin{equation}\label{Eq4.146}
  \mathbb{P}(\mathcal{H}(\textbf{z},\bm{\theta},\bm{\theta}'))\leq C\exp(-c t M k^2). 
\end{equation}
By (\ref{Eq2.11}), for any $l\in\{0\}\cup [T]$, as $M\geq M_0\geq 10$, we have 
\begin{equation}
|\mathcal{K}_l|\leq 2+\log(M(l+1)^{20})\slash \log(2)\leq 3\log(M(T+1)^{20}).
\end{equation}
Hence by (\ref{Eq4.131})-(\ref{Eq4.133}), we have
\begin{equation*}
 |\mathcal{J}_T|\leq (2^{12 T})^{4T^2}\leq \exp(C T^3), \quad |\Theta'_T|\leq (3\log(M(T+1)^{20}))^{4T^2},
\end{equation*}
\begin{equation*}
 |\Theta_T|\leq (3\log(M(T+1)^{20}))^{16 T^4}.
\end{equation*}
Hence
\begin{equation}\label{Eq4.147}
   |\mathcal{Z}|=|\mathcal{J}_T||\Theta_T||\Theta'_T|\leq \exp(CT^3)(3\log(M(T+1)^{20}))^{20T^4}. 
\end{equation}
As $k\geq K_0(M,\epsilon,\delta)$ is sufficiently large (depending on $\beta,M,\epsilon,\delta$) and $T$ only depends on $M,\epsilon$, by (\ref{Eq4.145}), (\ref{Eq4.146}), and (\ref{Eq4.147}), for any $t\geq 1$,
\begin{equation}\label{Eq4.150}
  \mathbb{P}(\tilde{\mathcal{H}}_t)\leq C\exp(CT^3)(3\log(M(T+1)^{20}))^{20T^4}\exp(-c t M k^2)\leq \exp(-ct M k^2). 
\end{equation}
By Proposition \ref{P4.2}, (\ref{Eq4.149}), (\ref{Eq4.150}), and the union bound, for any $t\geq 1$,
\begin{equation}\label{Eq4.159}
  \mathbb{P}((\mathcal{H}_t)^c)\leq C\exp(-c M k^2). 
\end{equation}

\paragraph{Step 3}

Below we assume that the event $\mathcal{H}_t$ holds (where $t\geq 1$). Then there exists $(\mathbf{z},\bm{\theta},\bm{\theta}')\in\mathcal{Z}_t$ with $\mathbf{z}=(z_{l,j})_{(l,j)\in\mathcal{R}_T}$, $\bm{\theta}=(\theta_{l,j;l',j'})_{(l,j)\in\mathcal{R}_T,(l',j')\in\mathcal{P}_{l,j}}$, and $\bm{\theta}'=(\theta_{l,j}')_{(l,j)\in\mathcal{R}_T}$, such that the event $\mathcal{H}(\textbf{z},\bm{\theta},\bm{\theta}')$ holds. Below we fix such a choice of $(\mathbf{z},\bm{\theta},\bm{\theta}')$, and assume that the event $\mathcal{H}(\textbf{z},\bm{\theta},\bm{\theta}')$ holds.

For any $(l,j)\in\mathcal{R}_T$ and $(l',j')\in\mathcal{P}_{l,j}$, we have 
\begin{equation}
\Delta_{l,j;l',j';1}\leq 2\theta_{l,j;l',j'} k+(l+1)^{-10} k.
\end{equation}
For any $(l,j)\in\mathcal{R}_T$, we have 
\begin{equation}
\Delta_{l,j;2}\leq  2\theta'_{l,j} k+(l+1)^{-10} k.
\end{equation}
We also note that for any $(l,j)\in\mathcal{R}_T$,
\begin{equation}
  |\mathcal{P}_{l,j}|\leq 2+2\sum_{i=1}^l i=2+l(l+1)\leq 4(l+1)^2.
\end{equation}
Hence by (\ref{Eq2.19}) and (\ref{Eq2.20}), for any $(l,j)\in\mathcal{R}_T$, we have
\begin{eqnarray}\label{Eq4.154}
\Delta_{l,j}&\leq&  \sum_{(l',j')\in\mathcal{P}_{l,j}} (2\theta_{l,j;l',j'}k +(l+1)^{-10}k )+ 2\theta'_{l,j} k+(l+1)^{-10} k \nonumber\\
 &\leq& 2\mathcal{S}_{l,j}(\bm{\theta})k+2\theta'_{l,j} k +(l+1)^{-10} k (1+|\mathcal{P}_{l,j}|)\nonumber\\
 &\leq& 2\mathcal{S}_{l,j}(\bm{\theta})k+2\theta'_{l,j} k+5(l+1)^{-8} k.
\end{eqnarray}

In the following, we consider any $(l,j)\in\mathcal{R}_T$. If $(l,j)\neq (0,1)$, we let $(l_1,j_1)$ be the largest element in $\mathcal{R}_T$ such that $(l_1,j_1)\prec (l,j)$ and take $I'_{l,j}:=I_{l_1,j_1}$; if $(l,j)=(0,1)$, we let $(l_1,j_1)=(0,0)$ and take $I'_{l,j}:=\emptyset$. If $(l,j)\neq (T,2T)$, we let $(l_2,j_2)$ be the smallest element in $\mathcal{R}_T$ such that $(l_2,j_2)\succ (l,j)$ and $I''_{l,j}:=I_{l_2,j_2}$; if $(l,j)=(T,2T)$, we let $(l_2,j_2)=(T,2T+1)$ and take $I''_{l,j}:=\emptyset$. We also define $\Delta_{0,0}=\mathcal{S}_{0,0}(\bm{\theta})=\theta'_{0,0}=0$ and $\Delta_{T,2T+1}=\mathcal{S}_{T,2T+1}(\bm{\theta})=\theta'_{T,2T+1}=0$. Note that when $l=0$, we have $l_1=0$; when $l\geq 1$, we have $l_1\geq l-1$. Hence
\begin{equation}\label{Eq4.157}
   \max\{l-1,0\}\leq l_1\leq l, \quad l\leq l_2\leq l+1,
\end{equation}
which leads to
\begin{equation}\label{Eq4.158}
  l_1+1\geq \frac{1}{2}(l+1), \quad l_2+1\leq l+2\leq 2(l+1).
\end{equation}

Recall the definitions of $N(x)$ and $\tilde{N}(x)$ from (\ref{Eq2.2}) and (\ref{Eq2.21}). For any $x\in I_{l,j}$, there exist $x_1,x_2\in\mathcal{Q}\cap(I_{l,j}\cup I'_{l,j}\cup I''_{l,j})$, such that $x_1\leq x\leq x_2$ and $x_2-x_1=2^{-10(T+1)}$. Note that
\begin{equation*}
 N(x_1 k^{2\slash 3})\leq N(x k^{2\slash 3})\leq N(x_2 k^{2\slash 3}), \quad N_0(x_1 k^{2\slash 3})\leq N_0(x k^{2\slash 3})\leq N_0(x_2 k^{2\slash 3}), 
\end{equation*}
which leads to
\begin{eqnarray*}
&&  N(x k^{2\slash 3})-N_0(x k^{2\slash 3})\leq N(x_2 k^{2\slash 3})-N_0(x_1 k^{2\slash 3})\nonumber\\
&\leq&   N(x_2 k^{2\slash 3})-N_0(x_2 k^{2\slash 3})+|N_0(x_1 k^{2\slash 3})-N_0(x_2 k^{2\slash 3})|,
\end{eqnarray*}
\begin{eqnarray*}
&&  N(x k^{2\slash 3})-N_0(x k^{2\slash 3})\geq N(x_1 k^{2\slash 3})-N_0(x_2 k^{2\slash 3})\nonumber\\
&\geq& N(x_1 k^{2\slash 3})-N_0(x_1 k^{2\slash 3})-|N_0(x_1 k^{2\slash 3})-N_0(x_2 k^{2\slash 3})|.
\end{eqnarray*}
Hence for any $x\in I_{l,j}$,
\begin{eqnarray}\label{Eq4.152}
&& |N(x k^{2\slash 3})-N_0(x k^{2\slash 3})|\leq |N_0(x_1 k^{2\slash 3})-N_0(x_2 k^{2\slash 3})|\nonumber\\
&+& \max\{|N(x_1 k^{2\slash 3})-N_0(x_1 k^{2\slash 3})|,|N(x_2 k^{2\slash 3})-N_0(x_2 k^{2\slash 3})|\}.
\end{eqnarray}
By Proposition \ref{AiryOperator}, there exists a positive absolute constant $C_0$, such that for any $i\in\mathbb{N}^{*}$ with $i\geq C_0$, we have $(3\pi(i-1)\slash 2)^{2\slash 3}\leq \gamma_i \leq (3\pi i\slash 2)^{2\slash 3}$, which leads to $2\gamma_i^{3\slash 2}\slash (3\pi)\leq i\leq 2\gamma_i^{3\slash 2}\slash (3\pi)+1$. For any $\lambda\geq (3\pi (C_0+1)\slash 2)^{2\slash 3}$, letting $i_1=\lfloor 2\lambda^{3\slash 2}\slash (3\pi)\rfloor$ and $i_2=\lceil 2\lambda^{3\slash 2}\slash (3\pi) \rceil+2$, we have $i_1,i_2\geq C_0$, hence $\gamma_{i_1}\leq \lambda$ and $\gamma_{i_2}> \lambda$. Therefore, for any $\lambda\geq (3\pi (C_0+1)\slash 2)^{2\slash 3}$, we have
\begin{equation}\label{Eq4.151}
 N_0(\lambda)\geq i_1\geq \frac{2}{3\pi}\lambda^{3\slash 2}-1,\quad N_0(\lambda)\leq i_2-1\leq \frac{2}{3\pi}\lambda^{3\slash 2}+2. 
\end{equation}
Now if $x_1=0$, then $x_2=2^{-10(T+1)}$. As $k\geq K_0(M,\epsilon,\delta)$ is sufficiently large (depending on $\beta,M,\epsilon,\delta$) and $T$ only depends on $M,\epsilon$, by (\ref{Eq4.151}), we have
\begin{equation}
 N_0(x_2 k^{2\slash 3})-N_0(x_1 k^{2\slash 3})\leq N_0(x_2 k^{2\slash 3})\leq \frac{2}{3\pi} x_2^{3\slash 2}k+2\leq \frac{C k}{2^{15(T+1)}}.
\end{equation}
If $x_1\neq 0$, then $x_1,x_2\geq 2^{-10(T+1)}$, and we have
\begin{equation*}
 N_0(x_1 k^{2\slash 3})\geq \frac{2}{3\pi} x_1^{3\slash 2} k -1,\quad N_0(x_2 k^{2\slash 3})\leq \frac{2}{3\pi} x_2^{3\slash 2} k +2.
\end{equation*}
Hence using the fact that $x_2\leq 2^{T+1}$, we have 
\begin{eqnarray}
 N_0(x_2 k^{2\slash 3})-N_0(x_1 k^{2\slash 3})&\leq& 3+\frac{2}{3\pi}(x_2^{3\slash 2}-x_1^{3\slash 2}) k \leq 3+\frac{1}{\pi} (x_2-x_1) \sqrt{x_2}k \nonumber\\
 &\leq& 3+\frac{\sqrt{2^{T+1}}k }{\pi\cdot 2^{10(T+1)}}\leq \frac{C k}{2^{9(T+1)}}.
\end{eqnarray}
Hence for both cases, we have
\begin{equation}\label{Eq4.153}
  N_0(x_2 k^{2\slash 3})-N_0(x_1 k^{2\slash 3})\leq \frac{C k}{2^{9(T+1)}}.
\end{equation}

By (\ref{Eq4.154}), (\ref{Eq4.152}), and (\ref{Eq4.153}), for any $x\in I_{l,j}$, we have
\begin{eqnarray}\label{Eq4.156}
 && |N(x k^{2\slash 3})-N_0(x k^{2\slash 3})|\nonumber\\
&\leq& \max\{|N(x_1 k^{2\slash 3})-N_0(x_1 k^{2\slash 3})|,|N(x_2 k^{2\slash 3})-N_0(x_2 k^{2\slash 3})|\}\nonumber\\
&& +C\cdot 2^{-9(T+1)} k \nonumber\\
&\leq&\max\{\Delta_{l,j},\Delta_{l_1,j_1},\Delta_{l_2,j_2}\} + C\cdot 2^{-9(T+1)} k\nonumber\\
&\leq& 2\max\{\mathcal{S}_{l,j}(\bm{\theta}),\mathcal{S}_{l_1,j_1}(\bm{\theta}),\mathcal{S}_{l_2,j_2}(\bm{\theta})\}k+2\max\{\theta'_{l,j},\theta'_{l_1,j_1},\theta'_{l_2,j_2}\} k  \nonumber\\
&& +C(l+1)^{-8}k \nonumber\\
&\leq& 2\max\{\mathcal{S}_{l,j}(\bm{\theta}),\mathcal{S}_{l_1,j_1}(\bm{\theta}),\mathcal{S}_{l_2,j_2}(\bm{\theta})\}k+C\cdot 2^{-3l \slash 2}(l+1)^{3\slash 2}tM k\nonumber\\
&& + C(l+1)^{-8}k,
\end{eqnarray}
where we use (\ref{Eq4.157}) and (\ref{Eq4.158}) in the third and last inequalities, and use (\ref{Eq4.155}) and the fact that $(\mathbf{z},\bm{\theta},\bm{\theta}')\in\mathcal{Z}_t$ in the last inequality.

By (\ref{Eq2.22}) and (\ref{Eq4.156}), for any $(l,j)\in\mathcal{R}_T$, we have
\begin{eqnarray}
 \overline{\Delta}_{l,j}&\leq& 2\max\{\mathcal{S}_{l,j}(\bm{\theta}),\mathcal{S}_{l_1,j_1}(\bm{\theta}),\mathcal{S}_{l_2,j_2}(\bm{\theta})\}k\nonumber\\
 &&+C\cdot 2^{-3l \slash 2}(l+1)^{3\slash 2}tM k+C(l+1)^{-8}k.
\end{eqnarray}
Hence we have
\begin{eqnarray}
 && \sum_{(l,j)\in\mathcal{R}_T} \frac{(\overline{\Delta}_{l,j})^2}{l+1}\nonumber\\
 &\leq& Ck^2 \Big(\sum_{(l,j)\in\mathcal{R}_T}\frac{\mathcal{S}_{l,j}(\bm{\theta})^2}{l+1}+ \sum_{(l,j)\in\mathcal{R}_T}\frac{\mathcal{S}_{l_1,j_1}(\bm{\theta})^2}{l+1}+\sum_{(l,j)\in\mathcal{R}_T}\frac{\mathcal{S}_{l_2,j_2}(\bm{\theta})^2}{l+1}\Big)\nonumber\\
 &&+Ct^2M^2k^2\Big(2+\sum_{l=1}^T 2l\cdot 2^{-3l}(l+1)^2\Big)+
 Ck^2\Big(2+\sum_{l=1}^T 2l(l+1)^{-17}\Big)\nonumber\\
 &\leq& Ck^2 \Big(\sum_{(l,j)\in\mathcal{R}_T}\frac{\mathcal{S}_{l,j}(\bm{\theta})^2}{l+1}+ \sum_{(l,j)\in\mathcal{R}_T}\frac{\mathcal{S}_{l_1,j_1}(\bm{\theta})^2}{l_1+1}+\sum_{(l,j)\in\mathcal{R}_T}\frac{\mathcal{S}_{l_2,j_2}(\bm{\theta})^2}{l_2+1}\Big)\nonumber\\
 && + Ct^2M^2k^2\Big(\sum_{l=1}^T 2^{-3 l}(l+1)^3\Big)+Ck^2\Big(\sum_{l=1}^T (l+1)^{-16}\Big)\nonumber\\
 &\leq& Ck^2\Big(\sum_{(l,j)\in\mathcal{R}_T}\frac{\mathcal{S}_{l,j}(\bm{\theta})^2}{l+1}\Big)+Ct^2M^2k^2\leq Ct^2M^2k^2,
\end{eqnarray}
where we use (\ref{Eq4.157}) and (\ref{Eq4.158}) in the second inequality, and use (\ref{Eq4.155}) and the fact that $(\mathbf{z},\bm{\theta},\bm{\theta}')\in\mathcal{Z}_t$ in the last inequality.

Taking $t=1$, we conclude that there exists a positive constant $\Lambda_1$ that only depends on $\beta$, such that 
\begin{equation}
  \mathcal{H}_1\subseteq \Big\{\sum_{(l,j)\in\mathcal{R}_T} \frac{(\overline{\Delta}_{l,j})^2}{l+1}\leq \Lambda_1 M^2 k^2\Big\}.
\end{equation}
Hence by (\ref{Eq4.159}), we have
\begin{equation}\label{Eq4.163}
 \mathbb{P}\Big(\sum_{(l,j)\in\mathcal{R}_T} \frac{(\overline{\Delta}_{l,j})^2}{l+1}>\Lambda_1 M^2 k^2\Big)\leq C\exp(-cMk^2).
\end{equation}

\paragraph{Step 4}

In Proposition \ref{P3.4}, we take $\lambda=k^{2\slash 3}$ and $x=Mk$. Noting that $k\geq K_0(M,\epsilon,\delta)$ is sufficiently large (depending on $\beta,M,\epsilon,\delta$) and $M\geq M_0$ is sufficiently large (depending on $\beta$), we obtain that 
\begin{equation}\label{Eq4.160}
 \mathbb{P}(N(k^{2\slash 3})\geq Mk)\leq C\exp(-cM^2 k^2)\leq C\exp(-c M k^2).
\end{equation}
Similarly, in Proposition \ref{P3.4}, we take $\lambda=0$ and $x=2^{-3\slash 2}Mk$, and obtain that
\begin{equation}\label{Eq4.161}
 \mathbb{P}(N(0)\geq 2^{-3\slash 2}Mk)\leq C\exp(-c M^2 k^2) \leq C\exp(-c M k^2).  
\end{equation}
Now consider any $j\in [k^{10}]$. In Proposition \ref{P3.5}, we take $\lambda=-j k^{2\slash 3}$ and $x=Mk(j+2)^{-3\slash 2}$. As $k\geq K_0(M,\epsilon,\delta)$ is sufficiently large (depending on $\beta,M,\epsilon,\delta$), we have
\begin{eqnarray}\label{Eq4.162}
    \mathbb{P}(N(-j k^{2\slash 3})\geq M k(j+2)^{-3\slash 2})&\leq& C\exp(-c M j^{3\slash 2} (j+2)^{-3\slash 2} k^2)\nonumber\\
    &\leq& C\exp(-c M k^2).
\end{eqnarray}

By (\ref{Eq4.160})-(\ref{Eq4.162}) and the union bound, as $k\geq K_0(M,\epsilon,\delta)$ is sufficiently large (depending on $\beta,M,\epsilon,\delta$), we have
\begin{eqnarray}\label{Eq4.164}
 && \mathbb{P}\Big(N(-j k^{2\slash 3})\geq \frac{M k}{(j+2)^{3\slash 2}}\text{ for some }j\in\{-1,0\}\cup [k^{10}]\Big)\nonumber\\
 &\leq& C(2+k^{10}) \exp(-c M k^2)\leq C\exp(-c M k^2).
\end{eqnarray}

By (\ref{Eq4.163}), (\ref{Eq4.164}), and the union bound, taking $\Lambda_0=\max\{\Lambda_1,1\}$, we conclude that
\begin{equation}
  \mathbb{P}(\mathcal{A}_0(\Lambda_0)^c)\leq C\exp(-c M k^2). 
\end{equation}

\end{proof}

\begin{proof}[Proof of Proposition \ref{P4.4}]

Let $\mathcal{C}_0$ be the event that $N(2^{T+2} k^{2\slash 3})\leq 2^{3(T+2)\slash 2} M k$. In Proposition \ref{P3.4}, we take $\lambda=2^{T+2} k^{2\slash 3}$ and $x=2^{3(T+2)\slash 2} M k$. As $M\geq M_0$ is sufficiently large (depending on $\beta$), by Proposition \ref{P3.4}, we have
\begin{eqnarray}\label{Eq4.1.28}
 &&  \mathbb{P}((\mathcal{C}_0)^c)\leq \mathbb{P}(N(2^{T+2} k^{2\slash 3})\geq 2^{3(T+2)\slash 2} M k)\nonumber\\
 &\leq& C\exp(-c\cdot 2^{3(T+2)}M^2 k^2)\leq C\exp(-c M k^2).
\end{eqnarray}

By Proposition \ref{P4.3}, there exists a positive constant $\Lambda_0'\geq 1$ that only depends on $\beta$, such that
\begin{equation}\label{Eq4.1.29}
  \mathbb{P}(\mathcal{A}_0(\Lambda_0')^c)\leq C\exp(-c M k^2).
\end{equation}

In the following, we couple the sorted eigenvalues $\lambda_1< \lambda_2< \cdots$ of the stochastic Airy operator $H_{\beta}$ and the sorted eigenvalues $\lambda_1^{(n)}> \lambda_2^{(n)}> \cdots> \lambda_n^{(n)}$ of the Gaussian $\beta$ ensemble $H_{\beta,n}$ (recall Section \ref{Sect.1.3}) using the coupling given in \cite[Theorem 1.5]{Zho} (replacing $k$ by $(\lceil 2^{3(T+2)\slash 2}M\rceil+1)k$ in \cite[Theorem 1.5]{Zho}; note that $k\geq K_0(M,\epsilon,\delta)$ is sufficiently large, depending on $\beta,M,\epsilon,\delta$). Recall that $\tilde{\lambda}_i^{(n)}=n^{1\slash 6}(\lambda_i^{(n)}-2\sqrt{n})$ for every $i\in [n]$. By \cite[Theorem 1.5]{Zho}, there exist positive constants $C_0,C,c$ that only depend on $\beta$, such that the following holds. Let $\mathcal{C}_0'$ be the event that for any $i\in [(\lceil 2^{3(T+2)\slash 2}M\rceil+1)k]$, we have
\begin{equation}\label{Eq4.1.1}
  |\tilde{\lambda}_i^{(n)}+\lambda_i|\leq C_0 n^{-1\slash 24}.
\end{equation}
Then
\begin{equation}\label{Eq4.1.30}
  \mathbb{P}((\mathcal{C}_0')^c)\leq C\exp(-c k^3)\leq C\exp(-cMk^2). 
\end{equation}

In the following, we assume that the event $\mathcal{C}_0\cap\mathcal{C}_0'\cap\mathcal{A}_0(\Lambda_0')$ holds. By (\ref{Eq4.1.1}), as $b_i=-k^{-2\slash 3}\tilde{\lambda}_i^{(n)}$ for every $i\in [n]$ (recall Section \ref{Sect.1.3}), for every $i\in [(\lceil 2^{3(T+2)\slash 2}M\rceil+1)k]$, we have
\begin{equation}\label{Eq4.1.2}
  |b_i-\lambda_i k^{-2\slash 3}|\leq C_0 k^{-2\slash 3} n^{-1\slash 24}. 
\end{equation}

Consider any $x\leq 2^{T+2}$. Let $i_x=N(x k^{2\slash 3})$. As the event $\mathcal{C}_0$ holds, we have $i_x\leq N(2^{T+2}k^{2\slash 3})\leq 2^{3(T+2)\slash 2} Mk$. We also note that $\lambda_1,\cdots,\lambda_{i_x}\leq xk^{2\slash 3}$ and $\lambda_{i_x+1}> xk^{2\slash 3}$. Hence if $i_x\geq 1$, then by (\ref{Eq4.1.2}), we have
\begin{equation}
 b_{i_x} \leq \lambda_{i_x}k^{-2\slash 3}+C_0 k^{-2\slash 3} n^{-1\slash 24}\leq x+ C_0 k^{-2\slash 3} n^{-1\slash 24}.
\end{equation}
Hence 
\begin{equation}\label{Eq4.1.3}
  \tilde{N}(x+C_0 k^{-2\slash 3} n^{-1\slash 24})\geq i_x=N(x k^{2\slash 3}).
\end{equation}
As $i_x+1\leq 2^{3(T+2)\slash 2}Mk+1\leq (\lceil 2^{3(T+2)\slash 2}M\rceil +1)k$, by (\ref{Eq4.1.2}), we have
\begin{equation}
 b_{i_x+1}\geq \lambda_{i_x+1}k^{-2\slash 3}-C_0 k^{-2\slash 3} n^{-1\slash 24}>x-C_0 k^{-2\slash 3} n^{-1\slash 24}.
\end{equation}
Hence
\begin{equation}\label{Eq4.1.4}
 \tilde{N}(x-C_0 k^{-2\slash 3} n^{-1\slash 24})\leq i_x=N(x k^{2\slash 3}). 
\end{equation}
Combining (\ref{Eq4.1.3}) and (\ref{Eq4.1.4}), we conclude that for any $x\leq 2^{T+2}$, 
\begin{equation}\label{Eq4.1.5}
\tilde{N}(x-C_0 k^{-2\slash 3} n^{-1\slash 24})\leq N(x k^{2\slash 3}) \leq \tilde{N}(x+C_0 k^{-2\slash 3} n^{-1\slash 24}).
\end{equation}

Now consider any $x\leq 2^{T+1}$. In (\ref{Eq4.1.5}), replacing $x$ by $x-C_0 k^{-2\slash 3} n^{-1\slash 24}$, we obtain that
\begin{equation}\label{Eq4.1.6}
  \tilde{N}(x)\geq N(xk^{2\slash 3}-C_0 n^{-1\slash 24}).
\end{equation}
As $k\geq K_0(M,\epsilon,\delta)$ is sufficiently large (depending on $\beta,M,\epsilon,\delta$), we have that $x+C_0 k^{-2\slash 3} n^{-1\slash 24}\leq 2^{T+1}+1\leq 2^{T+2}$. In (\ref{Eq4.1.5}), replacing $x$ by $x+C_0 k^{-2\slash 3} n^{-1\slash 24}$, we obtain that
\begin{equation}\label{Eq4.1.7}
   \tilde{N}(x)\leq N(x k^{2\slash 3}+C_0 n^{-1\slash 24}). 
\end{equation}
Combining (\ref{Eq4.1.6}) and (\ref{Eq4.1.7}), we conclude that for any $x\leq 2^{T+1}$,
\begin{equation}\label{Eq4.1.18}
   N(xk^{2\slash 3}-C_0 n^{-1\slash 24}) \leq \tilde{N}(x) \leq N(x k^{2\slash 3}+C_0 n^{-1\slash 24}).
\end{equation}

By Proposition \ref{AiryOperator}, there exists a positive absolute constant $K_0\geq 10$, such that for any $i\in\mathbb{N}^{*}$ with $i\geq K_0$, we have
\begin{equation*}
(3\pi(i-1)\slash 2)^{2\slash 3}\leq \gamma_i\leq (3\pi i\slash 2)^{2\slash 3}.
\end{equation*}
For any $x\geq (3\pi(K_0+1)\slash 2)^{2\slash 3}$, taking $i_x=\lfloor 2x^{3\slash 2}\slash (3\pi)\rfloor\geq K_0$, we obtain 
\begin{equation}\label{Eq4.1.8}
  \gamma_{i_x}\leq \Big(\frac{3\pi i_x}{2}\Big)^{2\slash 3}\leq x, \quad \text{hence } N_0(x)\geq i_x\geq \frac{2}{3\pi} x^{3\slash 2}-1.
\end{equation}
For any $x\geq (3\pi(K_0+1)\slash 2)^{2\slash 3}$, taking $i_x'=\lceil 2x^{3\slash 2}\slash (3\pi)\rceil+2\geq K_0$, we obtain
\begin{equation}\label{Eq4.1.9}
   \gamma_{i_x'}\geq \Big(\frac{3\pi(i_x'-1)}{2}\Big)^{2\slash 3}>x, \quad \text{hence }  N_0(x)\leq i_x'-1\leq \frac{2}{3\pi} x^{3\slash 2}+2.
\end{equation}
Let $K_1=(3\pi(K_0+1)\slash 2)^{2\slash 3}$. By (\ref{Eq4.1.8}) and (\ref{Eq4.1.9}), for any $x\geq K_1$, we have
\begin{equation}\label{Eq4.1.12}
     \Big|N_0(x)-\frac{2}{3\pi}x^{3\slash 2}\Big|\leq 2.
\end{equation}

For any $\lambda<0$, $q_{\lambda}(x)$ as defined in (\ref{DefinitionQ}) is lower bounded by $w(x)$, defined by
\begin{equation}\label{Eq4.1.10}
   w'(x)=-\lambda-w(x)^2, \quad w(0)=\infty.
\end{equation}
Solving (\ref{Eq4.1.10}), we obtain that for any $x\geq 0$, 
\begin{equation}
   w(x)=\frac{\sqrt{-\lambda}(e^{2\sqrt{-\lambda}x}+1)}{e^{2\sqrt{-\lambda}x}-1},
\end{equation}
which does not blow up. Hence for any $\lambda<0$, $q_{\lambda}(x)$ does not blow up in $[0,\infty)$, and $N_0(\lambda)=0$.

As $k\leq n^{1\slash 20000}$, $k\geq K_0(M,\epsilon,\delta)$ is sufficiently large (depending on $\beta,M,\epsilon,\delta$), and $T$ only depends on $M,\epsilon$, for any $x\in [0, 2^{T+2}]$, we have
\begin{equation}\label{Eq4.1.11}
  \frac{1}{4}(k\slash n)^{2\slash 3}x\leq \frac{1}{4}k^{-1}n^{-1\slash 3} 2^{T+2}\leq \frac{1}{4}n^{-1\slash 3}. 
\end{equation}
By (\ref{Eq2.21}) and (\ref{Eq4.1.11}), for any $x\in [0, 2^{T+2}]$, we have
\begin{equation}
   \tilde{N}_0(x)\leq k\pi^{-1}\int_0^x \sqrt{y}dy=\frac{2}{3\pi}x^{3\slash 2}k,
\end{equation}
\begin{equation}
   \tilde{N}_0(x)\geq k\pi^{-1}\sqrt{1-\frac{1}{4}n^{-1\slash 3}}\int_0^x \sqrt{y}dy\geq \Big(1-\frac{1}{2}n^{-1\slash 3}\Big) \frac{2}{3\pi}x^{3\slash 2}k,
\end{equation}
where we use the inequality $\sqrt{1-y}\geq 1-2y$ for any $y\in [0,1\slash 2]$. 
Hence for any $x\in [0,2^{T+2}]$, we have
\begin{equation}\label{Eq4.1.13}
  \Big|\tilde{N}_0(x)-\frac{2}{3\pi}x^{3\slash 2}k\Big|\leq k n^{-1\slash 3} 2^{3(T+2)\slash 2}\leq n^{-1\slash 6}. 
\end{equation}
Moreover, for any $x<0$, we have $\tilde{N}_0(x)=0$.

As $k\geq K_0(M,\epsilon,\delta)$ is sufficiently large (depending on $\beta,M,\epsilon,\delta$), for any $x\in [0,2^{T+1}]$, we have $x+C_0 k^{-2\slash 3} n^{-1\slash 24}\leq 2^{T+2}$. Note that for any $x\in [0, 2^{T+1}]$, 
\begin{eqnarray}\label{Eq4.1.14}
  0&\leq& \tilde{N}_0(x+C_0 k^{-2\slash 3} n^{-1\slash 24})-\tilde{N}_0(x)\leq k\pi^{-1}\int_x^{x+C_0 k^{-2\slash 3} n^{-1\slash 24}}\sqrt{y}dy \nonumber\\
  &\leq& k\pi^{-1}\cdot 2^{(T+2)\slash 2} \cdot C_0 k^{-2\slash 3} n^{-1\slash 24}\leq C\cdot 2^{(T+2)\slash 2} k^{1\slash 3} n^{-1\slash 24}\leq C, \nonumber\\
  &&   
\end{eqnarray}
\begin{eqnarray}\label{Eq4.1.15}
  0&\leq& \tilde{N}_0(x)-\tilde{N}_0(x-C_0 k^{-2\slash 3} n^{-1\slash 24})\leq k\pi^{-1}\int_{x-C_0 k^{-2\slash 3} n^{-1\slash 24}}^x\sqrt{y}\mathbbm{1}_{[0,\infty)}(y)dy \nonumber\\
  &\leq& k\pi^{-1}\cdot 2^{(T+2)\slash 2} \cdot C_0 k^{-2\slash 3} n^{-1\slash 24}\leq C\cdot 2^{(T+2)\slash 2} k^{1\slash 3} n^{-1\slash 24}\leq C. \nonumber\\
  &&   
\end{eqnarray}

Consider any $x\in [2K_1 k^{-2\slash 3},2^{T+1}]$. As $k\leq n^{1\slash 20000}$ and $k\geq K_0(M,\epsilon,\delta)$ is sufficiently large (depending on $\beta,M,\epsilon,\delta$), we have 
\begin{equation*}
 xk^{2\slash 3}-C_0 n^{-1\slash 24}\geq 2K_1-C_0 n^{-1\slash 24}\geq K_1, \quad x+C_0 k^{-2\slash 3} n^{-1\slash 24}\leq 2^{T+2}.
\end{equation*}
By (\ref{Eq4.1.12}), (\ref{Eq4.1.13}), (\ref{Eq4.1.14}), and (\ref{Eq4.1.15}), we have
\begin{eqnarray}
  \tilde{N}_0(x)&\leq& \tilde{N}_0(x-C_0 k^{-2\slash 3} n^{-1\slash 24})+C\leq \frac{2}{3\pi} (x-C_0 k^{-2\slash 3} n^{-1\slash 24})^{3\slash 2} k +C \nonumber\\
  &\leq& N_0(x k^{2\slash 3}-C_0 n^{-1\slash 24})+C,
\end{eqnarray}
\begin{eqnarray}
  \tilde{N}_0(x)&\geq& \tilde{N}_0(x+C_0 k^{-2\slash 3} n^{-1\slash 24})-C\geq \frac{2}{3\pi}(x+C_0 k^{-2\slash 3} n^{-1\slash 24})^{3\slash 2} k -C\nonumber\\
  &\geq& N_0(x k^{2\slash 3}+C_0 n^{-1\slash 24})-C.
\end{eqnarray}
Hence for any $x\in [2K_1 k^{-2\slash 3},2^{T+1}]$, we have
\begin{equation}\label{Eq4.1.16}
 N_0(x k^{2\slash 3}+C_0 n^{-1\slash 24})-C\leq \tilde{N}_0(x) \leq N_0(x k^{2\slash 3}-C_0 n^{-1\slash 24})+C.
\end{equation}

Now consider any $x<2K_1 k^{-2\slash 3}$. By (\ref{Eq2.21}) and (\ref{Eq4.1.12}), we have
\begin{equation}
   \tilde{N}_0(x)\leq \tilde{N}_0(2K_1 k^{-2\slash 3})\leq k\pi^{-1}\int_0^{2K_1 k^{-2\slash 3}}\sqrt{y}dy\leq \frac{2(2K_1)^{3\slash 2}}{3\pi}\leq C,
\end{equation}
\begin{eqnarray}
 N_0(x k^{2\slash 3}+C_0 n^{-1\slash 24})\leq N_0(4K_1)
 \leq \frac{2}{3\pi}(4K_1)^{3\slash 2}+2\leq C.
\end{eqnarray}
Hence for any $x<2K_1 k^{-2\slash 3}$, we have
\begin{equation}\label{Eq4.1.17}
N_0(x k^{2\slash 3}+C_0 n^{-1\slash 24})-C\leq \tilde{N}_0(x) \leq N_0(x k^{2\slash 3}-C_0 n^{-1\slash 24})+C.
\end{equation}

Combining (\ref{Eq4.1.16}) and (\ref{Eq4.1.17}), we obtain that for any $x\leq 2^{T+1}$,
\begin{equation}\label{Eq4.1.19}
N_0(x k^{2\slash 3}+C_0 n^{-1\slash 24})-C\leq \tilde{N}_0(x) \leq N_0(x k^{2\slash 3}-C_0 n^{-1\slash 24})+C.
\end{equation}
By (\ref{Eq4.1.18}) and (\ref{Eq4.1.19}), we conclude that for any $x\leq 2^{T+1}$,
\begin{eqnarray}
 && N(xk^{2\slash 3}-C_0 n^{-1\slash 24})-N_0(xk^{2\slash 3}-C_0 n^{-1\slash 24})-C\leq \tilde{N}(x)-\tilde{N}_0(x)\nonumber\\
 &\leq& N(xk^{2\slash 3}+C_0 n^{-1\slash 24})-N_0(xk^{2\slash 3}+C_0 n^{-1\slash 24})+C,
\end{eqnarray}
which leads to
\begin{eqnarray}\label{Eq4.1.20}
|\tilde{N}(x)-\tilde{N}_0(x)|&\leq& C+ |N(xk^{2\slash 3}+C_0 n^{-1\slash 24})-N_0(xk^{2\slash 3}+C_0 n^{-1\slash 24})|\nonumber\\
&+&|N(xk^{2\slash 3}-C_0 n^{-1\slash 24})-N_0(xk^{2\slash 3}-C_0 n^{-1\slash 24})|.
\end{eqnarray}

Now we consider any $(l,j)\in\mathcal{R}_{T-1}$. If $(l,j)\neq (0,1)$, we let $(l_1,j_1)$ be the largest element in $\mathcal{R}_T$ such that $(l_1,j_1)\prec (l,j)$; if $(l,j)=(0,1)$, we let $(l_1,j_1)=(0,0)$. We also let $(l_2,j_2)$ be the smallest element in $\mathcal{R}_T$ such that $(l_2,j_2)\succ (l,j)$. We also define
\begin{equation*}
  I_{0,0}=[-1,0), \quad \overline{\Delta}_{0,0}=\sup_{x\in [-1,0)} |N(xk^{2\slash 3})-N_0(xk^{2\slash 3})|.
\end{equation*}
Note that $l=0$ if $l_2=0$ and $l+1\geq l_2\geq (l_2+1)\slash 2$ if $l_2\geq 1$. Hence
\begin{equation}\label{Eq4.1.23}
  l_1\leq l, \quad l+1\geq (l_2+1)\slash 2.
\end{equation}
As $k\geq K_0(M,\epsilon,\delta)$ is sufficiently large (depending on $\beta,M,\epsilon,\delta$), we have $C_0 k^{-2\slash 3}n^{-1\slash 24}\leq 1\slash 10$. Hence for any $x\in I_{l,j}$, we have 
\begin{equation*}
 x\pm C_0k^{-2\slash 3} n^{-1\slash 24}\in I_{l_1,j_1}\cup I_{l,j}\cup I_{l_2,j_2}.
\end{equation*}
By (\ref{Eq4.1.20}), for any $(l,j)\in\mathcal{R}_{T-1}$, we have
\begin{eqnarray}
  \tilde{\Delta}_{l,j}=\sup_{x\in I_{l,j}}|\tilde{N}(x)-\tilde{N}_0(x)|\leq C+2\max\{\overline{\Delta}_{l_1,j_1},\overline{\Delta}_{l,j},\overline{\Delta}_{l_2,j_2}\},
\end{eqnarray}
which leads to
\begin{eqnarray}\label{Eq4.1.22}
 (\tilde{\Delta}_{l,j})^2&\leq& 2C^2+8\max\{(\overline{\Delta}_{l_1,j_1})^2,(\overline{\Delta}_{l,j})^2, (\overline{\Delta}_{l_2,j_2}\})^2\} \nonumber\\
 &\leq& C(1+(\overline{\Delta}_{l_1,j_1})^2+(\overline{\Delta}_{l,j})^2+(\overline{\Delta}_{l_2,j_2})^2).
\end{eqnarray}
Note that the event $\mathcal{A}_0(\Lambda_0')$ holds. By (\ref{Eq2.12}) with $j=0$, for any $x<0$, we have $N(x)\leq N(0)\leq \Lambda_0' Mk$. As $N_0(x)=0$ for any $x<0$, we have
\begin{equation}\label{Eq4.1.21}
 \overline{\Delta}_{0,0}=\sup_{x\in [-1,0)} N(x k^{2\slash 3})\leq \Lambda_0' M k.
\end{equation}
By (\ref{Eq2.10}) and (\ref{Eq4.1.21}), we have
\begin{equation}\label{Eq4.1.24}
 \sum_{(l,j)\in \{(0,0)\}\cup\mathcal{R}_{T}}\frac{(\overline{\Delta}_{l,j})^2}{l+1}\leq ((\Lambda_0')^2+\Lambda_0') M^2 k^2. 
\end{equation}
By (\ref{Eq4.1.23}), (\ref{Eq4.1.22}), and (\ref{Eq4.1.24}), as $k\geq K_0(M,\epsilon,\delta)$ is sufficiently large (depending on $\beta,M,\epsilon,\delta$) and $T$ only depends on $M,\epsilon$, we have
\begin{eqnarray}\label{Eq4.1.25}
 \sum_{(l,j)\in\mathcal{R}_{T-1}}\frac{(\tilde{\Delta}_{l,j})^2}{l+1}&\leq& C\Big(\sum_{(l,j)\in\mathcal{R}_{T-1}}\frac{1}{l+1}+\sum_{(l,j)\in\mathcal{R}_{T-1}}\frac{(\overline{\Delta}_{l,j})^2}{l+1}\nonumber\\
 && \quad+\sum_{(l,j)\in\mathcal{R}_{T-1}}\frac{(\overline{\Delta}_{l_1,j_1})^2}{l+1}+\sum_{(l,j)\in\mathcal{R}_{T-1}}\frac{(\overline{\Delta}_{l_2,j_2})^2}{l+1}\Big)\nonumber\\
 &\leq& C\Big(2+\sum_{l=1}^{T-1}\frac{2l}{l+1}+\sum_{(l,j)\in\mathcal{R}_{T-1}}\frac{(\overline{\Delta}_{l,j})^2}{l+1}\nonumber\\
 && \quad+\sum_{(l,j)\in\mathcal{R}_{T-1}}\frac{(\overline{\Delta}_{l_1,j_1})^2}{l_1+1}+\sum_{(l,j)\in\mathcal{R}_{T-1}}\frac{2(\overline{\Delta}_{l_2,j_2})^2}{l_2+1}\Big)\nonumber\\
 &\leq& C\Big(T+\sum_{(l,j)\in\{(0,0)\}\cup \mathcal{R}_{T}}\frac{(\overline{\Delta}_{l,j})^2}{l+1}\Big)\nonumber\\
 &\leq& CT+C((\Lambda_0')^2+\Lambda_0') M^2 k^2\leq CM^2k^2.
\end{eqnarray}

Now consider any $j\in\{0\}\cup [k^{10}]$. As $C_0 k^{-2\slash 3}n^{-1\slash 24}\leq 1\slash 10$, by (\ref{Eq2.12}) and (\ref{Eq4.1.7}), we have 
\begin{equation}\label{Eq4.1.26}
\tilde{N}(-j)\leq N((-j+C_0 k^{-2\slash 3} n^{-1\slash 24})k^{2\slash 3})\leq N((-j+1)k^{2\slash 3}) \leq\frac{\Lambda_0' M k}{(j+1)^{3\slash 2}}.
\end{equation}

By (\ref{Eq4.1.25}) and (\ref{Eq4.1.26}), there exists a positive constant $\Lambda_0$ that only depends on $\beta$, such that
\begin{equation}\label{Eq4.1.27}
 \mathcal{C}_0\cap\mathcal{C}_0'\cap\mathcal{A}_0(\Lambda_0') \subseteq \mathcal{B}_0(\Lambda_0).
\end{equation}
By (\ref{Eq4.1.28}), (\ref{Eq4.1.29}), (\ref{Eq4.1.30}), (\ref{Eq4.1.27}), and the union bound, we conclude that
\begin{equation}
 \mathbb{P}(\mathcal{B}_0(\Lambda_0)^c)\leq \mathbb{P}((\mathcal{C}_0)^c)+\mathbb{P}((\mathcal{C}_0')^c)+\mathbb{P}(\mathcal{A}_0(\Lambda_0')^c)\leq C\exp(-c M k^2).
\end{equation}

\end{proof}

\subsection{Some results on $\mathcal{B}(\bm{\phi};\Lambda)$}

Recall that $\Psi(x)=\tilde{N}(x)-\tilde{N}_0(x)$ for any $x\in\mathbb{R}$. Recall the definitions of $\mathcal{B}(\bm{\phi};\Lambda)$ and related concepts below (\ref{Eq2.14}) in Section \ref{Sect.2}. In this subsection, we establish some results on $\mathcal{B}(\bm{\phi};\Lambda)$, as given in Propositions \ref{P5.1} and \ref{P5.2} below.

\begin{proposition}\label{P5.1}
Assume that $\beta,n\in\mathbb{N}^{*}$ and $n^{10^{-7}}\leq k\leq n^{1\slash 20000}$. Let $\Lambda_0$ be the constant that appears in Proposition \ref{P4.4}. Then we have
\begin{equation}
 \mathcal{B}_0(\Lambda_0)\subseteq \bigcup_{\bm{\phi}\in\Phi} \mathcal{B}(\bm{\phi};\Lambda_0).
\end{equation}
\end{proposition}

\begin{proposition}\label{P5.2}
Assume that $\beta,n\in\mathbb{N}^{*}$ and $n^{10^{-7}}\leq k\leq n^{1\slash 20000}$. Recall (\ref{Eq2.1.10}). Let $\Lambda_0$ be the constant that appears in Proposition \ref{P4.4}. Then there exists a positive constant $C$ that only depends on $\beta$, such that for any choice of $\bm{\phi}=(i_0,l_0,j_0,l_0',j_0')\in \Phi$, when the event $\mathcal{B}(\bm{\phi};\Lambda_0)$ holds, the following properties hold.
\begin{itemize}
  \item[(a)] For any 
  \begin{equation*}
   x\in \Big[R_0(\bm{\phi})-\frac{R_0(\bm{\phi})}{16\log(R_0(\bm{\phi}))},R_0(\bm{\phi})+\frac{R_0(\bm{\phi})}{16\log(R_0(\bm{\phi}))}\Big],
  \end{equation*}
  we have
  \begin{equation*}
    |\Psi(x)|\leq C\sqrt{\frac{\epsilon}{\log(R_0(\bm{\phi}))\log\log(R_0(\bm{\phi}))}}k.
  \end{equation*}
  \item[(b)] 
  \begin{equation*}
    \int_0^{2^T} \frac{\Psi(x)^2}{x+1}dx\leq C M^2 k^2,
  \end{equation*}
  which implies
  \begin{equation*}
    \int_0^{R_0'(\bm{\phi})}\frac{\Psi(x)^2}{x+1}dx\leq C M^2 k^2.
  \end{equation*}
  \item[(c)] For any $x\in [-k^2-1,0]$,
  \begin{equation*}
     |\Psi(x)|=\tilde{N}(x)\leq \frac{CM k}{(|x|+1)^{3\slash 2}}.
  \end{equation*}
  This implies that for any $x\leq -k^2$, $\tilde{N}(x)=0$.
  \item[(d)] 
  \begin{equation*}
    \int_{15 R_0(\bm{\phi})\slash 16}^{17 R_0(\bm{\phi})\slash 16} \frac{\Psi(x)^2}{x+1}dx\leq \frac{C\epsilon k^2}{\log(R_0(\bm{\phi}))\log\log(R_0(\bm{\phi}))}.
  \end{equation*}
  \item[(e)]
  \begin{equation*}
     |\Psi(R_0'(\bm{\phi}))|\leq C \sqrt{\frac{\epsilon}{\log(R_0(\bm{\phi}))}}k.
  \end{equation*}
  \item[(f)]
  \begin{equation*}
      \int_{R_0(\bm{\phi})}^{R_0'(\bm{\phi})}\frac{\Psi(x)^2}{x+1}dx\leq C \epsilon k^2.
  \end{equation*}
\end{itemize}
\end{proposition}

\begin{proof}[Proof of Proposition \ref{P5.1}]

In the following, we assume that the event $\mathcal{B}_0(\Lambda_0)$ holds. Note that this implies (\ref{Eq2.1.1}) and (\ref{Eq2.1.2}) with $\Lambda$ replaced by $\Lambda_0$.

Note that $2^{W_{\mathcal{I}_0+1}}=T-1$. By (\ref{Eq2.1.1}) (with $\Lambda$ replaced by $\Lambda_0$), we have
\begin{equation}
 \sum_{i=\mathcal{I}_0\slash 2}^{\mathcal{I}_0} \sum_{l=2^{W_i}+1}^{2^{W_{i+1}}} \sum_{j=1}^{2l} \frac{(\tilde{\Delta}_{l,j})^2}{l+1}\leq \sum_{(l,j)\in\mathcal{R}_{T-1}}\frac{(\tilde{\Delta}_{l,j})^2}{l+1}\leq \Lambda_0 M^2 k^2. 
\end{equation}
Hence there exists $i_0\in\mathbb{N}$ with $\mathcal{I}_0\slash 2\leq i_0\leq \mathcal{I}_0$ (note that $\mathcal{I}_0\slash 2\in\mathbb{N}^{*}$), such that
\begin{equation}\label{Eq5.1}
\sum_{l=2^{W_{i_0}}+1}^{2^{W_{i_0+1}}} \sum_{j=1}^{2l} \frac{(\tilde{\Delta}_{l,j})^2}{l+1}\leq \frac{\Lambda_0 M^2 k^2}{\mathcal{I}_0\slash 2}\leq \Lambda_0 \epsilon  k^2,
\end{equation}
where we use the fact that $\mathcal{I}_0=2\lceil M^2\slash \epsilon\rceil\geq 2M^2\slash \epsilon$. Below we fix a choice of such $i_0$. Note that (\ref{Eq2.1.3}) holds with $\Lambda$ replaced by $\Lambda_0$. 

Now suppose that for any $l\in\mathbb{N}$ such that $2^{W_{i_0}}+1\leq l\leq 2^{\overline{W}_{i_0+1}}$, we have
\begin{equation}
 \sum_{j=1}^{2l}\frac{(\tilde{\Delta}_{l,j})^2}{l+1}\geq \frac{\epsilon k^2}{(l+1)\log(l+1)}. 
\end{equation}
As $2^{W_{i_0}}+2\leq 2^{W_{i_0}+1}$ (note that $W_{i_0}\geq W_1\geq 10$) and $\overline{W}_{i_0+1}\geq e^{M\slash \epsilon}(W_{i_0}+1)$, noting that $M\geq M_0$ is sufficiently large (depending on $\beta$), we have
\begin{eqnarray}
&& \sum_{l=2^{W_{i_0}}+1}^{2^{W_{i_0+1}}} \sum_{j=1}^{2l} \frac{(\tilde{\Delta}_{l,j})^2}{l+1}\geq \sum_{l=2^{W_{i_0}}+1}^{2^{\overline{W}_{i_0+1}}} \sum_{j=1}^{2l} \frac{(\tilde{\Delta}_{l,j})^2}{l+1}\geq \sum_{l=2^{W_{i_0}}+1}^{2^{\overline{W}_{i_0+1}}}\frac{\epsilon k^2}{(l+1)\log(l+1)}\nonumber\\
&\geq& \epsilon k^2 \int_{2^{W_{i_0}}+1}^{2^{\overline{W}_{i_0+1}}+1}\frac{1}{(x+1)\log(x+1)}dx\nonumber\\
&=& \epsilon k^2 (\log\log(2^{\overline{W}_{i_0+1}}+2)-\log\log(2^{W_{i_0}}+2))\nonumber\\
&\geq& \epsilon k^2(\log(\overline{W}_{i_0+1}\log(2))-\log((W_{i_0}+1)\log(2)))\nonumber\\
&=&\epsilon k^2 (\log(\overline{W}_{i_0+1})-\log(W_{i_0}+1))\geq Mk^2 > \Lambda_0 \epsilon k^2,
\end{eqnarray}
which contradicts (\ref{Eq5.1}). Hence there exists $l_0\in\mathbb{N}$ with $2^{W_{i_0}}+1\leq l_0\leq 2^{\overline{W}_{i_0+1}}$, such that  
\begin{equation}\label{Eq5.2}
 \sum_{j=1}^{2l_0}\frac{(\tilde{\Delta}_{l_0,j})^2}{l_0+1}\leq \frac{\epsilon k^2}{(l_0+1)\log(l_0+1)}. 
\end{equation}
Below we fix a choice of such $l_0$. Note that (\ref{Eq2.1.4}) holds. 

By (\ref{Eq5.2}), there exists $j_0\in\mathbb{N}$ with $l_0\slash 2\leq j_0\leq 3l_0\slash 2$, such that
\begin{equation}
  \frac{(\tilde{\Delta}_{l_0,j_0})^2}{l_0+1}\leq \frac{1}{(l_0+1)\slash 2}\cdot\frac{\epsilon k^2}{(l_0+1)\log(l_0+1)}, 
\end{equation}
where we use the fact that $\lfloor 3 l_0\slash 2\rfloor-\lceil l_0\slash 2\rceil\geq l_0-2\geq  (l_0+1)\slash 2$ (noting (\ref{Eq2.1.10})). Hence
\begin{equation}
  \tilde{\Delta}_{l_0,j_0}\leq 2\sqrt{\frac{\epsilon}{(l_0+1)\log(l_0+1)}}k,
\end{equation}
and (\ref{Eq2.1.5}) holds.

As $l_0\geq 2000$ (see (\ref{Eq2.1.10})), we have
\begin{eqnarray}
2^{W_{i_0}}+1\leq l_0\leq 10(l_0+1)\leq 11(l_0+1)\leq l_0^3\leq 2^{3\overline{W}_{i_0+1}}\leq 2^{W_{i_0+1}}.
\end{eqnarray}
Hence
\begin{equation}
\sum_{l=10(l_0+1)}^{11(l_0+1)}\sum_{j=1}^{2l} \frac{(\tilde{\Delta}_{l,j})^2}{l+1}\leq \sum_{l=2^{W_{i_0}}+1}^{2^{W_{i_0+1}}} \sum_{j=1}^{2l} \frac{(\tilde{\Delta}_{l,j})^2}{l+1}\leq \Lambda_0 \epsilon  k^2.
\end{equation}
Hence there exists $l_0'\in\mathbb{N}$ with $10(l_0+1)\leq l_0'\leq 11(l_0+1)$ (note that this implies $2^{W_{i_0}}+1\leq l_0'\leq 2^{W_{i_0+1}}$), such that
\begin{equation}\label{Eq5.3}
 \sum_{j=1}^{2l_0'}\frac{(\tilde{\Delta}_{l_0',j})^2}{l_0'+1}\leq \frac{\Lambda_0\epsilon k^2}{l_0+1}.
\end{equation}
Below we fix a choice of such $l_0'$. By (\ref{Eq5.3}), there exists $j_0'\in[2l_0']$, such that
\begin{equation}
   \frac{(\tilde{\Delta}_{l_0',j_0'})^2}{l_0'+1}\leq \frac{\Lambda_0 \epsilon k^2}{2l_0'(l_0+1)},\quad \text{hence } \tilde{\Delta}_{l_0',j_0'}\leq \sqrt{\frac{\Lambda_0 \epsilon}{l_0+1}}k\leq \sqrt{\frac{12\Lambda_0\epsilon}{l_0'+1}}k.
\end{equation}
Note that (\ref{Eq2.1.6}) holds with $\Lambda$ replaced by $\Lambda_0$. 

Therefore, letting $\bm{\phi}=(i_0,l_0,j_0,l_0',j_0')\in \Phi$, the event $\mathcal{B}(\bm{\phi};\Lambda_0)$ holds. We conclude that
\begin{equation}
    \mathcal{B}_0(\Lambda_0)\subseteq \bigcup_{\bm{\phi}\in\Phi} \mathcal{B}(\bm{\phi};\Lambda_0).
\end{equation}

\end{proof}

\begin{proof}[Proof of Proposition \ref{P5.2}]

In the following, we consider an arbitrary choice of $\bm{\phi}=(i_0,l_0,j_0,l_0',j_0')\in \Phi$, and assume that the event $\mathcal{B}(\bm{\phi};\Lambda_0)$ holds. 

We start with the proof of (a). By (\ref{Eq2.1.12}), we have
\begin{equation}\label{Eq5.4}
 R_0(\bm{\phi})-2^{l_0}\Big(1+\frac{j_0-1}{2l_0}\Big)=\frac{2^{l_0+1}}{8l_0}\geq \frac{R_0(\bm{\phi})}{8\log(R_0(\bm{\phi}))\slash \log(2)}> \frac{R_0(\bm{\phi})}{16\log(R_0(\bm{\phi}))},
\end{equation}
\begin{equation}\label{Eq5.5}
 2^{l_0}\Big(1+\frac{j_0}{2l_0}\Big)-R_0(\bm{\phi})=\frac{2^{l_0+1}}{8l_0}> \frac{R_0(\bm{\phi})}{16\log(R_0(\bm{\phi}))}.
\end{equation}
By (\ref{Eq2.1.12}) and (\ref{Eq2.1.10}), we have
\begin{eqnarray}\label{Eq5.10}
 (l_0+1)\log(l_0+1) &\geq& \log(2)^{-1}\log(R_0(\bm{\phi}))(\log\log(R_0(\bm{\phi}))-\log\log(2)) \nonumber\\
 &\geq& \log(2)^{-1}\log(R_0(\bm{\phi}))\log\log(R_0(\bm{\phi})).
\end{eqnarray}
Below we consider any 
\begin{equation*}
   x\in \Big[R_0(\bm{\phi})-\frac{R_0(\bm{\phi})}{16\log(R_0(\bm{\phi}))},R_0(\bm{\phi})+\frac{R_0(\bm{\phi})}{16\log(R_0(\bm{\phi}))}\Big].
\end{equation*}
By (\ref{Eq5.4}) and (\ref{Eq5.5}), we have $x\in I_{l_0,j_0}$. Hence by (\ref{Eq2.1.5}) and (\ref{Eq5.10}), we have
\begin{eqnarray}
    |\Psi(x)|&=&|\tilde{N}(x)-\tilde{N}_0(x)|\leq \tilde{\Delta}_{l_0,j_0}\leq 2\sqrt{\frac{\epsilon}{(l_0+1)\log(l_0+1)}}k\nonumber\\
   &\leq& C\sqrt{\frac{\epsilon}{\log(R_0(\bm{\phi}))\log\log(R_0(\bm{\phi}))}}k.
\end{eqnarray}

We show (b) as follows. For any $l\in [T-1]$, we have
\begin{eqnarray}
&& \int_{I_l}\frac{\Psi(x)^2}{x+1}dx=\sum_{j=1}^{2l} \int_{I_{l,j}}\frac{\Psi(x)^2}{x+1}dx\leq \sum_{j=1}^{2l} (\tilde{\Delta}_{l,j})^2 \int_{I_{l,j}}\frac{1}{x+1}dx \nonumber\\
&=& \sum_{j=1}^{2l} (\tilde{\Delta}_{l,j})^2 \log\Big(\frac{2^l(1+j\slash (2l))+1}{2^l(1+(j-1)\slash (2l))+1}\Big)\leq \sum_{j=1}^{2l} \frac{(\tilde{\Delta}_{l,j})^2 }{2l+j-1}\nonumber\\
&\leq& \sum_{j=1}^{2l} \frac{(\tilde{\Delta}_{l,j})^2 }{2l}\leq \sum_{j=1}^{2l}\frac{(\tilde{\Delta}_{l,j})^2 }{l+1},
\end{eqnarray}
where we use
\begin{eqnarray*}
 && \log\Big(\frac{2^l(1+j\slash (2l))+1}{2^l(1+(j-1)\slash (2l))+1}\Big)=\log\Big(1+\frac{2^l\slash (2l)}{2^l(1+(j-1)\slash (2l))+1}\Big)\nonumber\\
 &\leq& \log\Big(1+\frac{1}{2l+j-1}\Big)\leq\frac{1}{2l+j-1} 
\end{eqnarray*}
for any $l\in [T-1]$ and $j\in [2l]$ in the second line. Moreover, 
\begin{equation}
 \int_{I_0} \frac{\Psi(x)^2}{x+1}dx=\int_{[0,1]}\frac{\Psi(x)^2}{x+1}dx+\int_{(1,2]}\frac{\Psi(x)^2}{x+1} dx\leq (\tilde{\Delta}_{0,1})^2+(\tilde{\Delta}_{0,2})^2.
\end{equation}
Therefore, noting (\ref{Eq2.1.1}) (with $\Lambda$ replaced by $\Lambda_0$), we have
\begin{equation}\label{Eq5.6}
\int_{0}^{2^T}\frac{\Psi(x)^2}{x+1}dx=\sum_{l=0}^{T-1}\int_{I_l}\frac{\Psi(x)^2}{x+1}dx\leq \sum_{(l,j)\in\mathcal{R}_{T-1}}\frac{(\tilde{\Delta}_{l,j})^2}{l+1}\leq C M^2 k^2. 
\end{equation}
By (\ref{Eq2.1.15}) and (\ref{Eq5.6}), we have
\begin{equation}
 \int_0^{R_0'(\bm{\phi})}\frac{\Psi(x)^2}{x+1}dx\leq \int_{0}^{2^T}\frac{\Psi(x)^2}{x+1}dx\leq CM^2 k^2.
\end{equation}

We show (c) as follows. For any $x\in [-k^2-1,0]$, if $x\in [-j-1,-j]$ (where $j\in\{0\}\cup [k^2]$), we have $j+1\geq \max\{|x|,1\}\geq (|x|+1)\slash 2$. Hence by (\ref{Eq2.1.2}) (with $\Lambda$ replaced by $\Lambda_0$), we have
\begin{equation}
  \tilde{N}(x)\leq \tilde{N}(-j)\leq \frac{C M k}{(j+1)^{3\slash 2}}\leq \frac{C M k}{(|x|+1)^{3\slash 2}}. 
\end{equation}
As $k\geq K_0(M,\epsilon,\delta)$ is taken to be sufficiently large (depending on $\beta,M,\epsilon,\delta$), for any $x\leq -k^2$, we have
\begin{equation}
 \tilde{N}(x)\leq \tilde{N}(-k^2)\leq \frac{CMk}{(k^2+1)^{3\slash 2}}\leq CM k^{-2}<1, \quad \text{hence }\tilde{N}(x)=0.
\end{equation}

We show (d) as follows. As $l_0\geq 2000$ (see (\ref{Eq2.1.10})) and $l_0\slash 2\leq j_0\leq 3l_0\slash 2$, by (\ref{Eq2.1.12}), we have
\begin{equation}
  R_0(\bm{\phi})-2^{l_0}=\frac{2^{l_0}(2j_0-1)}{4l_0}\geq \frac{2^{l_0}(l_0-1)}{4l_0}\geq \frac{2^{l_0}}{8}\geq \frac{1}{16}R_0(\bm{\phi}),
\end{equation}
\begin{equation}
  2^{l_0+1}-R_0(\bm{\phi})=\frac{2^{l_0}(4l_0-2j_0+1)}{4l_0}\geq \frac{2^{l_0}}{4} \geq \frac{1}{8}R_0(\bm{\phi})\geq \frac{1}{16}R_0(\bm{\phi}).
\end{equation}
Hence
\begin{equation}\label{Eq5.7}
  2^{l_0}\leq \frac{15}{16}R_0(\bm{\phi})\leq \frac{17}{16}R_0(\bm{\phi})\leq 2^{l_0+1}.
\end{equation}
As $l_0\geq 2000$, by (\ref{Eq2.1.4}), (\ref{Eq5.10}), and (\ref{Eq5.7}), we have
\begin{eqnarray}
&& \int_{15R_0(\bm{\phi})\slash 16}^{17 R_0(\bm{\phi})\slash 16} \frac{\Psi(x)^2}{x+1}dx  \leq \int_{2^{l_0}}^{2^{l_0+1}}\frac{\Psi(x)^2}{x+1}dx=\sum_{j=1}^{2l_0}\int_{2^{l_0}(1+(j-1)\slash (2l_0))}^{2^{l_0}(1+j\slash (2l_0))}\frac{\Psi(x)^2}{x+1}dx\nonumber\\
&&  \leq \sum_{j=1}^{2l_0} (\tilde{\Delta}_{l_0,j})^2\int_{2^{l_0}(1+(j-1)\slash (2l_0))}^{2^{l_0}(1+j\slash (2l_0))}\frac{1}{x}dx=\sum_{j=1}^{2l_0} (\tilde{\Delta}_{l_0,j})^2 \log\Big(\frac{2l_0+j}{2l_0+j-1}\Big) \nonumber\\
&& = \sum_{j=1}^{2 l_0} (\tilde{\Delta}_{l_0,j})^2\log\Big(1+\frac{1}{2l_0+j-1}\Big)\leq \sum_{j=1}^{2l_0}\frac{(\tilde{\Delta}_{l_0,j})^2}{2l_0+j-1}\nonumber\\
&& \leq \sum_{j=1}^{2l_0}\frac{(\tilde{\Delta}_{l_0,j})^2}{l_0+1} \leq \frac{\epsilon k^2}{(l_0+1)\log(l_0+1)}\leq \frac{C\epsilon k^2}{\log(R_0(\bm{\phi}))\log\log(R_0(\bm{\phi}))}.
\end{eqnarray}

We show (e) as follows. As $R_0'(\bm{\phi})\in I_{l_0',j_0'}$, by (\ref{Eq2.1.6}) (with $\Lambda$ replaced by $\Lambda_0$), (\ref{Eq2.2.1}), and (\ref{Eq2.1.15}), we have  
\begin{equation}
 |\Psi(R_0'(\bm{\phi}))|\leq \tilde{\Delta}_{l_0',j_0'} \leq C\sqrt{\frac{\epsilon}{l_0'+1}}k \leq C\sqrt{\frac{\epsilon}{\log(R_0'(\bm{\phi}))}}k\leq C\sqrt{\frac{\epsilon}{\log(R_0(\bm{\phi}))}}k.
\end{equation}

Finally we show (f). By (\ref{Eq2.1.3}) (with $\Lambda$ replaced by $\Lambda_0$), (\ref{Eq2.1.12}), (\ref{Eq2.1.10}), and (\ref{Eq2.1.15}), we have
\begin{eqnarray}
&& \int_{R_0(\bm{\phi})}^{R_0'(\bm{\phi})}\frac{\Psi(x)^2}{x+1}dx\leq \int_{2^{l_0}}^{2^{l_0'+1}} \frac{\Psi(x)^2}{x+1} dx  =\sum_{l=l_0}^{l_0'}\sum_{j=1}^{2l} \int_{2^l(1+(j-1)\slash (2l))}^{2^l(1+j\slash (2l))}\frac{\Psi(x)^2}{x+1} dx \nonumber\\
&\leq& \sum_{l=l_0}^{l_0'}\sum_{j=1}^{2l} (\tilde{\Delta}_{l,j})^2 \int_{2^l(1+(j-1)\slash (2l))}^{2^l(1+j\slash (2l))}\frac{1}{x}dx= \sum_{l=l_0}^{l_0'}\sum_{j=1}^{2l} (\tilde{\Delta}_{l,j})^2\log\Big(\frac{2l+j}{2l+j-1}\Big)\nonumber\\
&=& \sum_{l=l_0}^{l_0'}\sum_{j=1}^{2l} (\tilde{\Delta}_{l,j})^2\log\Big(1+\frac{1}{2l+j-1}\Big) \leq \sum_{l=l_0}^{l_0'}\sum_{j=1}^{2l} \frac{(\tilde{\Delta}_{l,j})^2}{2l+j-1}\nonumber\\
& \leq  & \sum_{l=l_0}^{l_0'}\sum_{j=1}^{2l} \frac{(\tilde{\Delta}_{l,j})^2}{l+1}\leq \sum_{l=2^{W_{i_0}}+1}^{2^{W_{i_0+1}}}\sum_{j=1}^{2l}\frac{(\tilde{\Delta}_{l,j})^2}{l+1}\leq C\epsilon k^2.
\end{eqnarray}

\end{proof}

\section{Proof of the large deviation principle}\label{Sect.5}

In this section, we establish the large deviation principle for $\{\nu_{k;R}\}_{k=1}^{\infty}$, thereby proving Theorem \ref{Thm1}. We note that by \cite[Theorem 1.1(c)]{Zho}, $\{\nu_{k;R}\}_{k=1}^{\infty}$ is exponentially tight. Hence to establish the large deviation princile, it suffices to establish the local large deviation upper and lower bounds. We establish these two bounds in Sections \ref{local_upp} and \ref{local_low}. Some preliminary bounds are presented in Section \ref{Sect.5.1}. 

Throughout this section, we assume that
\begin{equation}
\beta, n\in\mathbb{N}^{*}, \quad n^{10^{-7}}\leq k\leq n^{1\slash 20000}.
\end{equation}
As $T\geq 1$ only depends on $M,\epsilon$ and $k\geq K_0(M,\epsilon,\delta)$ is sufficiently large (depending on $\beta,M,\epsilon,\delta$), we have $n^{10^{-7}}\leq 2^T k \leq k^2\leq n^{1\slash 10000}$. By \cite[Theorem 1.5]{Zho} (replacing $k$ by $2^T k$), there exists a coupling (which we fix throughout this section) between the stochastic Airy operator $H_{\beta}$ and the Gaussian $\beta$ ensemble $H_{\beta,n}$ and positive constants $C,C_0,c$ that only depend on $\beta$, such that the following holds. Denote by $\mathscr{E}$ the event that for any $i\in [2^T k]$, $|\tilde{\lambda}_i^{(n)}+\lambda_i|\leq C_0 n^{-1\slash 24}$. Then 
\begin{equation}\label{Eq12.40}
\mathbb{P}(\mathscr{E}^c)\leq C\exp(-c k^3).
\end{equation}
Note that when the event $\mathscr{E}$ holds, for any $i\in [2^Tk]$, we have 
\begin{equation}\label{Eq5.23.1}
  |k^{2\slash 3}b_i-\lambda_i|\leq C_0 n^{-1\slash 24}.
\end{equation}

\subsection{Some preliminary bounds}\label{Sect.5.1}

In this subsection, we establish some preliminary bounds that will be used in later parts of this section. Recall the definition of $\Phi$ below (\ref{Eq2.14}) in Section \ref{Sect.2} and the definitions of $\{b_i\}_{i=1}^n$ and $\mu_{n,k}$ from Section \ref{Sect.1.3}. Throughout this subsection, we fix an arbitrary $\bm{\phi}\in\Phi$. 

We start by the following definition.

\begin{definition}\label{Defn5.1}
For any $\bm{\phi}\in\Phi$ and any $n_0\in\{0\}\cup [n]$, we define $\mathcal{D}(\bm{\phi},n_0)$ to be the event that $|\{i\in [n]: b_i\in [-R_0(\bm{\phi}),R_0(\bm{\phi})]\}|=n_0$.
\end{definition}

By (\ref{Eq2.1.10})-(\ref{Eq2.1.15}), as $k\geq K_0(M,\epsilon,\delta)$ is sufficiently large (depending on $\beta,M,\epsilon,\delta$) and $T$ only depends on $M,\epsilon$, we have
\begin{equation}\label{Eq5.11}
 2000\leq R_0(\bm{\phi})\leq 2^{T\slash 10}\leq k^2.
\end{equation}
As $M\geq M_0$ is sufficiently large and $\epsilon\in (0,\epsilon_0)$ is sufficiently small, by (\ref{Eq2.1.10}) (which implies that $R_0(\bm{\phi})$ is sufficiently large), we have
\begin{equation}\label{Eq5.13}
 \frac{(1+R_0(\bm{\phi}))^{3\slash 2}}{\sqrt{\log(R_0(\bm{\phi}))\log\log(R_0(\bm{\phi}))}}\geq R_0(\bm{\phi})\geq 10^{M^2\slash \epsilon}\geq \frac{M^2}{\epsilon}\geq \frac{M}{\sqrt{\epsilon}}.
\end{equation}
Hence when the event $\mathcal{B}(\bm{\phi};\Lambda_0)$ holds (where $\Lambda_0$ is the constant that appears in Proposition \ref{P4.4}), by properties (a) and (c) in Proposition \ref{P5.2}, we have 
\begin{equation}\label{Eq5.32}
  |\Psi(R_0(\bm{\phi}))|\leq C\sqrt{\frac{\epsilon}{\log(R_0(\bm{\phi}))\log\log(R_0(\bm{\phi}))}}k,
\end{equation}
\begin{equation}\label{Eq5.12}
  \tilde{N}(-R_0(\bm{\phi}))\leq \frac{CMk}{(1+R_0(\bm{\phi}))^{3\slash 2}}\leq C\sqrt{\frac{\epsilon}{\log(R_0(\bm{\phi}))\log\log(R_0(\bm{\phi}))}}k,
\end{equation}
where we use (\ref{Eq5.13}) for the second inequality in (\ref{Eq5.12}). From (\ref{Eq2.21}), we have 
\begin{equation}
  |\{i\in [n]: b_i\in (-R_0(\bm{\phi}),R_0(\bm{\phi})]\}|=\tilde{N}(R_0(\bm{\phi}))-\tilde{N}(-R_0(\bm{\phi})),
\end{equation}
hence
\begin{eqnarray}
 && |\{i\in [n]: b_i\in [-R_0(\bm{\phi}),R_0(\bm{\phi})]\}|\geq \tilde{N}(R_0(\bm{\phi}))-\tilde{N}(-R_0(\bm{\phi})) \nonumber\\
 &&=\Psi(R_0(\bm{\phi}))-\tilde{N}(-R_0(\bm{\phi}))+\tilde{N}_0(R_0(\bm{\phi})),
\end{eqnarray}
\begin{eqnarray}
 && |\{i\in [n]: b_i\in [-R_0(\bm{\phi}),R_0(\bm{\phi})]\}|\leq \tilde{N}(R_0(\bm{\phi}))-\tilde{N}(-R_0(\bm{\phi}))+1 \nonumber\\
 && =\Psi(R_0(\bm{\phi}))-\tilde{N}(-R_0(\bm{\phi}))+\tilde{N}_0(R_0(\bm{\phi}))+1.
\end{eqnarray}
Therefore, when the event $\mathcal{B}(\bm{\phi};\Lambda_0)$ holds, as $\tilde{N}_0(R_0(\bm{\phi}))=k\mu_0([0,R_0(\bm{\phi})])$, 
\begin{eqnarray}\label{Eq5.14}
   &&||\{i\in [n]: b_i\in [-R_0(\bm{\phi}),R_0(\bm{\phi})]\}|-k\mu_0([0,R_0(\bm{\phi})])|\nonumber\\
   &\leq& |\Psi(R_0(\bm{\phi}))|+\tilde{N}(-R_0(\bm{\phi}))+1\leq C_0\sqrt{\frac{\epsilon}{\log(R_0(\bm{\phi}))\log\log(R_0(\bm{\phi}))}}k,\nonumber\\
   &&
\end{eqnarray}
where $C_0$ is a positive constant that only depends on $\beta$. 

Let $C_0$ be the constant in (\ref{Eq5.14}). Throughout the rest of this subsection, we fix an arbitrary $n_0\in \{0\}\cup[n]$, such that    
\begin{equation}\label{Eq5.15}
  |n_0-k\mu_0([0,R_0(\bm{\phi})])|\leq C_0\sqrt{\frac{\epsilon}{\log(R_0(\bm{\phi}))\log\log(R_0(\bm{\phi}))}}k. 
\end{equation}

In the following, we set up some notations. Let $\Delta:=\{(x,x):x\in\mathbb{R}\}$. We define
\begin{eqnarray}
 && \mathcal{L}_1:=(-\infty,-R_0(\bm{\phi})), \quad \mathcal{L}_2:=[-R_0(\bm{\phi}),0], \quad \mathcal{L}_3:=(0,R_0(\bm{\phi})], \nonumber\\
 && \mathcal{L}_4:=(R_0(\bm{\phi}),R_0'(\bm{\phi})], \quad \mathcal{L}_5:=(R_0'(\bm{\phi}),\infty).
\end{eqnarray}


For every $i\in [5]$, we define $\kappa_{i}$ to be the signed Borel measure on $\mathbb{R}$ such that
\begin{equation}
  \kappa_{i}(A)=\mu_{n,k}(A\cap \mathcal{L}_i)  \quad \text{ for any }A\in\mathcal{B}_{\mathbb{R}}.
\end{equation}
We define signed Borel measures $\Upsilon_{1}$ and $\Upsilon_{2}$ on $\mathbb{R}$ such that for any $A\in\mathcal{B}_{\mathbb{R}}$,
\begin{equation}\label{Eq5.21.2}
 \Upsilon_{1}(A)=\mu_{n,k}(A\cap [-R_0(\bm{\phi}),R_0(\bm{\phi})]), \quad \Upsilon_{2}(A)=\mu_{n,k}(A\cap [-R_0(\bm{\phi}),R_0(\bm{\phi})]^c).
\end{equation}
Note that 
\begin{equation}\label{Eq6.5.1}
\Upsilon_{1}=\kappa_{2}+\kappa_{3}, \quad \Upsilon_{2}=\kappa_{1}+\kappa_{4}+\kappa_{5}.
\end{equation}
For any $i,j\in [5]$, we define
\begin{equation}
   \Xi_{i,j}:=-\int_{\mathbb{R}^2\backslash \Delta} \log(|x-y|) d \kappa_{i}(x) d \kappa_{j}(y).
\end{equation}


As $k\leq n^{1\slash 20000}$, $k\geq K_0(M,\epsilon,\delta)$ is sufficiently large (depending on $\beta,M,\epsilon,\delta$), and $T$ only depends on $M,\epsilon$, noting (\ref{Eq5.11}), we have 
\begin{equation}\label{Eq5.8}
 (k\slash n)^{2\slash 3} R_0(\bm{\phi})\leq (k\slash n)^{2\slash 3} R_0(\bm{\phi})^{5\slash 2}\leq k^{-1} 2^T \leq 1\slash 10.
\end{equation}
Hence we have
\begin{equation*}
 k\mu_0([10,R_0(\bm{\phi})])\leq \pi^{-1} k\int_{10}^{R_0(\bm{\phi})}\sqrt{x}dx=\frac{2}{3\pi} (R_0(\bm{\phi})^{3\slash 2}-10^{3\slash 2}) k,
\end{equation*}
\begin{eqnarray*}
 k\mu_0([1,R_0(\bm{\phi})]) &\geq& \pi^{-1} k \Big(1-\frac{1}{4}(k\slash n)^{2\slash 3}R_0(\bm{\phi})\Big)^{1\slash 2}\int_1^{R_0(\bm{\phi})}\sqrt{x} dx\\
 &\geq& \frac{2}{3\pi}\Big(1-\frac{1}{4}(k\slash n)^{2\slash 3}R_0(\bm{\phi})\Big)(R_0(\bm{\phi})^{3\slash 2}-1)k,
\end{eqnarray*}
which by (\ref{Eq5.8}) lead to
\begin{eqnarray*}
 && k\mu_0([1,R_0(\bm{\phi})])-k\mu_0([10,R_0(\bm{\phi})])\nonumber\\
 &\geq& \frac{2}{3\pi}(10^{3\slash 2}-1)k-\frac{1}{6\pi}(k\slash n)^{2\slash 3} R_0(\bm{\phi})^{5\slash 2} k\geq k\geq 10.
\end{eqnarray*}
Hence there exists $r_0\in [1,10]$, such that $k\mu_0([r_0,R_0(\bm{\phi})])\in\mathbb{N}$. We fix such a choice of $r_0$, and let
\begin{equation}\label{Eq5.16}
 m_0:=k\mu_0([r_0,R_0(\bm{\phi})])\in\mathbb{N}, \quad m_0'=n_0-k\mu_0([r_0,R_0(\bm{\phi})])\in\mathbb{Z}.
\end{equation}
As $k\leq n^{1\slash 20000}$ and $k\geq K_0(M,\epsilon,\delta)$ is sufficiently large, by (\ref{Eq5.11}), we have
\begin{equation}\label{Eq5.17}
  \mu_0([0,1])\geq \frac{1}{2\pi}\int_{0}^1\sqrt{x}dx\geq \frac{1}{10}, \quad \mu_0([0,10])\leq \frac{1}{\pi}\int_0^{10}\sqrt{x}dx\leq 20,
\end{equation}
\begin{equation}\label{Eq5.18}
 \mu_0([r_0,R_0(\bm{\phi})])\geq \mu_0([10,2000])\geq \frac{1}{2\pi}\int_{10}^{2000}\sqrt{x}dx\geq 10
\end{equation}
By (\ref{Eq5.15}) and (\ref{Eq5.16})-(\ref{Eq5.18}), we have $m_0\geq 10$ and
\begin{eqnarray}\label{Eq9.3}
 m_0'&\geq& k\mu_0([0,R_0(\bm{\phi})])-C_0\sqrt{\frac{\epsilon}{\log(R_0(\bm{\phi}))\log\log(R_0(\bm{\phi}))}}k-k\mu_0([r_0,R_0(\bm{\phi})])\nonumber\\
 &\geq& k\mu_0([0,1])-C_0\sqrt{\frac{\epsilon}{\log(R_0(\bm{\phi}))\log\log(R_0(\bm{\phi}))}}k\geq ck\geq 10.
\end{eqnarray}
Hence by (\ref{Eq5.11}), (\ref{Eq5.15}), and (\ref{Eq5.16}), we have
\begin{eqnarray}\label{Eq5.18.5}
 m_0 &\leq& n_0\leq k\mu_0([0,R_0(\bm{\phi})])+C_0\sqrt{\frac{\epsilon}{\log(R_0(\bm{\phi}))\log\log(R_0(\bm{\phi}))}}k\nonumber\\
 &\leq& Ck\int_{0}^{R_0(\bm{\phi})}\sqrt{x}dx \leq C R_0(\bm{\phi})^{3\slash 2} k,
\end{eqnarray}
\begin{eqnarray}\label{Eq5.43}
  m_0'&\leq& k\mu_0([0,R_0(\bm{\phi})])+C_0\sqrt{\frac{\epsilon}{\log(R_0(\bm{\phi}))\log\log(R_0(\bm{\phi}))}}k-k\mu_0([r_0,R_0(\bm{\phi})])\nonumber\\
  &\leq& k\mu_0([0,r_0])+Ck\leq k\pi^{-1}\int_0^{r_0}\sqrt{x}dx+Ck\leq Ck.
\end{eqnarray}

For every $i\in [m_0']$, we let $\rho_i\geq 0$ be such that
\begin{equation}\label{rhh}
 \mu_0([0,\rho_i])=\frac{i}{m_0'+1}\mu_0([0,r_0]).
\end{equation}
For every $i\in [m_0]$, we let $\rho_{m_0'+i}\geq r_0$ be such that
\begin{equation}\label{Eq5.39}
  k\mu_0([r_0,\rho_{m_0'+i}])=i-1.
\end{equation}
We also define $\rho_0:=0$ and $\rho_{n_0+1}:=R_0(\bm{\phi})$. As $k\mu_0([r_0,\rho_{n_0}])=m_0-1$, by (\ref{Eq5.16}), we have $\rho_{n_0}\leq R_0(\bm{\phi})$. Hence we have
\begin{equation}\label{Eq5.21}
 0=\rho_0<\rho_1< \cdots< \rho_{m_0'}< r_0=\rho_{m_0'+1}< \cdots< \rho_{n_0}\leq R_0(\bm{\phi}).
\end{equation}

For any $i\in [m_0'+1,n_0]\cap\mathbb{Z}$, by (\ref{Eq5.16}) and (\ref{Eq5.39}), 
\begin{equation}\label{Eq5.19}
 1=k\mu_0([\rho_i,\rho_{i+1}])\leq k\pi^{-1}\int_{\rho_i}^{\rho_{i+1}}\sqrt{x}dx\leq k\pi^{-1}\sqrt{R_0(\bm{\phi})}(\rho_{i+1}-\rho_i).
\end{equation}
As $k\leq n^{1\slash 20000}$ and $k\geq K_0(M,\epsilon,\delta)$ is sufficiently large (depending on $\beta,M,\epsilon,\delta$), by (\ref{Eq5.11}) and (\ref{Eq5.19}), for any $i\in [m_0'+1,n_0]\cap\mathbb{Z}$, we have 
\begin{equation}\label{Eq5.20}
 \rho_{i+1}-\rho_{i}\geq \pi k^{-1} R_0(\bm{\phi})^{-1\slash 2}\geq 10 n^{-1}.
\end{equation}
Moreover, for any $i\in [m_0'+1,n_0]\cap\mathbb{Z}$, as $\rho_i\geq r_0\geq 1$, we have  
\begin{equation}
 1=k\mu_0([\rho_i,\rho_{i+1}])\geq \frac{k}{2\pi}\int_{\rho_i}^{\rho_{i+1}}\sqrt{x}dx \geq \frac{k(\rho_{i+1}-\rho_i)}{2\pi},
\end{equation}
hence
\begin{equation}\label{Eq5.15.1}
 \rho_{i+1}-\rho_i\leq 2\pi k^{-1}.
\end{equation}

Consider any $i\in [m_0'+1]$. As $\rho_i\in [0,r_0]\subseteq [0,10]$ and $k\geq K_0(M,\epsilon,\delta)$ is sufficiently large, by (\ref{rhh}), we have
\begin{equation}\label{Eq9.1}
    \frac{i}{m_0'+1}\mu_0([0,r_0])=\mu_0([0,\rho_i])
    \begin{cases}
     \leq \frac{1}{\pi} \int_0^{\rho_i}\sqrt{x} dx=\frac{2}{3\pi}\rho_i^{3\slash 2} \\
     \geq \frac{1}{2\pi} \int_0^{\rho_i}\sqrt{x}dx=\frac{1}{3\pi}\rho_i^{3\slash 2} \\
    \end{cases}.
\end{equation}
As $r_0\in [1,10]$, by (\ref{Eq5.17}), (\ref{Eq9.3}), (\ref{Eq5.43}), and (\ref{Eq9.1}), we have
\begin{equation}\label{Eq9.4}
    c\Big(\frac{i}{k}\Big)^{2\slash 3}\leq \rho_i\leq C\Big(\frac{i}{k}\Big)^{2\slash 3}. 
\end{equation}
By (\ref{rhh}) and (\ref{Eq9.4}), we have
\begin{equation}\label{Eq9.5}
    \frac{\mu_0([0,r_0])}{m_0'+1}=\mu_0([\rho_{i-1},\rho_i])\begin{cases}
     \leq \frac{1}{\pi}\int_{\rho_{i-1}}^{\rho_i}\sqrt{x}dx\leq C\big(\frac{i}{k}\big)^{1\slash 3}(\rho_i-\rho_{i-1})\\
     \geq \frac{1}{2\pi}\int_{\rho_{i-1}}^{\rho_i}\sqrt{x}dx\geq c\big(\frac{i-1}{k}\big)^{1\slash 3}(\rho_i-\rho_{i-1})
    \end{cases}.
\end{equation}
By (\ref{Eq5.17}), (\ref{Eq9.3}), (\ref{Eq5.43}), and (\ref{Eq9.5}), for any $i\in [m_0'+1]$ such that $i\geq 2$,
\begin{equation}\label{Eq9.6}
  \frac{c}{k}\Big(\frac{k}{i}\Big)^{1\slash 3} \leq \rho_i-\rho_{i-1}\leq \frac{C}{k}\Big(\frac{k}{i}\Big)^{1\slash 3}.
\end{equation}
Note that by (\ref{Eq9.4}), (\ref{Eq9.6}) also holds for $i=1$.

We let
\begin{equation}
 \mathscr{T}_{n_0}:=\{\mathbf{t}=(t_1,t_2,\cdots,t_{n_0}): 0\leq t_1\leq t_2\leq \cdots \leq t_{n_0}\leq n^{-1}\}.
\end{equation}
Below we consider an arbitrary $\mathbf{t}=(t_1,t_2,\cdots,t_{n_0})\in \mathscr{T}_{n_0}$. By (\ref{Eq5.21}) and (\ref{Eq5.20}), 
\begin{equation}\label{Eq5.25}
  0<\rho_1+t_1< \rho_2+t_2< \cdots < \rho_{n_0}+t_{n_0} \leq  \rho_{n_0}+n^{-1}< R_0(\bm{\phi}).
\end{equation}
We define
\begin{equation}\label{Eq8.2}
 \tilde{\mu}_{n_0,\mathbf{t}}:=\frac{1}{k}\sum_{i=1}^{n_0} \delta_{\rho_i+t_i}+\frac{1}{k}\sum_{\substack{i\in [n]:\\ b_i\notin [-R_0(\bm{\phi}),R_0(\bm{\phi})]}} \delta_{b_i}-\mu_0.
\end{equation}
For every $i\in [5]$, we define $\tilde{\kappa}_{i;n_0,\mathbf{t}}$ to be the signed Borel measure on $\mathbb{R}$ such that \begin{equation}
  \tilde{\kappa}_{i;n_0,\mathbf{t}}(A)=\tilde{\mu}_{n_0,\mathbf{t}}(A\cap \mathcal{L}_i)\quad \text{ for any }A\in\mathcal{B}_{\mathbb{R}}.
\end{equation}
We define signed Borel measures $\tilde{\Upsilon}_{1;n_0,\mathbf{t}}$ and $\tilde{\Upsilon}_{2;n_0,\mathbf{t}}$ on $\mathbb{R}$ such that for any $A\in\mathcal{B}_{\mathbb{R}}$,
\begin{equation}\label{Eq8.1}
 \tilde{\Upsilon}_{1;n_0,\mathbf{t}}(A)=\tilde{\mu}_{n_0,\mathbf{t}}(A\cap [-R_0(\bm{\phi}),R_0(\bm{\phi})]),  
\end{equation}
\begin{equation}\label{Eq14.2}
\tilde{\Upsilon}_{2;n_0,\mathbf{t}}(A)=\tilde{\mu}_{n_0,\mathbf{t}}(A\cap [-R_0(\bm{\phi}),R_0(\bm{\phi})]^c).
\end{equation}
Note that 
\begin{equation}
\tilde{\Upsilon}_{1;n_0,\mathbf{t}}=\tilde{\kappa}_{2;n_0,\mathbf{t}}+\tilde{\kappa}_{3;n_0,\mathbf{t}}, \quad \tilde{\Upsilon}_{2;n_0,\mathbf{t}}=\tilde{\kappa}_{1;n_0,\mathbf{t}}+\tilde{\kappa}_{4;n_0,\mathbf{t}}+\tilde{\kappa}_{5;n_0,\mathbf{t}}.
\end{equation}
For any $i,j\in [5]$, we define
\begin{equation}
   \tilde{\Xi}_{i,j;n_0,\mathbf{t}}:=-\int_{\mathbb{R}^2\backslash \Delta} \log(|x-y|) d \tilde{\kappa}_{i;n_0,\mathbf{t}}(x) d \tilde{\kappa}_{j;n_0,\mathbf{t}}(y).
\end{equation}

We also introduce the following three definitions.

\begin{definition}\label{Defn5.3.n}
For any $R\geq 10$, we define $\mathscr{G}_R$ to be the event that 
\begin{eqnarray}\label{Eq5.24.2}
&& |\{i\in [n]: b_i\in [-R-n^{-1\slash 30},-R+n^{-1\slash 30}]\nonumber\\
&&\quad\quad\quad\quad\quad\cup [R-n^{-1\slash 30},R+n^{-1\slash 30}]\}|\leq \frac{k}{\log(n)^{1\slash 4}}.
\end{eqnarray}
\end{definition}

\begin{definition}\label{Defn5.2}
For any $\bm{\phi}\in\Phi$ and $\Lambda\geq 0$, we define $\mathcal{C}(\bm{\phi};\Lambda)$ to be the event that for any $x\in [-R_0(\bm{\phi}),R_0(\bm{\phi})]$, 
\begin{equation}\label{Eq5.19.2}
   \Big|\int\frac{1}{x-y}d\kappa_5(y)\Big|\leq \Lambda  M R_0'(\bm{\phi})^{-1\slash 2}.
\end{equation} 
\end{definition}

By (\ref{Eq2.1.10})-(\ref{Eq2.1.15}), we have $2000\leq 10^{M^2\slash \epsilon}\leq 2 R_0(\bm{\phi})\leq R_0'(\bm{\phi})\leq 2^T$. Note that $T$ only depends on $M,\epsilon$ and $k\geq K_0(M,\epsilon,\delta)$ is sufficiently large (depending on $\beta, M,\epsilon,\delta$). Following the proof of \cite[Lemma 6.4]{Zho}, assuming that $n^{5\cdot 10^{-7}}\leq k\leq  n^{10^{-6}}$, we obtain that there exists a positive absolute constant $\Lambda_0'$, such that
\begin{equation}\label{Eq5.19.1}
\mathbb{P}(\mathcal{C}(\bm{\phi};\Lambda_0')^c)\leq C\exp(-Mk^2).
\end{equation}

\begin{definition}\label{Defn5.3}
Let $\rho(x)=\frac{\sqrt{4-x^2}}{2\pi}\mathbbm{1}_{[-2,2]}(x)$ be the density of the Wigner semicircle law. For any $x\in\mathbb{R}$, we define 
\begin{equation}\label{Eq7.6.1}
    \xi(x):=-\int_{-2}^2  \log(|x-y|)\rho(y)dy+\frac{1}{4}x^2-\frac{1}{2}, \quad \tilde{\xi}(x):=\frac{n}{k} \xi(2-(k\slash n)^{2\slash 3}x).
\end{equation}
\end{definition}

We have the following five propositions. In all these propositions, we let $\Lambda_0$ be the constant that appears in Proposition \ref{P4.4} and $\Lambda_0'$ the constant in (\ref{Eq5.19.1}). We also recall that we have fixed an arbitrary $\bm{\phi}\in \Phi$ and an arbitrary number $n_0\in \{0\}\cup [n]$ such that (\ref{Eq5.15}) holds. 

\begin{proposition}\label{P5.3}
There exists a positive constant $C$ that only depends on $\beta$, such that when the event $\mathcal{B}(\bm{\phi};\Lambda_0)\cap \mathcal{D}(\bm{\phi},n_0)\cap\mathcal{C}(\bm{\phi};\Lambda_0')$ holds, for any\\ $\mathbf{t}=(t_1,t_2,\cdots,t_{n_0})\in \mathscr{F}_{n_0}$, we have  
\begin{equation}
-\int \log(|x-y|) d\Upsilon_1(x)d\Upsilon_2(y)+\int \log(|x-y|) d\tilde{\Upsilon}_{1;n_0,\mathbf{t}}(x)d\tilde{\Upsilon}_{2;n_0,\mathbf{t}}(y)\geq -CM\sqrt{\epsilon}.
\end{equation}
\end{proposition}

\begin{proposition}\label{P5.4}
There exist positive constants $C,c$ that only depend on $\beta$, such that for any $R\in [10, R_0(\bm{\phi})-10]$, the following holds. When the event $\mathcal{B}(\bm{\phi};\Lambda_0)\cap \mathcal{D}(\bm{\phi},n_0)\cap\mathcal{C}(\bm{\phi};\Lambda_0')\cap (\mathscr{G}_R)^c$ holds, we have
\begin{equation}
     -\int_{\mathbb{R}^2\backslash \Delta} \log(|x-y|) d\Upsilon_1(x)d\Upsilon_1(y)+2 \int \tilde{\xi}(x) d\Upsilon_1(x)\geq c\sqrt{\log(n)}-C\cdot 2^T.
\end{equation}
\end{proposition}

\begin{proposition}\label{P5.5} 
For any $R\in [10,R_0(\bm{\phi})-10]$ and any $\mu\in \mathcal{X}_R$, we have
\begin{eqnarray}\label{Eq5.28.1}
&& \mathcal{B}(\bm{\phi};\Lambda_0)\cap \mathcal{D}(\bm{\phi},n_0)\cap\mathcal{C}(\bm{\phi};\Lambda_0')\cap \mathscr{G}_R \cap \mathscr{E}\cap \{d_R(\nu_{k;R},\mu)\leq \delta\} \nonumber\\
&\subseteq& \mathcal{B}(\bm{\phi};\Lambda_0)\cap \mathcal{D}(\bm{\phi},n_0)\cap\mathcal{C}(\bm{\phi};\Lambda_0')\cap \mathscr{G}_R \cap \{d_R(\mu_{n,k;R},\mu)\leq 2\delta\},
\end{eqnarray}
where we recall Definition \ref{Defn1.1}.
\end{proposition}

\begin{proposition}\label{P5.6}
Recall Definition \ref{Defn1.3}. There exist positive constants $C,C_1$ that only depend on $\beta$, such that for any $R\in [10, R_0(\bm{\phi})-10]$ and any $\mu\in \mathcal{X}_R$, the following holds. When the event
\begin{equation*}
    \mathcal{B}(\bm{\phi};\Lambda_0)\cap \mathcal{D}(\bm{\phi},n_0)\cap\mathcal{C}(\bm{\phi};\Lambda_0')\cap \mathscr{G}_R \cap \{d_R(\mu_{n,k;R},\mu)\leq 2\delta\}
\end{equation*}
holds, we have
\begin{eqnarray}
&& -\int_{\mathbb{R}^2\backslash \Delta} \log(|x-y|)d\Upsilon_1(x)d\Upsilon_1(y)+2\int\tilde{\xi}(x)d\Upsilon_1(x)\nonumber\\
&&\geq\Big(1+\frac{C_1}{\log(n)}\Big)^{-1}I_R(\mu,3\delta)-C(M\sqrt{\epsilon}+M^{-1}).
\end{eqnarray}
\end{proposition}

\begin{proposition}\label{P5.7}
There exists a positive constant $C$ that only depends on $\beta$, such that for any $\mathbf{t}=(t_1,t_2,\cdots,t_{n_0})\in \mathscr{F}_{n_0}$, we have
\begin{equation}
   \Big|-\int_{\mathbb{R}^2\backslash\Delta}\log(|x-y|)d\tilde{\Upsilon}_{1;n_0,\mathbf{t}}(x)d\tilde{\Upsilon}_{1;n_0,\mathbf{t}}(y)+2\int \tilde{\xi}(x)d\tilde{\Upsilon}_{1;n_0,\mathbf{t}}(x)\Big|\leq C\sqrt{\epsilon}.
\end{equation}
\end{proposition}

The rest of this subsection is devoted to the proofs of Propositions \ref{P5.3}-\ref{P5.7}.

\subsubsection{Proof of Proposition \ref{P5.3}}

\begin{proof}[Proof of Proposition \ref{P5.3}]

Throughout the proof, we assume that the event\\ $\mathcal{B}(\bm{\phi};\Lambda_0)\cap \mathcal{D}(\bm{\phi},n_0)\cap\mathcal{C}(\bm{\phi};\Lambda_0')$ holds. We also fix any $\mathbf{t}=(t_1,t_2,\cdots,t_{n_0})\in \mathscr{F}_{n_0}$. Note that
\begin{eqnarray}\label{Eq5.19.8}
&&-\int \log(|x-y|) d\Upsilon_1(x)d\Upsilon_2(y)+\int \log(|x-y|) d\tilde{\Upsilon}_{1;n_0,\mathbf{t}}(x)d\tilde{\Upsilon}_{2;n_0,\mathbf{t}}(y)\nonumber\\
&=&\sum_{i\in \{2,3\}}\sum_{j\in\{1,4,5\}}(\Xi_{i,j}-\tilde{\Xi}_{i,j;n_0,\mathbf{t}})\nonumber\\
&\geq& \sum_{i\in \{2,3\}}\sum_{j\in\{1,4\}}(\Xi_{i,j}-\tilde{\Xi}_{i,j;n_0,\mathbf{t}})-|(\Xi_{2,5}+\Xi_{3,5})-(\tilde{\Xi}_{2,5;n_0,\mathbf{t}}+\tilde{\Xi}_{3,5;n_0,\mathbf{t}})|.\nonumber\\
&& 
\end{eqnarray}
We bound the quantities on the right-hand side of (\ref{Eq5.19.8}) in \textbf{Steps 1-5} below.

\paragraph{Step 1} 

In this step, we bound $\Xi_{2,1}$ and $\tilde{\Xi}_{2,1;n_0,\mathbf{t}}$. 

As the event $\mathcal{B}(\bm{\phi},\Lambda_0)$ holds, by property (c) of Proposition \ref{P5.2}, we have $\tilde{N}(-k^2)=0$. Hence $b_i\geq -k^2$ for any $i\in [n]$. Therefore, 
\begin{equation}\label{Eq5.26}
  \kappa_1=\tilde{\kappa}_{1;n_0,\mathbf{t}}=\frac{1}{k}\sum_{\substack{i\in [n]:\\ -k^2\leq b_i< -R_0(\bm{\phi})}}\delta_{b_i}, \quad \kappa_2=\frac{1}{k}\sum_{\substack{i\in [n]:\\ -R_0(\bm{\phi})\leq b_i\leq 0}}\delta_{b_i}.
\end{equation}
For any $x\in \mathcal{L}_1, y\in\mathcal{L}_2$, we have $|x-y|\leq |x|$. Hence
\begin{eqnarray}\label{Eq5.23}
 && \Xi_{2,1}\geq -\int_{\mathbb{R}^2\backslash \Delta} \log(|x|) d\kappa_1(x) d\kappa_2(y)\nonumber\\
 &=& -k^{-2}|\{i\in [n]: -R_0(\bm{\phi})\leq b_i\leq 0\}|\Big(\sum_{\substack{i\in [n]: \\ -k^2\leq b_i<-R_0(\bm{\phi})}}\log(|b_i|)\Big).
\end{eqnarray}
By property (c) of Proposition \ref{P5.2} and (\ref{Eq5.11}), we have
\begin{equation}\label{Eq5.24}
 |\{i\in [n]: -R_0(\bm{\phi})\leq b_i\leq 0\}|\leq \tilde{N}(0)\leq CM k,
\end{equation}
\begin{eqnarray}\label{Eq5.22}
&& \sum_{\substack{i\in [n]: \\ -k^2\leq b_i<-R_0(\bm{\phi})}}\log(|b_i|)\nonumber\\
&\leq& \sum_{\substack{i\in [n]:\\ -\lceil R_0(\bm{\phi})\rceil \leq b_i<-R_0(\bm{\phi})}}\log(|b_i|)+\sum_{j=\lceil R_0(\bm{\phi})\rceil}^{k^2-1}\sum_{\substack{i\in [n]:\\ -(j+1)\leq b_i<-j}} \log(|b_i|)\nonumber\\
&\leq& \log(\lceil R_0(\bm{\phi})\rceil )\tilde{N}(-R_0(\bm{\phi}))+\sum_{j=\lceil R_0(\bm{\phi})\rceil}^{k^2-1}\log(j+1) \tilde{N}(-j)\nonumber\\
&\leq& \log(\lceil R_0(\bm{\phi})\rceil )\cdot\frac{CMk}{(R_0(\bm{\phi})+1)^{3\slash 2}}+\sum_{j=\lceil R_0(\bm{\phi})\rceil}^{k^2-1}\log(j+1)\cdot\frac{CMk }{(j+1)^{3\slash 2}}\nonumber\\
&\leq& \sum_{j=\lceil R_0(\bm{\phi})\rceil }^{\infty} CMk \log(j) j^{-3\slash 2}\leq \sum_{j=\lceil R_0(\bm{\phi})\rceil }^{\infty} CMk j^{-4\slash 3}\nonumber\\
&\leq& CMk \int_{\lceil R_0(\bm{\phi})\rceil-1}^{\infty} x^{-4\slash 3} dx\leq CMR_0(\bm{\phi})^{-1\slash 3} k,
\end{eqnarray}
where we use the inequality $\log(x)\leq 6 x^{1\slash 6}$ for any $x\geq 1$ in the fifth inequality in (\ref{Eq5.22}). By (\ref{Eq5.23})-(\ref{Eq5.22}), we have
\begin{equation}\label{Eq5.19.9}
  \Xi_{2,1}\geq -C M^2 R_0(\bm{\phi})^{-1\slash 3}.
\end{equation}

By (\ref{Eq5.25}), we have $\tilde{\kappa}_{2;n_0,\mathbf{t}}=0$. Hence
\begin{equation}\label{Eq5.19.10}
 \tilde{\Xi}_{2,1;n_0,\mathbf{t}}=0.
\end{equation}

\paragraph{Step 2}

In this step, we bound $\Xi_{3,1}$ and $\tilde{\Xi}_{3,1;n_0,\mathbf{t}}$.

By (\ref{Eq5.26}), we have
\begin{eqnarray}\label{Eq5.30}
\Xi_{3,1}&=&-\int_{\mathbb{R}^2\backslash \Delta} \log(|x-y|) d\kappa_1(x)d\kappa_3(y)\nonumber\\
&=& -k^{-1} \sum_{\substack{i\in [n]: \\ -k^2\leq b_i<-R_0(\bm{\phi})}} \int \log(|b_i-x|) d\kappa_3(x).
\end{eqnarray}
For any $i\in [n]$ such that $-k^2\leq b_i<-R_0(\bm{\phi})$, we have
\begin{eqnarray}\label{Eq5.27}
 && \int \log(|b_i-x|)d\kappa_3(x) = \frac{1}{k}\sum_{\substack{j\in [n]:\\0<b_j\leq R_0(\bm{\phi})}}\log(|b_i-b_j|)-\int_0^{R_0(\bm{\phi})}\log(|b_i-x|)d\mu_0(x) \nonumber\\
 &&= \frac{1}{k}\Big(\int_0^{R_0(\bm{\phi})} \log(|b_i-x|)d\tilde{N}(x)-\int_0^{R_0(\bm{\phi})}\log(|b_i-x|)d\tilde{N}_0(x)\Big), 
\end{eqnarray}
where the integrals in the second line are Riemann-Stieltjes integrals. By integration by parts, 
\begin{eqnarray}\label{Eq5.28}
\int_0^{R_0(\bm{\phi})} \log(|b_i-x|)d\tilde{N}(x)&=&\tilde{N}(R_0(\bm{\phi}))\log(|b_i-R_0(\bm{\phi})|)-\tilde{N}(0)\log(|b_i|)\nonumber\\
&&-\int_{0}^{R_0(\bm{\phi})}\frac{\tilde{N}(x)}{x-b_i}dx,
\end{eqnarray}
\begin{eqnarray}\label{Eq5.29}
\int_0^{R_0(\bm{\phi})}\log(|b_i-x|)d\tilde{N}_0(x)&=&\tilde{N}_0(R_0(\bm{\phi}))\log(|b_i-R_0(\bm{\phi})|)-\tilde{N}_0(0)\log(|b_i|)\nonumber\\
&&-\int_0^{R_0(\bm{\phi})}\frac{\tilde{N}_0(x)}{x-b_i}dx.
\end{eqnarray}
By (\ref{Eq5.27})-(\ref{Eq5.29}), for any $i\in [n]$ such that $-k^2\leq b_i<-R_0(\bm{\phi})$, we have
\begin{eqnarray}\label{Eq5.33}
\int \log(|b_i-x|)d\kappa_3(x)&=&\frac{1}{k}\Big((\Psi(R_0(\bm{\phi}))\log(|R_0(\bm{\phi})-b_i|)-\Psi(0)\log(|b_i|)\nonumber\\
&& -\int_0^{R_0(\bm{\phi})}\frac{\Psi(x)}{x-b_i}dx\Big).
\end{eqnarray}
By (\ref{Eq5.32}) and property (c) of Proposition \ref{P5.2},
\begin{equation}\label{Eq5.34}
 |\Psi(R_0(\bm{\phi}))|\leq Ck, \quad |\Psi(0)|=\tilde{N}(0)\leq CMk.
\end{equation}

Below we consider any $i\in [n]$ such that $-k^2\leq b_i<-R_0(\bm{\phi})$. By (\ref{Eq5.11}), $|b_i|\geq R_0(\bm{\phi})\geq 2000$, hence
\begin{equation}\label{Eq5.35}
 |\log(|R_0(\bm{\phi})-b_i|)|=\log(R_0(\bm{\phi})-b_i)\leq \log(2|b_i|)\leq 2\log(|b_i|),
\end{equation}
\begin{equation}\label{Eq5.36}
  |\log(|b_i|)|=\log(|b_i|).
\end{equation}
By (\ref{Eq2.2.1}) and property (b) of Proposition \ref{P5.2}, 
\begin{equation}
 \int_0^{R_0(\bm{\phi})}\frac{\Psi(x)^2}{x-b_i}dx\leq \int_0^{R_0(\bm{\phi})}\frac{\Psi(x)^2}{x+1}dx\leq \int_0^{R_0'(\bm{\phi})}\frac{\Psi(x)^2}{x+1}dx\leq CM^2 k^2.
\end{equation}
Hence by the Cauchy-Schwarz inequality,
\begin{eqnarray}\label{Eq5.37}
   \Big|\int_0^{R_0(\bm{\phi})}\frac{\Psi(x)}{x-b_i}dx\Big|&\leq& \Big(\int_0^{R_0(\bm{\phi})}\frac{\Psi(x)^2}{x-b_i}dx\Big)^{1\slash 2}\Big(\int_0^{R_0(\bm{\phi})}\frac{1}{x-b_i}dx\Big)^{1\slash 2}   \nonumber\\
 &\leq& CMk \Big(\int_0^{R_0(\bm{\phi})}\frac{1}{x+R_0(\bm{\phi})}dx\Big)^{1\slash 2} \leq CMk.
\end{eqnarray}
By (\ref{Eq5.33})-(\ref{Eq5.36}) and (\ref{Eq5.37}), 
\begin{equation}\label{Eq5.38}
 \Big|\int \log(|b_i-x|)d\kappa_3(x)\Big|\leq CM\log(|b_i|).
\end{equation}

By (\ref{Eq5.22}), (\ref{Eq5.30}), and (\ref{Eq5.38}), we conclude that
\begin{equation}\label{Eq5.19.11}
 |\Xi_{3,1}|\leq CMk^{-1}\sum_{\substack{i\in [n]:\\-k^2\leq b_i<-R_0(\bm{\phi})}}\log(|b_i|)\leq CM^2 R_0(\bm{\phi})^{-1\slash 3}.
\end{equation}

By (\ref{Eq5.26}), we have
\begin{equation}\label{Eq5.48}
 \tilde{\Xi}_{3,1;n_0,\mathbf{t}}=-k^{-1}\sum_{\substack{i\in[n]:\\-k^2\leq b_i<-R_0(\bm{\phi})}} \int\log(|b_i-x|)d\tilde{\kappa}_{3;n_0,\mathbf{t}}(x).
\end{equation}

Below we consider any $i\in [n]$ such that $-k^2\leq b_i<-R_0(\bm{\phi})$. By (\ref{Eq5.25}), 
\begin{equation}\label{Eq5.47}
\int\log(|b_i-x|)d\tilde{\kappa}_{3;n_0,\mathbf{t}}(x)=\frac{1}{k}\sum_{j=1}^{n_0}\log(|\rho_j+t_j-b_i|)-\int_0^{R_0(\bm{\phi})}\log(|x-b_i|)d\mu_0(x).
\end{equation}

For any $j\in [m_0'+1,n_0]\cap\mathbb{Z}$, $\rho_j\geq r_0\geq 1$. By (\ref{Eq5.16}), (\ref{Eq5.39}), and (\ref{Eq5.25}), 
\begin{equation}
 \frac{1}{k}=\mu_0([\rho_j,\rho_{j+1}])\geq \frac{1}{2\pi}\int_{\rho_j}^{\rho_{j+1}}\sqrt{x}dx\geq \frac{\rho_{j+1}-\rho_j}{2\pi}, \text{ hence } \rho_{j+1}-\rho_j\leq \frac{2\pi}{k},
\end{equation}
Hence for any $j\in [m_0'+1,n_0]\cap\mathbb{Z}$ and $x\in [\rho_j,\rho_{j+1}]$, 
\begin{equation}\label{Eq5.41}
 |\rho_j+t_j-x|\leq |\rho_j-x|+t_j\leq \rho_{j+1}-\rho_j+n^{-1}\leq 2\pi k^{-1}+n^{-1}\leq 10 k^{-1}.
\end{equation}
By (\ref{Eq5.41}) and the inequality $\log(1+z)\leq z$ for any $z>-1$, we have
\begin{equation}
 \log\Big(\frac{\rho_j+t_j-b_i}{x-b_i}\Big)=\log\Big(1+\frac{\rho_j+t_j-x}{x-b_i}\Big)\leq \frac{\rho_j+t_j-x}{x-b_i}\leq \frac{10}{R_0(\bm{\phi}) k},
\end{equation}
\begin{equation}
\log\Big(\frac{x-b_i}{\rho_j+t_j-b_i}\Big)=\log\Big(1+\frac{x-(\rho_j+t_j)}{\rho_j+t_j-b_i}\Big)\leq \frac{x-(\rho_j+t_j)}{\rho_j+t_j-b_i}\leq \frac{10}{R_0(\bm{\phi}) k}.
\end{equation}
Hence we have
\begin{equation}\label{Eq5.42}
 |\log(|\rho_j+t_j-b_i|)-\log(|x-b_i|)|\leq \frac{10}{R_0(\bm{\phi}) k}. 
\end{equation}

By (\ref{Eq5.39}) and (\ref{Eq5.42}), we have
\begin{eqnarray}\label{Eq5.44}
&& \Big|\frac{1}{k}\sum_{j=m_0'+1}^{n_0}\log(|\rho_j+t_j-b_i|)-\int_{r_0}^{R_0(\bm{\phi})}\log(|x-b_i|)d\mu_0(x)\Big|\nonumber\\
&=& \Big|\sum_{j=m_0'+1}^{n_0}\int_{\rho_j}^{\rho_{j+1}}(\log(|\rho_j+t_j-b_i|)-\log(|x-b_i|))d\mu_0(x)\Big|\nonumber\\
&\leq& \frac{10}{R_0(\bm{\phi})k}\sum_{j=m_0'+1}^{n_0}\mu_0([\rho_j,\rho_{j+1}])=\frac{10\mu_0([r_0,R_0(\bm{\phi})])}{R_0(\bm{\phi})k}\nonumber\\
&\leq& \frac{10}{\pi R_0(\bm{\phi}) k}\int_{r_0}^{R_0(\bm{\phi})}\sqrt{x}dx\leq \frac{10}{\pi R_0(\bm{\phi}) k}\cdot \frac{2}{3}R_0(\bm{\phi})^{3\slash 2}\nonumber\\
&\leq& 5 R_0(\bm{\phi})^{1\slash 2} k^{-1}.
\end{eqnarray}
By (\ref{Eq5.11}), for any $x\in [0,r_0]$, we have
\begin{equation*}
    x-b_i\geq -b_i\geq R_0(\bm{\phi})\geq 2000, \quad x-b_i\leq 10-b_i\leq  R_0(\bm{\phi})-b_i\leq -2b_i=2|b_i|.
\end{equation*}
Hence we have
\begin{eqnarray}\label{Eq5.45}
&& \Big|\int_0^{r_0}   \log(|x-b_i|) d\mu_0(x) \Big|=\int_0^{r_0}   \log(|x-b_i|) d\mu_0(x) \nonumber\\
&\leq& \log(2|b_i|)\pi^{-1}\int_{0}^{r_0}\sqrt{x}dx\leq 2\pi^{-1}r_0^{3\slash 2} \log(|b_i|)\leq 30\log(|b_i|).
\end{eqnarray}
For any $j\in [m_0']$, by (\ref{Eq5.11}) and (\ref{Eq5.25}), $\rho_j+t_j-b_i\leq R_0(\bm{\phi})-b_i\leq -2b_i=2|b_i|$ and $\rho_j+t_j-b_i\geq -b_i\geq R_0(\bm{\phi})\geq 2000$. Hence by (\ref{Eq5.43}), we have
\begin{eqnarray}\label{Eq5.46}
    \Big|\frac{1}{k}\sum_{j=1}^{m_0'}\log(|\rho_j+t_j-b_i|)\Big|&=&\frac{1}{k}\sum_{j=1}^{m_0'}\log(\rho_j+t_j-b_i)\leq m_0' k^{-1}\log(2|b_i|)\nonumber \\
   &\leq& 2 m_0' k^{-1} \log(|b_i|)\leq C\log(|b_i|).
\end{eqnarray}
By (\ref{Eq5.47}) and (\ref{Eq5.44})-(\ref{Eq5.46}), as $\log(|b_i|)\geq \log(R_0(\bm{\phi}))\geq 1\geq 5R_0(\bm{\phi})^{1\slash 2} k^{-1}$, 
\begin{equation}\label{Eq5.49}
 \Big|\int\log(|b_i-x|)d\tilde{\kappa}_{3;n_0,\mathbf{t}}(x)\Big|\leq C\log(|b_i|).
\end{equation}

By (\ref{Eq5.22}), (\ref{Eq5.48}), and (\ref{Eq5.49}), we conclude that
\begin{equation}\label{Eq5.19.12}
|\tilde{\Xi}_{3,1;n_0,\mathbf{t}}|\leq Ck^{-1}\sum_{\substack{i\in [n]:\\ -k^2\leq b_i<-R_0(\bm{\phi})}} \log(|b_i|)\leq CM R_0(\bm{\phi})^{-1\slash 3}.
\end{equation}

\paragraph{Step 3}

In this step, we bound $\Xi_{2,4}+\Xi_{3,4}$.

We define signed Borel measures $\{\eta_i\}_{i=0}^3$ and $\{\eta'_i\}_{i=1}^2$ on $\mathbb{R}$ such that for any $A\in\mathcal{B}_{\mathbb{R}}$,
\begin{eqnarray}
 && \eta_0(A)=\mu_{n,k}(A\cap (-R_0(\bm{\phi}),R_0(\bm{\phi})]),\quad \eta_1(A)=\mu_{n,k}(A\cap \{-R_0(\bm{\phi})\}),\nonumber\\
 && \eta_2(A)=\mu_{n,k}(A\cap (-R_0(\bm{\phi}),R_0(\bm{\phi})-R_0(\bm{\phi})^{-4}]),\nonumber\\
 && \eta_3(A)=\mu_{n,k}(A\cap (R_0(\bm{\phi})-R_0(\bm{\phi})^{-4},R_0(\bm{\phi})]),\nonumber\\
 && \eta'_1(A)=\mu_{n,k}(A\cap (R_0(\bm{\phi}),R_0(\bm{\phi})+R_0(\bm{\phi})^{-4}]),\nonumber\\
 && \eta'_2(A)=\mu_{n,k}(A\cap (R_0(\bm{\phi})+R_0(\bm{\phi})^{-4},R_0'(\bm{\phi})]).
\end{eqnarray}

Let 
\begin{eqnarray}
&& \mathscr{R}_1:=-\int \log(|x-y|)d\eta_1(x)d\kappa_4(y),  \mathscr{R}_2:=-\int \log(|x-y|)d\eta_3(x)d\eta'_1(y),\nonumber\\
&& \mathscr{R}_3:=-\int \log(|x-y|)d\eta_0(x) d\eta'_2(y), \mathscr{R}_4:=-\int \log(|x-y|)d\eta_2(x) d\eta'_1(y). \nonumber\\
&&
\end{eqnarray}
Note that
\begin{equation}\label{Eq5.20.1}
\Xi_{2,4}+\Xi_{3,4}=\mathscr{R}_1+\mathscr{R}_2+\mathscr{R}_3+\mathscr{R}_4.
\end{equation}
In the following, we bound $\mathscr{R}_1, \mathscr{R}_2, \mathscr{R}_3, \mathscr{R}_4$ in \textbf{Sub-steps 3.1-3.4}, respectively. 

\subparagraph{Sub-step 3.1}

In this sub-step, we bound $\mathscr{R}_1$. If there does not exist any $i\in [n]$ such that $b_i=-R_0(\bm{\phi})$, then $\eta_1=0$ and $\mathscr{R}_1=0$. Otherwise, we have $\eta_1=k^{-1}\delta_{-R_0(\bm{\phi})}$ and $\mathscr{R}_1=-k^{-1}\int \log(y+R_0(\bm{\phi}))d\kappa_4(y)$. Hence
\begin{eqnarray}
 |\mathscr{R}_1|&\leq& 
(k^{-1}|\{i\in [n]:R_0(\bm{\phi})<b_i\leq R_0'(\bm{\phi})\}|+\mu_0((R_0(\bm{\phi}),R_0'(\bm{\phi})]))\nonumber\\
 && \times k^{-1}\log(R_0(\bm{\phi})+R_0'(\bm{\phi})).
\end{eqnarray}
By (\ref{Eq2.2.1}), $\log(R_0(\bm{\phi})+R_0'(\bm{\phi}))\leq \log(2R_0(\bm{\phi})^{20})\leq 40\log(R_0(\bm{\phi}))$. By properties (a) and (e) of Proposition \ref{P5.2},  
\begin{eqnarray}
 &&  |\{i\in [n]:R_0(\bm{\phi})<b_i\leq R_0'(\bm{\phi})\}|=\tilde{N}(R_0'(\bm{\phi}))-\tilde{N}(R_0(\bm{\phi})) \nonumber\\
 &=& \Psi(R_0'(\bm{\phi}))-\Psi(R_0(\bm{\phi}))+\tilde{N}_0(R   _0'(\bm{\phi}))-\tilde{N}_0(R_0(\bm{\phi}))\nonumber\\
 &\leq& C\sqrt{\frac{\epsilon}{\log(R_0(\bm{\phi}))}}k+\mu_0([R_0(\bm{\phi}),R_0'(\bm{\phi})])k.
\end{eqnarray}
Hence noting $\mu_0([R_0(\bm{\phi}),R_0'(\bm{\phi})])\leq \pi^{-1}\int_{R_0(\bm{\phi})}^{R_0'(\bm{\phi})}\sqrt{x}dx\leq R_0'(\bm{\phi})^{3\slash 2}$, we obtain
\begin{equation}\label{Eq5.20.2}
 |\mathscr{R}_1|\leq C R_0'(\bm{\phi})^{3\slash 2}\log(R_0(\bm{\phi}))k^{-1}.
\end{equation}

\subparagraph{Sub-step 3.2}

In this sub-step, we bound $\mathscr{R}_2$. By (\ref{Eq5.11}), for any 
\begin{equation*}
x\in (R_0(\bm{\phi})-R_0(\bm{\phi})^{-4},R_0(\bm{\phi})], \quad y\in (R_0(\bm{\phi}),R_0(\bm{\phi})+R_0(\bm{\phi})^{-4}],
\end{equation*}
we have $|x-y|\leq 2R_0(\bm{\phi})^{-4}<1$. Hence
\begin{eqnarray}\label{Eq5.51}
&& \mathscr{R}_2= -k^{-2}\int_{R_0(\bm{\phi})-R_0(\bm{\phi})^{-4}}^{R_0(\bm{\phi})}\int_{R_0(\bm{\phi})}^{R_0(\bm{\phi})+R_0(\bm{\phi})^{-4}} \log(|x-y|)d\Psi(x) d\Psi(y)\nonumber\\
&\geq& k^{-2}\int_{R_0(\bm{\phi})-R_0(\bm{\phi})^{-4}}^{R_0(\bm{\phi})}\int_{R_0(\bm{\phi})}^{R_0(\bm{\phi})+R_0(\bm{\phi})^{-4}} \log(|x-y|)d\tilde{N}(x) d\tilde{N}_0(y)\nonumber\\
&&+k^{-2}\int_{R_0(\bm{\phi})-R_0(\bm{\phi})^{-4}}^{R_0(\bm{\phi})}\int_{R_0(\bm{\phi})}^{R_0(\bm{\phi})+R_0(\bm{\phi})^{-4}} \log(|x-y|)d\tilde{N}_0(x) d\tilde{N}(y).
\end{eqnarray}

For any $x\in (R_0(\bm{\phi})-R_0(\bm{\phi})^{-4},R_0(\bm{\phi})]$, we have
\begin{eqnarray}
&& -\int_{R_0(\bm{\phi})}^{R_0(\bm{\phi})+R_0(\bm{\phi})^{-4}}\log(|x-y|)d\tilde{N}_0(y)\nonumber\\
&\leq&\frac{k}{\pi}\int_{R_0(\bm{\phi})}^{R_0(\bm{\phi})+R_0(\bm{\phi})^{-4}}(-\log(y-R_0(\bm{\phi})))\sqrt{y}dy\nonumber\\
&\leq& \frac{k\sqrt{2 R_0(\bm{\phi})}}{\pi}\int_0^{R_0(\bm{\phi})^{-4}}(-\log(y))dy=\frac{\sqrt{2}k(4\log(R_0(\bm{\phi}))+1)}{\pi R_0(\bm{\phi})^{7\slash 2}}\nonumber\\
&\leq& C R_0(\bm{\phi})^{-7\slash 2}\log(R_0(\bm{\phi})) k.
\end{eqnarray}
By (\ref{Eq2.1.10}), as $M\geq M_0$ is sufficiently large and $\epsilon\in (0,\epsilon_0)$ is sufficiently small, 
\begin{equation}\label{Eq5.50}
 R_0(\bm{\phi})^{-4}\leq \frac{R_0(\bm{\phi})}{16\log(R_0(\bm{\phi}))}.
\end{equation}
Hence by property (a) of Proposition \ref{P5.2}, we have
\begin{eqnarray}
&& \tilde{N}(R_0(\bm{\phi}))-\tilde{N}(R_0(\bm{\phi})-R_0(\bm{\phi})^{-4}) \nonumber\\
&=& \Psi(R_0(\bm{\phi}))-\Psi(R_0(\bm{\phi})-R_0(\bm{\phi})^{-4})+\tilde{N}_0(R_0(\bm{\phi}))-\tilde{N}_0(R_0(\bm{\phi})-R_0(\bm{\phi})^{-4})\nonumber\\
&\leq& C\sqrt{\frac{\epsilon}{\log(R_0(\bm{\phi}))\log\log(R_0(\bm{\phi}))}}k+k\mu_0([R_0(\bm{\phi})-R_0(\bm{\phi})^{-4},R_0(\bm{\phi})])\nonumber\\
&\leq& C\sqrt{\frac{\epsilon}{\log(R_0(\bm{\phi}))\log\log(R_0(\bm{\phi}))}}k+\frac{k}{\pi}\int^{R_0(\bm{\phi})}_{R_0(\bm{\phi})-R_0(\bm{\phi})^{-4}}\sqrt{x}dx\nonumber\\
&\leq& C\sqrt{\frac{\epsilon}{\log(R_0(\bm{\phi}))\log\log(R_0(\bm{\phi}))}}k+\frac{k}{\pi R_0(\bm{\phi})^{7\slash 2}}\leq Ck.
\end{eqnarray}
Hence
\begin{eqnarray}\label{Eq5.61}
&& -k^{-2}\int_{R_0(\bm{\phi})-R_0(\bm{\phi})^{-4}}^{R_0(\bm{\phi})}\int_{R_0(\bm{\phi})}^{R_0(\bm{\phi})+R_0(\bm{\phi})^{-4}} \log(|x-y|)d\tilde{N}(x) d\tilde{N}_0(y) \nonumber\\
&\leq& C R_0(\bm{\phi})^{-7\slash 2}\log(R_0(\bm{\phi})) k^{-1}(\tilde{N}(R_0(\bm{\phi}))-\tilde{N}(R_0(\bm{\phi})-R_0(\bm{\phi})^{-4}))\nonumber\\
&\leq &  C R_0(\bm{\phi})^{-7\slash 2}\log(R_0(\bm{\phi})).
\end{eqnarray}

For any $y\in (R_0(\bm{\phi}),R_0(\bm{\phi})+R_0(\bm{\phi})^{-4}]$, we have
\begin{eqnarray}
&& -\int_{R_0(\bm{\phi})-R_0(\bm{\phi})^{-4}}^{R_0(\bm{\phi})}\log(|x-y|)d\tilde{N}_0(x) \nonumber\\
&\leq& \frac{k}{\pi}\int_{R_0(\bm{\phi})-R_0(\bm{\phi})^{-4}}^{R_0(\bm{\phi})}(-\log(R_0(\bm{\phi})-x))\sqrt{x} dx\nonumber\\
&\leq& \frac{k\sqrt{R_0(\bm{\phi})}}{\pi}\int_0^{R_0(\bm{\phi})^{-4}}(-\log(y))dy=\frac{k(4\log(R_0(\bm{\phi}))+1)}{\pi R_0(\bm{\phi})^{7\slash 2}}\nonumber\\
&\leq& C R_0(\bm{\phi})^{-7\slash 2}\log(R_0(\bm{\phi})) k.
\end{eqnarray}
Noting (\ref{Eq5.50}), by property (a) of Proposition \ref{P5.2}, we have
\begin{eqnarray}
 && \tilde{N}(R_0(\bm{\phi})+R_0(\bm{\phi})^{-4})-\tilde{N}(R_0(\bm{\phi})) \nonumber\\
 &=& \Psi(R_0(\bm{\phi})+R_0(\bm{\phi})^{-4})-\Psi(R_0(\bm{\phi}))+\tilde{N}_0(R_0(\bm{\phi})+R_0(\bm{\phi})^{-4})-\tilde{N}_0(R_0(\bm{\phi}))\nonumber\\
 &\leq& C\sqrt{\frac{\epsilon}{\log(R_0(\bm{\phi}))\log\log(R_0(\bm{\phi}))}}k+k\mu_0([R_0(\bm{\phi}),R_0(\bm{\phi})+R_0(\bm{\phi})^{-4}])\nonumber\\
 &\leq& C\sqrt{\frac{\epsilon}{\log(R_0(\bm{\phi}))\log\log(R_0(\bm{\phi}))}}k+\frac{k}{\pi}\int_{R_0(\bm{\phi})}^{R_0(\bm{\phi})+R_0(\bm{\phi})^{-4}}\sqrt{x}dx\nonumber\\
 &\leq& C\sqrt{\frac{\epsilon}{\log(R_0(\bm{\phi}))\log\log(R_0(\bm{\phi}))}}k+\frac{\sqrt{2}k}{\pi R_0(\bm{\phi})^{7\slash 2}}\leq Ck.
\end{eqnarray}
Hence
\begin{eqnarray}\label{Eq5.62}
&& -k^{-2}\int_{R_0(\bm{\phi})-R_0(\bm{\phi})^{-4}}^{R_0(\bm{\phi})}\int_{R_0(\bm{\phi})}^{R_0(\bm{\phi})+R_0(\bm{\phi})^{-4}} \log(|x-y|)d\tilde{N}_0(x) d\tilde{N}(y) \nonumber\\
&\leq& C R_0(\bm{\phi})^{-7\slash 2}\log(R_0(\bm{\phi}))k^{-1}(\tilde{N}(R_0(\bm{\phi})+R_0(\bm{\phi})^{-4})-\tilde{N}(R_0(\bm{\phi})) )\nonumber\\
&\leq& C R_0(\bm{\phi})^{-7\slash 2}\log(R_0(\bm{\phi})).
\end{eqnarray}

By (\ref{Eq5.51}), (\ref{Eq5.61}), and (\ref{Eq5.62}), we conclude that
\begin{equation}\label{Eq5.20.3}
\mathscr{R}_2\geq -C R_0(\bm{\phi})^{-7\slash 2}\log(R_0(\bm{\phi})).
\end{equation}

\subparagraph{Sub-step 3.3}

In this sub-step, we bound $\mathscr{R}_3$. Let 
\begin{equation}\label{Eq5.63}
 \mathscr{S}_3=\int_{-R_0(\bm{\phi})}^{R_0(\bm{\phi})}\int_{R_0(\bm{\phi})+R_0(\bm{\phi})^{-4}}^{R_0'(\bm{\phi})}\log(|x-y|) d\Psi(x)d\Psi(y).
\end{equation}
Note that
\begin{equation}\label{Eq5.1.1}
\mathscr{R}_3=-k^{-2}\mathscr{S}_3.
\end{equation}

For any $y\in (R_0(\bm{\phi})+R_0(\bm{\phi})^{-4}, R_0'(\bm{\phi})]$, by integration by parts, we have
\begin{eqnarray}\label{Eq5.64}
 \int_{-R_0(\bm{\phi})}^{R_0(\bm{\phi})}\log(|x-y|)d\Psi(x)&=& \int_{-R_0(\bm{\phi})}^{R_0(\bm{\phi})}\frac{\Psi(x)}{y-x}dx+\log(y-R_0(\bm{\phi}))\Psi(R_0(\bm{\phi}))\nonumber\\
 && -\log(y+R_0(\bm{\phi}))\Psi(-R_0(\bm{\phi})).
\end{eqnarray}
By (\ref{Eq5.63}) and (\ref{Eq5.64}), 
\begin{eqnarray}\label{Eq5.66}
\mathscr{S}_3&=&\int_{-R_0(\bm{\phi})}^{R_0(\bm{\phi})}\Psi(x)dx \int_{R_0(\bm{\phi})+R_0(\bm{\phi})^{-4}}^{R_0'(\bm{\phi})}\frac{1}{y-x}d\Psi(y)\nonumber\\
&& +\Psi(R_0(\bm{\phi}))\int_{R_0(\bm{\phi})+R_0(\bm{\phi})^{-4}}^{R_0'(\bm{\phi})} \log(y-R_0(\bm{\phi}))d\Psi(y)\nonumber\\
&& - \Psi(-R_0(\bm{\phi}))\int_{R_0(\bm{\phi})+R_0(\bm{\phi})^{-4}}^{R_0'(\bm{\phi})}\log(y+R_0(\bm{\phi}))d\Psi(y).
\end{eqnarray}
Now note that by integration by parts,
\begin{eqnarray}
 && \int_{R_0(\bm{\phi})+R_0(\bm{\phi})^{-4}}^{R_0'(\bm{\phi})} \log(y-R_0(\bm{\phi}))d\Psi(y)= -\int_{R_0(\bm{\phi})+R_0(\bm{\phi})^{-4}}^{R_0'(\bm{\phi})}\frac{\Psi(y)}{y-R_0(\bm{\phi})}dy\nonumber\\
 &&+\log(R_0'(\bm{\phi})-R_0(\bm{\phi}))\Psi(R_0'(\bm{\phi}))-\log(R_0(\bm{\phi})^{-4})\Psi(R_0(\bm{\phi})+R_0(\bm{\phi})^{-4}),\nonumber\\
 &&
\end{eqnarray}
\begin{eqnarray}
&& \int_{R_0(\bm{\phi})+R_0(\bm{\phi})^{-4}}^{R_0'(\bm{\phi})}\log(y+R_0(\bm{\phi}))d\Psi(y)=-\int_{R_0(\bm{\phi})+R_0(\bm{\phi})^{-4}}^{R_0'(\bm{\phi})}\frac{\Psi(y)}{y+R_0(\bm{\phi})}dy  \nonumber\\
&&+\log(R_0'(\bm{\phi})+R_0(\bm{\phi}))\Psi(R_0'(\bm{\phi}))-\log(2R_0(\bm{\phi})+R_0(\bm{\phi})^{-4})\Psi(R_0(\bm{\phi})+R_0(\bm{\phi})^{-4}). \nonumber\\
&&
\end{eqnarray}
For any $x\in (-R_0(\bm{\phi}),R_0(\bm{\phi})]$, by integration by parts,
\begin{eqnarray}\label{Eq5.67}
 \int_{R_0(\bm{\phi})+R_0(\bm{\phi})^{-4}}^{R_0'(\bm{\phi})}\frac{1}{y-x}d\Psi(y)&=&\int_{R_0(\bm{\phi})+R_0(\bm{\phi})^{-4}}^{R_0'(\bm{\phi})}\frac{\Psi(y)}{(y-x)^2}dy+\frac{\Psi(R_0'(\bm{\phi}))}{R_0'(\bm{\phi})-x}\nonumber\\
 &&-\frac{\Psi(R_0(\bm{\phi})+R_0(\bm{\phi})^{-4})}{R_0(\bm{\phi})+R_0(\bm{\phi})^{-4}-x}.
\end{eqnarray}
By (\ref{Eq5.66})-(\ref{Eq5.67}), we obtain that
\begin{eqnarray}\label{Eq5.68}
 \mathscr{S}_3&=& \Psi(R_0(\bm{\phi}))\Psi(R_0'(\bm{\phi}))\log(R_0'(\bm{\phi})-R_0(\bm{\phi}))\nonumber\\
 &&-\Psi(R_0(\bm{\phi}))\Psi(R_0(\bm{\phi})+R_0(\bm{\phi})^{-4})\log(R_0(\bm{\phi})^{-4})\nonumber\\
 && -\Psi(-R_0(\bm{\phi}))\Psi(R_0'(\bm{\phi}))\log(R_0'(\bm{\phi})+R_0(\bm{\phi}))\nonumber\\
 && +\Psi(-R_0(\bm{\phi}))\Psi(R_0(\bm{\phi})+R_0(\bm{\phi})^{-4})\log(2R_0(\bm{\phi})+R_0(\bm{\phi})^{-4})\nonumber\\
 && -\Psi(R_0(\bm{\phi}))\int_{R_0(\bm{\phi})+R_0(\bm{\phi})^{-4}}^{R_0'(\bm{\phi})}\frac{\Psi(y)}{y-R_0(\bm{\phi})}dy\nonumber\\
 && +\Psi(-R_0(\bm{\phi}))\int_{R_0(\bm{\phi})+R_0(\bm{\phi})^{-4}}^{R_0'(\bm{\phi})}\frac{\Psi(y)}{y+R_0(\bm{\phi})}dy\nonumber\\
 && +\Psi(R_0'(\bm{\phi}))\int_{-R_0(\bm{\phi})}^{R_0(\bm{\phi})}\frac{\Psi(x)}{R_0'(\bm{\phi})-x}dx\nonumber\\
 && -\Psi(R_0(\bm{\phi})+R_0(\bm{\phi})^{-4})\int_{-R_0(\bm{\phi})}^{R_0(\bm{\phi})} \frac{\Psi(x)}{R_0(\bm{\phi})+R_0(\bm{\phi})^{-4}-x}dx\nonumber\\
 &&+\int_{-R_0(\bm{\phi})}^{R_0(\bm{\phi})} \int_{R_0(\bm{\phi})+R_0(\bm{\phi})^{-4}}^{R_0'(\bm{\phi})}\frac{\Psi(x)\Psi(y)}{(y-x)^2}dxdy.
\end{eqnarray}

In the following, we bound the terms on the right-hand side of (\ref{Eq5.68}) in \textbf{Parts 3.3.1-3.3.9}.

\textbf{Part 3.3.1} By properties (a) and (e) of Proposition \ref{P5.2} and (\ref{Eq2.2.1}), we have
\begin{eqnarray}\label{Eq5.1.2}
&&  |\Psi(R_0(\bm{\phi}))\Psi(R_0'(\bm{\phi}))\log(R_0'(\bm{\phi})-R_0(\bm{\phi}))|\nonumber\\
&\leq& \frac{C\epsilon k^2}{\log(R_0(\bm{\phi}))\sqrt{\log\log(R_0(\bm{\phi}))}}\cdot \log(R_0(\bm{\phi})^{20})\leq  \frac{C\epsilon k^2}{\sqrt{\log\log(R_0(\bm{\phi}))}}.\nonumber\\
&&
\end{eqnarray}

\textbf{Part 3.3.2} By (\ref{Eq5.50}) and property (a) of Proposition \ref{P5.2}, we have
\begin{eqnarray}\label{Eq5.1.3}
&&  |\Psi(R_0(\bm{\phi}))\Psi(R_0(\bm{\phi})+R_0(\bm{\phi})^{-4})\log(R_0(\bm{\phi})^{-4})| \nonumber\\
&\leq& \frac{C\epsilon k^2}{\log(R_0(\bm{\phi}))\log\log(R_0(\bm{\phi}))}\cdot \log(R_0(\bm{\phi}))\leq \frac{C\epsilon k^2}{\log\log(R_0(\bm{\phi}))}.
\end{eqnarray}

\textbf{Part 3.3.3} By (\ref{Eq2.2.1}), (\ref{Eq5.11}), and properties (c) and (e) of Proposition \ref{P5.2},
\begin{eqnarray}\label{Eq5.1.4}
&& |\Psi(-R_0(\bm{\phi}))\Psi(R_0'(\bm{\phi}))\log(R_0'(\bm{\phi})+R_0(\bm{\phi}))|\nonumber\\
&\leq& \frac{CM k}{R_0(\bm{\phi})^{3\slash 2}}\cdot \sqrt{\frac{\epsilon}{\log(R_0(\bm{\phi}))}}k\cdot \log(2R_0(\bm{\phi})^{20})\leq \frac{CM\log(R_0(\bm{\phi})) k^2}{R_0(\bm{\phi})^{3\slash 2}}.\nonumber\\
&&
\end{eqnarray}

\textbf{Part 3.3.4} By (\ref{Eq2.2.1}), (\ref{Eq5.11}), (\ref{Eq5.50}), and properties (a) and (c) of Proposition \ref{P5.2}, 
we have 
\begin{eqnarray}\label{Eq5.1.5}
&& |\Psi(-R_0(\bm{\phi}))\Psi(R_0(\bm{\phi})+R_0(\bm{\phi})^{-4})\log(2R_0(\bm{\phi})+R_0(\bm{\phi})^{-4})| \nonumber\\
&\leq& \frac{CMk}{R_0(\bm{\phi})^{3\slash 2}} \cdot \sqrt{\frac{\epsilon}{\log(R_0(\bm{\phi}))\log\log(R_0(\bm{\phi}))}}k\cdot \log(R_0(\bm{\phi}))\nonumber\\
&\leq& \frac{CM \log(R_0(\bm{\phi})) k^2}{R_0(\bm{\phi})^{3\slash 2}}. 
\end{eqnarray}

\textbf{Part 3.3.5} By (\ref{Eq5.50}) and property (a) of Proposition \ref{P5.2}, we have
\begin{eqnarray}\label{Eq5.70}
&& \Big|\int_{R_0(\bm{\phi})+R_0(\bm{\phi})^{-4}}^{R_0(\bm{\phi})+\frac{R_0(\bm{\phi})}{16\log(R_0(\bm{\phi}))}} \frac{\Psi(y)}{y-R_0(\bm{\phi})}dy\Big|\leq \int_{R_0(\bm{\phi})+R_0(\bm{\phi})^{-4}}^{R_0(\bm{\phi})+\frac{R_0(\bm{\phi})}{16\log(R_0(\bm{\phi}))}} \frac{|\Psi(y)|}{y-R_0(\bm{\phi})}dy\nonumber\\
&\leq& C\sqrt{\frac{\epsilon}{\log(R_0(\bm{\phi}))\log\log(R_0(\bm{\phi}))}}k \int_{R_0(\bm{\phi})+R_0(\bm{\phi})^{-4}}^{R_0(\bm{\phi})+\frac{R_0(\bm{\phi})}{16\log(R_0(\bm{\phi}))}}\frac{1}{y-R_0(\bm{\phi})} dy\nonumber\\
&\leq& C\sqrt{\frac{\epsilon}{\log(R_0(\bm{\phi}))\log\log(R_0(\bm{\phi}))}}k \log\Big(\frac{R_0(\bm{\phi})^5}{16\log(R_0(\bm{\phi}))}\Big)\nonumber\\
&\leq& C\sqrt{\frac{\epsilon \log(R_0(\bm{\phi}))}{\log\log(R_0(\bm{\phi}))}}k.
\end{eqnarray}

By property (d) of Proposition \ref{P5.2}, 
\begin{eqnarray}
\int_{R_0(\bm{\phi})+\frac{R_0(\bm{\phi})}{16\log(R_0(\bm{\phi}))}}^{\frac{17}{16}R_0(\bm{\phi})}\Psi(y)^2 dy&\leq& 2R_0(\bm{\phi})\int_{R_0(\bm{\phi})+\frac{R_0(\bm{\phi})}{16\log(R_0(\bm{\phi}))}}^{\frac{17}{16}R_0(\bm{\phi})}\frac{\Psi(y)^2}{y+1}dy \nonumber\\
&\leq& \frac{CR_0(\bm{\phi}) \epsilon k^2}{\log(R_0(\bm{\phi})) \log\log(R_0(\bm{\phi}))}.
\end{eqnarray}
Moreover,
\begin{equation}
\int_{R_0(\bm{\phi})+\frac{R_0(\bm{\phi})}{16\log(R_0(\bm{\phi}))}}^{\frac{17}{16}R_0(\bm{\phi})}\frac{1}{(y-R_0(\bm{\phi}))^2}dy=\frac{16(\log(R_0(\bm{\phi}))-1)}{R_0(\bm{\phi})}\leq \frac{16\log(R_0(\bm{\phi}))}{R_0(\bm{\phi})}. 
\end{equation}
Hence by the Cauchy-Schwarz inequality, we have
\begin{eqnarray}\label{Eq5.72}
&& \Big|\int_{R_0(\bm{\phi})+\frac{R_0(\bm{\phi})}{16\log(R_0(\bm{\phi}))}}^{\frac{17}{16}R_0(\bm{\phi})}\frac{\Psi(y)}{y-R_0(\bm{\phi})}dy\Big|\leq \int_{R_0(\bm{\phi})+\frac{R_0(\bm{\phi})}{16\log(R_0(\bm{\phi}))}}^{\frac{17}{16}R_0(\bm{\phi})}\frac{|\Psi(y)|}{y-R_0(\bm{\phi})}dy \nonumber\\
&\leq&  \Big(\int_{R_0(\bm{\phi})+\frac{R_0(\bm{\phi})}{16\log(R_0(\bm{\phi}))}}^{\frac{17}{16}R_0(\bm{\phi})}\Psi(y)^2 dy\Big)^{1\slash 2}\Big(\int_{R_0(\bm{\phi})+\frac{R_0(\bm{\phi})}{16\log(R_0(\bm{\phi}))}}^{\frac{17}{16}R_0(\bm{\phi})}\frac{1}{(y-R_0(\bm{\phi}))^2}dy\Big)^{1\slash 2}\nonumber\\
&\leq& C\sqrt{\frac{\epsilon}{\log\log(R_0(\bm{\phi}))}} k.
\end{eqnarray}

Combining (\ref{Eq5.70}) and (\ref{Eq5.72}), by property (a) of Proposition \ref{P5.2}, we have
\begin{eqnarray}\label{Eq5.73}
 && \Big|\Psi(R_0(\bm{\phi}))\int_{R_0(\bm{\phi})+R_0(\bm{\phi})^{-4}}^{\frac{17}{16}R_0(\bm{\phi})}\frac{\Psi(y)}{y-R_0(\bm{\phi})}dy\Big|\nonumber\\
 &\leq& C\sqrt{\frac{\epsilon}{\log(R_0(\bm{\phi}))\log\log(R_0(\bm{\phi}))}}k\cdot \sqrt{\frac{\epsilon \log(R_0(\bm{\phi}))}{\log\log(R_0(\bm{\phi}))}}k\leq \frac{C\epsilon k^2}{\log\log(R_0(\bm{\phi}))}.\nonumber\\
 &&
\end{eqnarray}

For any $y\geq 17R_0(\bm{\phi})\slash 16$, we have $y\geq 2000$ (see (\ref{Eq5.11})) and hence
\begin{equation*}
y-R_0(\bm{\phi})\geq y\slash 17\geq (y+1)\slash 20.
\end{equation*}
Hence by property (f) of Proposition \ref{P5.2}, 
\begin{equation}
 \int_{\frac{17}{16}R_0(\bm{\phi})}^{R_0'(\bm{\phi})}\frac{\Psi(y)^2}{y-R_0(\bm{\phi})}dy\leq 20\int_{\frac{17}{16}R_0(\bm{\phi})}^{R_0'(\bm{\phi})}\frac{\Psi(y)^2}{y+1}dy\leq C\epsilon k^2.
\end{equation}
Hence by (\ref{Eq2.2.1}), the AM-GM inequality, and property (a) of Proposition \ref{P5.2},
\begin{eqnarray}\label{Eq5.74}
&& \Big|\Psi(R_0(\bm{\phi}))\int_{\frac{17}{16}R_0(\bm{\phi})}^{R_0'(\bm{\phi})}\frac{\Psi(y)}{y-R_0(\bm{\phi})}dy\Big|\leq \int_{\frac{17}{16}R_0(\bm{\phi})}^{R_0'(\bm{\phi})}\frac{|\Psi(R_0(\bm{\phi}))||\Psi(y)|}{y-R_0(\bm{\phi})}dy \nonumber\\
&\leq& \int_{\frac{17}{16}R_0(\bm{\phi})}^{R_0'(\bm{\phi})}\frac{\Psi(R_0(\bm{\phi}))^2}{y-R_0(\bm{\phi})}dy+\int_{\frac{17}{16}R_0(\bm{\phi})}^{R_0'(\bm{\phi})}\frac{\Psi(y)^2}{y-R_0(\bm{\phi})}dy\nonumber\\
&\leq& \frac{C \epsilon k^2}{\log(R_0(\bm{\phi}))\log\log(R_0(\bm{\phi}))}\cdot\log\Big(\frac{R_0'(\bm{\phi})-R_0(\bm{\phi})}{R_0(\bm{\phi})\slash 16}\Big)+C\epsilon k^2\nonumber\\
&\leq& \frac{C \epsilon k^2}{\log(R_0(\bm{\phi}))\log\log(R_0(\bm{\phi}))}\cdot \log(R_0(\bm{\phi}))+C\epsilon k^2\leq C\epsilon k^2.
\end{eqnarray}

By (\ref{Eq5.73}) and (\ref{Eq5.74}), we conclude that
\begin{equation}\label{Eq5.1.6}
\Big|\Psi(R_0(\bm{\phi}))\int_{R_0(\bm{\phi})+R_0(\bm{\phi})^{-4}}^{R_0'(\bm{\phi})}\frac{\Psi(y)}{y-R_0(\bm{\phi})}dy\Big|\leq C\epsilon k^2.
\end{equation}

\textbf{Part 3.3.6} By (\ref{Eq2.2.1}), the Cauchy-Schwarz inequality, and property (f) of Proposition \ref{P5.2}, 
\begin{eqnarray}
&& \Big|\int_{R_0(\bm{\phi})+R_0(\bm{\phi})^{-4}}^{R_0'(\bm{\phi})}\frac{\Psi(y)}{y+R_0(\bm{\phi})}dy\Big|\leq \int_{R_0(\bm{\phi})+R_0(\bm{\phi})^{-4}}^{R_0'(\bm{\phi})}\frac{|\Psi(y)|}{y+R_0(\bm{\phi})}dy \nonumber\\
&\leq& \Big(\int_{R_0(\bm{\phi})+R_0(\bm{\phi})^{-4}}^{R_0'(\bm{\phi})}\frac{\Psi(y)^2}{y+R_0(\bm{\phi})}dy\Big)^{1\slash 2}\Big(\int_{R_0(\bm{\phi})+R_0(\bm{\phi})^{-4}}^{R_0'(\bm{\phi})}\frac{1}{y+R_0(\bm{\phi})}dy\Big)^{1\slash 2}\nonumber\\
&\leq& \Big(\int_{R_0(\bm{\phi})+R_0(\bm{\phi})^{-4}}^{R_0'(\bm{\phi})}\frac{\Psi(y)^2}{y+1}dy\Big)^{1\slash 2} \log\Big(\frac{R_0'(\bm{\phi})+R_0(\bm{\phi})}{2R_0(\bm{\phi})+R_0(\bm{\phi})^{-4}}\Big)^{1\slash 2}\nonumber\\
&\leq& C\sqrt{\epsilon \log(R_0(\bm{\phi}))} k\leq C\log(R_0(\bm{\phi})) k.
\end{eqnarray}
Hence by (\ref{Eq5.11}) and property (c) of Proposition \ref{P5.2},
\begin{eqnarray}\label{Eq5.2.1}
  \Big| \Psi(-R_0(\bm{\phi}))\int_{R_0(\bm{\phi})+R_0(\bm{\phi})^{-4}}^{R_0'(\bm{\phi})}\frac{\Psi(y)}{y+R_0(\bm{\phi})}dy\Big|\leq \frac{CM\log(R_0(\bm{\phi})) k^2}{R_0(\bm{\phi})^{3\slash 2}}.
\end{eqnarray}

\textbf{Part 3.3.7} By property (b) of Proposition \ref{P5.2} and (\ref{Eq2.2.1}), 
\begin{equation}\label{Eq5.75}
  \int_0^{R_0(\bm{\phi})}\Psi(x)^2 dx \leq (R_0(\bm{\phi})+1)\int_{0}^{R_0(\bm{\phi})}\frac{\Psi(x)^2}{x+1}dx \leq CM^2 R_0(\bm{\phi}) k^2,
\end{equation}
\begin{equation}
  \int_0^{R_0(\bm{\phi})}\frac{1}{(R_0'(\bm{\phi})-x)^2}dx=\frac{R_0(\bm{\phi})}{(R_0'(\bm{\phi})-R_0(\bm{\phi}))R_0'(\bm{\phi})}\leq \frac{2R_0(\bm{\phi})}{R_0'(\bm{\phi})^2}.
\end{equation}
Hence by the Cauchy-Schwarz inequality,
\begin{eqnarray}
&& \Big|\int_0^{R_0(\bm{\phi})}\frac{\Psi(x)}{R_0'(\bm{\phi})-x}dx\Big|\leq \int_0^{R_0(\bm{\phi})}\frac{|\Psi(x)|}{R_0'(\bm{\phi})-x}dx\nonumber\\
&\leq& \Big(\int_0^{R_0(\bm{\phi})}\Psi(x)^2 dx\Big)^{1\slash 2}\Big(\int_0^{R_0(\bm{\phi})}\frac{1}{(R_0'(\bm{\phi})-x)^2}dx\Big)^{1\slash 2}\leq \frac{CM R_0(\bm{\phi}) k}{R_0'(\bm{\phi})}.\nonumber\\
&& 
\end{eqnarray}
By (\ref{Eq5.11}) and property (c) of Proposition \ref{P5.2},
\begin{eqnarray}
&&  \Big|\int_{-R_0(\bm{\phi})}^0 \frac{\Psi(x)}{R_0'(\bm{\phi})-x}dx\Big|\leq R_0'(\bm{\phi})^{-1}\int_{-R_0(\bm{\phi})}^0 |\Psi(x)|dx \nonumber\\
&\leq& CMR_0'(\bm{\phi})^{-1} k\int_{-R_0(\bm{\phi})}^0(|x|+1)^{-3\slash 2}dx \leq CMR_0'(\bm{\phi})^{-1} k.
\end{eqnarray}
Hence by property (e) of Proposition \ref{P5.2} and (\ref{Eq2.2.1}),
\begin{eqnarray}\label{Eq5.1.7}
 \Big|\Psi(R_0'(\bm{\phi}))\int_{-R_0(\bm{\phi})}^{R_0(\bm{\phi})}\frac{\Psi(x)}{R_0'(\bm{\phi})-x}dx\Big|&\leq&  C \sqrt{\frac{\epsilon}{\log(R_0(\bm{\phi}))}}k\cdot \frac{CM R_0(\bm{\phi}) k}{R_0'(\bm{\phi})}\nonumber\\
 &\leq& CM R_0(\bm{\phi})^{-9} k^2.
\end{eqnarray}

\textbf{Part 3.3.8} By (\ref{Eq5.11}) and property (c) of Proposition \ref{P5.2}, 
\begin{eqnarray}\label{Eq5.76}
&& \int_{-R_0(\bm{\phi})}^{0} \frac{|\Psi(x)|}{R_0(\bm{\phi})+R_0(\bm{\phi})^{-4}-x}dx\leq R_0(\bm{\phi})^{-1}\int_{-R_0(\bm{\phi})}^{0} |\Psi(x)|dx \nonumber\\
&\leq& CMR_0(\bm{\phi})^{-1}k \int_{-R_0(\bm{\phi})}^{0}(|x|+1)^{-3\slash 2}dx\leq CMR_0(\bm{\phi})^{-1}k.
\end{eqnarray}

Note that
\begin{eqnarray}
&& \int_0^{R_0(\bm{\phi})-\frac{R_0(\bm{\phi})}{16\log(R_0(\bm{\phi}))}}\frac{1}{R_0(\bm{\phi})+R_0(\bm{\phi})^{-4}-x}dx \nonumber\\
&=&\log\Big(\frac{R_0(\bm{\phi})+R_0(\bm{\phi})^{-4}}{\frac{R_0(\bm{\phi})}{ 16\log(R_0(\bm{\phi}))}+R_0(\bm{\phi})^{-4}}\Big)\leq \log\Big(\frac{2R_0(\bm{\phi})}{R_0(\bm{\phi})\slash (16 \log(R_0(\bm{\phi})))}\Big)\nonumber\\
&=& \log(32)+\log\log(R_0(\bm{\phi}))\leq C\log\log(R_0(\bm{\phi})).
\end{eqnarray}
By (\ref{Eq5.75}), we have
\begin{eqnarray}
 && \int_0^{R_0(\bm{\phi})-\frac{R_0(\bm{\phi})}{16\log(R_0(\bm{\phi}))}}\frac{\Psi(x)^2}{R_0(\bm{\phi})+R_0(\bm{\phi})^{-4}-x}dx\nonumber\\
 &&\leq \frac{16\log(R_0(\bm{\phi}))}{R_0(\bm{\phi})}\int_0^{R_0(\bm{\phi})}\Psi(x)^2dx \leq CM^2\log(R_0(\bm{\phi})) k^2.
\end{eqnarray}
Hence by the Cauchy-Schwarz inequality, we have
\begin{eqnarray}\label{Eq5.77}
&&\int_0^{R_0(\bm{\phi})-\frac{R_0(\bm{\phi})}{16\log(R_0(\bm{\phi}))}}\frac{|\Psi(x)|}{R_0(\bm{\phi})+R_0(\bm{\phi})^{-4}-x}dx\nonumber\\
&\leq& \Big(\int_0^{R_0(\bm{\phi})-\frac{R_0(\bm{\phi})}{16\log(R_0(\bm{\phi}))}}\frac{1}{R_0(\bm{\phi})+R_0(\bm{\phi})^{-4}-x}dx\Big)^{1\slash 2}\nonumber\\
&&\times \Big(\int_0^{R_0(\bm{\phi})-\frac{R_0(\bm{\phi})}{16\log(R_0(\bm{\phi}))}}\frac{\Psi(x)^2}{R_0(\bm{\phi})+R_0(\bm{\phi})^{-4}-x}dx\Big)^{1\slash 2}\nonumber\\
&\leq& CM\sqrt{\log(R_0(\bm{\phi}))\log\log(R_0(\bm{\phi}))}k.
\end{eqnarray}

By property (a) of Proposition \ref{P5.2}, we have
\begin{eqnarray}\label{Eq5.78}
&& \int_{R_0(\bm{\phi})-\frac{R_0(\bm{\phi})}{16\log(R_0(\bm{\phi}))}}^{R_0(\bm{\phi})}\frac{|\Psi(x)|}{R_0(\bm{\phi})+R_0(\bm{\phi})^{-4}-x}dx\nonumber\\
&\leq& C\sqrt{\frac{\epsilon}{\log(R_0(\bm{\phi}))\log\log(R_0(\bm{\phi}))}}k\int_{R_0(\bm{\phi})-\frac{R_0(\bm{\phi})}{16\log(R_0(\bm{\phi}))}}^{R_0(\bm{\phi})}\frac{1}{R_0(\bm{\phi})+R_0(\bm{\phi})^{-4}-x}dx\nonumber\\
&=& C\sqrt{\frac{\epsilon}{\log(R_0(\bm{\phi}))\log\log(R_0(\bm{\phi}))}}k \log\Big(\frac{\frac{R_0(\bm{\phi})}{16\log(R_0(\bm{\phi}))}+R_0(\bm{\phi})^{-4}}{R_0(\bm{\phi})^{-4}}\Big)\nonumber\\
&\leq& C\sqrt{\frac{\epsilon}{\log(R_0(\bm{\phi}))\log\log(R_0(\bm{\phi}))}}k \log(R_0(\bm{\phi})^5)\leq C\sqrt{\frac{\epsilon \log(R_0(\bm{\phi}))}{\log\log(R_0(\bm{\phi}))}}k.
\end{eqnarray}

By (\ref{Eq5.76}), (\ref{Eq5.77}), and (\ref{Eq5.78}), we have
\begin{eqnarray}\label{Eq5.79}
&& \Big|\int_{-R_0(\bm{\phi})}^{R_0(\bm{\phi})} \frac{\Psi(x)}{R_0(\bm{\phi})+R_0(\bm{\phi})^{-4}-x}dx\Big|\leq \int_{-R_0(\bm{\phi})}^{R_0(\bm{\phi})} \frac{|\Psi(x)|}{R_0(\bm{\phi})+R_0(\bm{\phi})^{-4}-x}dx\nonumber\\
&& \leq CM\sqrt{\log(R_0(\bm{\phi}))\log\log(R_0(\bm{\phi}))} k.
\end{eqnarray}
Hence by (\ref{Eq5.50}) and property (a) of Proposition \ref{P5.2}, we conclude that
\begin{eqnarray}\label{Eq5.1.8}
&& \Big|\Psi(R_0(\bm{\phi})+R_0(\bm{\phi})^{-4})\int_{-R_0(\bm{\phi})}^{R_0(\bm{\phi})} \frac{\Psi(x)}{R_0(\bm{\phi})+R_0(\bm{\phi})^{-4}-x}dx\Big|\nonumber\\
&\leq& C\sqrt{\frac{\epsilon}{\log(R_0(\bm{\phi}))\log\log(R_0(\bm{\phi}))}}k\cdot M\sqrt{\log(R_0(\bm{\phi}))\log\log(R_0(\bm{\phi}))}k\nonumber\\
&\leq& CM\sqrt{\epsilon} k^2.
\end{eqnarray}

\textbf{Part 3.3.9} We let
\begin{eqnarray}
&& \mathscr{U}_{11}:=\int_{-R_0(\bm{\phi})}^{R_0(\bm{\phi})-\frac{R_0(\bm{\phi})}{16\log(R_0(\bm{\phi}))}}\int_{R_0(\bm{\phi})+R_0(\bm{\phi})^{-4}}^{R_0(\bm{\phi})+\frac{R_0(\bm{\phi})}{16\log(R_0(\bm{\phi}))}}\frac{|\Psi(x)||\Psi(y)|}{(y-x)^2}dxdy, \nonumber\\
&& \mathscr{U}_{12}:=\int^{R_0(\bm{\phi})}_{R_0(\bm{\phi})-\frac{R_0(\bm{\phi})}{16\log(R_0(\bm{\phi}))}}\int_{R_0(\bm{\phi})+R_0(\bm{\phi})^{-4}}^{R_0(\bm{\phi})+\frac{R_0(\bm{\phi})}{16\log(R_0(\bm{\phi}))}}\frac{|\Psi(x)||\Psi(y)|}{(y-x)^2}dxdy,\nonumber\\
&& \mathscr{U}_{21}:=\int_{-R_0(\bm{\phi})}^0 \int_{R_0(\bm{\phi})+\frac{R_0(\bm{\phi})}{16\log(R_0(\bm{\phi}))}}^{\frac{17}{16}R_0(\bm{\phi})}\frac{|\Psi(x)||\Psi(y)|}{(y-x)^2}dxdy,\nonumber\\
&& \mathscr{U}_{22}:=\int_0^{\frac{15}{16}R_0(\bm{\phi})} \int_{R_0(\bm{\phi})+\frac{R_0(\bm{\phi})}{16\log(R_0(\bm{\phi}))}}^{\frac{17}{16}R_0(\bm{\phi})}\frac{|\Psi(x)||\Psi(y)|}{(y-x)^2}dxdy,\nonumber\\
&& \mathscr{U}_{23}:=\int_{\frac{15}{16}R_0(\bm{\phi})}^{R_0(\bm{\phi})-\frac{R_0(\bm{\phi})}{16\log(R_0(\bm{\phi}))}} \int_{R_0(\bm{\phi})+\frac{R_0(\bm{\phi})}{16\log(R_0(\bm{\phi}))}}^{\frac{17}{16}R_0(\bm{\phi})}\frac{|\Psi(x)||\Psi(y)|}{(y-x)^2}dxdy,\nonumber\\
&& \mathscr{U}_{24}:=\int_{R_0(\bm{\phi})-\frac{R_0(\bm{\phi})}{16\log(R_0(\bm{\phi}))}}^{R_0(\bm{\phi})}
\int_{R_0(\bm{\phi})+\frac{R_0(\bm{\phi})}{16\log(R_0(\bm{\phi}))}}^{\frac{17}{16}R_0(\bm{\phi})}\frac{|\Psi(x)||\Psi(y)|}{(y-x)^2}dxdy,\nonumber\\
&& \mathscr{U}_{31}:=\int_{-R_0(\bm{\phi})}^0
\int_{\frac{17}{16}R_0(\bm{\phi})}^{R_0'(\bm{\phi})}\frac{|\Psi(x)||\Psi(y)|}{(y-x)^2}dxdy, \nonumber\\
&& \mathscr{U}_{32}:=\int_0^{R_0(\bm{\phi})}
\int_{\frac{17}{16}R_0(\bm{\phi})}^{R_0'(\bm{\phi})}\frac{|\Psi(x)||\Psi(y)|}{(y-x)^2}dxdy.
\end{eqnarray}
Note that
\begin{eqnarray}\label{Eq5.95}
 && \Big|\int_{-R_0(\bm{\phi})}^{R_0(\bm{\phi})} \int_{R_0(\bm{\phi})+R_0(\bm{\phi})^{-4}}^{R_0'(\bm{\phi})}\frac{\Psi(x)\Psi(y)}{(y-x)^2}dxdy\Big|\nonumber\\
 &\leq& \mathscr{U}_{11}+\mathscr{U}_{12}+\mathscr{U}_{21}+\mathscr{U}_{22}+\mathscr{U}_{23}+\mathscr{U}_{24}+\mathscr{U}_{31}+\mathscr{U}_{32}.
\end{eqnarray}

We first bound $\mathscr{U}_{11}$. For any $x\in \big[-R_0(\bm{\phi}), R_0(\bm{\phi})-\frac{R_0(\bm{\phi})}{16\log(R_0(\bm{\phi}))}\big]$, by property (a) of Proposition \ref{P5.2}, we have
\begin{eqnarray}
&& \int_{R_0(\bm{\phi})+R_0(\bm{\phi})^{-4}}^{R_0(\bm{\phi})+\frac{R_0(\bm{\phi})}{16\log(R_0(\bm{\phi}))}}\frac{|\Psi(y)|}{(y-x)^2} dy\nonumber\\
&\leq& C\sqrt{\frac{\epsilon}{\log(R_0(\bm{\phi}))\log\log(R_0(\bm{\phi}))}}k \int_{R_0(\bm{\phi})+R_0(\bm{\phi})^{-4}}^{R_0(\bm{\phi})+\frac{R_0(\bm{\phi})}{16\log(R_0(\bm{\phi}))}}\frac{1}{(y-x)^2}dy\nonumber\\
&\leq& C\sqrt{\frac{\epsilon}{\log(R_0(\bm{\phi}))\log\log(R_0(\bm{\phi}))}}k \cdot\frac{1}{R_0(\bm{\phi})+R_0(\bm{\phi})^{-4}-x}.
\end{eqnarray}
Hence by (\ref{Eq5.79}), we have
\begin{eqnarray}
\mathscr{U}_{11}&\leq& C\sqrt{\frac{\epsilon}{\log(R_0(\bm{\phi}))\log\log(R_0(\bm{\phi}))}}k \int_{-R_0(\bm{\phi})}^{R_0(\bm{\phi})} \frac{|\Psi(x)|}{R_0(\bm{\phi})+R_0(\bm{\phi})^{-4}-x}dx\nonumber\\
&\leq& CM\sqrt{\epsilon} k^2.
\end{eqnarray}

We bound $\mathscr{U}_{12}$ as follows. By property (a) of Proposition (\ref{P5.2}),
\begin{eqnarray}\label{Eq5.80}
    \mathscr{U}_{12}&\leq& \frac{C\epsilon k^2}{\log(R_0(\bm{\phi}))\log\log(R_0(\bm{\phi}))}\nonumber\\
    &&\times \int^{R_0(\bm{\phi})}_{R_0(\bm{\phi})-\frac{R_0(\bm{\phi})}{16\log(R_0(\bm{\phi}))}}\int_{R_0(\bm{\phi})+R_0(\bm{\phi})^{-4}}^{R_0(\bm{\phi})+\frac{R_0(\bm{\phi})}{16\log(R_0(\bm{\phi}))}}\frac{1}{(y-x)^2}dxdy.
\end{eqnarray}
For any $x\in \big[R_0(\bm{\phi})-\frac{R_0(\bm{\phi})}{16\log(R_0(\bm{\phi}))},R_0(\bm{\phi})\big]$, we have
\begin{equation}
\int_{R_0(\bm{\phi})+R_0(\bm{\phi})^{-4}}^{R_0(\bm{\phi})+\frac{R_0(\bm{\phi})}{16\log(R_0(\bm{\phi}))}}\frac{1}{(y-x)^2}dy\leq \frac{1}{R_0(\bm{\phi})+R_0(\bm{\phi})^{-4}-x}.
\end{equation}
Hence
\begin{eqnarray}\label{Eq5.81}
&& \int^{R_0(\bm{\phi})}_{R_0(\bm{\phi})-\frac{R_0(\bm{\phi})}{16\log(R_0(\bm{\phi}))}}\int_{R_0(\bm{\phi})+R_0(\bm{\phi})^{-4}}^{R_0(\bm{\phi})+\frac{R_0(\bm{\phi})}{16\log(R_0(\bm{\phi}))}}\frac{1}{(y-x)^2}dxdy \nonumber\\
&\leq& \int^{R_0(\bm{\phi})}_{R_0(\bm{\phi})-\frac{R_0(\bm{\phi})}{16\log(R_0(\bm{\phi}))}}\frac{1}{R_0(\bm{\phi})+R_0(\bm{\phi})^{-4}-x}dx\nonumber\\
&=& \log\Big(\frac{\frac{R_0(\bm{\phi})}{16\log(R_0(\bm{\phi}))}+R_0(\bm{\phi})^{-4}}{R_0(\bm{\phi})^{-4}}\Big)\leq \log(R_0(\bm{\phi})^5).
\end{eqnarray}
By (\ref{Eq5.80}) and (\ref{Eq5.81}),
\begin{equation}
 \mathscr{U}_{12}\leq \frac{C\epsilon k^2}{\log\log(R_0(\bm{\phi}))}\leq C\epsilon k^2\leq CM\sqrt{\epsilon}k^2.
\end{equation}

We bound $\mathscr{U}_{21}$ as follows. For any $y\in \big[R_0(\bm{\phi})+\frac{R_0(\bm{\phi})}{16\log(R_0(\bm{\phi}))},\frac{17}{16}R_0(\bm{\phi})\big]$ and $x\in [-R_0(\bm{\phi}),0]$, we have $y-x\geq R_0(\bm{\phi})$, hence
\begin{equation}\label{Eq5.82}
\mathscr{U}_{21}\leq R_0(\bm{\phi})^{-2} \Big(\int_{-R_0(\bm{\phi})}^0 |\Psi(x)|dx\Big)\Big(\int_{R_0(\bm{\phi})+\frac{R_0(\bm{\phi})}{16\log(R_0(\bm{\phi}))}}^{\frac{17}{16}R_0(\bm{\phi})}|\Psi(y)|dy\Big)
\end{equation}
By (\ref{Eq5.11}) and property (c) of Proposition \ref{P5.2}, 
\begin{equation}\label{Eq5.83}
\int_{-R_0(\bm{\phi})}^0 |\Psi(x)|dx\leq CMk \int_{-R_0(\bm{\phi})}^0 (|x|+1)^{-3\slash 2} dx \leq CMk.
\end{equation}
By the Cauchy-Schwarz inequality and property (f) of Proposition \ref{P5.2},
\begin{eqnarray}\label{Eq5.84}
&& \int_{R_0(\bm{\phi})+\frac{R_0(\bm{\phi})}{16\log(R_0(\bm{\phi}))}}^{\frac{17}{16}R_0(\bm{\phi})}|\Psi(y)|dy\nonumber\\
&\leq& \Big(\frac{R_0(\bm{\phi})}{16}\Big)^{1\slash 2}\Big(\int_{R_0(\bm{\phi})+\frac{R_0(\bm{\phi})}{16\log(R_0(\bm{\phi}))}}^{\frac{17}{16}R_0(\bm{\phi})}\Psi(y)^2dy\Big)^{1\slash 2}\nonumber\\
&\leq& \frac{1}{4}\sqrt{R_0(\bm{\phi})}\Big((2R_0(\bm{\phi})+1)\int_{R_0(\bm{\phi})+\frac{R_0(\bm{\phi})}{16\log(R_0(\bm{\phi}))}}^{\frac{17}{16}R_0(\bm{\phi})}\frac{\Psi(y)^2}{y+1}dy\Big)^{1\slash 2}\nonumber\\
&\leq& C\sqrt{\epsilon} R_0(\bm{\phi}) k.
\end{eqnarray}
By (\ref{Eq5.82})-(\ref{Eq5.84}),
\begin{equation}
 \mathscr{U}_{21}\leq CM\sqrt{\epsilon} R_0(\bm{\phi})^{-1} k^2\leq CM\sqrt{\epsilon} k^2.
\end{equation}

We bound $\mathscr{U}_{22}$ as follows. For any $y\in \big[R_0(\bm{\phi})+\frac{R_0(\bm{\phi})}{16\log(R_0(\bm{\phi}))},\frac{17}{16}R_0(\bm{\phi})\big]$ and $x\in \big[0,\frac{15}{16}R_0(\bm{\phi})\big]$, we have $y-x\geq \frac{1}{16}R_0(\bm{\phi})$, hence
\begin{equation}\label{Eq5.85}
 \mathscr{U}_{22}\leq C  R_0(\bm{\phi})^{-2} \Big(\int_0^{\frac{15}{16}R_0(\bm{\phi})} |\Psi(x)|  dx\Big)\Big(\int_{R_0(\bm{\phi})+\frac{R_0(\bm{\phi})}{16\log(R_0(\bm{\phi}))}}^{\frac{17}{16}R_0(\bm{\phi})}|\Psi(y)| dy\Big).
\end{equation}
By the Cauchy-Schwarz inequality and property (b) of Proposition \ref{P5.2},
\begin{eqnarray}\label{Eq5.86}
&& \int_0^{\frac{15}{16}R_0(\bm{\phi})} |\Psi(x)|  dx\leq  \Big(\frac{15}{16}R_0(\bm{\phi})\Big)^{1\slash 2} \Big(\int_0^{\frac{15}{16}R_0(\bm{\phi})} \Psi(x)^2 dx\Big)^{1\slash 2} \nonumber\\
&\leq& \sqrt{R_0(\bm{\phi})}\Big((R_0(\bm{\phi})+1)\int_0^{\frac{15}{16}R_0(\bm{\phi})} \frac{\Psi(x)^2}{x+1} dx\Big)^{1\slash 2}\leq CMR_0(\bm{\phi})k.\nonumber\\
&&
\end{eqnarray}
By (\ref{Eq5.84}), (\ref{Eq5.85}), and (\ref{Eq5.86}), we have
\begin{equation}
 \mathscr{U}_{22}\leq CM\sqrt{\epsilon} k^2.
\end{equation}

We bound $\mathscr{U}_{23}$ as follows. By the AM-GM inequality, we have
\begin{eqnarray}
 \mathscr{U}_{23}&\leq& \int_{\frac{15}{16}R_0(\bm{\phi})}^{R_0(\bm{\phi})-\frac{R_0(\bm{\phi})}{16\log(R_0(\bm{\phi}))}} \int_{R_0(\bm{\phi})+\frac{R_0(\bm{\phi})}{16\log(R_0(\bm{\phi}))}}^{\frac{17}{16}R_0(\bm{\phi})}\frac{\Psi(x)^2}{(y-x)^2}dxdy \nonumber\\
 && + \int_{\frac{15}{16}R_0(\bm{\phi})}^{R_0(\bm{\phi})-\frac{R_0(\bm{\phi})}{16\log(R_0(\bm{\phi}))}} \int_{R_0(\bm{\phi})+\frac{R_0(\bm{\phi})}{16\log(R_0(\bm{\phi}))}}^{\frac{17}{16}R_0(\bm{\phi})}\frac{\Psi(y)^2}{(y-x)^2}dxdy. 
\end{eqnarray}
For any $x\in \big[\frac{15}{16}R_0(\bm{\phi}),R_0(\bm{\phi})-\frac{R_0(\bm{\phi})}{16\log(R_0(\bm{\phi}))}\big]$, 
\begin{equation}\label{Eq5.87}
\int_{R_0(\bm{\phi})+\frac{R_0(\bm{\phi})}{16\log(R_0(\bm{\phi}))}}^{\frac{17}{16}R_0(\bm{\phi})}\frac{1}{(y-x)^2}dy\leq \frac{1}{R_0(\bm{\phi})+\frac{R_0(\bm{\phi})}{16\log(R_0(\bm{\phi}))}-x}\leq \frac{16\log(R_0(\bm{\phi}))}{R_0(\bm{\phi})}.
\end{equation}
For any $y\in \big[R_0(\bm{\phi})+\frac{R_0(\bm{\phi})}{16\log(R_0(\bm{\phi}))}, \frac{17}{16}R_0(\bm{\phi})\big]$, 
\begin{equation}
\int_{\frac{15}{16}R_0(\bm{\phi})}^{R_0(\bm{\phi})-\frac{R_0(\bm{\phi})}{16\log(R_0(\bm{\phi}))}} \frac{1}{(y-x)^2}dx\leq \frac{1}{y-\Big(R_0(\bm{\phi})-\frac{R_0(\bm{\phi})}{16\log(R_0(\bm{\phi}))}\Big)}\leq \frac{16\log(R_0(\bm{\phi}))}{R_0(\bm{\phi})}.
\end{equation}
Hence by property (d) of Proposition \ref{P5.2},
\begin{eqnarray}
&&\mathscr{U}_{23}\nonumber\\
&\leq&
\frac{16\log(R_0(\bm{\phi}))}{R_0(\bm{\phi})}  \times\Big(\int_{\frac{15}{16}R_0(\bm{\phi})}^{R_0(\bm{\phi})-\frac{R_0(\bm{\phi})}{16\log(R_0(\bm{\phi}))}}\Psi(x)^2 dx+\int_{R_0(\bm{\phi})+\frac{R_0(\bm{\phi})}{16\log(R_0(\bm{\phi}))}}^{\frac{17}{16}R_0(\bm{\phi})}\Psi(y)^2dy\Big)\nonumber\\
&\leq& \frac{16\log(R_0(\bm{\phi}))}{R_0(\bm{\phi})}\int_{\frac{15}{16}R_0(\bm{\phi})}^{\frac{17}{16}R_0(\bm{\phi})}\Psi(x)^2dx\nonumber\\
&\leq& \frac{16\log(R_0(\bm{\phi}))}{R_0(\bm{\phi})}\cdot (2R_0(\bm{\phi})+1)\int_{\frac{15}{16}R_0(\bm{\phi})}^{\frac{17}{16}R_0(\bm{\phi})}\frac{\Psi(x)^2}{x+1}dx\nonumber\\
&\leq& \frac{C\epsilon k^2}{\log\log(R_0(\bm{\phi}))}\leq C\epsilon k^2\leq CM\sqrt{\epsilon}k^2.
\end{eqnarray}

We bound $\mathscr{U}_{24}$ as follows. For any $y\in \big[R_0(\bm{\phi})+\frac{R_0(\bm{\phi})}{16\log(R_0(\bm{\phi}))},\frac{17}{16}R_0(\bm{\phi})\big]$,
\begin{equation}
\int_{R_0(\bm{\phi})-\frac{R_0(\bm{\phi})}{16\log(R_0(\bm{\phi}))}}^{R_0(\bm{\phi})}\frac{1}{(x-y)^2}dx\leq \frac{1}{y-R_0(\bm{\phi})}. 
\end{equation}
Hence by property (a) of Proposition \ref{P5.2},
\begin{eqnarray}\label{Eq5.88}
 \mathscr{U}_{24}&\leq& C\sqrt{\frac{\epsilon}{\log(R_0(\bm{\phi}))\log\log(R_0(\bm{\phi}))}}k\nonumber\\
 &&\times\int_{R_0(\bm{\phi})+\frac{R_0(\bm{\phi})}{16\log(R_0(\bm{\phi}))}}^{\frac{17}{16}R_0(\bm{\phi})}|\Psi(y)|dy \int_{R_0(\bm{\phi})-\frac{R_0(\bm{\phi})}{16\log(R_0(\bm{\phi}))}}^{R_0(\bm{\phi})}\frac{1}{(x-y)^2}dx \nonumber\\
 &\leq& C\sqrt{\frac{\epsilon}{\log(R_0(\bm{\phi}))\log\log(R_0(\bm{\phi}))}}k \int_{R_0(\bm{\phi})+\frac{R_0(\bm{\phi})}{16\log(R_0(\bm{\phi}))}}^{\frac{17}{16}R_0(\bm{\phi})}\frac{|\Psi(y)|}{y-R_0(\bm{\phi})}dy.\nonumber\\
 &&
\end{eqnarray}
Note that
\begin{equation}
    \int_{R_0(\bm{\phi})+\frac{R_0(\bm{\phi})}{16\log(R_0(\bm{\phi}))}}^{\frac{17}{16}R_0(\bm{\phi})}\frac{1}{(y-R_0(\bm{\phi}))^2} dy=\frac{16(\log(R_0(\bm{\phi}))-1)}{R_0(\bm{\phi})}\leq \frac{16\log(R_0(\bm{\phi}))}{R_0(\bm{\phi})}.
\end{equation}
Hence by the Cauchy-Schwarz inequality and property (d) of Proposition \ref{P5.2},
\begin{eqnarray}\label{Eq5.89}
&& \int_{R_0(\bm{\phi})+\frac{R_0(\bm{\phi})}{16\log(R_0(\bm{\phi}))}}^{\frac{17}{16}R_0(\bm{\phi})}\frac{|\Psi(y)|}{y-R_0(\bm{\phi})}dy \nonumber\\
&\leq& \Big(\int_{R_0(\bm{\phi})+\frac{R_0(\bm{\phi})}{16\log(R_0(\bm{\phi}))}}^{\frac{17}{16}R_0(\bm{\phi})}\frac{1}{(y-R_0(\bm{\phi}))^2} dy\Big)^{1\slash 2} \Big(\int_{R_0(\bm{\phi})+\frac{R_0(\bm{\phi})}{16\log(R_0(\bm{\phi}))}}^{\frac{17}{16}R_0(\bm{\phi})}\Psi(y)^2 dy\Big)^{1\slash 2}\nonumber\\
&\leq& \sqrt{\frac{16\log(R_0(\bm{\phi}))}{R_0(\bm{\phi})}}\Big((2R_0(\bm{\phi})+1)\int_{R_0(\bm{\phi})+\frac{R_0(\bm{\phi})}{16\log(R_0(\bm{\phi}))}}^{\frac{17}{16}R_0(\bm{\phi})}\frac{\Psi(y)^2}{y+1} dy\Big)^{1\slash 2}\nonumber\\
&\leq& C\sqrt{\frac{\epsilon}{\log\log(R_0(\bm{\phi}))}}k.
\end{eqnarray}
By (\ref{Eq5.88}) and (\ref{Eq5.89}), 
\begin{equation}
    \mathscr{U}_{24}\leq  C\epsilon k^2\leq CM\sqrt{\epsilon} k^2.
\end{equation}

We bound $\mathscr{U}_{31}$ as follows. For any $x\in [-R_0(\bm{\phi}),0]$ and $y\in [\frac{17}{16}R_0(\bm{\phi}),R_0'(\bm{\phi})]$, we have $y-x\geq y>0$. Hence 
\begin{equation}\label{Eq5.90}
     \mathscr{U}_{31}\leq \Big(\int_{-R_0(\bm{\phi})}^0 |\Psi(x)|dx\Big)
\Big(\int_{\frac{17}{16}R_0(\bm{\phi})}^{R_0'(\bm{\phi})}\frac{|\Psi(y)|}{y^2}dy\Big).
\end{equation}
Note that
\begin{equation}
    \int_{\frac{17}{16}R_0(\bm{\phi})}^{R_0'(\bm{\phi})}\frac{1}{y^3}dy\leq \frac{1}{2}\Big(\frac{17}{16}R_0(\bm{\phi})\Big)^{-2}\leq R_0(\bm{\phi})^{-2}.
\end{equation}
Hence by the Cauchy-Schwarz inequality and property (f) of Proposition \ref{P5.2},
\begin{eqnarray}\label{Eq5.91}
&&\int_{\frac{17}{16}R_0(\bm{\phi})}^{R_0'(\bm{\phi})}\frac{|\Psi(y)|}{y^2}dy\leq \Big(\int_{\frac{17}{16}R_0(\bm{\phi})}^{R_0'(\bm{\phi})}\frac{\Psi(y)^2}{y}dy\Big)^{1\slash 2}\Big(\int_{\frac{17}{16}R_0(\bm{\phi})}^{R_0'(\bm{\phi})}\frac{1}{y^3}dy\Big)^{1\slash 2}\nonumber\\
&\leq& R_0(\bm{\phi})^{-1}\Big(2\int_{\frac{17}{16}R_0(\bm{\phi})}^{R_0'(\bm{\phi})}\frac{\Psi(y)^2}{y+1}dy\Big)^{1\slash 2}\leq C\sqrt{\epsilon}R_0(\bm{\phi})^{-1} k.
\end{eqnarray}
By (\ref{Eq5.83}), (\ref{Eq5.90}), and (\ref{Eq5.91}),
\begin{equation}
    \mathscr{U}_{31}\leq CM\sqrt{\epsilon}R_0(\bm{\phi})^{-1}k^2\leq CM\sqrt{\epsilon}k^2.
\end{equation}

We bound $\mathscr{U}_{32}$ as follows. For any $x\in [0,R_0(\bm{\phi})]$ and $y\in \big[\frac{17}{16}R_0(\bm{\phi}),R_0'(\bm{\phi})\big]$, we have $x\leq 16y\slash 17$ and $y-x\geq y\slash 17$. Hence by the Cauchy-Schwarz inequality,
\begin{eqnarray}\label{Eq5.92}
&& \mathscr{U}_{32}\leq C \int_0^{R_0(\bm{\phi})}
\int_{\frac{17}{16}R_0(\bm{\phi})}^{R_0'(\bm{\phi})}\frac{|\Psi(x)||\Psi(y)|}{y^2}dxdy\nonumber\\
&\leq& C\Big(\int_0^{R_0(\bm{\phi})}
\int_{\frac{17}{16}R_0(\bm{\phi})}^{R_0'(\bm{\phi})}\frac{1}{y^3}dxdy\Big)^{1\slash 2}\Big(\int_0^{R_0(\bm{\phi})}
\int_{\frac{17}{16}R_0(\bm{\phi})}^{R_0'(\bm{\phi})}\frac{\Psi(x)^2\Psi(y)^2}{y}dxdy\Big)^{1\slash 2}.\nonumber\\
&& 
\end{eqnarray}
Note that
\begin{equation}\label{Eq5.93}
    \int_0^{R_0(\bm{\phi})}
\int_{\frac{17}{16}R_0(\bm{\phi})}^{R_0'(\bm{\phi})}\frac{1}{y^3}dxdy\leq R_0(\bm{\phi})\cdot \frac{1}{2}\Big(\frac{17}{16}R_0(\bm{\phi})\Big)^{-2}\leq R_0(\bm{\phi})^{-1}.
\end{equation}
By properties (b) and (f) of Proposition \ref{P5.2},
\begin{equation}
    \int_0^{R_0(\bm{\phi})}\Psi(x)^2dx\leq (R_0(\bm{\phi})+1)\int_0^{R_0(\bm{\phi})}\frac{\Psi(x)^2}{x+1}dx\leq CM^2 R_0(\bm{\phi}) k^2,
\end{equation}
\begin{equation}
    \int_{\frac{17}{16}R_0(\bm{\phi})}^{R_0'(\bm{\phi})}\frac{\Psi(y)^2}{y}dy\leq 2  \int_{\frac{17}{16}R_0(\bm{\phi})}^{R_0'(\bm{\phi})}\frac{\Psi(y)^2}{y+1}dy\leq C\epsilon k^2,
\end{equation}
hence
\begin{equation}\label{Eq5.94}
    \int_0^{R_0(\bm{\phi})}
\int_{\frac{17}{16}R_0(\bm{\phi})}^{R_0'(\bm{\phi})}\frac{\Psi(x)^2\Psi(y)^2}{y}dxdy\leq CM^2\epsilon R_0(\bm{\phi})k^4.
\end{equation}
By (\ref{Eq5.92}), (\ref{Eq5.93}), and (\ref{Eq5.94}),
\begin{equation}
    \mathscr{U}_{32}\leq CM\sqrt{\epsilon} k^2.
\end{equation}

By (\ref{Eq5.95}) and the above estimates, we conclude that
\begin{equation}\label{Eq5.1.9}
    \Big|\int_{-R_0(\bm{\phi})}^{R_0(\bm{\phi})} \int_{R_0(\bm{\phi})+R_0(\bm{\phi})^{-4}}^{R_0'(\bm{\phi})}\frac{\Psi(x)\Psi(y)}{(y-x)^2}dxdy\Big|\leq CM\sqrt{\epsilon} k^2.
\end{equation}

\medskip

By (\ref{Eq5.1.1}), (\ref{Eq5.68})-(\ref{Eq5.1.5}), (\ref{Eq5.1.6}), (\ref{Eq5.2.1}), (\ref{Eq5.1.7}), (\ref{Eq5.1.8}), and (\ref{Eq5.1.9}), we conclude that \begin{equation}\label{Eq5.20.4}
    |\mathscr{R}_3|\leq CM\sqrt{\epsilon}+\frac{CM\log(R_0(\bm{\phi}))}{R_0(\bm{\phi})^{3\slash 2}}.
\end{equation}

\subparagraph{Sub-step 3.4}

In this sub-step, we bound $\mathscr{R}_4$. Let
\begin{equation}\label{Eq5.3.1}
    \mathscr{S}_4=\int_{-R_0(\bm{\phi})}^{R_0(\bm{\phi})-R_0(\bm{\phi})^{-4}}\int_{R_0(\bm{\phi})}^{R_0(\bm{\phi})+R_0(\bm{\phi})^{-4}}\log(|x-y|)d\Psi(x)d\Psi(y).
\end{equation}
Note that
\begin{equation}\label{Eq5.9.1}
    \mathscr{R}_4=-k^{-2}\mathscr{S}_4.
\end{equation}

For any $y\in [R_0(\bm{\phi}), R_0(\bm{\phi})+R_0(\bm{\phi})^{-4}]$, by integration by parts, we have
\begin{eqnarray}\label{Eq5.3.2}
   && \int_{-R_0(\bm{\phi})}^{R_0(\bm{\phi})-R_0(\bm{\phi})^{-4}}\log(|x-y|)d\Psi(x)\nonumber\\
   &=&
   -\int_{-R_0(\bm{\phi})}^{R_0(\bm{\phi})-R_0(\bm{\phi})^{-4}}\frac{\Psi(x)}{x-y}dx-\log(y+R_0(\bm{\phi}))\Psi(-R_0(\bm{\phi}))\nonumber\\
   &&+\log(y-(R_0(\bm{\phi})-R_0(\bm{\phi})^{-4}))\Psi(R_0(\bm{\phi})-R_0(\bm{\phi})^{-4}).
\end{eqnarray}
By (\ref{Eq5.3.1}) and (\ref{Eq5.3.2}), we have
\begin{eqnarray}\label{Eq5.3.3}
 \mathscr{S}_4&=& \int_{-R_0(\bm{\phi})}^{R_0(\bm{\phi})-R_0(\bm{\phi})^{-4}}\Psi(x) dx \int_{R_0(\bm{\phi})}^{R_0(\bm{\phi})+R_0(\bm{\phi})^{-4}}\frac{1}{y-x}d\Psi(y)\nonumber\\
 &+&\Psi(R_0(\bm{\phi})-R_0(\bm{\phi})^{-4})\int_{R_0(\bm{\phi})}^{R_0(\bm{\phi})+R_0(\bm{\phi})^{-4}}\log(y-(R_0(\bm{\phi})-R_0(\bm{\phi})^{-4}))d\Psi(y)\nonumber\\
 &-&\Psi(-R_0(\bm{\phi}))\int_{R_0(\bm{\phi})}^{R_0(\bm{\phi})+R_0(\bm{\phi})^{-4}}\log(y+R_0(\bm{\phi}))d\Psi(y).
\end{eqnarray}

Now note that by integration by parts,
\begin{eqnarray}
&& \int_{R_0(\bm{\phi})}^{R_0(\bm{\phi})+R_0(\bm{\phi})^{-4}}\log(y-(R_0(\bm{\phi})-R_0(\bm{\phi})^{-4}))d\Psi(y)\nonumber\\
&=&\log(2R_0(\bm{\phi})^{-4})  \Psi(R_0(\bm{\phi})+R_0(\bm{\phi})^{-4})-\log(R_0(\bm{\phi})^{-4})\Psi(R_0(\bm{\phi}))\nonumber\\
&&-\int_{R_0(\bm{\phi})}^{R_0(\bm{\phi})+R_0(\bm{\phi})^{-4}}\frac{\Psi(y)}{y-(R_0(\bm{\phi})-R_0(\bm{\phi})^{-4})}dy,
\end{eqnarray}
\begin{eqnarray}
 &&   \int_{R_0(\bm{\phi})}^{R_0(\bm{\phi})+R_0(\bm{\phi})^{-4}}\log(y+R_0(\bm{\phi}))d\Psi(y)\nonumber\\
 &=& \log(2R_0(\bm{\phi})+R_0(\bm{\phi})^{-4})\Psi(R_0(\bm{\phi})+R_0(\bm{\phi})^{-4})-\log(2R_0(\bm{\phi}))\Psi(R_0(\bm{\phi}))\nonumber\\
 && -\int_{R_0(\bm{\phi})}^{R_0(\bm{\phi})+R_0(\bm{\phi})^{-4}}\frac{\Psi(y)}{y+R_0(\bm{\phi})}dy.
\end{eqnarray}
For any $x\in [-R_0(\bm{\phi}),R_0(\bm{\phi})-R_0(\bm{\phi})^{-4}]$, by integration by parts,
\begin{eqnarray}\label{Eq5.3.4}
\int_{R_0(\bm{\phi})}^{R_0(\bm{\phi})+R_0(\bm{\phi})^{-4}}\frac{1}{y-x}d\Psi(y)&=&\frac{\Psi(R_0(\bm{\phi})+R_0(\bm{\phi})^{-4})}{R_0(\bm{\phi})+R_0(\bm{\phi})^{-4}-x}-\frac{\Psi(R_0(\bm{\phi}))}{R_0(\bm{\phi})-x}\nonumber\\
&&+\int_{R_0(\bm{\phi})}^{R_0(\bm{\phi})+R_0(\bm{\phi})^{-4}}\frac{\Psi(y)}{(y-x)^2}dy.
\end{eqnarray}
By (\ref{Eq5.3.3})-(\ref{Eq5.3.4}), we obtain that
\begin{eqnarray}\label{Eq5.3.5}
\mathscr{S}_4&=& \Psi(R_0(\bm{\phi})-R_0(\bm{\phi})^{-4})\Psi(R_0(\bm{\phi})+R_0(\bm{\phi})^{-4})\log(2R_0(\bm{\phi})^{-4})  \nonumber\\
&& -\Psi(R_0(\bm{\phi})-R_0(\bm{\phi})^{-4})\Psi(R_0(\bm{\phi}))\log(R_0(\bm{\phi})^{-4})  \nonumber\\
&&-\Psi(-R_0(\bm{\phi}))\Psi(R_0(\bm{\phi})+R_0(\bm{\phi})^{-4})\log(2R_0(\bm{\phi})+R_0(\bm{\phi})^{-4})  \nonumber\\
&&+\Psi(-R_0(\bm{\phi}))\Psi(R_0(\bm{\phi}))\log(2R_0(\bm{\phi})) \nonumber\\
&& -\Psi(R_0(\bm{\phi})-R_0(\bm{\phi})^{-4})\int_{R_0(\bm{\phi})}^{R_0(\bm{\phi})+R_0(\bm{\phi})^{-4}}\frac{\Psi(y)}{y-(R_0(\bm{\phi})-R_0(\bm{\phi})^{-4})}dy\nonumber\\
&&+\Psi(-R_0(\bm{\phi}))\int_{R_0(\bm{\phi})}^{R_0(\bm{\phi})+R_0(\bm{\phi})^{-4}}\frac{\Psi(y)}{y+R_0(\bm{\phi})}dy\nonumber\\
&&+\Psi(R_0(\bm{\phi})+R_0(\bm{\phi})^{-4})\int_{-R_0(\bm{\phi})}^{R_0(\bm{\phi})-R_0(\bm{\phi})^{-4}}\frac{\Psi(x)}{R_0(\bm{\phi})+R_0(\bm{\phi})^{-4}-x}dx\nonumber\\
&&-\Psi(R_0(\bm{\phi}))\int_{-R_0(\bm{\phi})}^{R_0(\bm{\phi})-R_0(\bm{\phi})^{-4}}\frac{\Psi(x)}{R_0(\bm{\phi})-x}dx\nonumber\\
&&+\int_{-R_0(\bm{\phi})}^{R_0(\bm{\phi})-R_0(\bm{\phi})^{-4}} \int_{R_0(\bm{\phi})}^{R_0(\bm{\phi})+R_0(\bm{\phi})^{-4}}\frac{\Psi(x)\Psi(y)}{(y-x)^2}dxdy.
\end{eqnarray}

In the following, we bound the terms on the right-hand side of (\ref{Eq5.3.5}) in \textbf{Parts 3.4.1-3.4.9}.

\textbf{Part 3.4.1} By (\ref{Eq5.50}) and property (a) of Proposition \ref{P5.2}, 
\begin{eqnarray}\label{Eq5.10.1}
&& |\Psi(R_0(\bm{\phi})-R_0(\bm{\phi})^{-4})\Psi(R_0(\bm{\phi})+R_0(\bm{\phi})^{-4})\log(2R_0(\bm{\phi})^{-4})  | \nonumber\\
&\leq& \frac{C\epsilon k^2}{\log(R_0(\bm{\phi}))\log\log(R_0(\bm{\phi}))}\cdot \log(R_0(\bm{\phi}))\leq \frac{C\epsilon k^2}{\log\log(R_0(\bm{\phi}))}.
\end{eqnarray}

\textbf{Part 3.4.2} By (\ref{Eq5.50}) and property (a) of Proposition \ref{P5.2}, 
\begin{eqnarray}\label{Eq5.10.2}
&& |\Psi(R_0(\bm{\phi})-R_0(\bm{\phi})^{-4})\Psi(R_0(\bm{\phi}))\log(R_0(\bm{\phi})^{-4})|\nonumber\\
&\leq& \frac{C\epsilon k^2}{\log(R_0(\bm{\phi}))\log\log(R_0(\bm{\phi}))}\cdot \log(R_0(\bm{\phi}))\leq \frac{C\epsilon k^2}{\log\log(R_0(\bm{\phi}))}.
\end{eqnarray}

\textbf{Part 3.4.3} By (\ref{Eq5.11}), (\ref{Eq5.50}), and properties (a) and (c) of Proposition \ref{P5.2}, 
\begin{eqnarray}\label{Eq5.10.3}
&& |\Psi(-R_0(\bm{\phi}))\Psi(R_0(\bm{\phi})+R_0(\bm{\phi})^{-4})\log(2R_0(\bm{\phi})+R_0(\bm{\phi})^{-4})| \nonumber\\
&\leq& \frac{CMk}{R_0(\bm{\phi})^{3\slash 2}}\cdot \sqrt{\frac{\epsilon}{\log(R_0(\bm{\phi}))\log\log(R_0(\bm{\phi}))}}k \cdot \log(R_0(\bm{\phi}))\nonumber\\
&\leq& \frac{CM\log(R_0(\bm{\phi})) k^2}{R_0(\bm{\phi})^{3\slash 2}}.
\end{eqnarray}

\textbf{Part 3.4.4} By (\ref{Eq5.11}) and properties (a) and (c) of Proposition \ref{P5.2}
\begin{eqnarray}\label{Eq5.10.4}
 && |\Psi(-R_0(\bm{\phi}))\Psi(R_0(\bm{\phi}))\log(2R_0(\bm{\phi}))|\nonumber\\
 &\leq& \frac{CMk}{R_0(\bm{\phi})^{3\slash 2}}\cdot \sqrt{\frac{\epsilon}{\log(R_0(\bm{\phi}))\log\log(R_0(\bm{\phi}))}}k \cdot \log(R_0(\bm{\phi}))\nonumber\\
&\leq& \frac{CM\log(R_0(\bm{\phi})) k^2}{R_0(\bm{\phi})^{3\slash 2}}.
\end{eqnarray}

\textbf{Part 3.4.5} By (\ref{Eq5.50}) and property (a) of Proposition \ref{P5.2},
\begin{eqnarray}\label{Eq5.10.5}
&&\Big|\Psi(R_0(\bm{\phi})-R_0(\bm{\phi})^{-4})\int_{R_0(\bm{\phi})}^{R_0(\bm{\phi})+R_0(\bm{\phi})^{-4}}\frac{\Psi(y)}{y-(R_0(\bm{\phi})-R_0(\bm{\phi})^{-4})}dy\Big|\nonumber\\
&\leq& |\Psi(R_0(\bm{\phi})-R_0(\bm{\phi})^{-4})|\int_{R_0(\bm{\phi})}^{R_0(\bm{\phi})+R_0(\bm{\phi})^{-4}}\frac{|\Psi(y)|}{y-(R_0(\bm{\phi})-R_0(\bm{\phi})^{-4})}dy\nonumber\\
&\leq& \frac{C\epsilon k^2}{\log(R_0(\bm{\phi}))\log\log(R_0(\bm{\phi}))}\int_{R_0(\bm{\phi})}^{R_0(\bm{\phi})+R_0(\bm{\phi})^{-4}}\frac{1}{y-(R_0(\bm{\phi})-R_0(\bm{\phi})^{-4})}dy\nonumber\\
&\leq& \frac{C\epsilon k^2}{\log(R_0(\bm{\phi}))\log\log(R_0(\bm{\phi}))}.
\end{eqnarray}

\textbf{Part 3.4.6} By (\ref{Eq5.11}), (\ref{Eq5.50}), and properties (a) and (c) of Proposition \ref{P5.2},
\begin{eqnarray}\label{Eq5.10.6}
&& \Big|\Psi(-R_0(\bm{\phi}))\int_{R_0(\bm{\phi})}^{R_0(\bm{\phi})+R_0(\bm{\phi})^{-4}}\frac{\Psi(y)}{y+R_0(\bm{\phi})}dy\Big|\nonumber\\
&\leq& |\Psi(-R_0(\bm{\phi}))|\int_{R_0(\bm{\phi})}^{R_0(\bm{\phi})+R_0(\bm{\phi})^{-4}}\frac{|\Psi(y)|}{y+R_0(\bm{\phi})}dy\nonumber\\
&\leq& \frac{CMk}{R_0(\bm{\phi})^{3\slash 2}}\cdot \sqrt{\frac{\epsilon}{\log(R_0(\bm{\phi}))\log\log(R_0(\bm{\phi}))}}k\cdot R_0(\bm{\phi})^{-1}\cdot R_0(\bm{\phi})^{-4}\nonumber\\
&\leq& \frac{CMk^2}{R_0(\bm{\phi})^{3\slash 2}}.
\end{eqnarray}

\textbf{Part 3.4.7} By (\ref{Eq5.50}), (\ref{Eq5.79}), and property (a) of Proposition \ref{P5.2},
\begin{eqnarray}\label{Eq5.10.7}
&& \Big|\Psi(R_0(\bm{\phi})+R_0(\bm{\phi})^{-4})\int_{-R_0(\bm{\phi})}^{R_0(\bm{\phi})-R_0(\bm{\phi})^{-4}}\frac{\Psi(x)}{R_0(\bm{\phi})+R_0(\bm{\phi})^{-4}-x}dx\Big|\nonumber\\
&\leq& C\sqrt{\frac{\epsilon}{\log(R_0(\bm{\phi}))\log\log(R_0(\bm{\phi}))}}k \int_{-R_0(\bm{\phi})}^{R_0(\bm{\phi})-R_0(\bm{\phi})^{-4}}\frac{|\Psi(x)|}{R_0(\bm{\phi})+R_0(\bm{\phi})^{-4}-x}dx\nonumber\\
&\leq& C\sqrt{\frac{\epsilon}{\log(R_0(\bm{\phi}))\log\log(R_0(\bm{\phi}))}}k\cdot M\sqrt{\log(R_0(\bm{\phi}))\log\log(R_0(\bm{\phi}))} k\nonumber\\
&\leq& CM\sqrt{\epsilon} k^2.
\end{eqnarray}

\textbf{Part 3.4.8} By (\ref{Eq5.11}) and property (c) of Proposition \ref{P5.2}, 
\begin{eqnarray}\label{Eq5.5.1}
&& \int_{-R_0(\bm{\phi})}^{0}\frac{|\Psi(x)|}{R_0(\bm{\phi})-x}dx\leq R_0(\bm{\phi})^{-1}\int_{-R_0(\bm{\phi})}^{0}|\Psi(x)|dx \nonumber\\
&\leq& CMkR_0(\bm{\phi})^{-1}\int_{-R_0(\bm{\phi})}^{0}(|x|+1)^{-3\slash 2}dx\leq CMkR_0(\bm{\phi})^{-1}.
\end{eqnarray}

Note that
\begin{equation}
    \int_0^{R_0(\bm{\phi})-\frac{R_0(\bm{\phi})}{16\log(R_0(\bm{\phi}))}}\frac{1}{(R_0(\bm{\phi})-x)^2}dx=\frac{16\log(R_0(\bm{\phi}))-1}{R_0(\bm{\phi})}\leq \frac{16\log(R_0(\bm{\phi}))}{R_0(\bm{\phi})}.
\end{equation}
By property (b) of Proposition \ref{P5.2},
\begin{eqnarray}
    \int_0^{R_0(\bm{\phi})-\frac{R_0(\bm{\phi})}{16\log(R_0(\bm{\phi}))}}\Psi(x)^2 dx&\leq& (R_0(\bm{\phi})+1)\int_0^{R_0(\bm{\phi})-\frac{R_0(\bm{\phi})}{16\log(R_0(\bm{\phi}))}}\frac{\Psi(x)^2}{x+1} dx\nonumber\\
 &\leq&CM^2 R_0(\bm{\phi}) k^2.
\end{eqnarray}
Hence by the Cauchy-Schwarz inequality,
\begin{eqnarray}\label{Eq5.5.2}
&& \int_0^{R_0(\bm{\phi})-\frac{R_0(\bm{\phi})}{16\log(R_0(\bm{\phi}))}}\frac{|\Psi(x)|}{R_0(\bm{\phi})-x}dx \nonumber\\
&\leq& \Big(\int_0^{R_0(\bm{\phi})-\frac{R_0(\bm{\phi})}{16\log(R_0(\bm{\phi}))}}\frac{1}{(R_0(\bm{\phi})-x)^2}dx\Big)^{1\slash 2}\Big(\int_0^{R_0(\bm{\phi})-\frac{R_0(\bm{\phi})}{16\log(R_0(\bm{\phi}))}}\Psi(x)^2dx\Big)^{1\slash 2}\nonumber\\
&\leq& CM\sqrt{\log(R_0(\bm{\phi}))} k.
\end{eqnarray}

By property (a) of Proposition \ref{P5.2},
\begin{eqnarray}\label{Eq5.5.3}
 &&\int_{R_0(\bm{\phi})-\frac{R_0(\bm{\phi})}{16\log(R_0(\bm{\phi}))}}^{R_0(\bm{\phi})-R_0(\bm{\phi})^{-4}}\frac{|\Psi(x)|}{R_0(\bm{\phi})-x}dx\nonumber\\
 &\leq& C\sqrt{\frac{\epsilon}{\log(R_0(\bm{\phi}))\log\log(R_0(\bm{\phi}))}}k \int_{R_0(\bm{\phi})-\frac{R_0(\bm{\phi})}{16\log(R_0(\bm{\phi}))}}^{R_0(\bm{\phi})-R_0(\bm{\phi})^{-4}}\frac{1}{R_0(\bm{\phi})-x}dx\nonumber\\
 &=& C\sqrt{\frac{\epsilon}{\log(R_0(\bm{\phi}))\log\log(R_0(\bm{\phi}))}}k \log\Big(\frac{R_0(\bm{\phi})\slash (16\log(R_0(\bm{\phi})))}{R_0(\bm{\phi})^{-4}}\Big)\nonumber\\
 &\leq& C\sqrt{\frac{\epsilon \log(R_0(\bm{\phi}))}{\log\log(R_0(\bm{\phi}))}}k.
\end{eqnarray}

By (\ref{Eq5.5.1}), (\ref{Eq5.5.2}), and (\ref{Eq5.5.3}), we have
\begin{equation}\label{Eq5.6.1}
 \int_{-R_0(\bm{\phi})}^{R_0(\bm{\phi})-R_0(\bm{\phi})^{-4}}\frac{|\Psi(x)|}{R_0(\bm{\phi})-x}dx\leq CM\sqrt{\log(R_0(\bm{\phi}))} k.
\end{equation}
Hence by property (a) of Proposition \ref{P5.2},
\begin{eqnarray}\label{Eq5.10.8}
 && \Big|\Psi(R_0(\bm{\phi}))\int_{-R_0(\bm{\phi})}^{R_0(\bm{\phi})-R_0(\bm{\phi})^{-4}}\frac{\Psi(x)}{R_0(\bm{\phi})-x}dx\Big|\nonumber\\
 &\leq& C\sqrt{\frac{\epsilon}{\log(R_0(\bm{\phi}))\log\log(R_0(\bm{\phi}))}}k\cdot M\sqrt{\log(R_0(\bm{\phi}))} k \leq  CM\sqrt{\epsilon} k^2.\nonumber\\
 &&
\end{eqnarray}

\textbf{Part 3.4.9} By (\ref{Eq5.50}) and property (a) of Proposition \ref{P5.2}, 
\begin{eqnarray}\label{Eq5.7.1}
 && \Big|\int_{-R_0(\bm{\phi})}^{R_0(\bm{\phi})-R_0(\bm{\phi})^{-4}} \int_{R_0(\bm{\phi})}^{R_0(\bm{\phi})+R_0(\bm{\phi})^{-4}}\frac{\Psi(x)\Psi(y)}{(y-x)^2}dxdy\Big|\nonumber\\
 &\leq& \int_{-R_0(\bm{\phi})}^{R_0(\bm{\phi})-R_0(\bm{\phi})^{-4}} \int_{R_0(\bm{\phi})}^{R_0(\bm{\phi})+R_0(\bm{\phi})^{-4}}\frac{|\Psi(x)||\Psi(y)|}{(y-x)^2}dxdy\nonumber\\
 &\leq& C\sqrt{\frac{\epsilon}{\log(R_0(\bm{\phi}))\log\log(R_0(\bm{\phi}))}}k\nonumber\\
 &&\times \int_{-R_0(\bm{\phi})}^{R_0(\bm{\phi})-R_0(\bm{\phi})^{-4}} |\Psi(x)|dx \int_{R_0(\bm{\phi})}^{R_0(\bm{\phi})+R_0(\bm{\phi})^{-4}}\frac{1}{(y-x)^2}dy.
\end{eqnarray}

For any $x\in [-R_0(\bm{\phi}),R_0(\bm{\phi})-R_0(\bm{\phi})^{-4}]$, 
\begin{equation}\label{Eq5.6.2}
\int_{R_0(\bm{\phi})}^{R_0(\bm{\phi})+R_0(\bm{\phi})^{-4}}\frac{1}{(y-x)^2}dy\leq \frac{1}{R_0(\bm{\phi})-x}.
\end{equation}
Hence by (\ref{Eq5.50}) and (\ref{Eq5.6.1}),
\begin{eqnarray}\label{Eq5.7.2}
&& \int_{-R_0(\bm{\phi})}^{R_0(\bm{\phi})-\frac{R_0(\bm{\phi})}{16\log(R_0(\bm{\phi}))}} |\Psi(x)|dx \int_{R_0(\bm{\phi})}^{R_0(\bm{\phi})+R_0(\bm{\phi})^{-4}}\frac{1}{(y-x)^2}dy \nonumber\\
&\leq& \int_{-R_0(\bm{\phi})}^{R_0(\bm{\phi})-\frac{R_0(\bm{\phi})}{16\log(R_0(\bm{\phi}))}} \frac{|\Psi(x)|}{R_0(\bm{\phi})-x}dx \leq CM\sqrt{\log(R_0(\bm{\phi}))} k.
\end{eqnarray}

By (\ref{Eq5.6.2}), we have
\begin{eqnarray}\label{Eq5.6.7}
 && \int_{R_0(\bm{\phi})-\frac{R_0(\bm{\phi})}{16\log(R_0(\bm{\phi}))}}^{R_0(\bm{\phi})-R_0(\bm{\phi})^{-4}} dx \int_{R_0(\bm{\phi})}^{R_0(\bm{\phi})+R_0(\bm{\phi})^{-4}}\frac{1}{(y-x)^2}dy \nonumber\\
 &\leq& \int_{R_0(\bm{\phi})-\frac{R_0(\bm{\phi})}{16\log(R_0(\bm{\phi}))}}^{R_0(\bm{\phi})-R_0(\bm{\phi})^{-4}} \frac{1}{R_0(\bm{\phi})-x}dx=\log\Big(\frac{R_0(\bm{\phi})\slash (16\log(R_0(\bm{\phi})))}{R_0(\bm{\phi})^{-4}}\Big)\nonumber\\
 &\leq& C\log(R_0(\bm{\phi})).
\end{eqnarray}
Hence by property (a) of Proposition \ref{P5.2}, 
\begin{eqnarray}\label{Eq5.6.6}
&& \int_{R_0(\bm{\phi})-\frac{R_0(\bm{\phi})}{16\log(R_0(\bm{\phi}))}}^{R_0(\bm{\phi})-R_0(\bm{\phi})^{-4}} |\Psi(x)|dx \int_{R_0(\bm{\phi})}^{R_0(\bm{\phi})+R_0(\bm{\phi})^{-4}}\frac{1}{(y-x)^2}dy \nonumber\\
&\leq& C\sqrt{\frac{\epsilon}{\log(R_0(\bm{\phi}))\log\log(R_0(\bm{\phi}))}}k\nonumber\\
&&\times\int_{R_0(\bm{\phi})-\frac{R_0(\bm{\phi})}{16\log(R_0(\bm{\phi}))}}^{R_0(\bm{\phi})-R_0(\bm{\phi})^{-4}} dx \int_{R_0(\bm{\phi})}^{R_0(\bm{\phi})+R_0(\bm{\phi})^{-4}}\frac{1}{(y-x)^2}dy\nonumber\\
&\leq& C\sqrt{\frac{\epsilon \log(R_0(\bm{\phi}))}{\log\log(R_0(\bm{\phi}))}}k.
\end{eqnarray}

By (\ref{Eq5.7.1}), (\ref{Eq5.7.2}), and (\ref{Eq5.6.6}), we conclude that
\begin{equation}\label{Eq5.10.9}
\Big|\int_{-R_0(\bm{\phi})}^{R_0(\bm{\phi})-R_0(\bm{\phi})^{-4}} \int_{R_0(\bm{\phi})}^{R_0(\bm{\phi})+R_0(\bm{\phi})^{-4}}\frac{\Psi(x)\Psi(y)}{(y-x)^2}dxdy\Big|\leq CM\sqrt{\epsilon} k^2.
\end{equation}

\medskip

By (\ref{Eq5.9.1}), (\ref{Eq5.3.5}), (\ref{Eq5.10.1})-(\ref{Eq5.10.7}), (\ref{Eq5.10.8}), and (\ref{Eq5.10.9}), we conclude that
\begin{equation}\label{Eq5.20.5}
  |\mathscr{R}_4|\leq CM\sqrt{\epsilon}+\frac{CM\log(R_0(\bm{\phi}))}{R_0(\bm{\phi})^{3\slash 2}}.
\end{equation}

\bigskip

By (\ref{Eq5.20.1}), (\ref{Eq5.20.2}), (\ref{Eq5.20.3}), (\ref{Eq5.20.4}), and (\ref{Eq5.20.5}), we conclude that
\begin{equation}\label{Eq5.19.13}
\Xi_{2,4}+\Xi_{3,4}\geq -CM\sqrt{\epsilon}-\frac{CM\log(R_0(\bm{\phi}))}{R_0(\bm{\phi})^{3\slash 2}},
\end{equation}
where we note (\ref{Eq2.1.15}) and that $T$ only depends on $M,\epsilon$ and $k\geq K_0(M,\epsilon,\delta)$ is sufficiently large (depending on $\beta,M,\epsilon,\delta$).

\paragraph{Step 4} In this step, we bound $\tilde{\Xi}_{2,4;n_0,\mathbf{t}}$ and $\tilde{\Xi}_{3,4;n_0,\mathbf{t}}$. By (\ref{Eq5.25}), we have $\tilde{\kappa}_{2;n_0,\mathbf{t}}=0$. Hence
\begin{equation}\label{Eq5.19.14}
    \tilde{\Xi}_{2,4;n_0,\mathbf{t}}=0. 
\end{equation}

For any $x\in [0,R_0(\bm{\phi})]$, let 
\begin{equation}\label{Eq5.14.1}
    h(x):=\int \log(|x-y|)d\tilde{\kappa}_{4;n_0,\mathbf{t}}(y).
\end{equation}
By (\ref{Eq5.25}) and integration by parts, for any $x\in [0,R_0(\bm{\phi}))$, we have
\begin{eqnarray}\label{Eq5.12.3}
    h(x)&=&\frac{1}{k}\int_{R_0(\bm{\phi})}^{R_0'(\bm{\phi})}\log(|x-y|)d\Psi(y),\nonumber\\
    &=& \frac{1}{k}\Big(\log(R_0'(\bm{\phi})-x)\Psi(R_0'(\bm{\phi}))-\log(R_0(\bm{\phi})-x)\Psi(R_0(\bm{\phi}))\nonumber\\
    &&\quad\quad
    -\int_{R_0(\bm{\phi})}^{R_0'(\bm{\phi})}\frac{\Psi(y)}{y-x}dy\Big);
\end{eqnarray}
moreover,
\begin{eqnarray}\label{Eq5.18.8}
 \tilde{\Xi}_{3,4;n_0,\mathbf{t}}&=&-\int d\tilde{\kappa}_{3;n_0,\mathbf{t}}(x)\int \log(|x-y|)d\tilde{\kappa}_{4;n_0,\mathbf{t}}(y)\nonumber\\
 &=& -\frac{1}{k}\sum_{i=1}^{n_0} h(\rho_i+t_i)+\int_{0}^{R_0(\bm{\phi})}h(x)d\mu_0(x).
\end{eqnarray}

Below we consider an arbitrary $x\in [0,r_0+1]$. As $r_0\in [1,10]$, by (\ref{Eq5.11}), for any $y\in [R_0(\bm{\phi}), R_0'(\bm{\phi})]$, we have $x\leq 11\leq R_0(\bm{\phi})\slash 2\leq  y\slash 2$ and $y-x\geq y\slash 2$. Hence 
\begin{equation}\label{Eq5.11.1}
\Big|\int_{R_0(\bm{\phi})}^{R_0'(\bm{\phi})}\Psi(y)\Big(\frac{1}{y-x}-\frac{1}{y}\Big)dy\Big|\leq \int_{R_0(\bm{\phi})}^{R_0'(\bm{\phi})} \frac{|\Psi(y)| x}{(y-x)y}dy\leq 22  \int_{R_0(\bm{\phi})}^{R_0'(\bm{\phi})} \frac{|\Psi(y)|}{y^2}dy.
\end{equation}
By property (f) of Proposition \ref{P5.2}, 
\begin{equation}\label{Eq5.11.2}
\int_{R_0(\bm{\phi})}^{R_0'(\bm{\phi})} \frac{\Psi(y)^2}{y}dy\leq 2\int_{R_0(\bm{\phi})}^{R_0'(\bm{\phi})} \frac{\Psi(y)^2}{y+1}dy\leq C\epsilon k^2.
\end{equation}
By (\ref{Eq5.11.1}), (\ref{Eq5.11.2}), and the Cauchy-Schwarz inequality,
\begin{eqnarray}\label{Eq5.12.1}
\Big|\int_{R_0(\bm{\phi})}^{R_0'(\bm{\phi})}\Psi(y)\Big(\frac{1}{y-x}-\frac{1}{y}\Big)dy\Big|&\leq& C\Big(\int_{R_0(\bm{\phi})}^{R_0'(\bm{\phi})} \frac{\Psi(y)^2}{y}dy\Big)^{1\slash 2}\Big(\int_{R_0(\bm{\phi})}^{R_0'(\bm{\phi})} \frac{1}{y^3}dy\Big)^{1\slash 2}\nonumber\\
&\leq& C\sqrt{\epsilon} R_0(\bm{\phi})^{-1} k. 
\end{eqnarray}
Note that $R_0'(\bm{\phi})-x\geq R_0(\bm{\phi})-x\geq R_0(\bm{\phi})\slash 2$. Hence
\begin{eqnarray}\label{Eq5.12.2}
&& 0\leq \log(R_0'(\bm{\phi}))-\log(R_0'(\bm{\phi})-x)=\log\Big(1+\frac{x}{R_0'(\bm{\phi})-x}\Big)\leq \frac{x}{R_0'(\bm{\phi})-x}\leq \frac{22}{R_0(\bm{\phi})},\nonumber\\
&& 0\leq \log(R_0(\bm{\phi}))-\log(R_0(\bm{\phi})-x)=\log\Big(1+\frac{x}{R_0(\bm{\phi})-x}\Big)\leq \frac{x}{R_0(\bm{\phi})-x}\leq \frac{22}{R_0(\bm{\phi})}.\nonumber\\
&& 
\end{eqnarray}
By (\ref{Eq5.12.3}), (\ref{Eq5.12.1}), (\ref{Eq5.12.2}), and properties (a) and (e) of Proposition \ref{P5.2}, 
\begin{eqnarray}\label{Eq5.12.4}
&& |h(x)-h(0)| \leq \frac{1}{k}\Big|\int_{R_0(\bm{\phi})}^{R_0'(\bm{\phi})}\Psi(y)\Big(\frac{1}{y-x}-\frac{1}{y}\Big)dy\Big|\nonumber\\
&&\quad\quad\quad\quad\quad\quad\quad+\frac{1}{k}|\log(R_0'(\bm{\phi})-x)-\log(R_0'(\bm{\phi}))||\Psi(R_0'(\bm{\phi}))|\nonumber\\
&&\quad\quad\quad\quad\quad\quad\quad + \frac{1}{k}|\log(R_0(\bm{\phi})-x)-\log(R_0(\bm{\phi}))||\Psi(R_0(\bm{\phi}))|\nonumber\\
&\leq& C\sqrt{\epsilon}R_0(\bm{\phi})^{-1}+CR_0(\bm{\phi})^{-1}k^{-1}\cdot\sqrt{\frac{\epsilon}{\log(R_0(\bm{\phi}))}}k\nonumber\\
&& +CR_0(\bm{\phi})^{-1}k^{-1}\cdot \sqrt{\frac{\epsilon}{\log(R_0(\bm{\phi}))\log\log(R_0(\bm{\phi}))}}k\nonumber\\
&\leq& C\sqrt{\epsilon}R_0(\bm{\phi})^{-1}.
\end{eqnarray}
By (\ref{Eq5.21}), for any $i\in [m_0']$, $0<\rho_i+t_i\leq r_0+n^{-1}\leq r_0+1$. Hence by (\ref{Eq5.43}) and (\ref{Eq5.12.4}), 
\begin{eqnarray}\label{Eq5.12.6}
 && \Big|\frac{1}{k}\sum_{i=1}^{m_0'}h(\rho_i+t_i)-\int_0^{r_0}h(x)d\mu_0(x)-\Big(\frac{m_0'}{k}-\mu_0([0,r_0])\Big) h(0)\Big| \nonumber\\
 &\leq& \frac{1}{k}\sum_{i=1}^{m_0'}|h(\rho_i+t_i)-h(0)|+\int_{0}^{r_0}|h(x)-h(0)|d\mu_0(x)\nonumber\\
 &\leq& C\sqrt{\epsilon}R_0(\bm{\phi})^{-1}\Big(\frac{m_0'}{k}+\mu_0([0,r_0])\Big)\leq C\sqrt{\epsilon}R_0(\bm{\phi})^{-1}. 
\end{eqnarray}

By the Cauchy-Schwarz inequality, (\ref{Eq2.2.1}), and (\ref{Eq5.11.2}),
\begin{eqnarray}\label{Eq5.12.5}
 && \int_{R_0(\bm{\phi})}^{R_0'(\bm{\phi})}\frac{|\Psi(y)|}{y}dy \leq \Big(\int_{R_0(\bm{\phi})}^{R_0'(\bm{\phi})}\frac{\Psi(y)^2}{y}dy\Big)^{1\slash 2} \Big(\int_{R_0(\bm{\phi})}^{R_0'(\bm{\phi})}\frac{1}{y}dy\Big)^{1\slash 2} \nonumber\\
 &\leq& C\sqrt{\epsilon} k \cdot \sqrt{\log\Big(\frac{R_0'(\bm{\phi})}{R_0(\bm{\phi})}\Big)}\leq C\sqrt{\epsilon \log(R_0(\bm{\phi}))} k.
\end{eqnarray}
By (\ref{Eq2.2.1}), (\ref{Eq5.12.3}), (\ref{Eq5.12.5}), and properties (a) and (e) of Proposition \ref{P5.2},
\begin{eqnarray}\label{Eq5.12.8}
|h(0)|&\leq& \frac{1}{k}\Big(\log(R_0'(\bm{\phi}))|\Psi(R_0'(\bm{\phi}))|+\log(R_0(\bm{\phi}))|\Psi(R_0(\bm{\phi}))|\nonumber\\
&& \quad\quad+\int_{R_0(\bm{\phi})}^{R_0'(\bm{\phi})}\frac{|\Psi(y)|}{y}dy\Big)\nonumber\\
&\leq& \frac{C}{k}\Big(\log(R_0(\bm{\phi}))\cdot \sqrt{\frac{\epsilon}{\log(R_0(\bm{\phi}))}}k+\sqrt{\epsilon \log(R_0(\bm{\phi}))} k\Big)\nonumber\\
&\leq& C\sqrt{\epsilon \log(R_0(\bm{\phi}))}.
\end{eqnarray}
By (\ref{Eq5.15}) and (\ref{Eq5.16}),
\begin{equation}\label{Eq5.12.9}
 \Big|\frac{m_0'}{k}-\mu_0([0,r_0])\Big|=\Big|\frac{n_0}{k}-\mu_0([0,R_0(\bm{\phi})])\Big|\leq C_0\sqrt{\frac{\epsilon}{\log(R_0(\bm{\phi}))\log\log(R_0(\bm{\phi}))}}.
\end{equation}

By (\ref{Eq2.1.10}), as $M\geq M_0$ is sufficiently large and $\epsilon\in (0,\epsilon_0)$ is sufficiently small, we have $R_0(\bm{\phi})\geq \sqrt{\log\log(R_0(\bm{\phi}))}$. Hence by (\ref{Eq5.12.6}), (\ref{Eq5.12.8}), and (\ref{Eq5.12.9}),
\begin{equation}\label{Eq5.18.7}
\Big|\frac{1}{k}\sum_{i=1}^{m_0'}h(\rho_i+t_i)-\int_0^{r_0}h(x)d\mu_0(x)\Big|\leq C\sqrt{\frac{\epsilon}{\log\log(R_0(\bm{\phi}))}}. 
\end{equation}

By (\ref{Eq5.16}) and (\ref{Eq5.39}), for any $i\in [m_0'+1,n_0]\cap\mathbb{Z}$, $\mu_0([\rho_i,\rho_{i+1}])=k^{-1}$. Hence noting (\ref{Eq5.14.1}), we have
\begin{eqnarray}\label{Eq5.18.1}
&&  \frac{1}{k}\sum_{i=m_0'+1}^{n_0}h(\rho_i+t_i)-\int_{r_0}^{R_0(\bm{\phi})}h(x)d\mu_0(x) \nonumber\\
&=& \sum_{i=m_0'+1}^{n_0}\int_{\rho_i}^{\rho_{i+1}} (h(\rho_i+t_i)-h(x))d\mu_0(x)\nonumber\\
&=& \sum_{i=m_0'+1}^{n_0}\int \Big(\int_{\rho_i}^{\rho_{i+1}}(\log(|\rho_i+t_i-y|)-\log(|x-y|))d\mu_0(x) \Big) d\tilde{\kappa}_{4;n_0,\mathbf{t}}(y).   \nonumber\\
&&
\end{eqnarray}

Below we consider any $i\in [m_0'+1,n_0-1]\cap\mathbb{Z}$ and any $x\in [\rho_i,\rho_{i+1}]$, $y\in [R_0(\bm{\phi}),R_0'(\bm{\phi})]$. By (\ref{Eq5.20}), we have
\begin{equation}
 y-x\geq R_0(\bm{\phi})-\rho_{i+1}=\sum_{l=i+1}^{n_0}(\rho_{l+1}-\rho_l)\geq \frac{c(n_0-i)}{k\sqrt{R_0(\bm{\phi})}},
\end{equation}
\begin{equation}
 y-(\rho_i+t_i)\geq (R_0(\bm{\phi})-\rho_{i+1})+(\rho_{i+1}-\rho_i-n^{-1})\geq R_0(\bm{\phi})-\rho_{i+1} \geq \frac{c(n_0-i)}{k\sqrt{R_0(\bm{\phi})}}.
\end{equation}
By (\ref{Eq5.15.1}), 
\begin{equation}
 x-(\rho_i+t_i)\leq \rho_{i+1}-\rho_i\leq C k^{-1}, \quad \rho_i+t_i-x\leq n^{-1}\leq k^{-1}. 
\end{equation}
Hence
\begin{eqnarray}
&& \log(|\rho_i+t_i-y|)-\log(|x-y|)=\log\Big(\frac{y-(\rho_i+t_i)}{y-x}\Big)=\log\Big(1+\frac{x-(\rho_i+t_i)}{y-x}\Big) \nonumber\\
&&\leq \frac{x-(\rho_i+t_i)}{y-x}\leq \frac{C\slash k}{c(n_0-i)\slash (k\sqrt{R_0(\bm{\phi})})}\leq \frac{C\sqrt{R_0(\bm{\phi})}}{n_0-i},
\end{eqnarray}
\begin{eqnarray}
&& \log(|x-y|)-\log(|\rho_i+t_i-y|)=\log\Big(\frac{y-x}{y-(\rho_i+t_i)}\Big)=\log\Big(1+\frac{\rho_i+t_i-x}{y-(\rho_i+t_i)}\Big) \nonumber\\
&&\leq \frac{\rho_i+t_i-x}{y-(\rho_i+t_i)}\leq \frac{k^{-1}}{c(n_0-i)\slash (k\sqrt{R_0(\bm{\phi})})}\leq \frac{C\sqrt{R_0(\bm{\phi})}}{n_0-i},
\end{eqnarray}
which lead to
\begin{equation}\label{Eq5.16.1}
 |\log(|\rho_i+t_i-y|)-\log(|x-y|)|\leq \frac{C\sqrt{R_0(\bm{\phi})}}{n_0-i}.
\end{equation}

For any $i\in [m_0'+1,n_0-1]\cap\mathbb{Z}$ and $y\in [R_0(\bm{\phi}),R_0'(\bm{\phi})]$, by (\ref{Eq5.16}) and (\ref{Eq5.39}), 
\begin{eqnarray}
&& \Big|\int_{\rho_i}^{\rho_{i+1}}(\log(|\rho_i+t_i-y|)-\log(|x-y|))d\mu_0(x)\Big|\nonumber\\
&\leq& \frac{C\sqrt{R_0(\bm{\phi})}}{n_0-i}\cdot \mu_0([\rho_i,\rho_{i+1}]) =\frac{C\sqrt{R_0(\bm{\phi})}}{(n_0-i)k}.
\end{eqnarray}
Hence for any $i\in [m_0'+1,n_0-1]\cap\mathbb{Z}$, 
\begin{eqnarray}
&& \Big|\int \Big(\int_{\rho_i}^{\rho_{i+1}}(\log(|\rho_i+t_i-y|)-\log(|x-y|))d\mu_0(x) \Big) d\tilde{\kappa}_{4;n_0,\mathbf{t}}(y)\Big| \nonumber\\
&\leq& \frac{C\sqrt{R_0(\bm{\phi})}}{(n_0-i)k}\cdot (k^{-1}|\{j\in [n]: b_j\in (R_0(\bm{\phi}),R_0'(\bm{\phi})]\}|+\mu_0([R_0(\bm{\phi}),R_0'(\bm{\phi})])).\nonumber\\
&& 
\end{eqnarray}
Note that by (\ref{Eq2.2.1}) and property (e) of Proposition \ref{P5.2},
\begin{equation}\label{Eq5.17.3}
 \mu_0([R_0(\bm{\phi}),R_0'(\bm{\phi})])\leq \mu_0([0,R_0'(\bm{\phi})]) \leq\frac{1}{\pi}\int_{0}^{R_0'(\bm{\phi})}\sqrt{x}dx\leq R_0'(\bm{\phi})^{3\slash 2}\leq R_0(\bm{\phi})^{30},
\end{equation}
\begin{eqnarray}\label{Eq5.17.4}
&& |\{j\in [n]: b_j\in (R_0(\bm{\phi}),R_0'(\bm{\phi})]\}|\leq |\{j\in [n]: b_j\leq R_0'(\bm{\phi})\}| \nonumber\\
&=& \tilde{N}(R_0'(\bm{\phi}))\leq |\Psi(R_0'(\bm{\phi}))|+\tilde{N}_0(R_0'(\bm{\phi}))\nonumber\\
&\leq& C \sqrt{\frac{\epsilon}{\log(R_0(\bm{\phi}))}}k+R_0(\bm{\phi})^{30}k\leq CR_0(\bm{\phi})^{30}k.
\end{eqnarray}
Hence for any $i\in [m_0'+1,n_0-1]\cap\mathbb{Z}$, we have
\begin{equation}\label{Eq5.18.2}
\Big|\int \Big(\int_{\rho_i}^{\rho_{i+1}}(\log(|\rho_i+t_i-y|)-\log(|x-y|))d\mu_0(x) \Big) d\tilde{\kappa}_{4;n_0,\mathbf{t}}(y)\Big|\leq \frac{CR_0(\bm{\phi})^{32}}{(n_0-i)k}.
\end{equation}

For any $y\in [R_0(\bm{\phi}),R_0'(\bm{\phi})]$, by (\ref{Eq2.2.1}) and (\ref{Eq5.20}), we have
\begin{equation}
 y-(\rho_{n_0}+t_{n_0})\geq R_0(\bm{\phi})-\rho_{n_0}-n^{-1}\geq \frac{1}{2}(R_0(\bm{\phi})-\rho_{n_0})\geq \frac{c}{k\sqrt{R_0(\bm{\phi})}},
\end{equation}
\begin{equation}
y-(\rho_{n_0}+t_{n_0})\leq y\leq R_0'(\bm{\phi})\leq R_0(\bm{\phi})^{20}.
\end{equation}
Hence by (\ref{Eq5.11}), we have $c k^{-2}\leq y-(\rho_{n_0}+t_{n_0})\leq k^{40}$, and
\begin{equation}
  |\log(|\rho_{n_0}+t_{n_0}-y|)|\leq C\log(k).
\end{equation}
Hence by (\ref{Eq5.16}) and (\ref{Eq5.39}), for any $y\in [R_0(\bm{\phi}),R_0'(\bm{\phi})]$,
\begin{equation}\label{Eq5.17.1}
 \Big|\int_{\rho_{n_0}}^{\rho_{n_0+1}} \log(|\rho_{n_0}+t_{n_0}-y|)d\mu_0(x)\Big|\leq C\log(k)\mu_0([\rho_{n_0},\rho_{n_0+1}])\leq \frac{C\log(k)}{k}.
\end{equation}

Below we consider any $y\in [R_0(\bm{\phi}),R_0'(\bm{\phi})]$. For any $x\in [\rho_{n_0},\rho_{n_0+1})$, we have $0<y-x\leq R_0'(\bm{\phi})$. By (\ref{Eq2.2.1}), $\log(y-x)\leq C\log(R_0(\bm{\phi}))$.
Hence
\begin{eqnarray}
 && |\log(|x-y|)|=|\log(y-x)|=(\log(y-x))_{+}+(\log(y-x))_{-} \nonumber\\
 &=& 2(\log(y-x))_{+}-\log(y-x)\leq C\log(R_0(\bm{\phi}))-\log(y-x),
\end{eqnarray}
where $(a)_{+}:=\max\{a,0\}$ and $(a)_{-}:=-\min\{a,0\}$ for any $a\in\mathbb{R}$. Hence
\begin{eqnarray}
&& \Big|\int_{\rho_{n_0}}^{\rho_{n_0+1}} \log(|x-y|)d\mu_0(x)\Big|\leq \sqrt{R_0(\bm{\phi})}\int_{\rho_{n_0}}^{\rho_{n_0+1}} |\log(|x-y|)| dx\nonumber\\
&\leq&  \sqrt{R_0(\bm{\phi})}\Big(C\log(R_0(\bm{\phi}))(\rho_{n_0+1}-\rho_{n_0})-\int_{\rho_{n_0}}^{\rho_{n_0+1}}\log(y-x)dx \Big).
\end{eqnarray}
By (\ref{Eq5.11}) and (\ref{Eq5.20}), $\log(\rho_{n_0+1}-\rho_{n_0})\geq \log(k^{-1}R_0(\bm{\phi})^{-1\slash 2})\geq -2\log(k)$. Hence by (\ref{Eq5.11}) and (\ref{Eq5.15.1}), 
\begin{equation}
C\log(R_0(\bm{\phi}))(\rho_{n_0+1}-\rho_{n_0})\leq \frac{C\log(k)}{k},
\end{equation}
\begin{eqnarray}
&& -\int_{\rho_{n_0}}^{\rho_{n_0+1}}\log(y-x)dx  \leq -\int_{\rho_{n_0}}^{\rho_{n_0+1}}\log(\rho_{n_0+1}-x)dx\nonumber\\
&=& (\rho_{n_0+1}-\rho_{n_0})(-\log(\rho_{n_0+1}-\rho_{n_0})+1)\leq \frac{C\log(k)}{k}.
\end{eqnarray}
Hence
\begin{equation}\label{Eq5.17.2}
\Big|\int_{\rho_{n_0}}^{\rho_{n_0+1}} \log(|x-y|)d\mu_0(x)\Big|\leq \frac{C\sqrt{R_0(\bm{\phi})}\log(k)}{k}.
\end{equation}

By (\ref{Eq5.17.1}) and (\ref{Eq5.17.2}), for any $y\in [R_0(\bm{\phi}),R_0'(\bm{\phi})]$, 
\begin{equation}
\Big|\int_{\rho_{n_0}}^{\rho_{n_0+1}}(\log(|\rho_{n_0}+t_{n_0}-y|)-\log(|x-y|))d\mu_0(x)\Big|\leq \frac{C\sqrt{R_0(\bm{\phi})}\log(k)}{k}.
\end{equation}
Hence by (\ref{Eq5.17.3}) and (\ref{Eq5.17.4}), we have
\begin{eqnarray}\label{Eq5.18.3}
 && \Big|\int \Big( \int_{\rho_{n_0}}^{\rho_{n_0+1}}(\log(|\rho_{n_0}+t_{n_0}-y|)-\log(|x-y|))d\mu_0(x) \Big) d\tilde{\kappa}_{4;n_0,\mathbf{t}}(y)\Big|\nonumber\\
 &\leq& \frac{C\sqrt{R_0(\bm{\phi})}\log(k)}{k}\cdot (k^{-1}|\{j\in [n]: b_j\in (R_0(\bm{\phi}),R_0'(\bm{\phi})]\}|+\mu_0([R_0(\bm{\phi}),R_0'(\bm{\phi})]))\nonumber\\
 &\leq& \frac{C R_0(\bm{\phi})^{32}\log(k)}{k}.
\end{eqnarray}

Note that by (\ref{Eq5.11}) and (\ref{Eq5.18.5}), 
\begin{equation}
 \sum_{i=m_0'+1}^{n_0-1}\frac{1}{n_0-i}\leq \sum_{i=1}^{n_0}\frac{1}{i}\leq 1+\int_{1}^{n_0}\frac{1}{x}dx=\log(n_0)+1\leq C\log(k).
\end{equation}
Hence by (\ref{Eq5.18.1}), (\ref{Eq5.18.2}), and (\ref{Eq5.18.3}), 
\begin{eqnarray}\label{Eq5.18.6}
 && \Big|\frac{1}{k}\sum_{i=m_0'+1}^{n_0}h(\rho_i+t_i)-\int_{r_0}^{R_0(\bm{\phi})}h(x)d\mu_0(x)\Big| \nonumber\\
 &\leq& \sum_{i=m_0'+1}^{n_0-1} \frac{CR_0(\bm{\phi})^{32}}{(n_0-i)k}+\frac{C R_0(\bm{\phi})^{32}\log(k)}{k}\leq \frac{C R_0(\bm{\phi})^{32}\log(k)}{k}.
\end{eqnarray}

By (\ref{Eq5.18.8}), (\ref{Eq5.18.7}), and (\ref{Eq5.18.6}), we conclude that
\begin{equation}\label{Eq5.19.15}
|\tilde{\Xi}_{3,4;n_0,\mathbf{t}}|\leq C\sqrt{\frac{\epsilon}{\log\log(R_0(\bm{\phi}))}}+ \frac{C R_0(\bm{\phi})^{32}\log(k)}{k}\leq C\sqrt{\epsilon},
\end{equation}
where we use (\ref{Eq5.11}) and the facts that $T$ only depends on $M,\epsilon$ and $k\geq K_0(M,\epsilon,\delta)$ is sufficiently large (depending on $\beta,M,\epsilon,\delta$) for the second inequality. 

\paragraph{Step 5} In this step, we bound $|(\Xi_{2,5}+\Xi_{3,5})-(\tilde{\Xi}_{2,5;n_0,\mathbf{t}}+\tilde{\Xi}_{3,5;n_0,\mathbf{t}})|$. Note that the event $\mathcal{D}(\bm{\phi},n_0)$ holds (recall Definition \ref{Defn5.1}). Let $\alpha_1\leq \alpha_2\leq \cdots\leq \alpha_{n_0}$ be such that $\{\alpha_1,\alpha_2,\cdots,\alpha_{n_0}\}=\{b_i: i\in [n], b_i\in [-R_0(\bm{\phi}),R_0(\bm{\phi})]\}$.  

For any $x\in [-R_0(\bm{\phi}),R_0(\bm{\phi})]$, we let 
\begin{equation}
    g(x):=\int \log(|x-y|)d\kappa_5(y).
\end{equation}
Note that $\tilde{\kappa}_{5;n_0,\mathbf{t}}=\kappa_5$. Hence by (\ref{Eq5.25}),
\begin{equation}\label{Eq5.19.4}
   (\Xi_{2,5}+\Xi_{3,5})-(\tilde{\Xi}_{2,5;n_0,\mathbf{t}}+\tilde{\Xi}_{3,5;n_0,\mathbf{t}})=-\frac{1}{k}\Big(\sum_{i=1}^{n_0}g(\alpha_i)-\sum_{i=1}^{n_0}g(\rho_i+t_i)\Big).
\end{equation}

Note that for any $x\in [-R_0(\bm{\phi}),R_0(\bm{\phi})]$, 
\begin{equation}
    g'(x)=\int \frac{1}{x-y} d\kappa_5(y).
\end{equation}
Hence as the event $\mathcal{C}(\bm{\phi};\Lambda_0')$ holds, by (\ref{Eq5.19.2}), we have 
\begin{equation}\label{Eq5.19.3}
  \sup_{x\in [-R_0(\bm{\phi}),R_0(\bm{\phi})]}|g'(x)|\leq CMR_0'(\bm{\phi})^{-1\slash 2}. 
\end{equation}

For any $i\in [n_0]$, there exists $\tilde{\alpha}_i\in [-R_0(\bm{\phi}),R_0(\bm{\phi})]$, such that 
\begin{equation*}
 g(\alpha_i)-g(\rho_i+t_i)=g'(\tilde{\alpha}_i)(\alpha_i-(\rho_i+t_i)).
\end{equation*}
Hence by (\ref{Eq5.19.3}),
\begin{eqnarray}\label{Eq5.19.7}
&&  |g(\alpha_i)-g(\rho_i+t_i)|\leq CM R_0'(\bm{\phi})^{-1\slash 2}|\alpha_i-(\rho_i+t_i)|\nonumber\\
&\leq& CM R_0'(\bm{\phi})^{-1\slash 2}R_0(\bm{\phi})\leq CM R_0(\bm{\phi})^{-4}.
\end{eqnarray}

By (\ref{Eq5.18.5}), (\ref{Eq5.19.4}), and (\ref{Eq5.19.7}), we have 
\begin{equation}\label{Eq5.19.16}
 |(\Xi_{2,5}+\Xi_{3,5})-(\tilde{\Xi}_{2,5;n_0,\mathbf{t}}+\tilde{\Xi}_{3,5;n_0,\mathbf{t}})|\leq C k^{-1}n_0\cdot M R_0(\bm{\phi})^{-4} \leq CMR_0(\bm{\phi})^{-5\slash 2}. 
\end{equation}

\bigskip

By (\ref{Eq5.19.8}), (\ref{Eq5.19.9}), (\ref{Eq5.19.10}), (\ref{Eq5.19.11}), (\ref{Eq5.19.12}), (\ref{Eq5.19.13}), (\ref{Eq5.19.14}), (\ref{Eq5.19.15}), and (\ref{Eq5.19.16}), 
\begin{eqnarray}
&&  -\int \log(|x-y|) d\Upsilon_1(x)d\Upsilon_2(y)+\int \log(|x-y|) d\tilde{\Upsilon}_{1;n_0,\mathbf{t}}(x)d\tilde{\Upsilon}_{2;n_0,\mathbf{t}}(y) \nonumber\\
& \geq &  -CM^2 R_0(\bm{\phi})^{-1\slash 3}-CM\sqrt{\epsilon}-\frac{CM\log(R_0(\bm{\phi}))}{R_0(\bm{\phi})^{3\slash 2}}.
\end{eqnarray}
By (\ref{Eq2.1.10}), as $M\geq M_0$ is sufficiently large and $\epsilon\in (0,\epsilon_0)$ is sufficiently small, 
\begin{equation*}
 \frac{M\log(R_0(\bm{\phi}))}{R_0(\bm{\phi})^{3\slash 2}}\leq M^2 R_0(\bm{\phi})^{-1\slash 3}\leq\frac{M^2}{10^{M^2\slash (3\epsilon)}}\leq \frac{M^2}{(M^2\slash (3\epsilon)}\leq C\epsilon.
\end{equation*}
Hence we conclude that
\begin{equation}
-\int \log(|x-y|) d\Upsilon_1(x)d\Upsilon_2(y)+\int \log(|x-y|) d\tilde{\Upsilon}_{1;n_0,\mathbf{t}}(x)d\tilde{\Upsilon}_{2;n_0,\mathbf{t}}(y)\geq -CM\sqrt{\epsilon}.
\end{equation}

\end{proof}

\subsubsection{Proof of Proposition \ref{P5.4}}

\begin{proof}[Proof of Proposition \ref{P5.4}]

Throughout the proof, we fix any $R\in [10,R_0(\bm{\phi})-10]$, and assume that the event $\mathcal{B}(\bm{\phi};\Lambda_0)\cap \mathcal{D}(\bm{\phi},n_0)\cap\mathcal{C}(\bm{\phi};\Lambda_0')\cap  (\mathscr{G}_R)^c$ holds. 

By the definition of $\Upsilon_1$ in (\ref{Eq5.21.2}), we have 
\begin{eqnarray}\label{Eq5.22.1}
&&-\int_{\mathbb{R}^2\backslash \Delta} \log(|x-y|) d\Upsilon_1(x)d\Upsilon_1(y)  =  -\frac{2}{k^2}\sum_{\substack{1\leq i<j\leq n:\\ b_i,b_j\in [-R_0(\bm{\phi}),R_0(\bm{\phi})]}}\log(|b_i-b_j|)\nonumber\\
&&\quad\quad\quad\quad\quad\quad\quad\quad\quad\quad+\frac{2}{k}\sum_{\substack{i\in [n]:\\ b_i\in [-R_0(\bm{\phi}),R_0(\bm{\phi})]}}\int_{-R_0(\bm{\phi})}^{R_0(\bm{\phi})}\log(|b_i-x|)d\mu_0(x)\nonumber\\
&&\quad\quad\quad\quad\quad\quad\quad\quad\quad\quad-\int_{-R_0(\bm{\phi})}^{R_0(\bm{\phi})} \int_{-R_0(\bm{\phi})}^{R_0(\bm{\phi})} \log(|x-y|)d\mu_0(x)d\mu_0(y).
\end{eqnarray}

Let $h(x):=x\log(x)-x$ for any $x>0$ and $h(0):=0$. As $h(x) \geq -1$ for any $x\in [0,1]$, for any $z_1,z_2\in [0,1]$, we have
\begin{equation*}
 \int_{x-z_1}^{x+z_2}\log(|x-y|)dy=h(z_1)+h(z_2)\geq -2.
\end{equation*}
For any $x\in [-R_0(\bm{\phi}),R_0(\bm{\phi})]$, we have
\begin{eqnarray}
&&  \int_{-R_0(\bm{\phi})}^{R_0(\bm{\phi})}\log(|x-y|)d\mu_0(y)\geq \int_{[-R_0(\bm{\phi}),R_0(\bm{\phi})]\cap [x-1,x+1]}\log(|x-y|)d\mu_0(y) \nonumber\\
&\geq& \frac{\sqrt{R_0(\bm{\phi})}}{\pi}\int_{[-R_0(\bm{\phi}),R_0(\bm{\phi})]\cap [x-1,x+1]}\log(|x-y|)dy\nonumber\\
&\geq& -\frac{2\sqrt{R_0(\bm{\phi})}}{\pi}\geq -\sqrt{R_0(\bm{\phi})}.
\end{eqnarray}
Hence by Definition \ref{Defn5.1} and (\ref{Eq5.18.5}), 
\begin{equation}\label{Eq5.22.2}
\frac{2}{k}\sum_{\substack{i\in [n]:\\ b_i\in [-R_0(\bm{\phi}),R_0(\bm{\phi})]}}\int_{-R_0(\bm{\phi})}^{R_0(\bm{\phi})}\log(|b_i-x|)d\mu_0(x)\geq -\frac{2}{k}\cdot \sqrt{R_0(\bm{\phi})} n_0\geq -C R_0(\bm{\phi})^2.
\end{equation}

As $\mu_0([-R_0(\bm{\phi}),R_0(\bm{\phi})])\leq \pi^{-1}\int_0^{R_0(\bm{\phi})}\sqrt{x}dx\leq R_0(\bm{\phi})^{3\slash 2}$, by (\ref{Eq5.11}),
\begin{eqnarray}\label{Eq5.22.3}
&& \int_{-R_0(\bm{\phi})}^{R_0(\bm{\phi})} \int_{-R_0(\bm{\phi})}^{R_0(\bm{\phi})} \log(|x-y|)d\mu_0(x)d\mu_0(y)\nonumber\\
&\leq& \log(2R_0(\bm{\phi}))\mu_0([-R_0(\bm{\phi}),R_0(\bm{\phi})])^2 
\leq CR_0(\bm{\phi})^3\log(R_0(\bm{\phi})).
\end{eqnarray}

Moreover, by (\ref{Eq5.18.5}),
\begin{equation}\label{Eq5.22.4}
\sum_{\substack{1\leq i<j\leq n:\\ b_i,b_j\in [-R_0(\bm{\phi}),R_0(\bm{\phi})],\\|b_i-b_j|>1}} \log(|b_i-b_j|)\leq \log(2R_0(\bm{\phi})) n_0^2\leq CR_0(\bm{\phi})^3\log(R_0(\bm{\phi})) k^2.  
\end{equation}
Let 
\begin{eqnarray}
 &&  n_1:=|\{i\in [n]: b_i\in [R-n^{-1\slash 30},R+n^{-1\slash 30}]\}|,\nonumber\\
 &&  n_2:=|\{i\in [n]: b_i\in [-R-n^{-1\slash 30},-R+n^{-1\slash 30}]\}|.
\end{eqnarray}
As the event $(\mathscr{G}_R)^c$ holds and $k$ is sufficiently large, by Definition \ref{Defn5.3.n}, 
\begin{equation}
  n_1+n_2\geq \frac{k}{\log(n)^{1\slash 4}}\geq \frac{ck}{\log(k)^{1\slash 4}}\geq 10.
\end{equation}
Hence
\begin{eqnarray}
&&\frac{n_1(n_1-1)}{2}+\frac{n_2(n_2-1)}{2}=\frac{n_1^2+n_2^2}{2}-\frac{n_1+n_2}{2}\nonumber\\
&\geq & \Big(\frac{n_1+n_2}{2}\Big)^2-\frac{n_1+n_2}{2}\geq \frac{(n_1+n_2)^2}{8}\geq \frac{k^2}{8\sqrt{\log(n)}}.
\end{eqnarray}
Note that $\log(2n^{-1\slash 30})=\log(2)-\log(n)\slash 30\leq -\log(n)\slash 50$. Hence we have
\begin{eqnarray}\label{Eq5.22.5}
&& \sum_{\substack{1\leq i<j\leq n:\\ b_i,b_j\in [-R_0(\bm{\phi}),R_0(\bm{\phi})],\\|b_i-b_j|\leq 1}} \log(|b_i-b_j|)\nonumber\\
&\leq& \sum_{\substack{1\leq i<j\leq n:\\ b_i,b_j\in [R-n^{-1\slash 30},R+n^{-1\slash 30}],\\\text{ or }b_i,b_j\in [-R-n^{-1\slash 30},-R+n^{-1\slash 30}]}}\log(|b_i-b_j|)\nonumber\\
&\leq& \log(2n^{-1\slash 30})\Big(\frac{n_1(n_1-1)}{2}+\frac{n_2(n_2-1)}{2}\Big)\nonumber\\
&\leq &  -\frac{1}{50}\log(n)\cdot \frac{k^2}{8\sqrt{\log(n)}}\leq -ck^2\sqrt{\log(n)}.
\end{eqnarray}

By (\ref{Eq5.22.1}), (\ref{Eq5.22.2})-(\ref{Eq5.22.4}), and (\ref{Eq5.22.5}), we have
\begin{equation}
 -\int_{\mathbb{R}^2\backslash \Delta} \log(|x-y|) d\Upsilon_1(x)d\Upsilon_1(y)\geq c\sqrt{\log(n)}-CR_0(\bm{\phi})^3\log(R_0(\bm{\phi})).
\end{equation}
By (\ref{Eq5.11}) and (\ref{Eq5.13}), as $M\geq M_0$ is sufficiently large, we have 
\begin{equation*}
    R_0(\bm{\phi})^3\log(R_0(\bm{\phi}))\leq R_0(\bm{\phi})^{10}\leq 2^T. 
\end{equation*}
Hence 
\begin{equation}
     -\int_{\mathbb{R}^2\backslash \Delta} \log(|x-y|) d\Upsilon_1(x)d\Upsilon_1(y)\geq c\sqrt{\log(n)}-C\cdot 2^T.
\end{equation}
For any $x\in [0,R_0(\bm{\phi})]$, by (\ref{Eq5.11}), $2-(k\slash n)^{2\slash 3}x\in [-2,2]$. Hence by (\ref{Eq7.6.1}) and \cite[Lemma 6.1]{Zho}, $\tilde{\xi}(x)=0,\forall x\in [0,R_0(\bm{\phi})]$ and $\tilde{\xi}(x)\geq 0,  \forall x\in\mathbb{R}$. Hence 
\begin{equation}
    \int \tilde{\xi}(x) d\Upsilon_1(x)=\frac{1}{k}\sum_{\substack{i\in [n]:\\ b_i\in [-R_0(\bm{\phi}),R_0(\bm{\phi})]}}\tilde{\xi}(b_i)\geq 0. 
\end{equation}
Hence we conclude that
\begin{equation}
    -\int_{\mathbb{R}^2\backslash \Delta} \log(|x-y|) d\Upsilon_1(x)d\Upsilon_1(y)+2 \int \tilde{\xi}(x) d\Upsilon_1(x) \geq c\sqrt{\log(n)}-C\cdot 2^T.
\end{equation}

\end{proof}

\subsubsection{Proof of Proposition \ref{P5.5}}

\begin{proof}[Proof of Proposition \ref{P5.5}]

In the following, we fix any $R\in [10,R_0(\bm{\phi})-10]$, and assume that $\mathcal{B}(\bm{\phi};\Lambda_0)\cap \mathcal{D}(\bm{\phi},n_0)\cap\mathcal{C}(\bm{\phi};\Lambda_0')\cap \mathscr{G}_R \cap \mathscr{E}\cap \{d_R(\nu_{k;R},\mu)\leq \delta\}$ holds.

In the following, we consider any function $f:[-R,R]\rightarrow \mathbb{R}$ such that $\max\{\|f\|_{\infty},\|f\|_{Lip}\}\leq 1$. By (\ref{Eq5.11}), $R\leq R_0(\bm{\phi})\leq 2^{T\slash 10}$. As $T$ only depends on $M,\epsilon$ and $k\geq K_0(M,\epsilon,\delta)$ is sufficiently large (depending on $\beta,M,\epsilon,\delta$), we have $R\leq 4(n\slash k)^{2\slash 3}$. Hence
\begin{eqnarray}\label{Eq5.27.2}
 &&\Big|\int_{[-R,R]}fd\mu_0-\int_{[-R,R]}fd\nu_0\Big|\leq\frac{1}{\pi}\int_{0}^R |f(x)|\sqrt{x}\Big(1-\sqrt{1-\frac{1}{4}\Big(\frac{k}{n}\Big)^{2\slash 3}x}\Big)dx\nonumber\\
 &&\leq \frac{1}{4\pi}\Big(\frac{k}{n}\Big)^{2\slash 3}\int_0^R x^{3\slash 2}dx\leq \Big(\frac{k}{n}\Big)^{2\slash 3} R^{5\slash 2}\leq 2^T k^{2\slash 3} n^{-2\slash 3}.
\end{eqnarray}

Let 
\begin{equation}
 \mathscr{I}_0:=\{i\in [n]: b_i\in [-R,R], k^{-2\slash 3}\lambda_i \in [-R,R]\},
\end{equation}
\begin{equation}
  \mathscr{I}_1:=\{i\in [n]:b_i\in [-R,R], k^{-2\slash 3}\lambda_i\notin[-R,R]\},
\end{equation}
\begin{eqnarray}
 \mathscr{I}_2&:=&\{i\in\mathbb{N}^{*}:\text{ either }i\geq n+1\text{ and }k^{-2\slash 3}\lambda_i\in [-R,R],\nonumber\\
 &&\quad \text{ or }i\in [n], b_i\notin [-R,R], \text{ and }k^{-2\slash 3}\lambda_i\in [-R,R]\}. 
\end{eqnarray}
Note that
\begin{eqnarray}\label{Eq5.26.1}
&& \Big|\int_{[-R,R]}fd\Big(\frac{1}{k}\sum_{i=1}^{\infty}\delta_{k^{-2\slash 3}\lambda_i}\Big)-\int_{[-R,R]}fd\Big(\frac{1}{k}\sum_{i=1}^{n}\delta_{b_i}\Big)\Big| \nonumber\\
&=& \frac{1}{k}\Big|\sum_{\substack{i\in\mathbb{N}^{*}:\\k^{-2\slash 3}\lambda_i\in [-R,R]}}f(k^{-2\slash 3}\lambda_i)-\sum_{\substack{i\in [n]:\\ b_i\in [-R,R]}}f(b_i)\Big| \nonumber\\
&=& \frac{1}{k}\Big|\sum_{i\in \mathscr{I}_0\cup\mathscr{I}_2}f(k^{-2\slash 3}\lambda_i)-\sum_{i\in\mathscr{I}_0\cup\mathscr{I}_1}f(b_i)\Big|\nonumber\\
&\leq& \frac{1}{k}\Big(\sum_{i\in\mathscr{I}_0}|f(k^{-2\slash 3}\lambda_i)-f(b_i)|+\sum_{i\in\mathscr{I}_1}|f(b_i)|+\sum_{i\in\mathscr{I}_2}|f(k^{-2\slash 3}\lambda_i)|\Big)\nonumber\\
&\leq& \frac{1}{k}\Big(\sum_{i\in\mathscr{I}_0}|k^{-2\slash 3}\lambda_i-b_i|+|\mathscr{I}_1|+|\mathscr{I}_2|\Big).
\end{eqnarray}

By (\ref{Eq5.11}) and property (c) of Proposition \ref{P5.2},
\begin{equation}
    \tilde{N}(-R_0(\bm{\phi}))\leq \frac{CMk}{R_0(\bm{\phi})^{3\slash 2}}.
\end{equation}
Note that by (\ref{Eq5.13}), $R_0(\bm{\phi})\geq M$. Hence by (\ref{Eq5.18.5}), 
\begin{equation}\label{Eq5.24.4}
    \tilde{N}(R_0(\bm{\phi}))\leq \tilde{N}(-R_0(\bm{\phi}))+n_0\leq \frac{CMk}{R_0(\bm{\phi})^{3\slash 2}}+C R_0(\bm{\phi})^{3\slash 2} k\leq C R_0(\bm{\phi})^{3\slash 2} k. 
\end{equation}
As $R_0(\bm{\phi})\geq M\geq M_0$ is sufficiently large, by (\ref{Eq5.11}), we have 
\begin{equation}\label{Eq5.23.2}
    \tilde{N}(R_0(\bm{\phi}))\leq C R_0(\bm{\phi})^{3\slash 2} k\leq R_0(\bm{\phi})^{10}k\leq 2^T k.
\end{equation}

For any $i\in [n]$ such that $b_i\in [-R,R]$, we have $b_i\leq R\leq R_0(\bm{\phi})$, hence by (\ref{Eq5.23.2}), $i\leq \tilde{N}(R_0(\bm{\phi}))\leq 2^Tk$. As the event $\mathscr{E}$ holds, by (\ref{Eq5.23.1}), we have 
\begin{equation}\label{Eq5.24.1}
    |b_i-k^{-2\slash 3}\lambda_i|\leq C_0 k^{-2\slash 3} n^{-1\slash 24}.
\end{equation}
Moreover, by (\ref{Eq5.11}) and (\ref{Eq5.18.5}),
\begin{equation}
    |\mathscr{I}_0|\leq |\{i\in [n]: b_i\in [-R_0(\bm{\phi}),R_0(\bm{\phi})]\}|=n_0\leq CR_0(\bm{\phi})^{3\slash 2}k\leq C\cdot 2^T k.
\end{equation}
Hence 
\begin{equation}\label{Eq5.26.2}
    \sum_{i\in\mathscr{I}_0}|k^{-2\slash 3}\lambda_i-b_i|\leq C_0 k^{-2\slash 3} n^{-1\slash 24}|\mathscr{I}_0|\leq C\cdot 2^T k^{1\slash 3}n^{-1\slash 24}.
\end{equation}

We bound $|\mathscr{I}_1|$ as follows. Consider any $i\in\mathscr{I}_1$. If $k^{-2\slash 3}\lambda_i>R$, by (\ref{Eq5.24.1}), 
\begin{equation}
    b_i\geq k^{-2\slash 3}\lambda_i-C_0 k^{-2\slash 3} n^{-1\slash 24}>R-C_0 k^{-2\slash 3} n^{-1\slash 24}\geq R-n^{-1\slash 30}.
\end{equation}
If $k^{-2\slash 3}\lambda_i<-R$, by (\ref{Eq5.24.1}),
\begin{equation}
    b_i\leq k^{-2\slash 3}\lambda_i+C_0 k^{-2\slash 3} n^{-1\slash 24}< -R+C_0 k^{-2\slash 3} n^{-1\slash 24}\leq -R+n^{-1\slash 30}.
\end{equation}
Noting that $b_i\in [-R,R]$, we have $b_i\in [-R,-R+n^{-1\slash 30}]\cup [R-n^{-1\slash 30},R]$. As the event $\mathscr{G}_R$ holds, by (\ref{Eq5.24.2}), we have
\begin{equation}\label{Eq5.26.3}
    |\mathscr{I}_1|\leq \frac{k}{\log(n)^{1\slash 4}}.
\end{equation}

We bound $|\mathscr{I}_2|$ as follows. Consider any $i\in\mathscr{I}_2$. Let $C$ be the constant on the right-hand side of (\ref{Eq5.24.4}). By (\ref{Eq5.11}), $CR_0(\bm{\phi})^{3\slash 2}k\leq C\cdot 2^T k\leq n-1$. Hence $b_{\lfloor CR_0(\bm{\phi})^{3\slash 2} k\rfloor+1}>R_0(\bm{\phi})$. As $R_0(\bm{\phi})\geq M\geq M_0$ is sufficiently large, by (\ref{Eq5.11}), we have $\lfloor CR_0(\bm{\phi})^{3\slash 2} k\rfloor+1\leq R_0(\bm{\phi})^{10} k\leq 2^T k$. As the event $\mathscr{E}$ holds, by (\ref{Eq5.23.1}),
\begin{equation}
    k^{-2\slash 3}\lambda_{\lfloor CR_0(\bm{\phi})^{3\slash 2} k\rfloor+1} \geq b_{\lfloor CR_0(\bm{\phi})^{3\slash 2} k\rfloor+1}-C_0 k^{-2\slash 3} n^{-1\slash 24}> R_0(\bm{\phi})-1>R.
\end{equation}
Hence $i\leq \lfloor CR_0(\bm{\phi})^{3\slash 2} k\rfloor\leq 2^T k\leq n$. By (\ref{Eq5.23.1}), 
\begin{equation}
    b_i\geq k^{-2\slash 3}\lambda_i-C_0 k^{-2\slash 3} n^{-1\slash 24}\geq -R-n^{-1\slash 30},
\end{equation}
\begin{equation}
    b_i\leq k^{-2\slash 3}\lambda_i+C_0 k^{-2\slash 3} n^{-1\slash 24}\leq R+n^{-1\slash 30}.
\end{equation}
Hence $b_i\in [-R-n^{-1\slash 30},-R]\cup [R,R+n^{-1\slash 30}]$. As the event $\mathscr{G}_R$ holds, by (\ref{Eq5.24.2}), we have
\begin{equation}\label{Eq5.26.4}
    |\mathscr{I}_2|\leq \frac{k}{\log(n)^{1\slash 4}}.
\end{equation}

By (\ref{Eq5.26.1}), (\ref{Eq5.26.2}), (\ref{Eq5.26.3}), and (\ref{Eq5.26.4}), we have
\begin{eqnarray}\label{Eq5.27.3}
   && \Big|\int_{[-R,R]}fd\Big(\frac{1}{k}\sum_{i=1}^{\infty}\delta_{k^{-2\slash 3}\lambda_i}\Big)-\int_{[-R,R]}fd\Big(\frac{1}{k}\sum_{i=1}^{n}\delta_{b_i}\Big)\Big|\nonumber\\
   &\leq& C\cdot 2^T k^{-2\slash 3} n^{-1\slash 24}+\frac{2}{\log(n)^{1\slash 4}}\leq \frac{C}{\log(n)^{1\slash 4}}.  
\end{eqnarray}
By (\ref{Eq5.27.2}) and (\ref{Eq5.27.3}), we have
\begin{equation}
    \Big|\int_{[-R,R]}fd\nu_{k;R}-\int_{[-R,R]}fd\mu_{n,k;R}\Big|\leq \frac{C}{\log(n)^{1\slash 4}},
\end{equation}
hence $d_R(\nu_{k;R},\mu_{n,k;R})\leq C\log(n)^{-1\slash 4}\leq \delta$. As $d_R(\nu_{k;R},\mu)\leq \delta$, we have 
\begin{equation}
    d_R(\mu_{n,k;R},\mu)\leq d_R(\nu_{k;R},\mu_{n,k;R})+d_R(\nu_{k;R},\mu)\leq 2\delta.
\end{equation}
Hence the event $\mathcal{B}(\bm{\phi};\Lambda_0)\cap \mathcal{D}(\bm{\phi},n_0)\cap\mathcal{C}(\bm{\phi};\Lambda_0')\cap \mathscr{G}_R\cap \{d_R(\mu_{n,k;R},\mu)\leq 2\delta\}$ holds, and we conclude that (\ref{Eq5.28.1}) holds. 

\end{proof}

\subsubsection{Proof of Proposition \ref{P5.6}}

Before the proof of Proposition \ref{P5.6}, we present the following lemma.

\begin{lemma}\label{L5.1}
There exists a positive absolute constant $C$, such that for any $t_1,t_2\in\mathbb{R}$ such that $t_1\leq t_2$, we have
\begin{eqnarray}\label{Eq6.1.5}
&& \int_{t_1-n^{-1}}^{t_1+n^{-1}}\int_{t_2-n^{-1}}^{t_2+n^{-1}} \log(|x-y|)dxdy\nonumber\\
&\geq& \frac{4}{n^2}\log(t_2-t_1)-\frac{C}{n^2\max\{n(t_2-t_1),1\}}.
\end{eqnarray}
Moreover, for any $t_1,t_2\in\mathbb{R}$ such that $0\leq t_2-t_1\leq 10n^{-1}$, we have
\begin{equation}\label{Eq6.1.4}
    \int_{t_1-n^{-1}}^{t_1+n^{-1}}\int_{t_2-n^{-1}}^{t_2+n^{-1}} \log(|x-y|)dxdy\geq -\frac{4\log(n)}{n^2}-\frac{C}{n^2}.
\end{equation}
\end{lemma}
\begin{proof}

In the following, we assume that $0\log(0):=0$.

By a change of variables, we obtain that
\begin{equation}\label{Eq6.1.1}
    \int_{t_1-n^{-1}}^{t_1+n^{-1}}\int_{t_2-n^{-1}}^{t_2+n^{-1}} \log(|x-y|)dxdy=\int_{-n^{-1}}^{n^{-1}}dy \int_{-n^{-1}}^{n^{-1}} \log(|t_1-t_2+x-y|)dx.
\end{equation}
By integration by parts, for any $y\in [-n^{-1},n^{-1}]$, 
\begin{eqnarray}
    &&\int_{-n^{-1}}^{n^{-1}} \log(|t_1-t_2+x-y|)dx= (t_1-t_2-y+n^{-1})\log(|t_1-t_2-y+n^{-1}|)\nonumber\\
    &&\quad\quad\quad\quad\quad\quad\quad\quad\quad\quad
    -(t_1-t_2-y-n^{-1})\log(|t_1-t_2-y-n^{-1}|)-2n^{-1}.\nonumber\\
    &&
\end{eqnarray}
By integration by parts, 
\begin{eqnarray}
  &&  \int_{-n^{-1}}^{n^{-1}} (t_1-t_2-y+n^{-1})\log(|t_1-t_2-y+n^{-1}|)dy  \nonumber\\
  &=& \int_{t_1-t_2}^{t_1-t_2+2n^{-1}} s\log(|s|)ds\nonumber\\
  &=&\frac{1}{2}(t_1-t_2+2n^{-1})^2\log(|t_1-t_2+2n^{-1}|)
  -\frac{1}{2}(t_1-t_2)^2\log(|t_1-t_2|)\nonumber\\
  &&  -\frac{1}{4}(t_1-t_2+2n^{-1})^2+\frac{1}{4}(t_1-t_2)^2,
\end{eqnarray}
\begin{eqnarray}\label{Eq6.1.2}
   &&  \int_{-n^{-1}}^{n^{-1}} (t_1-t_2-y-n^{-1})\log(|t_1-t_2-y-n^{-1}|)dy  \nonumber\\
  &=& \int_{t_1-t_2-2n^{-1}}^{t_1-t_2} s\log(|s|)ds\nonumber\\
  &=&\frac{1}{2}(t_1-t_2)^2\log(|t_1-t_2|)
  -\frac{1}{2}(t_1-t_2-2n^{-1})^2\log(|t_1-t_2-2n^{-1}|)\nonumber\\
  &&  -\frac{1}{4}(t_1-t_2)^2+\frac{1}{4}(t_1-t_2-2n^{-1})^2.
\end{eqnarray}
By (\ref{Eq6.1.1})-(\ref{Eq6.1.2}), we have
\begin{eqnarray}\label{Eq6.1.3}
 &&  \int_{t_1-n^{-1}}^{t_1+n^{-1}}\int_{t_2-n^{-1}}^{t_2+n^{-1}} \log(|x-y|)dxdy\nonumber\\
 &=& \frac{1}{2}(t_1-t_2+2n^{-1})^2\log(|t_1-t_2+2n^{-1}|)
  -(t_1-t_2)^2\log(|t_1-t_2|) \nonumber \\
  &&+\frac{1}{2}(t_1-t_2-2n^{-1})^2\log(|t_1-t_2-2n^{-1}|)-6n^{-2}.
\end{eqnarray}

In the following, we consider two cases.

\paragraph{Case 1: $0\leq t_2-t_1\leq 10 n^{-1}$}

Let $h(x)=x^2\log(x)$ for any $x\geq 0$. Note that $h'(x)=x(2\log(x)+1)$ for any $x\geq 0$. Hence $h'(x)\leq 0$ for $x\in [0,e^{-1\slash 2}]$ and $h'(x)\geq 0$ for $x\geq e^{-1\slash 2}$. Thus for any $x\geq 0$, $h(x)\geq h(e^{-1\slash 2})=-e^{-1}\slash 2\geq -1$. Hence we have
\begin{eqnarray}
  && \frac{1}{2}(t_1-t_2+2n^{-1})^2\log(|t_1-t_2+2n^{-1}|) \nonumber\\
  &=&  \frac{1}{2n^2} h(n|t_1-t_2+2n^{-1}|)-\frac{1}{2}(t_1-t_2+2n^{-1})^2\log(n)\nonumber\\
  &\geq& -\frac{1}{2n^2}-\frac{1}{2}(t_1-t_2+2n^{-1})^2\log(n).
\end{eqnarray}
Note that $\log(|t_1-t_2-2n^{-1}|)=\log(t_2-t_1+2n^{-1})\geq \log(n^{-1})$ and\\ $\log(|t_1-t_2|)=\log(t_2-t_1)\leq \log(10 n^{-1})$. Hence by (\ref{Eq6.1.3}), we have
\begin{eqnarray}\label{Eq6.1.7}
   &&  \int_{t_1-n^{-1}}^{t_1+n^{-1}}\int_{t_2-n^{-1}}^{t_2+n^{-1}} \log(|x-y|)dxdy\nonumber\\
   &\geq& -\frac{1}{2n^2}-\frac{1}{2}(t_2-t_1-2n^{-1})^2\log(n)-\log(10)(t_2-t_1)^2+(t_2-t_1)^2\log(n)\nonumber\\
   && -\frac{1}{2}(t_2-t_1+2n^{-1})^2\log(n)-6n^{-2}\nonumber\\
   &\geq& -4n^{-2}\log(n)-C n^{-2}\geq 4n^{-2}\log(t_2-t_1)-C n^{-2},
\end{eqnarray}
which implies (\ref{Eq6.1.4}).

\paragraph{Case 2: $t_2-t_1 > 10 n^{-1}$}

Consider any $x\in [t_1-n^{-1},t_1+n^{-1}]$ and any $y\in [t_2-n^{-1},t_2+n^{-1}]$. We have 
\begin{equation}
    y-x \geq t_2-t_1-\frac{2}{n} =  (t_2-t_1)\Big(1-\frac{2}{n(t_2-t_1)}\Big)>0.
\end{equation}
As $\log(1-x)\geq -x\slash (1-x)\geq -2x$ for any $x\in [0,1\slash 2]$, we have
\begin{equation}
    \log(y-x)\geq \log(t_2-t_1)+\log\Big(1-\frac{2}{n(t_2-t_1)}\Big)\geq \log(t_2-t_1)-\frac{4}{n(t_2-t_1)}.
\end{equation}

Hence 
\begin{eqnarray}\label{Eq6.1.6}
   \int_{t_1-n^{-1}}^{t_1+n^{-1}}\int_{t_2-n^{-1}}^{t_2+n^{-1}} \log(|x-y|)dxdy\geq \frac{4}{n^2}\log(t_2-t_1)-\frac{16}{n^3(t_2-t_1)}. 
\end{eqnarray}
Now (\ref{Eq6.1.5}) follows from (\ref{Eq6.1.7}) and (\ref{Eq6.1.6}).

\end{proof}

We present the proof of Proposition \ref{P5.6} as follows.

\begin{proof}[Proof of Proposition \ref{P5.6}]

Throughout the proof, we fix any $R\in [10,R_0(\bm{\phi})-10]$ and any $\mu\in \mathcal{X}_R$, and assume that the event
\begin{equation*}
    \mathcal{B}(\bm{\phi};\Lambda_0)\cap \mathcal{D}(\bm{\phi},n_0)\cap\mathcal{C}(\bm{\phi};\Lambda_0')\cap \mathscr{G}_R \cap \{d_R(\mu_{n,k;R},\mu)\leq 2\delta\}
\end{equation*}
holds.


Noting Definition \ref{Defn5.1}, we let $\alpha_1\leq \alpha_2\leq \cdots\leq \alpha_{n_0}$ be such that
\begin{equation}
\{\alpha_1,\alpha_2,\cdots,\alpha_{n_0}\}=\{b_i: i\in [n], b_i\in [-R_0(\bm{\phi}),R_0(\bm{\phi})]\}.
\end{equation}
For every $i\in [n_0]$, we let $\pi_i$ be the Borel probability measure on $\mathbb{R}$ such that 
\begin{equation}\label{Eq6.1.9}
  \pi_i(A)=\frac{n}{2}\cdot Leb(A\cap [\alpha_i-n^{-1},\alpha_i+n^{-1}]) \text{ for any } A\in\mathcal{B}_{\mathbb{R}}.
\end{equation}
We also let $\mu',\mu_1',\mu_2'$ be finite signed Borel measures on $\mathbb{R}$ such that for any $A\in\mathcal{B}_{\mathbb{R}}$,
\begin{eqnarray}\label{Eq7.3.1}
 \mu'(A)&=& -\mu_{n,k}([-R_0(\bm{\phi}),R_0(\bm{\phi})])Leb(A\cap [R_0(\bm{\phi})+2,R_0(\bm{\phi})+3])\nonumber\\
 && +\frac{1}{k}\sum_{i=1}^{n_0}\pi_i(A)-\mu_0(A\cap [0,R_0(\bm{\phi})]).
\end{eqnarray}
\begin{equation}\label{Eq6.6.1}
 \mu'_1(A)= \frac{1}{k}\sum_{i=1}^{n_0}\pi_i(A)-\mu_0(A\cap [0,R_0(\bm{\phi})]).
\end{equation}
\begin{equation}\label{Eq6.4.1}
 \mu'_2(A)= -\mu_{n,k}([-R_0(\bm{\phi}),R_0(\bm{\phi})])Leb(A\cap [R_0(\bm{\phi})+2,R_0(\bm{\phi})+3]).
\end{equation}
Note that
\begin{equation}\label{Eq7.2.1}
    \mu'(\mathbb{R})=0, \quad \supp \mu'\subseteq [-R_0(\bm{\phi})-1,R_0(\bm{\phi})+3], \quad \mu'=\mu'_1+\mu'_2.
\end{equation}
Moreover, for any $A\in\mathcal{B}_{\mathbb{R}}$, we have
\begin{eqnarray}\label{Eq6.1}
&& (\mu'+\nu_0)(A) \geq -\mu_{n,k}([-R_0(\bm{\phi}),R_0(\bm{\phi})])Leb(A\cap [R_0(\bm{\phi})+2,R_0(\bm{\phi})+3])\nonumber\\
&&\quad\quad\quad\quad\quad\quad\quad
+\frac{1}{k}\sum_{i=1}^{n_0}\pi_i(A)+\nu_0(A\cap [R_0(\bm{\phi})+2,R_0(\bm{\phi})+3])\nonumber\\
&\geq& \int_{A\cap [R_0(\bm{\phi})+2,R_0(\bm{\phi})+3]}\Big(\frac{1}{\pi}\sqrt{x}-\mu_{n,k}([-R_0(\bm{\phi}),R_0(\bm{\phi})])\Big)dx.
\end{eqnarray}
Note that 
\begin{eqnarray*}
&&\mu_{n,k}([-R_0(\bm{\phi}),R_0(\bm{\phi})])\nonumber\\
&=& \frac{1}{k}|\{i\in [n]: b_i\in [-R_0(\bm{\phi}),R_0(\bm{\phi})]\}|-\mu_0([-R_0(\bm{\phi}),R_0(\bm{\phi})])\nonumber\\
&=& \frac{1}{k}(\tilde{N}(R_0(\bm{\phi}))-\tilde{N}_0(R_0(\bm{\phi}))-|\{i\in [n]: b_i<-R_0(\bm{\phi})\}|).
\end{eqnarray*}
Hence by (\ref{Eq5.11})-(\ref{Eq5.13}) and properties (a) and (c) of Proposition \ref{P5.2}, 
\begin{eqnarray}\label{Eq6.2}
   && |\mu_{n,k}([-R_0(\bm{\phi}),R_0(\bm{\phi})])|\leq \frac{1}{k}(|\Psi(R_0(\bm{\phi}))|+\tilde{N}(-R_0(\bm{\phi}))\nonumber\\
   &\leq&  C\sqrt{\frac{\epsilon}{\log(R_0(\bm{\phi}))\log\log(R_0(\bm{\phi}))}}+\frac{CM}{R_0(\bm{\phi})^{3\slash 2}} \leq C.
\end{eqnarray}
By (\ref{Eq5.13}), $R_0(\bm{\phi})\geq M$ is sufficiently large. Hence for any $x\geq R_0(\bm{\phi})+2$, $\pi^{-1}\sqrt{x}-\mu_{n,k}([-R_0(\bm{\phi}),R_0(\bm{\phi})])\geq \pi^{-1}\sqrt{R_0(\bm{\phi})+2}-C\geq 0$. Hence by (\ref{Eq6.1}), 
\begin{equation}
    (\mu'+\nu_0)(A)\geq 0 \text{ for any }A\in\mathcal{B}_{\mathbb{R}}, 
\end{equation}
and $\mu'+\nu_0$ is a positive measure. Recalling Definition \ref{Defn1.2}, we have
\begin{equation}\label{Eq7.8.3}
    \mu'\in\mathscr{X}.
\end{equation}

We bound $d_R(\mu'_R,\mu)$ as follows (see Definition \ref{Defn1.3} for the definition of $\mu'_R$). In the following, we consider an arbitrary function $f:[-R,R]\rightarrow \mathbb{R}$ such that $\|f\|_{BL}\leq 1$. By (\ref{Eq7.3.1}), we have
\begin{equation}\label{Eq7.3.2}
    \Big|\int_{[-R,R]}fd\mu'-\int_{[-R,R]}fd\mu_{n,k}\Big|\leq\frac{1}{k}\sum_{i=1}^{n_0}\Big|f(\alpha_i)\mathbbm{1}_{[-R,R]}(\alpha_i)-\int_{[-R,R]}f d\pi_i\Big|.
\end{equation}
For any $i\in [n_0]$ such that $\alpha_i\in [-R+n^{-1},R-n^{-1}]$, as $\|f\|_{Lip}\leq 1$,
\begin{eqnarray}
 && \Big|f(\alpha_i)\mathbbm{1}_{[-R,R]}(\alpha_i)-\int_{[-R,R]}f d\pi_i\Big|=\Big|f(\alpha_i)-\frac{n}{2}\int_{\alpha_i-n^{-1}}^{\alpha_i+n^{-1}}f(x)dx\Big| \nonumber\\
 &\leq& \frac{n}{2}\int_{\alpha_i-n^{-1}}^{\alpha_i+n^{-1}}|f(\alpha_i)-f(x)|dx\leq \frac{n}{2}\cdot \frac{1}{n}\cdot \frac{2}{n}=\frac{1}{n}.
\end{eqnarray}

For any $i\in [n_0]$ such that $\alpha_i\in [-R-n^{-1},-R+n^{-1}]\cup [R-n^{-1},R+n^{-1}]$, we have $\alpha_i\in[-R-n^{-1\slash 30},-R+n^{-1\slash 30}]\cup[R-n^{-1\slash 30},R+n^{-1\slash 30}]$. As $\|f\|_{\infty}\leq 1$,
\begin{equation}
    \Big|f(\alpha_i)\mathbbm{1}_{[-R,R]}(\alpha_i)-\int_{[-R,R]}f d\pi_i\Big|\leq 2
\end{equation}
Moreover, as the event $\mathscr{G}_R$ holds, we have
\begin{equation}
    |\{i\in [n_0]:\alpha_i\in [-R-n^{-1},-R+n^{-1}]\cup [R-n^{-1},R+n^{-1}]\}|\leq \frac{k}{\log(n)^{1\slash 4}}.
\end{equation}

For any $i\in [n_0]$ such that $\alpha_i\notin [-R-n^{-1},R+n^{-1}]$, we have
\begin{equation}\label{Eq7.3.3}
    \Big|f(\alpha_i)\mathbbm{1}_{[-R,R]}(\alpha_i)-\int_{[-R,R]}f d\pi_i\Big|=0.
\end{equation}

By (\ref{Eq7.3.2})-(\ref{Eq7.3.3}), noting (\ref{Eq5.11}) and (\ref{Eq5.18.5}), we obtain that
\begin{eqnarray}
   &&  \Big|\int_{[-R,R]}fd\mu'-\int_{[-R,R]}fd\mu_{n,k}\Big|\leq \frac{1}{k}\Big(\frac{n_0}{n}+\frac{2k}{\log(n)^{1\slash 4}}\Big)\nonumber\\
   &\leq& \frac{CR_0(\bm{\phi})^{3\slash 2}}{n}+\frac{2}{\log(n)^{1\slash 4}}\leq \frac{C}{\log(n)^{1\slash 4}}\leq \delta.
\end{eqnarray}
Hence $d_R(\mu'_R,\mu_{n,k;R})\leq \delta$. As $d_R(\mu_{n,k;R},\mu)\leq 2\delta$, we have
\begin{equation}\label{Eq7.8.4}
  d_R(\mu'_R,\mu)\leq 3\delta.    
\end{equation}

\paragraph{Step 1}

Consider any $i,j\in [n_0]$ such that $i<j$. By (\ref{Eq6.1.9}) and Lemma \ref{L5.1}, 
\begin{eqnarray}\label{Eq6.2.20}
 &&\int \int \log(|x-y|)d\pi_i(x)d\pi_j(y)=\frac{n^2}{4} \int_{\alpha_i-n^{-1}}^{\alpha_i+n^{-1}}\int_{\alpha_j-n^{-1}}^{\alpha_j+n^{-1}} \log(|x-y|)dxdy\nonumber\\
 && \geq\log(\alpha_j-\alpha_i)-\frac{C}{\max\{n(\alpha_j-\alpha_i),1\}}.
\end{eqnarray}

For any $x>0$, let $h(x)=-x\log(x)$ and note that $h'(x)=-1-\log(x)$. Hence $h'(x)\geq 0$ for $x\in (0,e^{-1}]$ and $h'(x)\leq 0$ for $x\geq e^{-1}$. If $n^{-1}\leq \alpha_j-\alpha_i\leq 1\slash 2$, we have $h(\alpha_j-\alpha_i)\geq \min\{h(n^{-1}),h(1\slash 2)\}\geq c\log(n)\slash n$, hence
\begin{equation}\label{Eq6.1.15}
    \frac{1}{\max\{n(\alpha_j-\alpha_i),1\}}=\frac{1}{n(\alpha_j-\alpha_i)}=\frac{-\log(\alpha_j-\alpha_i)}{n h(\alpha_j-\alpha_i)}\leq \frac{-C\log(\alpha_j-\alpha_i)}{\log(n)},
\end{equation}
where we note that $\log(\alpha_j-\alpha_i)\leq \log(1\slash 2)<0$. If $0<\alpha_j-\alpha_i<n^{-1}$, we have $-\log(\alpha_j-\alpha_i)\geq -\log(n^{-1})=\log(n)$, hence  
\begin{equation}
    \frac{1}{\max\{n(\alpha_j-\alpha_i),1\}}=1\leq \frac{-\log(\alpha_j-\alpha_i)}{\log(n)}.
\end{equation}
If $\alpha_j-\alpha_i>1\slash 2$, we have
\begin{equation}\label{Eq6.1.14}
    \frac{1}{\max\{n(\alpha_j-\alpha_i),1\}}\leq \frac{2}{n}. 
\end{equation}

By (\ref{Eq6.1.15})-(\ref{Eq6.1.14}), we have
\begin{equation}\label{Eq6.2.17}
\sum_{1\leq i<j\leq n_0}\frac{1}{\max\{n(\alpha_j-\alpha_i),1\}}\leq -\frac{C_0}{\log(n)}\sum_{\substack{1\leq i<j\leq n_0 : \\ 0<\alpha_j-\alpha_i\leq 1\slash 2}} \log(\alpha_j-\alpha_i)+\frac{2n_0^2}{n},
\end{equation}
where $C_0$ is a positive constant that only depends on $\beta$. As $0<\alpha_j-\alpha_i\leq 2R_0(\bm{\phi})$ for any $i,j\in [n_0]$ such that $i<j$, we have
\begin{eqnarray}\label{Eq6.2.19}
\sum_{1\leq i<j\leq n_0}\log(\alpha_j-\alpha_i)&=&\sum_{\substack{1\leq i<j\leq n_0 : \\ 0<\alpha_j-\alpha_i\leq 1\slash 2}} \log(\alpha_j-\alpha_i)+\sum_{\substack{1\leq i<j\leq n_0 : \\ \alpha_j-\alpha_i> 1\slash 2}} \log(\alpha_j-\alpha_i)\nonumber\\
&\leq& \sum_{\substack{1\leq i<j\leq n_0 : \\ 0<\alpha_j-\alpha_i\leq 1\slash 2}} \log(\alpha_j-\alpha_i)+\log(2R_0(\bm{\phi})) n_0^2.
\end{eqnarray}
By (\ref{Eq5.18.5}), (\ref{Eq6.2.17}), and (\ref{Eq6.2.19}), we have
\begin{eqnarray}\label{Eq6.2.21}
    \sum_{1\leq i<j\leq n_0}\frac{1}{\max\{n(\alpha_j-\alpha_i),1\}}&\leq&   
    -\frac{C_0}{\log(n)}\sum_{1\leq i<j\leq n_0}\log(\alpha_j-\alpha_i)\nonumber\\
    && +\frac{C R_0(\bm{\phi})^3 \log(R_0(\bm{\phi}))k^2}{\log(n)}.
\end{eqnarray}
By (\ref{Eq6.2.20}) and (\ref{Eq6.2.21}),
\begin{eqnarray}\label{Eq7.1.1}
&& \sum_{1\leq i<j\leq n_0} \int \int \log(|x-y|)d\pi_i(x)d\pi_j(y)\nonumber\\
&\geq& \Big(1+\frac{C_1}{\log(n)}\Big) \sum_{1\leq i<j\leq n_0}\log(\alpha_j-\alpha_i)-\frac{C R_0(\bm{\phi})^3 \log(R_0(\bm{\phi}))k^2}{\log(n)},\nonumber\\
&& 
\end{eqnarray}
where $C_1$ is a positive constant that only depends on $\beta$.

Moreover, by (\ref{Eq5.18.5}), (\ref{Eq6.1.9}), and Lemma \ref{L5.1}, we have
\begin{eqnarray}\label{Eq7.1.2}
&& \sum_{i=1}^{n_0} \int \int \log(|x-y|)d\pi_i(x)d\pi_i(y)\nonumber\\
&=& \sum_{i=1}^{n_0}\frac{n^2}{4} \int_{\alpha_i-n^{-1}}^{\alpha_i+n^{-1}}\int_{\alpha_i-n^{-1}}^{\alpha_i+n^{-1}} \log(|x-y|)dxdy\nonumber\\
&\geq& -(\log(n)+C)n_0\geq -C R_0(\bm{\phi})^{3\slash 2} k\log(n).
\end{eqnarray}

\paragraph{Step 2}

We consider an arbitrary $i\in [n_0]$. By (\ref{Eq6.1.9}), we have
\begin{equation}
    \int d\pi_i(x) \int_0^{R_0(\bm{\phi})} \log(|x-y|)d\mu_0(y)=\frac{n}{2}\int_{\alpha_i-n^{-1}}^{\alpha_i+n^{-1}}dx \int_0^{R_0(\bm{\phi})} \log(|x-y|)d\mu_0(y).
\end{equation}
Hence
\begin{eqnarray}\label{Eq6.3.1}
&& \int d\pi_i(x) \int_0^{R_0(\bm{\phi})} \log(|x-y|)d\mu_0(y)-\int_0^{R_0(\bm{\phi})} \log(|\alpha_i-y|)d\mu_0(y) \nonumber\\
&=& \frac{n}{2}\int_{\alpha_i-n^{-1}}^{\alpha_i+n^{-1}}dx \int_0^{R_0(\bm{\phi})}(\log(|x-y|)-\log(|\alpha_i-y|))d\mu_0(y).
\end{eqnarray}

Consider any $x\in [\alpha_i-n^{-1},\alpha_i+n^{-1}]$. For any $y\in [\alpha_i-2n^{-1},\alpha_i+2n^{-1}]$, $|x-y|\leq 4\slash n<1$ and $|\alpha_i-y|\leq 2\slash n<1$. Hence 
\begin{eqnarray}
&&\int_{[\alpha_i-2n^{-1},\alpha_i+2n^{-1}]\cap[0,R_0(\bm{\phi})]}(\log(|x-y|)-\log(|\alpha_i-y|))d\mu_0(y)\nonumber\\
&\leq& \int_{[\alpha_i-2n^{-1},\alpha_i+2n^{-1}]\cap[0,R_0(\bm{\phi})]}(-\log(|\alpha_i-y|))d\mu_0(y)\nonumber\\
&\leq& \sqrt{R_0(\bm{\phi})}\int_{[\alpha_i-2n^{-1},\alpha_i+2n^{-1}]}(-\log(|\alpha_i-y|))dy\nonumber\\
&=& \sqrt{R_0(\bm{\phi})}\Big(\frac{4}{n}\log\Big(\frac{n}{2}\Big)+\frac{4}{n}\Big)\leq \frac{8\sqrt{R_0(\bm{\phi})}\log(n)}{n}.
\end{eqnarray}
Using the inequality $\log(1+x)\leq x$ for any $x>-1$, we obtain that for any $y\in [0,R_0(\bm{\phi})]\backslash [\alpha_i-2n^{-1},\alpha_i+2n^{-1}]$,
\begin{equation*}
    \log(|x-y|)-\log(|\alpha_i-y|)\leq \log\Big(\frac{|x-\alpha_i|+|\alpha_i-y|}{|\alpha_i-y|}\Big)\leq \frac{|x-\alpha_i|}{|\alpha_i-y|}\leq \frac{1}{n|\alpha_i-y|}.
\end{equation*}
Hence
\begin{eqnarray}\label{Eq6.3.2}
&&\int_{[0,R_0(\bm{\phi})]\backslash [\alpha_i-2n^{-1},\alpha_i+2n^{-1}]}(\log(|x-y|)-\log(|\alpha_i-y|))d\mu_0(y)\nonumber\\
&\leq& \frac{1}{n}\int_{[0,R_0(\bm{\phi})]\backslash [\alpha_i-2n^{-1},\alpha_i+2n^{-1}]}\frac{1}{|y-\alpha_i|}d\mu_0(y)\nonumber\\
&\leq& \frac{\sqrt{R_0(\bm{\phi})}}{n}\int_{[0,R_0(\bm{\phi})]\backslash [\alpha_i-2n^{-1},\alpha_i+2n^{-1}]}\frac{1}{|y-\alpha_i|}dy\nonumber\\
&\leq& \frac{\sqrt{R_0(\bm{\phi})}}{n}\Big(\int_{-R_0(\bm{\phi})-2n^{-1}}^{\alpha_i-2n^{-1}}\frac{1}{|y-\alpha_i|}dy+\int_{\alpha_i+2n^{-1}}^{R_0(\bm{\phi})+2n^{-1}}\frac{1}{|y-\alpha_i|}dy\Big)\nonumber\\
&=& \frac{\sqrt{R_0(\bm{\phi})}}{n}\Big(\log\Big(R_0(\bm{\phi})+\alpha_i+\frac{2}{n}\Big)+\log\Big(R_0(\bm{\phi})-\alpha_i+\frac{2}{n}\Big)-2\log\Big(\frac{2}{n}\Big)\Big)\nonumber\\
&\leq&\frac{\sqrt{R_0(\bm{\phi})}}{n}(2\log(3R_0(\bm{\phi}))+2\log(n))\leq \frac{C\sqrt{R_0(\bm{\phi})}\log(n)}{n},
\end{eqnarray}
where we use (\ref{Eq5.11}) and the facts that $T$ only depends on $M,\epsilon$ and $k\geq K_0(M,\epsilon,\delta)$ is sufficiently large (depending on $\beta,M,\epsilon,\delta$) in the last line.

By (\ref{Eq6.3.1})-(\ref{Eq6.3.2}), for any $i\in [n_0]$, we have
\begin{eqnarray}
   && \int d\pi_i(x) \int_0^{R_0(\bm{\phi})} \log(|x-y|)d\mu_0(y)-\int_0^{R_0(\bm{\phi})} \log(|\alpha_i-y|)d\mu_0(y)\nonumber\\
   && \leq \frac{C\sqrt{R_0(\bm{\phi})}\log(n)}{n}. 
\end{eqnarray}
Hence by (\ref{Eq5.18.5}),
\begin{eqnarray}\label{Eq7.1.3}
&& \frac{1}{k}\sum_{i=1}^{n_0}\int d\pi_i(x) \int_0^{R_0(\bm{\phi})} \log(|x-y|)d\mu_0(y)-\frac{1}{k}\sum_{i=1}^{n_0} \int_0^{R_0(\bm{\phi})} \log(|\alpha_i-y|)d\mu_0(y) \nonumber\\
&&\leq\frac{n_0}{k}\cdot \frac{C\sqrt{R_0(\bm{\phi})}\log(n)}{n}\leq \frac{CR_0(\bm{\phi})^2\log(n)}{n}.
\end{eqnarray}

\paragraph{Step 3}

Consider any $x\in [-R_0(\bm{\phi}),R_0(\bm{\phi})]$. We have
\begin{equation}
    \Big|\int_0^{R_0(\bm{\phi})}\log(|x-y|)d\mu_0(y)\Big|\leq \sqrt{R_0(\bm{\phi})}\int_0^{R_0(\bm{\phi})}|\log(|x-y|)|dy.
\end{equation}
Note that
\begin{equation*}
    \int_{x-1}^{x+1}|\log(|x-y|)|dy=-2\int_0^1\log(s)ds=2,
\end{equation*}
\begin{equation*}
    \int_{-R_0(\bm{\phi})-1}^{x-1}|\log(|x-y|)|dy\leq \log(x+R_0(\bm{\phi})+1)(x+R_0(\bm{\phi}))\leq C R_0(\bm{\phi})\log(R_0(\bm{\phi})),
\end{equation*}
\begin{equation*}
    \int_{x+1}^{R_0(\bm{\phi})+1}|\log(|x-y|)|dy\leq \log(R_0(\bm{\phi})+1-x)(R_0(\bm{\phi})-x)\leq C R_0(\bm{\phi})\log(R_0(\bm{\phi})).
\end{equation*}
Hence for any $x \in [-R_0(\bm{\phi}),R_0(\bm{\phi})]$,
\begin{equation}\label{Eq7.1.5}
    \Big|\int_0^{R_0(\bm{\phi})}\log(|x-y|)d\mu_0(y)\Big|\leq  CR_0(\bm{\phi})^{3\slash 2}\log(R_0(\bm{\phi})).
\end{equation}
By (\ref{Eq5.18.5}) and (\ref{Eq7.1.5}), we have
\begin{eqnarray}\label{Eq7.1.8}
    \Big|\frac{2}{k}\sum_{i=1}^{n_0}\int_0^{R_0(\bm{\phi})}\log(|\alpha_i-y|)d\mu_0(y)\Big|&\leq& CR_0(\bm{\phi})^{3\slash 2}\log(R_0(\bm{\phi}))\cdot \frac{2n_0}{k}\nonumber\\
    &\leq& C R_0(\bm{\phi})^3\log(R_0(\bm{\phi})),
\end{eqnarray}
\begin{eqnarray}\label{Eq7.1.7}
&& \Big| \int_0^{R_0(\bm{\phi})}\int_0^{R_0(\bm{\phi})}\log(|x-y|)d\mu_0(x) d\mu_0(y)\Big|\nonumber\\
&\leq& CR_0(\bm{\phi})^{3\slash 2}\log(R_0(\bm{\phi}))\cdot\mu_0([0,R_0(\bm{\phi})])\leq C R_0(\bm{\phi})^3\log(R_0(\bm{\phi})).\nonumber\\
&&
\end{eqnarray}
By (\ref{Eq5.21.2}), (\ref{Eq7.1.8}), and (\ref{Eq7.1.7}),
\begin{eqnarray}\label{Eq7.1.6}
 &&   \frac{2}{k^2}\sum_{1\leq i<j\leq n_0}\log(\alpha_j-\alpha_i)=\int_{\mathbb{R}^2\backslash \Delta} \log(|x-y|)d\Upsilon_1(x)d\Upsilon_1(y)\nonumber\\
 && +\frac{2}{k}\sum_{i=1}^{n_0}\int_0^{R_0(\bm{\phi})}\log(|\alpha_i-y|)d\mu_0(y)-\int_0^{R_0(\bm{\phi})}\int_0^{R_0(\bm{\phi})}\log(|x-y|)d\mu_0(x) d\mu_0(y)\nonumber\\
 &\geq& \int_{\mathbb{R}^2\backslash \Delta} \log(|x-y|)d\Upsilon_1(x)d\Upsilon_1(y)-C R_0(\bm{\phi})^3\log(R_0(\bm{\phi})).
\end{eqnarray}

By (\ref{Eq6.6.1}), (\ref{Eq7.1.1}), (\ref{Eq7.1.2}), (\ref{Eq7.1.3}), and (\ref{Eq7.1.6}),
\begin{eqnarray}\label{Eq7.2.2}
 &&\int\int \log(|x-y|)d\mu'_1(x)d\mu'_1(y)= \frac{1}{k^2}\sum_{i=1}^{n_0} \int \int \log(|x-y|)d\pi_i(x)d\pi_i(y)\nonumber\\
 &&+\frac{2}{k^2} \sum_{1\leq i<j\leq n_0} \int \int \log(|x-y|)d\pi_i(x)d\pi_j(y)-\frac{2}{k}\sum_{i=1}^{n_0}\int d\pi_i(x)\int_{0}^{R_0(\bm{\phi})}\log(|x-y|)d\mu_0(y)\nonumber\\
 &&+\int_0^{R_0(\bm{\phi})}\int_0^{R_0(\bm{\phi})}\log(|x-y|)d\mu_0(x)d\mu_0(y)\nonumber\\
 &\geq& \frac{2}{k^2}\Big(1+\frac{C_1}{\log(n)}\Big) \sum_{1\leq i<j\leq n_0}\log(\alpha_j-\alpha_i)-\frac{C R_0(\bm{\phi})^3 \log(R_0(\bm{\phi}))}{\log(n)}\nonumber\\
 &&-\frac{2}{k}\sum_{i=1}^{n_0} \int_0^{R_0(\bm{\phi})} \log(|\alpha_i-y|)d\mu_0(y)+\int_0^{R_0(\bm{\phi})}\int_0^{R_0(\bm{\phi})}\log(|x-y|)d\mu_0(x)d\mu_0(y)\nonumber\\
 &\geq& \Big(1+\frac{C_1}{\log(n)}\Big)\int_{\mathbb{R}^2\backslash \Delta} \log(|x-y|)d\Upsilon_1(x)d\Upsilon_1(y)-\frac{C R_0(\bm{\phi})^3 \log(R_0(\bm{\phi}))}{\log(n)}.
\end{eqnarray}

\paragraph{Step 4}

Consider any $i\in [n_0]$. For any $x\in [\alpha_i-n^{-1},\alpha_i+n^{-1}]$ and any $y\in [R_0(\bm{\phi})+2,R_0(\bm{\phi})+3]$, we have $y-\alpha_i\geq (R_0(\bm{\phi})+2)-R_0(\bm{\phi})=2$ and $y-x\geq y-(\alpha_i+n^{-1})\geq 1$. Using the inequality $\log(1+x)\leq x$ for any $x>-1$, we obtain that
\begin{equation*}
 \log(|x-y|)-\log(|\alpha_i-y|)\leq \log\Big(\frac{|x-\alpha_i|+|\alpha_i-y|}{|\alpha_i-y|}\Big)\leq \frac{|x-\alpha_i|}{|\alpha_i-y|}\leq \frac{1}{n},
\end{equation*}
\begin{equation*}
    \log(|\alpha_i-y|)-\log(|x-y|)\leq \log\Big(\frac{|x-\alpha_i|+|x-y|}{|x-y|}\Big)\leq \frac{|x-\alpha_i|}{|x-y|}\leq \frac{1}{n}.
\end{equation*}
Hence $|\log(|x-y|)-\log(|\alpha_i-y|)|\leq n^{-1}$. Thus for any $y\in [R_0(\bm{\phi})+2,R_0(\bm{\phi})+3]$, 
\begin{eqnarray}\label{Eq6.4.2}
&& \Big|\int\log(|x-y|)d\pi_i(x)-\log(|\alpha_i-y|)\Big|\nonumber\\
&\leq &\frac{n}{2}\int_{\alpha_i-n^{-1}}^{\alpha_i+n^{-1}}|\log(|x-y|)-\log(|\alpha_i-y|)|dx\leq \frac{1}{n}.
\end{eqnarray}

By (\ref{Eq5.18.5}) and (\ref{Eq6.4.2}), for any $y\in [R_0(\bm{\phi})+2,R_0(\bm{\phi})+3]$, we have
\begin{eqnarray}\label{Eq6.4.3}
  &&  \Big|\int \log(|x-y|)d\Big(\frac{1}{k}\sum_{i=1}^{n_0}\pi_i-\frac{1}{k}\sum_{i=1}^{n_0}\delta_{\alpha_i}\Big)(x)\Big|\nonumber\\
  &\leq& \frac{1}{k}\sum_{i=1}^{n_0}\Big|\int\log(|x-y|)d\pi_i(x)-\log(|\alpha_i-y|)\Big|\leq \frac{n_0}{k n}\leq \frac{CR_0(\bm{\phi})^{3\slash 2}}{n}.\nonumber\\
  &&
\end{eqnarray}
Hence by (\ref{Eq6.4.1}) and (\ref{Eq6.2}),
\begin{eqnarray}\label{Eq6.7.1}
&& \Big|\int\int \log(|x-y|)d\Big(\frac{1}{k}\sum_{i=1}^{n_0}\pi_i-\frac{1}{k}\sum_{i=1}^{n_0}\delta_{\alpha_i}\Big)(x)d\mu'_2(y)\Big|\nonumber\\
&\leq & |\mu_{n,k}([-R_0(\bm{\phi}),R_0(\bm{\phi})])|\cdot \frac{CR_0(\bm{\phi})^{3\slash 2}}{n}\leq \frac{CR_0(\bm{\phi})^{3\slash 2}}{n}.
\end{eqnarray}

In the following, we consider any $y\in [R_0(\bm{\phi})+2,R_0(\bm{\phi})+3]$. Recall the definition of $\Upsilon_1$ from (\ref{Eq5.21.2}). We have
\begin{eqnarray}\label{Eq6.5.8}
    \int \log(|x-y|)d\Upsilon_1(x)&=&k^{-1}\log(R_0(\bm{\phi})+y)\mathbbm{1}_{b_i=-R_0(\bm{\phi})\text{ for some } i\in [n]}\nonumber\\
    && + k^{-1}\int_{-R_0(\bm{\phi})}^{R_0(\bm{\phi})}\log(|x-y|)d\Psi(x).
\end{eqnarray}
Note that
\begin{equation}\label{Eq6.5.9}
    \Big|k^{-1}\log(R_0(\bm{\phi})+y)\mathbbm{1}_{b_i=-R_0(\bm{\phi})\text{ for some } i\in [n]}\Big|\leq \frac{\log(2R_0(\bm{\phi})+3)}{k} \leq \frac{C\log(R_0(\bm{\phi}))}{k}. 
\end{equation}
By integration by parts, we have
\begin{eqnarray}\label{Eq6.5.6}
 && k^{-1}\int_{-R_0(\bm{\phi})}^{R_0(\bm{\phi})}\log(|x-y|)d\Psi(x)\nonumber\\
 &=& \frac{1}{k}\Big(\log(|y-R_0(\bm{\phi})|)\Psi(R_0(\bm{\phi}))-\log(|y+R_0(\bm{\phi})|)\Psi(-R_0(\bm{\phi}))\nonumber\\
 &&\quad -\int_{-R_0(\bm{\phi})}^{R_0(\bm{\phi})}\frac{\Psi(x)}{x-y}dx\Big).
\end{eqnarray}
As $|\log(|y-R_0(\bm{\phi})|)|\leq \log(3)\leq 2$, by property (a) of Proposition \ref{P5.2}, 
\begin{equation}
    \Big|\frac{1}{k}\log(|y-R_0(\bm{\phi})|)\Psi(R_0(\bm{\phi}))\Big|\leq C\sqrt{\frac{\epsilon}{\log(R_0(\bm{\phi}))\log\log(R_0(\bm{\phi}))}}.
\end{equation}
By (\ref{Eq5.13}),
\begin{equation}\label{Eq6.5.15}
    M\leq \log(R_0(\bm{\phi})).
\end{equation}
As $|\log(|y+R_0(\bm{\phi})|)|\leq \log(2R_0(\bm{\phi})+3)\leq C\log(R_0(\bm{\phi}))$, by (\ref{Eq5.11}) and property (c) of Proposition \ref{P5.2}, 
\begin{equation}
    \Big|\frac{1}{k}\log(|y+R_0(\bm{\phi})|)\Psi(-R_0(\bm{\phi}))\Big|\leq C\log(R_0(\bm{\phi}))\cdot \frac{CM}{R_0(\bm{\phi})^{3\slash 2}}\leq \frac{C \log(R_0(\bm{\phi}))^2}{R_0(\bm{\phi})^{3\slash 2}}. 
\end{equation}
As $y-x\geq (R_0(\bm{\phi})+2)-R_0(\bm{\phi})\geq 1$ for any $x\leq R_0(\bm{\phi})$, by (\ref{Eq5.50}), (\ref{Eq5.6.1}), and property (a) of Proposition \ref{P5.2},
\begin{eqnarray}\label{Eq6.5.7}
  && \Big| \int_{-R_0(\bm{\phi})}^{R_0(\bm{\phi})}\frac{\Psi(x)}{x-y}dx\Big| \leq  \int_{-R_0(\bm{\phi})}^{R_0(\bm{\phi})}\frac{|\Psi(x)|}{y-x}dx\nonumber\\
  &\leq& \int_{-R_0(\bm{\phi})}^{R_0(\bm{\phi})-R_0(\bm{\phi})^{-4}}\frac{|\Psi(x)|}{R_0(\bm{\phi})-x}dx+\int_{R_0(\bm{\phi})-R_0(\bm{\phi})^{-4}}^{R_0(\bm{\phi})}  |\Psi(x)|dx\nonumber\\
  &\leq& CM\sqrt{\log(R_0(\bm{\phi}))} k+C\sqrt{\frac{\epsilon}{\log(R_0(\bm{\phi}))\log\log(R_0(\bm{\phi}))}}k\cdot R_0(\bm{\phi})^{-4}\nonumber\\
  &\leq& CM\sqrt{\log(R_0(\bm{\phi}))} k.
\end{eqnarray}
By (\ref{Eq6.5.6})-(\ref{Eq6.5.7}),
\begin{equation}\label{Eq6.5.12}
   \Big| k^{-1}\int_{-R_0(\bm{\phi})}^{R_0(\bm{\phi})}\log(|x-y|)d\Psi(x)\Big|\leq CM\sqrt{\log(R_0(\bm{\phi}))}.
\end{equation}
By (\ref{Eq6.5.8}), (\ref{Eq6.5.9}), and (\ref{Eq6.5.12}),
\begin{equation}\label{Eq6.5.14}
    \Big| \int \log(|x-y|)d\Upsilon_1(x)\Big|\leq CM\sqrt{\log(R_0(\bm{\phi}))}.
\end{equation}

By (\ref{Eq6.4.1}), (\ref{Eq6.2}), (\ref{Eq6.5.15}), and (\ref{Eq6.5.14}), 
\begin{eqnarray}\label{Eq6.7.3}
&& \Big|\int\int \log(|x-y|)d\Upsilon_1(x)d\mu_2'(y)\Big|\nonumber\\
&=& \Big|\mu_{n,k}([-R_0(\bm{\phi}),R_0(\bm{\phi})])\Big|\Big|\int_{R_0(\bm{\phi})+2}^{R_0(\bm{\phi})+3}dy\int\log(|x-y|)d\Upsilon_1(x)\Big|\nonumber\\
&\leq& C\Big(\sqrt{\frac{\epsilon}{\log(R_0(\bm{\phi}))\log\log(R_0(\bm{\phi}))}}+\frac{M}{R_0(\bm{\phi})^{3\slash 2}}\Big)\cdot M\sqrt{\log(R_0(\bm{\phi}))}\nonumber\\
&\leq& C\Big(M\sqrt{\epsilon}+\frac{\log(R_0(\bm{\phi}))^3}{R_0(\bm{\phi})^{3\slash 2}}\Big).
\end{eqnarray}

By (\ref{Eq5.21.2}) and (\ref{Eq6.6.1}), we have
\begin{equation}
    \mu'_1=\frac{1}{k}\sum_{i=1}^{n_0}\pi_i-\frac{1}{k}\sum_{i=1}^{n_0}\delta_{\alpha_i}+\Upsilon_1.
\end{equation}
Hence by (\ref{Eq6.7.1}) and (\ref{Eq6.7.3}),
\begin{equation}\label{Eq7.2.3}
    \Big|\int\int \log(|x-y|)d\mu'_1(x)d\mu_2'(y)\Big|\leq C\Big(M\sqrt{\epsilon}+\frac{\log(R_0(\bm{\phi}))^3}{R_0(\bm{\phi})^{3\slash 2}}\Big).
\end{equation}

Using the inequality $|x\log(x)|\leq e^{-1}$ for any $x\in [0,1]$, we have
\begin{eqnarray*}
  &&   \Big|\int_{R_0(\bm{\phi})+2}^{R_0(\bm{\phi})+3}dy\int_{R_0(\bm{\phi})+2}^{R_0(\bm{\phi})+3}\log(|x-y|)dx\Big|=\Big|\int_{0}^1 dy\int_{0}^1\log(|x-y|)dx\Big|\nonumber\\
  &=& \Big|\int_0^1((1-y)\log(1-y)+y\log(y)-1)dy \Big|\leq 3.
\end{eqnarray*}
Hence by (\ref{Eq6.4.1}), (\ref{Eq6.2}), and (\ref{Eq6.5.15}),
\begin{eqnarray}\label{Eq7.2.4}
 && \Big|\int\int\log(|x-y|)d\mu'_2(x)d\mu'_2(y)\Big|\nonumber\\
 &=&(\mu_{n,k}([-R_0(\bm{\phi}),R_0(\bm{\phi})]))^2\Big|\int_{R_0(\bm{\phi})+2}^{R_0(\bm{\phi})+3}dy\int_{R_0(\bm{\phi})+2}^{R_0(\bm{\phi})+3}\log(|x-y|)dx\Big|\nonumber\\
 &\leq& C\Big(\sqrt{\frac{\epsilon}{\log(R_0(\bm{\phi}))\log\log(R_0(\bm{\phi}))}}+\frac{\log(R_0(\bm{\phi}))}{R_0(\bm{\phi})^{3\slash 2}}\Big).
\end{eqnarray}

By (\ref{Eq7.2.1}), (\ref{Eq7.2.2}), (\ref{Eq7.2.3}), (\ref{Eq7.2.4}), noting (\ref{Eq5.11}), we have
\begin{eqnarray}
   \int\int \log(|x-y|)d\mu'(x)d\mu'(y) &\geq& \Big(1+\frac{C_1}{\log(n)}\Big)\int_{\mathbb{R}^2\backslash \Delta} \log(|x-y|)d\Upsilon_1(x)d\Upsilon_1(y) \nonumber\\
    &&-C\Big(M\sqrt{\epsilon}+\frac{\log(R_0(\bm{\phi}))^3}{R_0(\bm{\phi})^{3\slash 2}}\Big),
\end{eqnarray}
which leads to
\begin{eqnarray}\label{Eq7.8.1}
  -\int_{\mathbb{R}^2\backslash \Delta} \log(|x-y|)d\Upsilon_1(x)d\Upsilon_1(y)&\geq& -\Big(1+\frac{C_1}{\log(n)}\Big)^{-1}\int\int \log(|x-y|)d\mu'(x)d\mu'(y)\nonumber\\
  &&  -C\Big(M\sqrt{\epsilon}+\frac{\log(R_0(\bm{\phi}))^3}{R_0(\bm{\phi})^{3\slash 2}}\Big).
\end{eqnarray}

\paragraph{Step 5}

For any $x\in [0,R_0(\bm{\phi})]$, by (\ref{Eq5.11}), $2-(k\slash n)^{2\slash 3}x\in [-2,2]$; hence by (\ref{Eq7.6.1}) and \cite[Lemma 6.1]{Zho}, $\tilde{\xi}(x)=0$. Hence by (\ref{Eq5.21.2}), 
\begin{eqnarray}\label{Eq12.39}
  2\int\tilde{\xi}(x)d\Upsilon_1(x)&=&\frac{2}{k}\sum_{i=1}^{n}\tilde{\xi}(b_i)\mathbbm{1}_{[-R_0(\bm{\phi}),R_0(\bm{\phi})]}(b_i)\nonumber\\
  &&-2\int\mathbbm{1}_{[-R_0(\bm{\phi}),R_0(\bm{\phi})]}(x)\tilde{\xi}(x)d\mu_0(x)\nonumber\\
  &=&\frac{2}{k}\sum_{i=1}^{n_0}\tilde{\xi}(\alpha_i)=\frac{2}{k}\sum_{i\in [n_0]:\alpha_i<0}\tilde{\xi}(\alpha_i).
\end{eqnarray}
For any $i\in [n_0]$ such that $\alpha_i<0$, by (\ref{Eq7.6.1}) and \cite[Lemma 6.1]{Zho}, we have
\begin{equation}
   \frac{2}{3}|\alpha_i|^{3\slash 2}\leq  \tilde{\xi}(\alpha_i)\leq \frac{2}{3}|\alpha_i|^{3\slash 2}\Big(1+\frac{1}{4}\Big(\frac{k}{n}\Big)^{2\slash 3}R_0(\bm{\phi})\Big).
\end{equation}
Hence by (\ref{Eq5.18.5}),
\begin{eqnarray}\label{Eq7.15.1}
&& \Big|2\int\tilde{\xi}(x)d\Upsilon_1(x)-\frac{4}{3k}\sum_{i\in [n_0]:\alpha_i<0}|\alpha_i|^{3\slash 2}\Big|\leq \frac{2}{k}\sum_{i\in [n_0]:\alpha_i<0}\Big|\tilde{\xi}(\alpha_i)-\frac{2}{3}|\alpha_i|^{3\slash 2}\Big|\nonumber\\
&\leq& \frac{1}{3k}\sum_{i\in [n_0]:\alpha_i<0} \Big(\frac{k}{n}\Big)^{2\slash 3}|\alpha_i|^{3\slash 2}R_0(\bm{\phi})\leq \frac{R_0(\bm{\phi})^{5\slash 2} n_0}{k^{1\slash 3}n^{2\slash 3}} \leq CR_0(\bm{\phi})^4 \Big(\frac{k}{n}\Big)^{2\slash 3} . \nonumber\\
&&
\end{eqnarray}

For every $i\in [n_0]$, let 
\begin{eqnarray}
 \mathscr{J}_i:=\frac{4}{3}\int |x|^{3\slash 2} \mathbbm{1}_{(-\infty,0]}(x) d\pi_i(x)=\frac{2n}{3}\int_{\alpha_i-n^{-1}}^{\alpha_i+n^{-1}}|x|^{3\slash 2} \mathbbm{1}_{(-\infty,0]}(x) dx.
\end{eqnarray}
For any $i\in [n_0]$ such that $\alpha_i\geq n^{-1}$, we have $\alpha_i-n^{-1}\geq 0$, hence $\mathscr{J}_i=0$. For any $i\in [n_0]$ such that $-n^{-1}\leq \alpha_i\leq n^{-1}$, we have
\begin{equation}
    |\mathscr{J}_i|\leq \frac{2n}{3}\int_{-2n^{-1}}^{2n^{-1}}|x|^{3\slash 2}\mathbbm{1}_{(-\infty,0]}(x) dx\leq \frac{2n}{3}\cdot (2n^{-1})^{3\slash 2}\cdot \frac{2}{n}\leq 6 n^{-3\slash 2}.
\end{equation}

For any $i\in [n_0]$ such that $\alpha_i<-n^{-1}$, we have
\begin{equation*}
    \mathscr{J}_i=\frac{2n}{3}\int_{\alpha_i-n^{-1}}^{\alpha_i+n^{-1}}|x|^{3\slash 2} dx.
\end{equation*}
For any $x\in [\alpha_i-n^{-1},\alpha_i+n^{-1}]$, there exists $y$ between $x$ and $\alpha_i$, such that
\begin{equation*}
    |x|^{3\slash 2}-|\alpha_i|^{3\slash 2}=(-x)^{3\slash 2}-(-\alpha_i)^{3\slash 2}=\frac{3}{2}(-y)^{1\slash 2}(\alpha_i-x).
\end{equation*}
As $y\in [\alpha_i-n^{-1},\alpha_i+n^{-1}]\subseteq [-2R_0(\bm{\phi}),0]$, we have 
\begin{equation*}
 \Big||x|^{3\slash 2}-|\alpha_i|^{3\slash 2}\Big|\leq \frac{3}{2}|y|^{1\slash 2}|x-\alpha_i|\leq 4\sqrt{R_0(\bm{\phi})}n^{-1}.
\end{equation*}
Hence
\begin{eqnarray}
&& \Big|\mathscr{J}_i-\frac{4}{3}|\alpha_i|^{3\slash 2}\Big|=\frac{2n}{3}\Big|\int_{\alpha_i-n^{-1}}^{\alpha_i+n^{-1}}(|x|^{3\slash 2}-|\alpha_i|^{3\slash 2}) dx\Big|\nonumber\\
&\leq& \frac{2n}{3}\int_{\alpha_i-n^{-1}}^{\alpha_i+n^{-1}}\Big||x|^{3\slash 2}-|\alpha_i|^{3\slash 2}\Big|dx \leq 6\sqrt{R_0(\bm{\phi})}n^{-1}.
\end{eqnarray}

Hence by (\ref{Eq5.18.5}) and (\ref{Eq7.3.1}), we have 
\begin{eqnarray}\label{Eq7.15.2}
&& \Big|\frac{4}{3}\int |x|^{3\slash 2} \mathbbm{1}_{(-\infty,0]}(x) d\mu'(x)-\frac{4}{3k}\sum_{i\in [n_0]:\alpha_i<0}|\alpha_i|^{3\slash 2}\Big|\nonumber\\
&=& \Big|\frac{1}{k}\sum_{i=1}^{n_0}\mathscr{J}_i-\frac{4}{3k}\sum_{i\in [n_0]:\alpha_i<0}|\alpha_i|^{3\slash 2}\Big|\nonumber\\
&\leq& \frac{1}{k}\sum_{\substack{i\in [n_0]:\\-n^{-1}\leq \alpha_i\leq n^{-1}}}|\mathscr{J}_i|+\frac{4}{3k}\sum_{\substack{i\in [n_0]:\\\alpha_i\in [-n^{-1},0)}}|\alpha_i|^{3\slash 2}+\frac{1}{k}\sum_{\substack{i\in [n_0]:\\\alpha_i<-n^{-1}}}\Big|\mathscr{J}_i-\frac{4}{3}|\alpha_i|^{3\slash 2}\Big|\nonumber\\
&\leq& \frac{1}{k}\cdot 6n^{-3\slash 2}\cdot n_0+\frac{4}{3k}\cdot n^{-3\slash 2}\cdot n_0+\frac{1}{k}\cdot 6\sqrt{R_0(\bm{\phi})}n^{-1}\cdot n_0\nonumber\\
&\leq& CR_0(\bm{\phi})^2 n^{-1}.
\end{eqnarray}

By (\ref{Eq7.15.1}) and (\ref{Eq7.15.2}), we have
\begin{equation}\label{Eq7.8.2}
    \Big|2\int\tilde{\xi}(x)d\Upsilon_1(x)-\frac{4}{3}\int |x|^{3\slash 2} \mathbbm{1}_{(-\infty,0]}(x) d\mu'(x)\Big|\leq CR_0(\bm{\phi})^4 \Big(\frac{k}{n}\Big)^{2\slash 3}.
\end{equation}

By (\ref{Eq7.8.1}) and (\ref{Eq7.8.2}), noting (\ref{Eq5.11}) and Definition \ref{Defn1.3}, we have
\begin{eqnarray}
&& -\int_{\mathbb{R}^2\backslash \Delta} \log(|x-y|)d\Upsilon_1(x)d\Upsilon_1(y)+2\int\tilde{\xi}(x)d\Upsilon_1(x)\nonumber\\
&\geq& -\Big(1+\frac{C_1}{\log(n)}\Big)^{-1}\int\int \log(|x-y|)d\mu'(x)d\mu'(y)\nonumber\\
&& +\frac{4}{3}\int |x|^{3\slash 2} \mathbbm{1}_{(-\infty,0]}(x) d\mu'(x)-C\Big(M\sqrt{\epsilon}+\frac{\log(R_0(\bm{\phi}))^3}{R_0(\bm{\phi})^{3\slash 2}}\Big)\nonumber\\
&\geq& \Big(1+\frac{C_1}{\log(n)}\Big)^{-1}\mathscr{I}(\mu')-C\Big(M\sqrt{\epsilon}+\frac{\log(R_0(\bm{\phi}))^3}{R_0(\bm{\phi})^{3\slash 2}}\Big),
\end{eqnarray}
where we use the fact that $\frac{4}{3}\int |x|^{3\slash 2} \mathbbm{1}_{(-\infty,0]}(x) d\mu'(x)=k^{-1}\sum_{i=1}^{n_0}\mathscr{J}_i\geq 0$. Hence by (\ref{Eq7.8.3}) and (\ref{Eq7.8.4}), noting Definition \ref{Defn1.3}, we obtain that
\begin{eqnarray}
&& -\int_{\mathbb{R}^2\backslash \Delta} \log(|x-y|)d\Upsilon_1(x)d\Upsilon_1(y)+2\int\tilde{\xi}(x)d\Upsilon_1(x)\nonumber\\
&\geq& \Big(1+\frac{C_1}{\log(n)}\Big)^{-1}I_R(\mu,3\delta)-C\Big(M\sqrt{\epsilon}+\frac{\log(R_0(\bm{\phi}))^3}{R_0(\bm{\phi})^{3\slash 2}}\Big).
\end{eqnarray}
By (\ref{Eq5.13}), as $M\geq M_0$ is sufficiently large, we have
\begin{equation}
    \frac{\log(R_0(\bm{\phi}))^3}{R_0(\bm{\phi})^{3\slash 2}}\leq R_0(\bm{\phi})^{-1}\leq M^{-1}.
\end{equation}
Hence we conclude that
\begin{eqnarray}
&& -\int_{\mathbb{R}^2\backslash \Delta} \log(|x-y|)d\Upsilon_1(x)d\Upsilon_1(y)+2\int\tilde{\xi}(x)d\Upsilon_1(x)\nonumber\\
&&\geq\Big(1+\frac{C_1}{\log(n)}\Big)^{-1}I_R(\mu,3\delta)-C\Big(M\sqrt{\epsilon}+M^{-1}\Big).
\end{eqnarray}

\end{proof}

\subsubsection{Proof of Proposition \ref{P5.7}}

\begin{proof}[Proof of Proposition \ref{P5.7}]

Throughout the proof, we fix an arbitrary
\begin{equation*}
    \mathbf{t}=(t_1,t_2,\cdots,t_{n_0})\in \mathscr{F}_{n_0}.
\end{equation*}
For every $i\in [m_0']$, we define $\mathscr{V}_i$ to be the signed Borel measure on $\mathbb{R}$ such that for any $A\in\mathcal{B}_{\mathbb{R}}$,
\begin{equation}\label{Eq10.20}
    \mathscr{V}_i(A)=k^{-1} \mathbbm{1}_{A}(\rho_i+t_i)-\mu_0(A\cap [\rho_{i-1},\rho_i]).
\end{equation}
For every $i\in [m_0'+1,n_0]\cap\mathbb{Z}$, we define $\mathscr{V}_i$ to be the signed Borel measure on $\mathbb{R}$ such that for any $A\in\mathcal{B}_{\mathbb{R}}$,
\begin{equation}
    \mathscr{V}_i(A)=k^{-1} \mathbbm{1}_{A}(\rho_i+t_i)-\mu_0(A\cap [\rho_{i},\rho_{i+1}]).
\end{equation}
We define $\mathscr{V}_0$ to be the signed Borel measure on $\mathbb{R}$ such that for any $A\in\mathcal{B}_{\mathbb{R}}$,
\begin{equation}
    \mathscr{V}_0(A)=-\mu_0(A\cap [\rho_{m_0'},r_0]).
\end{equation}
We also define 
\begin{equation}
    \mathscr{W}_1:=\sum_{i=0}^{m_0'} \mathscr{V}_i, \quad \mathscr{W}_2:=\sum_{i=m_0'+1}^{n_0} \mathscr{V}_i. 
\end{equation}
Note that by (\ref{Eq8.2}) and (\ref{Eq8.1}), we have 
\begin{equation}\label{Eq12.35}
    \tilde{\Upsilon}_{1;n_0,\mathbf{t}}=\mathscr{W}_1+\mathscr{W}_2.
\end{equation}
For every $i,j\in \{0\}\cup [n_0]$ such that $i\leq j$, we let 
\begin{equation}\label{Eq10.19}
    \mathscr{P}_{i,j}:=\int_{\mathbb{R}^2\backslash \Delta} \log(|x-y|) d\mathscr{V}_i(x)d\mathscr{V}_j(y).
\end{equation}

\paragraph{Step 1}

In this step, we bound $\int_{\mathbb{R}^2\backslash \Delta}\log(|x-y|)d\mathscr{W}_1(x)d\mathscr{W}_1(y)$. Note that
\begin{equation}\label{Eq12.13}
    \int_{\mathbb{R}^2\backslash \Delta}\log(|x-y|)d\mathscr{W}_1(x)d\mathscr{W}_1(y)=2\sum_{0\leq i<j\leq m_0'}\mathscr{P}_{i,j}+\sum_{i=0}^{m_0'}\mathscr{P}_{i,i}.
\end{equation}

\subparagraph{Sub-step 1.1}

In the following, we consider any $i,j\in [m_0']$ such that $i<j$. By (\ref{Eq5.25}), 
\begin{eqnarray}\label{Eq10.16}
\mathscr{P}_{i,j}&=&k^{-2}\log(|(\rho_i+t_i)-(\rho_j+t_j)|)-k^{-1}\int_{\rho_{i-1}}^{\rho_i}\log(|(\rho_j+t_j)-x|)d\mu_0(x)\nonumber\\
&&-k^{-1}\int_{\rho_{j-1}}^{\rho_j}\log(|(\rho_i+t_i)-x|)d\mu_0(x)\nonumber\\
&&+\int_{\rho_{i-1}}^{\rho_i}\int_{\rho_{j-1}}^{\rho_j}\log(|x-y|)d\mu_0(x)d\mu_0(y).
\end{eqnarray}

By (\ref{Eq5.43}) and (\ref{Eq9.6}), $\rho_j-\rho_i\geq \rho_j-\rho_{j-1}\geq ck^{-1}$, hence
\begin{equation}
k^{-2}\log(|(\rho_i+t_i)-(\rho_j+t_j)|)\geq k^{-2}\log(\rho_j-\rho_i),
\end{equation}
\begin{eqnarray}
&& k^{-2}\log(|(\rho_i+t_i)-(\rho_j+t_j)|)\leq k^{-2}\log(\rho_j-\rho_i+n^{-1})\nonumber\\
&=&k^{-2}\Big(\log(\rho_j-\rho_i)+\log\Big(1+\frac{1}{n(\rho_j-\rho_i)}\Big)\Big)\nonumber\\
&\leq& k^{-2}\Big(\log(\rho_j-\rho_i)+\frac{1}{n(\rho_j-\rho_i)}\Big)\leq k^{-2}\log(\rho_j-\rho_i)+\frac{C}{kn}.
\end{eqnarray}
Hence
\begin{equation}\label{Eq9.7}
    |k^{-2}\log(|(\rho_i+t_i)-(\rho_j+t_j)|)-k^{-2}\log(\rho_j-\rho_i)|\leq \frac{C}{kn}.
\end{equation}

Now by (\ref{Eq5.43}), $j\leq m_0'\leq Ck$. Hence by (\ref{Eq9.6}),
\begin{equation}
    \frac{\rho_i-\rho_{i-1}+n^{-1}}{\rho_j-\rho_i}\leq \frac{C k^{-2\slash 3} i^{-1\slash 3}}{k^{-2\slash 3}\sum_{l=i+1}^j l^{-1\slash 3}}\leq \frac{C k^{1\slash 3}}{j-i}.
\end{equation}
For any $x\in [\rho_{i-1},\rho_i]$, we have $\rho_j+t_j-x\in [\rho_j-\rho_i,\rho_j-\rho_{i-1}+n^{-1}]$, hence
\begin{equation}
    \log(|(\rho_j+t_j)-x|)\geq \log(\rho_j-\rho_i),
\end{equation}
\begin{eqnarray}
&& \log(|(\rho_j+t_j)-x|)\leq \log(\rho_j-\rho_{i-1}+n^{-1})\nonumber\\
&=& \log(\rho_j-\rho_{i})+\log\Big(1+\frac{\rho_i-\rho_{i-1}+n^{-1}}{\rho_j-\rho_i}\Big) \nonumber\\
&\leq & \log(\rho_j-\rho_{i})+\frac{\rho_i-\rho_{i-1}+n^{-1}}{\rho_j-\rho_i}\leq \log(\rho_j-\rho_i)+\frac{C k^{1\slash 3}}{j-i}.
\end{eqnarray}
Hence by (\ref{rhh}), we have
\begin{equation}
    k^{-1}\int_{\rho_{i-1}}^{\rho_i}\log(|(\rho_j+t_j)-x|)d\mu_0(x)\geq \frac{\mu_0([0,r_0]) \log(\rho_j-\rho_i)}{k(m_0'+1)},
\end{equation}
\begin{equation}
    k^{-1}\int_{\rho_{i-1}}^{\rho_i}\log(|(\rho_j+t_j)-x|)d\mu_0(x)\leq \frac{\mu_0([0,r_0]) \log(\rho_j-\rho_i)}{k(m_0'+1)}+\frac{C\mu_0([0,r_0])}{k^{2\slash 3}(m_0'+1)(j-i)}.
\end{equation}
Hence by (\ref{Eq5.17}) and (\ref{Eq9.3}), we have
\begin{eqnarray}\label{Eq9.8}
  &&  \Big|k^{-1}\int_{\rho_{i-1}}^{\rho_i}\log(|(\rho_j+t_j)-x|)d\mu_0(x)-\frac{\mu_0([0,r_0]) \log(\rho_j-\rho_i)}{k(m_0'+1)}\Big|\nonumber\\
  &\leq& \frac{C\mu_0([0,r_0])}{k^{2\slash 3}(m_0'+1)(j-i)}\leq \frac{C}{k^{5\slash 3}(j-i)}.
\end{eqnarray}
By (\ref{Eq5.15}) and (\ref{Eq5.16}), we have
\begin{equation*}
    |m_0'-k\mu_0([0,r_0])|=|n_0-k\mu_0([0,R_0(\bm{\phi})])|\leq C_0\sqrt{\frac{\epsilon}{\log(R_0(\bm{\phi}))\log\log(R_0(\bm{\phi}))}}k.
\end{equation*}
By (\ref{Eq5.11}), as $T$ only depends on $M,\epsilon$ and $k\geq K_0(M,\epsilon,\delta)$ is sufficiently large (depending on $\beta,M,\epsilon,\delta$), we have
\begin{equation*}
    |m_0'+1-k\mu_0([0,r_0])| \leq C\sqrt{\frac{\epsilon}{\log(R_0(\bm{\phi}))\log\log(R_0(\bm{\phi}))}}k.
\end{equation*}
Hence by (\ref{Eq9.3}), 
\begin{equation}\label{Eq9.9}
    \Big|\frac{\mu_0([0,r_0])}{m_0'+1}-\frac{1}{k}\Big|\leq Ck^{-1}  \sqrt{\frac{\epsilon}{\log(R_0(\bm{\phi}))\log\log(R_0(\bm{\phi}))}}.
\end{equation}
As $k\leq n^{1\slash 20000}$ and $k\geq K_0(M,\epsilon,\delta)$ is sufficiently large, we have\\ $k^{5\slash 3}(j-i)\leq Ck^{8\slash 3}\leq k^3\leq n$. Hence by (\ref{Eq9.7}), (\ref{Eq9.8}), and (\ref{Eq9.9}),
\begin{eqnarray}\label{Eq9.15}
  &&  \Big|k^{-1}\int_{\rho_{i-1}}^{\rho_i}\log(|(\rho_j+t_j)-x|)d\mu_0(x)-k^{-2}\log(|(\rho_i+t_i)-(\rho_j+t_j)|)\Big|\nonumber\\
  &\leq& \frac{1}{k}\Big|\frac{\mu_0([0,r_0])}{m_0'+1}-\frac{1}{k}\Big||\log(\rho_j-\rho_i)|+\frac{C}{k^{5\slash 3}(j-i)}\nonumber\\
  &\leq& Ck^{-2}|\log(\rho_j-\rho_i)|\sqrt{\frac{\epsilon}{\log(R_0(\bm{\phi}))\log\log(R_0(\bm{\phi}))}}+ \frac{C}{k^{5\slash 3}(j-i)}.
\end{eqnarray}

Now note that by (\ref{Eq5.43}) and (\ref{Eq9.6}), we have
\begin{equation*}
    \rho_j-\rho_i\begin{cases}
     \leq r_0\leq 10\\
     \geq ck^{-2\slash 3}\sum_{l=i+1}^j l^{-1\slash 3}\geq ck^{-1}(j-i)
    \end{cases}.
\end{equation*}
Hence 
\begin{equation*}
    \log(\rho_j-\rho_i)\begin{cases}
     \leq \log(10)\leq 3\\
     \geq \log(c)+\log\big(\frac{j-i}{k}\big)\geq -C+\log\big(\frac{j-i}{k}\big)
    \end{cases},
\end{equation*}
which leads to
\begin{eqnarray}\label{Eq9.12}
  &&|\log(\rho_j-\rho_i)|=(\log(\rho_{j}-\rho_i))_{+}+(\log(\rho_{j}-\rho_i))_{-}\nonumber\\
  &=& 2(\log(\rho_{j}-\rho_i))_{+}-\log(\rho_j-\rho_i)\leq -\log\Big(\frac{j-i}{k}\Big)+C.
\end{eqnarray}

Note that
\begin{equation}\label{Eq9.14}
    \sum_{l=1}^k \log\Big(\frac{l}{k}\Big)=\sum_{l=2}^k \log(l)-k\log(k)\geq \int_1^k \log(x)dx-k\log(k)=-k+1\geq -k.
\end{equation}
By (\ref{Eq5.43}), (\ref{Eq9.12}), and (\ref{Eq9.14}), we have
\begin{eqnarray}\label{Eq9.17}
&& \sum_{1\leq i<j\leq m_0'}|\log(\rho_j-\rho_i)|\leq -\sum_{1\leq i<j\leq m_0'}\log\Big(\frac{j-i}{k}\Big)+C (m_0')^2 \nonumber\\
&\leq& -\sum_{i=1}^{m_0'}\sum_{j=i+1}^{i+k}\log\Big(\frac{j-i}{k}\Big)+C k^2=-\sum_{i=1}^{m_0'}\sum_{l=1}^k \log\Big(\frac{l}{k}\Big)+C k^2\nonumber\\
&\leq& k m_0'+Ck^2\leq Ck^2.
\end{eqnarray}
We also note that by (\ref{Eq5.43}),
\begin{eqnarray}\label{Eq9.18}
  &&  \sum_{1\leq i<j\leq m_0'}\frac{1}{j-i}= \sum_{i=1}^{m_0'-1}\sum_{j=i+1}^{m_0'}\frac{1}{j-i}\leq m_0'\Big(\sum_{l=1}^{m_0'}\frac{1}{l}\Big)\nonumber\\
  &\leq& m_0'\Big(1+\int_1^{m_0'}\frac{1}{x}dx\Big)=m_0'(1+\log(m_0'))\leq Ck\log(k).
\end{eqnarray}
By (\ref{Eq9.15}), (\ref{Eq9.17}), and (\ref{Eq9.18}), we have
\begin{eqnarray}\label{Eq10.17}
&& \sum_{1\leq i<j\leq m_0'} \Big|k^{-1}\int_{\rho_{i-1}}^{\rho_i}\log(|(\rho_j+t_j)-x|)d\mu_0(x)-k^{-2}\log(|(\rho_i+t_i)-(\rho_j+t_j)|)\Big| \nonumber\\
&&  \leq C\sqrt{\frac{\epsilon}{\log(R_0(\bm{\phi}))\log\log(R_0(\bm{\phi}))}}+Ck^{-2\slash 3}\log(k)\nonumber\\
&&\leq C\sqrt{\frac{\epsilon}{\log(R_0(\bm{\phi}))\log\log(R_0(\bm{\phi}))}}.
\end{eqnarray}

\subparagraph{Sub-step 1.2}

In the following, we consider any $i\in [m_0'-1]$. Note that by (\ref{Eq9.6}), we have $\rho_{i+1}-\rho_i-t_i\geq ck^{-1}-n^{-1}>0$. Hence by (\ref{Eq9.6}), for any $x\in [\rho_i,\rho_{i+1}]$, as $k\geq K_0(M,\epsilon,\delta)$ is sufficiently large,
\begin{eqnarray}\label{Eq10.2}
    |\rho_i+t_i-x|&\leq& \max\{|t_i|,|\rho_i+t_i-\rho_{i+1}|\}\leq \max\{n^{-1},\rho_{i+1}-\rho_{i}\}\nonumber\\
    &\leq& \max\{n^{-1},Ck^{-2\slash 3}\}\leq C k^{-2\slash 3}<1.
\end{eqnarray}
As $\rho_{i+1}\leq r_0\leq 10$, by (\ref{Eq9.6}) and (\ref{Eq10.2}), we have
\begin{eqnarray}\label{Eq10.14}
&& \Big|\frac{1}{k}\int_{\rho_i}^{\rho_{i+1}}\log(|(\rho_i+t_i)-x|)d\mu_0(x)\Big|=\frac{1}{k}\int_{\rho_i}^{\rho_{i+1}}(-\log(|\rho_i+t_i-x|))d\mu_0(x) \nonumber\\
&&\leq \frac{\sqrt{10}}{\pi k}\int_{\rho_i}^{\rho_{i+1}}(-\log(|\rho_i+t_i-x|))dx\leq \frac{2}{k}\int_{\rho_i-\rho_{i+1}+t_i}^{t_i}(-\log(|s|))ds\nonumber\\
&&\leq \frac{2}{k}\int_{-Ck^{-2\slash 3}}^{C k^{-2\slash 3}}(-\log(|s|))ds\leq Ck^{-5\slash 3}\log(k).
\end{eqnarray}

For any $x\in [\rho_{i-1},\rho_{i}]$ and $y\in [\rho_{i},\rho_{i+1}]$, by (\ref{Eq9.6}), we have
\begin{equation}\label{Eq10.1}
    |x-y|=y-x\leq \rho_{i+1}-\rho_{i-1}\leq Ck^{-2\slash 3}<1.
\end{equation}
For any $y\in [\rho_i,\rho_{i+1}]$, as $y-\rho_i\leq y-\rho_{i-1}\leq \rho_{i+1}-\rho_{i-1}<1$, we have
\begin{eqnarray}\label{Eq10.3}
&&\int_{\rho_{i-1}}^{\rho_i}(-\log(|x-y|))dx=\int_{y-\rho_i}^{y-\rho_{i-1}} (-\log(s))ds\nonumber\\
&=& -\log(y-\rho_{i-1})(y-\rho_{i-1})+\log(y-\rho_i)(y-\rho_i)+\rho_i-\rho_{i-1}\nonumber\\
&\leq& -(\rho_{i+1}-\rho_{i-1})\log(y-\rho_{i-1})+\rho_i-\rho_{i-1}\nonumber\\
&\leq& Ck^{-2\slash 3}(1-\log(y-\rho_{i-1})).
\end{eqnarray}
By (\ref{Eq9.6}), (\ref{Eq10.1}), and (\ref{Eq10.3}), as $\rho_{i+1}\leq r_0\leq 10$,
\begin{eqnarray}\label{Eq10.11}
&& \Big|\int_{\rho_{i-1}}^{\rho_i}\int_{\rho_{i}}^{\rho_{i+1}}\log(|x-y|)d\mu_0(x)d\mu_0(y)\Big|\nonumber\\
&=&\int_{\rho_{i-1}}^{\rho_i}\int_{\rho_{i}}^{\rho_{i+1}}(-\log(|x-y|))d\mu_0(x)d\mu_0(y) \nonumber\\
&\leq& \frac{10}{\pi^2}\int_{\rho_i}^{\rho_{i+1}}dy\int_{\rho_{i-1}}^{\rho_i}(-\log(|x-y|))dx\leq Ck^{-2\slash 3}\int_{\rho_i}^{\rho_{i+1}}(1-\log(y-\rho_{i-1}))dy\nonumber\\
&=&Ck^{-2\slash 3} (2(\rho_{i+1}-\rho_i)-(\rho_{i+1}-\rho_{i-1})\log(\rho_{i+1}-\rho_{i-1})\nonumber\\
&&+(\rho_i-\rho_{i-1})\log(\rho_i-\rho_{i-1}))\leq C k^{-4\slash 3}\log(k).
\end{eqnarray}

In the following, we consider any $1\leq i<j\leq m_0'$ such that $j\geq i+2$. For any $x\in [\rho_{j-1},\rho_j]$, by (\ref{Eq9.6}), we have $x-(\rho_i+t_i)\leq \rho_j-\rho_i$ and 
\begin{equation*}
        x-(\rho_i+t_i)
        \geq \rho_{j-1}-\rho_i-n^{-1}\geq \rho_{i+1}-\rho_{i}-n^{-1} \geq ck^{-1}-n^{-1}>0.
\end{equation*}
Hence by (\ref{Eq9.6}) and the inequality $\log(1+t)\geq t\slash (1+t)$ for any $t>-1$, 
\begin{equation}\label{Eq10.4}
    \log(|(\rho_i+t_i)-x|)= \log(x-(\rho_i+t_i))\leq \log(\rho_j-\rho_i),
\end{equation}
\begin{eqnarray}\label{Eq10.7}
 && \log(|(\rho_i+t_i)-x|)\geq \log(\rho_{j-1}-\rho_i-n^{-1})\nonumber\\
 &=& \log(\rho_j-\rho_i)+\log\Big(1-\frac{\rho_{j}-\rho_{j-1}+n^{-1}}{\rho_j-\rho_i}\Big)\nonumber\\
 &\geq& \log(\rho_j-\rho_i)-\frac{\rho_j-\rho_{j-1}+n^{-1}}{\rho_{j-1}-\rho_i-n^{-1}}\geq \log(\rho_j-\rho_i)-\frac{4(\rho_j-\rho_{j-1})}{\rho_{j-1}-\rho_i}\nonumber\\
 &\geq& \log(\rho_j-\rho_i)-\frac{C k^{-2\slash 3} j^{-1\slash 3}}{k^{-2\slash 3}\sum_{l=i+1}^{j-1}l^{-1\slash 3}}\geq \log(\rho_j-\rho_i)-\frac{C}{j-i-1},\nonumber\\
 &&
\end{eqnarray}
where we use the fact that $\min\{\rho_j-\rho_{j-1},\rho_{j-1}-\rho_i\}\geq ck^{-1}\geq 2n^{-1}$ in the third line of (\ref{Eq10.7}). Hence by (\ref{rhh}), we have
\begin{equation}
    k^{-1}\int_{\rho_{j-1}}^{\rho_j}\log(|(\rho_i+t_i)-x|)d\mu_0(x)\leq \log(\rho_j-\rho_i)\cdot\frac{\mu_0([0,r_0])}{(m_0'+1)k},
\end{equation}
\begin{equation}
 k^{-1}\int_{\rho_{j-1}}^{\rho_j}\log(|(\rho_i+t_i)-x|)d\mu_0(x)\geq \Big(\log(\rho_j-\rho_i)-\frac{C}{j-i-1}\Big)\cdot\frac{\mu_0([0,r_0])}{(m_0'+1)k}.
\end{equation}
Hence by (\ref{Eq5.17}) and (\ref{Eq9.3}), we have
\begin{eqnarray}\label{Eq10.9}
  &&  \Big|-k^{-1}\int_{\rho_{j-1}}^{\rho_j}\log(|(\rho_i+t_i)-x|)d\mu_0(x)+\log(\rho_j-\rho_i)\cdot\frac{\mu_0([0,r_0])}{(m_0'+1)k}\Big|  \nonumber\\
  &\leq&\frac{C}{j-i-1}\cdot \frac{\mu_0([0,r_0])}{(m_0'+1)k}\leq \frac{C}{k^2(j-i-1)}.
\end{eqnarray}

Now for any $x\in [\rho_{i-1},\rho_i]$ and $y\in [\rho_{j-1},\rho_j]$, we have
\begin{equation*}
    0<\rho_{j-1}-\rho_i\leq y-x\leq \rho_j-\rho_{i-1},
\end{equation*}
\begin{equation}\label{Eq10.6}
    \log(|x-y|)=\log(y-x)\begin{cases}
     \geq \log(\rho_j-\rho_i)+\log\big(\frac{\rho_{j-1}-\rho_i}{\rho_j-\rho_i}\big)\\
     \leq \log(\rho_j-\rho_i)+\log\big(\frac{\rho_j-\rho_{i-1}}{\rho_j-\rho_i}\big)
    \end{cases}.
\end{equation}
By (\ref{Eq5.43}), (\ref{Eq9.6}), and the inequality $t\slash (1+t)\leq \log(1+t)\leq t$ for any $t>-1$, 
\begin{eqnarray}\label{Eq10.7n}
 \log\Big(\frac{\rho_{j-1}-\rho_i}{\rho_j-\rho_i}\Big)&=&\log\Big(1-\frac{\rho_j-\rho_{j-1}}{\rho_j-\rho_i}\Big)\geq -\frac{\rho_j-\rho_{j-1}}{\rho_{j-1}-\rho_i}\nonumber\\
 &\geq& -\frac{Ck^{-2\slash 3}j^{-1\slash 3}}{k^{-2\slash 3}\sum_{l=i+1}^{j-1}l^{-1\slash 3}}\geq -\frac{C}{j-i-1}, 
\end{eqnarray}
\begin{eqnarray}\label{Eq10.8}
  &&\log\Big(\frac{\rho_j-\rho_{i-1}}{\rho_j-\rho_i}\Big)=\log\Big(1+\frac{\rho_i-\rho_{i-1}}{\rho_j-\rho_i}\Big)\leq \frac{\rho_i-\rho_{i-1}}{\rho_j-\rho_i}\nonumber\\
 &\leq& \frac{Ck^{-2\slash 3}i^{-1\slash 3}}{k^{-2\slash 3}\sum_{l=i+1}^j l^{-1\slash 3}}\leq \frac{C j^{1\slash 3}}{j-i-1}\leq \frac{C k^{1\slash 3}}{j-i-1}.
\end{eqnarray}
By (\ref{Eq5.17}), (\ref{Eq9.3}), (\ref{rhh}), (\ref{Eq9.6}), and (\ref{Eq10.6})-(\ref{Eq10.8}), 
\begin{eqnarray}\label{Eq10.10}
 && \Big|\int_{\rho_{i-1}}^{\rho_i}\int_{\rho_{j-1}}^{\rho_j}\log(|x-y|)d\mu_0(x)d\mu_0(y)-\log(\rho_j-\rho_i)\cdot\frac{\mu_0([0,r_0])^2}{(m_0'+1)^2}\Big|\nonumber\\
 &\leq& \frac{Ck^{1\slash 3}}{j-i-1}\cdot \frac{\mu_0([0,r_0])^2}{(m_0'+1)^2}\leq \frac{C}{k^{5\slash 3}(j-i-1)}.
\end{eqnarray}

By (\ref{Eq5.17}), (\ref{Eq9.3}), (\ref{Eq9.9}), (\ref{Eq10.9}), and (\ref{Eq10.10}), for any $1\leq i<j\leq m_0'$ such that $j\geq i+2$, we have
\begin{eqnarray}\label{Eq10.12}
&&\Big|-k^{-1}\int_{\rho_{j-1}}^{\rho_j}\log(|(\rho_i+t_i)-x|)d\mu_0(x)+\int_{\rho_{i-1}}^{\rho_i}\int_{\rho_{j-1}}^{\rho_j}\log(|x-y|)d\mu_0(x)d\mu_0(y)\Big|\nonumber\\
&&\leq |\log(\rho_j-\rho_i)|\cdot \frac{\mu_0([0,r_0])}{m_0'+1}\cdot \Big|\frac{\mu_0([0,r_0])}{m_0'+1}-\frac{1}{k}\Big|+\frac{C}{k^{5\slash 3}(j-i-1)}\nonumber\\
&&\leq Ck^{-2}|\log(\rho_j-\rho_i)|  \sqrt{\frac{\epsilon}{\log(R_0(\bm{\phi}))\log\log(R_0(\bm{\phi}))}}+\frac{C}{k^{5\slash 3}(j-i)}.
\end{eqnarray}

By (\ref{Eq5.11}), (\ref{Eq5.43}), (\ref{Eq9.17}), (\ref{Eq9.18}), (\ref{Eq10.14}), (\ref{Eq10.11}), and (\ref{Eq10.12}), we have
\begin{eqnarray}\label{Eq10.18}
&& \sum_{1\leq i<j\leq m_0'}\Big|-k^{-1}\int_{\rho_{j-1}}^{\rho_j}\log(|(\rho_i+t_i)-x|)d\mu_0(x)\nonumber\\
&& \quad\quad\quad\quad\quad +\int_{\rho_{i-1}}^{\rho_i}\int_{\rho_{j-1}}^{\rho_j}\log(|x-y|)d\mu_0(x)d\mu_0(y)\Big|\nonumber\\
&\leq& Ck^{-4\slash 3}\log(k)m_0'+Ck^{-5\slash 3}\sum_{1\leq i<j\leq m_0'}\frac{1}{j-i}\nonumber\\
&& +Ck^{-2} \sqrt{\frac{\epsilon}{\log(R_0(\bm{\phi}))\log\log(R_0(\bm{\phi}))}}\sum_{1\leq i<j\leq m_0'} |\log(\rho_j-\rho_i)|\nonumber\\
&\leq& Ck^{-1\slash 3}\log(k)+C\sqrt{\frac{\epsilon}{\log(R_0(\bm{\phi}))\log\log(R_0(\bm{\phi}))}}\nonumber\\
&\leq& C\sqrt{\frac{\epsilon}{\log(R_0(\bm{\phi}))\log\log(R_0(\bm{\phi}))}}.
\end{eqnarray}

By (\ref{Eq10.16}), (\ref{Eq10.17}), and (\ref{Eq10.18}), we conclude that
\begin{equation}\label{Eq12.9}
    \sum_{1\leq i<j\leq m_0'}|\mathscr{P}_{i,j}|\leq C\sqrt{\frac{\epsilon}{\log(R_0(\bm{\phi}))\log\log(R_0(\bm{\phi}))}}.
\end{equation}

\subparagraph{Sub-step 1.3}

In the following, we consider any $i\in [m_0']$. By (\ref{Eq10.20}) and (\ref{Eq10.19}), we have
\begin{eqnarray}\label{Eq10.25}
\mathscr{P}_{i,i}&=&-\frac{2}{k}\int_{\rho_{i-1}}^{\rho_{i}}\log(|(\rho_i+t_i)-x|)d\mu_0(x)\nonumber\\
&&+\int_{\rho_{i-1}}^{\rho_{i}}\int_{\rho_{i-1}}^{\rho_{i}}\log(|x-y|)d\mu_0(x)d\mu_0(y).
\end{eqnarray}
For any $x\in [\rho_{i-1},\rho_i)$, by (\ref{Eq9.6}), we have
\begin{eqnarray}
 \rho_i+t_i-x\begin{cases}
  >t_i\geq 0\\
  \leq \rho_i-\rho_{i-1}+n^{-1}\leq Ck^{-2\slash 3}<1
 \end{cases}.
\end{eqnarray}
As $\rho_i\leq r_0\leq 10$ and $\rho_i-\rho_{i-1}+t_i\leq \rho_i-\rho_{i-1}+n^{-1}\leq C k^{-2\slash 3}<1$, we have
\begin{eqnarray}\label{Eq10.27}
&& \Big|-\frac{2}{k}\int_{\rho_{i-1}}^{\rho_{i}}\log(|(\rho_i+t_i)-x|)d\mu_0(x)\Big|=\frac{2}{k}\int_{\rho_{i-1}}^{\rho_i}(-\log(\rho_i+t_i-x))d\mu_0(x) \nonumber\\
 && \leq \frac{\sqrt{10}}{\pi}\cdot \frac{2}{k} \int_{\rho_{i-1}}^{\rho_i}(-\log(\rho_i+t_i-x))dx \leq \frac{C}{k}\int_{t_i}^{\rho_i-\rho_{i-1}+t_i}(-\log(s))ds\nonumber\\
 &&\leq \frac{C}{k}\int_{0}^{Ck^{-2\slash 3}}(-\log(s))ds= \frac{C}{k}\Big(-Ck^{-2\slash 3}\log(Ck^{-2\slash 3})+Ck^{-2\slash 3}\Big)\nonumber\\
 &&\leq C k^{-5\slash 3}\log(k).
\end{eqnarray}

For any $x\in [\rho_{i-1},\rho_i]$ and $y\in [\rho_{i-1},\rho_i]$ such that $x\neq y$, we have that $|x-y|\leq \rho_i-\rho_{i-1}\leq C k^{-2\slash 3}<1$. Hence
\begin{eqnarray}\label{Eq10.21}
   && \Big|\int_{\rho_{i-1}}^{\rho_{i}}\int_{\rho_{i-1}}^{\rho_{i}}\log(|x-y|)d\mu_0(x)d\mu_0(y)\Big|=\int_{\rho_{i-1}}^{\rho_{i}}\int_{\rho_{i-1}}^{\rho_{i}}(-\log(|x-y|))d\mu_0(x)d\mu_0(y) \nonumber\\
    && \leq \frac{10}{\pi^2}\int_{\rho_{i-1}}^{\rho_{i}}\int_{\rho_{i-1}}^{\rho_{i}}(-\log(|x-y|))dxdy.
\end{eqnarray}
For any $x\in (\rho_{i-1},\rho_i)$, 
\begin{eqnarray}
&& \int_{\rho_{i-1}}^{\rho_{i}}(-\log(|x-y|))dy=-\int_{\rho_{i-1}-x}^{\rho_i-x}\log(|s|)ds \nonumber\\
&=& -(\rho_i-x)\log(\rho_i-x)+(\rho_{i-1}-x)\log(x-\rho_{i-1})+(\rho_i-\rho_{i-1})\nonumber\\
&\leq& (\rho_i-\rho_{i-1})(-\log(\rho_i-x)-\log(x-\rho_{i-1})+1).
\end{eqnarray}
Hence by (\ref{Eq5.43}) and (\ref{Eq9.6}),
\begin{eqnarray}\label{Eq10.23}
&&\int_{\rho_{i-1}}^{\rho_{i}}\int_{\rho_{i-1}}^{\rho_{i}}(-\log(|x-y|))dxdy\nonumber\\
&\leq& (\rho_i-\rho_{i-1})\Big(-\int_{\rho_{i-1}}^{\rho_i}\log(\rho_i-x)dx-\int_{\rho_{i-1}}^{\rho_i}\log(x-\rho_{i-1})dx+\rho_i-\rho_{i-1}\Big) \nonumber\\
&=& (\rho_i-\rho_{i-1})(-2(\rho_i-\rho_{i-1})\log(\rho_i-\rho_{i-1})+3(\rho_i-\rho_{i-1}))\nonumber\\
&\leq& C k^{-4\slash 3} \log(k).
\end{eqnarray}
By (\ref{Eq10.21}) and (\ref{Eq10.23}),
\begin{equation}\label{Eq10.28}
    \Big|\int_{\rho_{i-1}}^{\rho_{i}}\int_{\rho_{i-1}}^{\rho_{i}}\log(|x-y|)d\mu_0(x)d\mu_0(y)\Big|\leq C k^{-4\slash 3} \log(k).
\end{equation}

By (\ref{Eq5.43}), (\ref{Eq10.25}), (\ref{Eq10.27}), and (\ref{Eq10.28}), we conclude that 
\begin{equation}\label{Eq12.10}
    \sum_{i=1}^{m_0'}|\mathscr{P}_{i,i}|\leq Cm_0' k^{-4\slash 3}\log(k)\leq C k^{-1\slash 3} \log(k).
\end{equation}

\subparagraph{Sub-step 1.4}

In this sub-step, we bound $\mathscr{P}_{0,0}$. For any $x,y\in [\rho_{m_0'},r_0]$ such that $x\neq y$, by (\ref{Eq9.3}) and (\ref{Eq9.6}), $|x-y|\leq r_0-\rho_{m_0'}\leq C k^{-1}<1$. Hence 
\begin{eqnarray}
|\mathscr{P}_{0,0}|&=&\int_{\rho_{m_0'}}^{r_0}\int_{\rho_{m_0'}}^{r_0} (-\log(|x-y|))d\mu_0(x)d\mu_0(y) \nonumber\\
&\leq& \frac{10}{\pi^2} \int_{\rho_{m_0'}}^{r_0}\int_{\rho_{m_0'}}^{r_0} (-\log(|x-y|)) dx dy.
\end{eqnarray}
For any $x\in (\rho_{m_0'},r_0)$, we have
\begin{eqnarray*}
 &&   \int_{\rho_{m_0'}}^{r_0} (-\log(|x-y|))  dy=-\int_{\rho_{m_0'}-x}^{r_0-x} \log(|s|)ds\nonumber\\
 &=&-(r_0-x)\log(r_0-x)+(\rho_{m_0'}-x)\log(x-\rho_{m_0'})+(r_0-\rho_{m_0'})\nonumber\\
 &\leq& (r_0-\rho_{m_0'})(-\log(r_0-x)-\log(x-\rho_{m_0'})+1).
\end{eqnarray*}
Hence by (\ref{Eq5.43}) and (\ref{Eq9.6}),
\begin{eqnarray}\label{Eq12.11}
&&|\mathscr{P}_{0,0}|\leq C(r_0-\rho_{m_0'})
\Big(-\int_{\rho_{m_0'}}^{r_0}\log(r_0-x)dx-\int_{\rho_{m_0'}}^{r_0}\log(x-\rho_{m_0'})dx+(r_0-\rho_{m_0'})\Big)\nonumber\\
&&= C(r_0-\rho_{m_0'})(-2(r_0-\rho_{m_0'})\log(r_0-\rho_{m_0'})+3(r_0-\rho_{m_0'}))\leq Ck^{-2}\log(k).
\end{eqnarray}

\subparagraph{Sub-step 1.5} 

In this sub-step, we bound $\mathscr{P}_{0,i}$ for every $i\in [m_0']$. In the following, we consider any $i\in [m_0']$. Note that
\begin{eqnarray}\label{Eq12.6}
\mathscr{P}_{0,i}&=&-\frac{1}{k}\int_{\rho_{m_0'}}^{r_0}\log(|\rho_i+t_i-x|)d\mu_0(x)\nonumber\\
&&+\int_{\rho_{i-1}}^{\rho_i}\int_{\rho_{m_0'}}^{r_0}\log(|x-y|)d\mu_0(x)d\mu_0(y).
\end{eqnarray}

We have
\begin{eqnarray}\label{Eq12.1}
&& \int_{\rho_{m_0'}}^{r_0}|\log(|\rho_{m_0'}+t_{m_0'}-x|)|dx=\int_{\rho_{m_0'}+t_{m_0'}-r_0}^{t_{m_0'}} |\log(|s|)|ds\leq \int_{-(r_0-\rho_{m_0'})}^{n^{-1}} |\log(|s|)|ds \nonumber\\
&&\leq \int_{-C k^{-1}}^{C k^{-1}}|\log(|s|)|ds=-\int_{-C k^{-1}}^{C k^{-1}}\log(|s|)ds=-2Ck^{-1}\log(C k^{-1})+2Ck^{-1}\nonumber\\
&&\leq Ck^{-1}\log(k).
\end{eqnarray}
If $i\in [m_0'-1]$, for any $x\in [\rho_{m_0'},r_0]$, by (\ref{Eq9.6}), we have
\begin{equation*}
    \rho_i+t_i-x\begin{cases}
     \geq -r_0\geq -10\\
     \leq n^{-1}-(\rho_{m_0'}-\rho_i)\leq n^{-1}-ck^{-1}\leq -ck^{-1}
    \end{cases},
\end{equation*}
hence $ck^{-1}\leq|\rho_i+t_i-x|\leq 10$ and $|\log(|\rho_i+t_i-x|)|\leq C\log(k)$. Hence if $i\in [m_0'-1]$, 
\begin{equation}\label{Eq12.2}
    \int_{\rho_{m_0'}}^{r_0}|\log(|\rho_i+t_i-x|)| dx\leq C\log(k)(r_0-\rho_{m_0'})\leq Ck^{-1}\log(k).
\end{equation}
By (\ref{Eq12.1}) and (\ref{Eq12.2}), for any $i\in [m_0']$, we have
\begin{equation}\label{Eq12.7}
    \Big|\int_{\rho_{m_0'}}^{r_0}\log(|\rho_i+t_i-x|)d\mu_0(x)\Big|\leq C\int_{\rho_{m_0'}}^{r_0}|\log(|\rho_i+t_i-x|)|dx\leq Ck^{-1}\log(k).
\end{equation}

For any $x\in [\rho_{m_0'-1},\rho_{m_0'}]$ and $y\in [\rho_{m_0'},r_0]$, by (\ref{Eq9.3}) and (\ref{Eq9.6}), we have
\begin{equation}\label{Eq12.3}
    0\leq y-x\leq r_0-\rho_{m_0'-1}=(r_0-\rho_{m_0'})+(\rho_{m_0'}-\rho_{m_0'-1})\leq C k^{-1}<1. 
\end{equation}
For any $x\in [\rho_{m_0'-1},\rho_{m_0'}]$,
\begin{eqnarray}
  &&  \int_{\rho_{m_0'}}^{r_0} \log(|x-y|) dy =\int_{\rho_{m_0'}-x}^{r_0-x}\log(|s|) ds \nonumber\\
  &=&(r_0-x)\log(r_0-x)-(\rho_{m_0'}-x)\log(\rho_{m_0'}-x)-(r_0-\rho_{m_0'})\nonumber\\
  &\geq& (r_0-\rho_{m_0'-1})\log(r_0-x)-(r_0-\rho_{m_0'}).
\end{eqnarray}
By (\ref{Eq5.43}) and (\ref{Eq9.6}), $r_0-\rho_{m_0'}\geq ck^{-1}$, hence
\begin{eqnarray}\label{Eq12.4}
 && \int_{\rho_{m_0'-1}}^{\rho_{m_0'}}\int_{\rho_{m_0'}}^{r_0}|\log(|x-y|)|dxdy=-\int_{\rho_{m_0'-1}}^{\rho_{m_0'}}\int_{\rho_{m_0'}}^{r_0} \log(|x-y|) dxdy \nonumber\\
 &\leq& -(r_0-\rho_{m_0'-1})\int_{\rho_{m_0'-1}}^{\rho_{m_0'}}\log(r_0-x)dx+(r_0-\rho_{m_0'})(\rho_{m_0'}-\rho_{m_0'-1})\nonumber\\
 &\leq& Ck^{-2}\log(k).
\end{eqnarray}
If $i\in [m_0'-1]$, for any $x\in [\rho_{i-1},\rho_i]$ and $y\in [\rho_{m_0'},r_0]$, by (\ref{Eq5.43}) and (\ref{Eq9.6}),
\begin{equation*}
   10\geq r_0\geq y-x\geq \rho_{m_0'}-\rho_{m_0'-1}\geq ck^{-1}, \text{ hence } |\log(|x-y|)|\leq C\log(k).
\end{equation*}
Hence by (\ref{Eq9.6}), if $i\in [m_0'-1]$,
\begin{equation}\label{Eq12.5}
    \int_{\rho_{i-1}}^{\rho_i}\int_{\rho_{m_0'}}^{r_0}|\log(|x-y|)|dxdy\leq C\log(k)(\rho_{i}-\rho_{i-1})(r_0-\rho_{m_0'})\leq Ck^{-5\slash 3}\log(k).
\end{equation}
By (\ref{Eq12.4}) and (\ref{Eq12.5}), for any $i\in [m_0']$,
\begin{eqnarray}\label{Eq12.8}
\Big|\int_{\rho_{i-1}}^{\rho_i}\int_{\rho_{m_0'}}^{r_0}\log(|x-y|)d\mu_0(x)d\mu_0(y)\Big|&\leq& C \int_{\rho_{i-1}}^{\rho_i}\int_{\rho_{m_0'}}^{r_0}|\log(|x-y|)|dx dy\nonumber\\
&\leq& C k^{-5\slash 3} \log(k).
\end{eqnarray}

By (\ref{Eq12.6}), (\ref{Eq12.7}), and (\ref{Eq12.8}), we have $|\mathscr{P}_{0,i}|\leq Ck^{-5\slash 3}\log(k)$ for any $i\in [m_0']$. Hence by (\ref{Eq5.43}),
\begin{equation}\label{Eq12.12}
    \sum_{i=1}^{m_0'} |\mathscr{P}_{0,i}|\leq Cm_0' k^{-5\slash 3}\log(k)\leq C k^{-2\slash 3}\log(k).
\end{equation}

By (\ref{Eq12.13}), (\ref{Eq12.9}), (\ref{Eq12.10}), (\ref{Eq12.11}), and (\ref{Eq12.12}), we conclude that
\begin{equation}\label{Eq12.36}
    \Big| \int_{\mathbb{R}^2\backslash \Delta}\log(|x-y|)d\mathscr{W}_1(x)d\mathscr{W}_1(y)\Big|\leq C\sqrt{\frac{\epsilon}{\log(R_0(\bm{\phi}))\log\log(R_0(\bm{\phi}))}}.
\end{equation}

\paragraph{Step 2}

In this step, we bound $\int_{\mathbb{R}^2\backslash \Delta}\log(|x-y|)d\mathscr{W}_2(x)d\mathscr{W}_2(y)$. Note that
\begin{equation}\label{Eq12.26}
    \int_{\mathbb{R}^2\backslash \Delta}\log(|x-y|)d\mathscr{W}_2(x)d\mathscr{W}_2(y)=2\sum_{m_0'+1\leq i<j\leq n_0}\mathscr{P}_{i,j}+\sum_{i=m_0'+1}^{n_0}\mathscr{P}_{i,i}.
\end{equation}

\subparagraph{Sub-step 2.1}

In the following, we consider any $i,j\in [m_0'+1,n_0]\cap\mathbb{Z}$ such that $i<j$. Note that
\begin{eqnarray}\label{Eq12.18}
 \mathscr{P}_{i,j}&=&k^{-2}\log(|(\rho_i+t_i)-(\rho_j+t_j)|)-k^{-1}\int_{\rho_i}^{\rho_{i+1}}\log(|(\rho_j+t_j)-x|)d\mu_0(x)\nonumber\\
 && -k^{-1}\int_{\rho_j}^{\rho_{j+1}}\log(|(\rho_i+t_i)-x|)d\mu_0(x)\nonumber\\
 && +\int_{\rho_i}^{\rho_{i+1}}\int_{\rho_j}^{\rho_{j+1}}\log(|x-y|)d\mu_0(x)d\mu_0(y).
\end{eqnarray}

For any $x\in [\rho_j,\rho_{j+1}]$, by (\ref{Eq5.11}) and (\ref{Eq5.20}), we have
\begin{equation}
    x-(\rho_i+t_i)\geq \rho_j-\rho_i-n^{-1}\geq c R_0(\bm{\phi})^{-1\slash 2}k^{-1}-n^{-1}\geq c R_0(\bm{\phi})^{-1\slash 2}k^{-1}>0.
\end{equation}
For any $x\in [\rho_j,\rho_{j+1}]$, we let
\begin{equation}
    g_{i,j}(x):=\log(|(\rho_i+t_i)-(\rho_j+t_j)|)-\log(|(\rho_i+t_i)-x|).
\end{equation}
For any $x\in [\rho_j,\rho_{j+1}]$, by (\ref{Eq5.20}), (\ref{Eq5.15.1}), as $\log(1+t)\leq t$ for any $t>-1$, 
\begin{eqnarray}
    && g_{i,j}(x) \geq  \log((\rho_j+t_j)-(\rho_i+t_i))-\log(\rho_{j+1}-(\rho_i+t_i)) \nonumber\\
    &=& -\log\Big(1+\frac{\rho_{j+1}-(\rho_j+t_j)}{(\rho_j+t_j)-(\rho_i+t_i)}\Big)\geq -\frac{\rho_{j+1}-(\rho_j+t_j)}{(\rho_j+t_j)-(\rho_i+t_i)}\nonumber\\
    &\geq& -\frac{C k^{-1}}{k^{-1}R_0(\bm{\phi})^{-1\slash 2}(j-i)}\geq -\frac{C R_0(\bm{\phi})^{1\slash 2}}{j-i},
\end{eqnarray}
\begin{eqnarray}
 && g_{i,j}(x)\leq \log((\rho_j+t_j)-(\rho_i+t_i))-\log(\rho_{j}-(\rho_i+t_i)) \nonumber\\
 &=& \log\Big(1+\frac{t_j}{\rho_j-(\rho_i+t_i)}\Big)\leq \frac{t_j}{\rho_j-(\rho_i+t_i)}\nonumber\\
 &\leq& \frac{n^{-1}}{ck^{-1}R_0(\bm{\phi})^{-1\slash 2}(j-i)}\leq \frac{C R_0(\bm{\phi})^{1\slash 2}}{j-i},
\end{eqnarray}
hence $|g_{i,j}(x)|\leq C R_0(\bm{\phi})^{1\slash 2}(j-i)^{-1}$. By (\ref{Eq5.39}), we have
\begin{eqnarray}\label{Eq12.14}
&& \Big|k^{-2}\log(|(\rho_i+t_i)-(\rho_j+t_j)|)-k^{-1}\int_{\rho_j}^{\rho_{j+1}}\log(|(\rho_i+t_i)-x|)d\mu_0(x)\Big| \nonumber\\
&=& k^{-1}\Big|\int_{\rho_j}^{\rho_{j+1}}g_{i,j}(x)d\mu_0(x)\Big|\leq \frac{CR_0(\bm{\phi})^{1\slash 2}}{k^2(j-i)} .
\end{eqnarray}

For any $x\in [\rho_i,\rho_{i+1})$, by (\ref{Eq5.15.1}), 
\begin{equation}\label{Eq12.16}
    0<\rho_{i+1}+t_{i+1}-x\leq \rho_{i+1}-\rho_i+n^{-1}\leq Ck^{-1}<1.
\end{equation}
As $\rho_{i+1}\leq R_0(\bm{\phi})$, by (\ref{Eq12.16}), we have
\begin{eqnarray}\label{Eq12.19}
&& \Big|\int_{\rho_{i}}^{\rho_{i+1}}\log(|(\rho_{i+1}+t_{i+1})-x|)d\mu_0(x)\Big|=\int_{\rho_{i}}^{\rho_{i+1}}(-\log(\rho_{i+1}+t_{i+1}-x))d\mu_0(x) \nonumber\\
&\leq& \sqrt{R_0(\bm{\phi})}\int_{\rho_{i}}^{\rho_{i+1}}(-\log(\rho_{i+1}+t_{i+1}-x))dx=\sqrt{R_0(\bm{\phi})}\int_{t_{i+1}}^{\rho_{i+1}-\rho_i+t_{i+1}} (-\log(s))ds\nonumber\\
&\leq & \sqrt{R_0(\bm{\phi})}\int_0^{Ck^{-1}}(-\log(s))ds=\sqrt{R_0(\bm{\phi})}(-Ck^{-1}\log(Ck^{-1})+Ck^{-1})\nonumber\\
&\leq& C\sqrt{R_0(\bm{\phi})}k^{-1}\log(k). 
\end{eqnarray}

For any $x\in [\rho_i,\rho_{i+1})$ and $y\in [\rho_{i+1},\rho_{i+2}]$, by (\ref{Eq5.15.1}), 
\begin{equation}\label{Eq12.17}
    0<y-x\leq \rho_{i+2}-\rho_i\leq C k^{-1}<1.
\end{equation}
As $\rho_{i+2}\leq R_0(\bm{\phi})$, by (\ref{Eq12.17}), for any $x\in [\rho_i,\rho_{i+1})$, we have
\begin{eqnarray}
  &&  \int_{\rho_{i+1}}^{\rho_{i+2}}|\log(|x-y|)|d\mu_0(y)= \int_{\rho_{i+1}}^{\rho_{i+2}}(-\log(y-x))d\mu_0(y) \nonumber\\
  &\leq& \sqrt{R_0(\bm{\phi})}\int_{\rho_{i+1}}^{\rho_{i+2}}(-\log(y-\rho_{i+1}))dy=\sqrt{R_0(\bm{\phi})}\int_{0}^{\rho_{i+2}-\rho_{i+1}}(-\log(s))ds\nonumber\\
  &\leq & \sqrt{R_0(\bm{\phi})}\int_{0}^{C k^{-1}}(-\log(s))ds\leq C\sqrt{R_0(\bm{\phi})}k^{-1}\log(k).
\end{eqnarray}
Hence by (\ref{Eq5.39}), we have
\begin{eqnarray}\label{Eq12.20}
&&\Big|\int_{\rho_i}^{\rho_{i+1}}\int_{\rho_{i+1}}^{\rho_{i+2}}\log(|x-y|)d\mu_0(x)d\mu_0(y)\Big|\leq \int_{\rho_i}^{\rho_{i+1}}d\mu_0(x)\int_{\rho_{i+1}}^{\rho_{i+2}}|\log(|x-y|)|d\mu_0(y)\nonumber\\
&&\leq C\sqrt{R_0(\bm{\phi})} k^{-1}\log(k) \mu_0([\rho_i,\rho_{i+1}])\leq C\sqrt{R_0(\bm{\phi})}k^{-2}\log(k).
\end{eqnarray}

If $j\geq i+2$, for any $x\in [\rho_i,\rho_{i+1})$ and $y\in [\rho_j,\rho_{j+1}]$, by (\ref{Eq5.20}), we have
\begin{equation}
 \min\{y-x,\rho_j+t_j-x\}\geq \rho_j-\rho_{i+1}\geq c R_0(\bm{\phi})^{-1\slash 2}k^{-1}(j-i-1)>0.
\end{equation}
Hence by (\ref{Eq5.15.1}) and the inequality $\log(1+t)\leq t$ for any $t>-1$,
\begin{eqnarray}
 && \log(|x-y|)-\log(|\rho_j+t_j-x|)=\log\Big(1+\frac{y-(\rho_j+t_j)}{\rho_j+t_j-x}\Big) \nonumber\\
 &\leq& \frac{y-(\rho_j+t_j)}{\rho_j+t_j-x}\leq \frac{\rho_{j+1}-\rho_j}{\rho_j+t_j-x}\leq \frac{Ck^{-1}}{ R_0(\bm{\phi})^{-1\slash 2} k^{-1} (j-i-1)}\nonumber\\
 &\leq& \frac{C\sqrt{R_0(\bm{\phi})}}{j-i-1},
\end{eqnarray}
\begin{eqnarray}
 && \log(|\rho_j+t_j-x|)-\log(|x-y|)=\log\Big(1+\frac{\rho_j+t_j-y}{y-x}\Big) \nonumber\\
 &\leq& \frac{\rho_j+t_j-y}{y-x}\leq \frac{n^{-1}}{R_0(\bm{\phi})^{-1\slash 2} k^{-1} (j-i-1)}\leq  \frac{C\sqrt{R_0(\bm{\phi})}}{j-i-1}.
\end{eqnarray}
Hence by (\ref{Eq5.39}), if $j\geq i+2$,
\begin{eqnarray}\label{Eq12.21}
&& \Big|-k^{-1}\int_{\rho_i}^{\rho_{i+1}}\log(|(\rho_j+t_j)-x|)d\mu_0(x)+\int_{\rho_i}^{\rho_{i+1}}\int_{\rho_j}^{\rho_{j+1}}\log(|x-y|)d\mu_0(x)d\mu_0(y)\Big| \nonumber\\
&&\leq \int_{\rho_i}^{\rho_{i+1}}\int_{\rho_j}^{\rho_{j+1}}|\log(|\rho_j+t_j-x|)-\log(|x-y|)|d\mu_0(x)d\mu_0(y)\nonumber\\
&&\leq \frac{C\sqrt{R_0(\bm{\phi})}}{j-i-1}\cdot \mu_0([\rho_i,\rho_{i+1}])\mu_0([\rho_j,\rho_{j+1}])\leq \frac{C\sqrt{R_0(\bm{\phi})}}{k^2(j-i)}.
\end{eqnarray}

By (\ref{Eq5.18.5}),
\begin{eqnarray}\label{Eq12.15}
&& \sum_{m_0'+1\leq i<j\leq n_0}\frac{1}{j-i}=\sum_{i=m_0'+1}^{n_0-1}\sum_{j=i+1}^{n_0}\frac{1}{j-i}\leq n_0\Big(\sum_{l=1}^{n_0}\frac{1}{l}\Big) \nonumber\\
&\leq& n_0\Big(1+\int_1^{n_0}\frac{1}{x}dx\Big)=n_0(1+\log(n_0))\leq C R_0(\bm{\phi})^{3\slash 2} k\log(k).
\end{eqnarray}
Hence by (\ref{Eq12.18}), (\ref{Eq12.14}), (\ref{Eq12.19}), (\ref{Eq12.20}), and (\ref{Eq12.21}), noting (\ref{Eq5.18.5}),
\begin{eqnarray}\label{Eq12.25}
&& \sum_{m_0'+1\leq i<j\leq n_0}|\mathscr{P}_{i,j}|\nonumber\\
&\leq& C\sqrt{R_0(\bm{\phi})}k^{-2}\log(k)\cdot n_0+CR_0(\bm{\phi})^{1\slash 2} k^{-2} \sum_{m_0'+1\leq i<j\leq n_0}\frac{1}{j-i}\nonumber\\
&\leq& C R_0(\bm{\phi})^2 k^{-1}\log(k).
\end{eqnarray}

\subparagraph{Sub-step 2.2}

In the following, we consider any $i\in [m_0'+1,n_0]\cap\mathbb{Z}$. Note that 
\begin{eqnarray}\label{Eq12.22}
 \mathscr{P}_{i,i}&=& -\frac{2}{k}\int_{\rho_i}^{\rho_{i+1}}\log(|x-(\rho_i+t_i)|)d\mu_0(x)\nonumber\\
 &&+\int_{\rho_i}^{\rho_{i+1}}\int_{\rho_i}^{\rho_{i+1}}\log(|x-y|)d\mu_0(x)d\mu_0(y).
\end{eqnarray}

For any $x\in [\rho_i,\rho_{i+1}]$, by (\ref{Eq5.15.1}), 
\begin{equation*}
    x-(\rho_i+t_i) \begin{cases}
     \geq -t_i \geq -n^{-1}>-1\\
     \leq \rho_{i+1}-\rho_i\leq Ck^{-1}<1
    \end{cases}.
\end{equation*}
As $\rho_{i+1}\leq R_0(\bm{\phi})$, we have
\begin{eqnarray}\label{Eq12.23}
&&\Big|\int_{\rho_i}^{\rho_{i+1}}\log(|x-(\rho_i+t_i)|)d\mu_0(x)\Big|=\int_{\rho_i}^{\rho_{i+1}}(-\log(|x-(\rho_i+t_i)|))d\mu_0(x)\nonumber\\
&\leq& \sqrt{R_0(\bm{\phi})} \int_{\rho_i}^{\rho_{i+1}}(-\log(|x-(\rho_i+t_i)|))dx=\sqrt{R_0(\bm{\phi})}\int_{-t_i}^{\rho_{i+1}-\rho_i-t_i} (-\log(|s|))ds\nonumber\\
&\leq& \sqrt{R_0(\bm{\phi})}\int_{-Ck^{-1}}^{Ck^{-1}}(-\log(|s|))ds=2\sqrt{R_0(\bm{\phi})}(-Ck^{-1}\log(Ck^{-1})+Ck^{-1})\nonumber\\
&\leq& C\sqrt{R_0(\bm{\phi})}k^{-1}\log(k).
\end{eqnarray}

For any $x,y\in [\rho_i,\rho_{i+1}]$ such that $x\neq y$, by (\ref{Eq5.15.1}), 
\begin{equation*}
    |x-y|\leq \rho_{i+1}-\rho_i\leq Ck^{-1}<1.
\end{equation*}
Hence for any $x\in [\rho_i,\rho_{i+1}]$, as $\rho_{i+1}\leq R_0(\bm{\phi})$, 
\begin{eqnarray}
&& \int_{\rho_i}^{\rho_{i+1}}|\log(|x-y|)| d\mu_0(y)=\int_{\rho_i}^{\rho_{i+1}}(-\log(|x-y|))d\mu_0(y) \nonumber\\
&\leq& \sqrt{R_0(\bm{\phi})}\int_{\rho_i}^{\rho_{i+1}}(-\log(|x-y|))dy=\sqrt{R_0(\bm{\phi})}\int_{\rho_i-x}^{\rho_{i+1}-x}(-\log(|s|))ds\nonumber\\
&\leq & \sqrt{R_0(\bm{\phi})}\int_{-Ck^{-1}}^{Ck^{-1}}(-\log(|s|))ds=2\sqrt{R_0(\bm{\phi})}(-Ck^{-1}\log(Ck^{-1})+Ck^{-1})\nonumber\\
&\leq& C\sqrt{R_0(\bm{\phi})} k^{-1}\log(k).
\end{eqnarray}
Hence by (\ref{Eq5.39}),
\begin{eqnarray}\label{Eq12.24}
&& \Big|\int_{\rho_i}^{\rho_{i+1}}\int_{\rho_i}^{\rho_{i+1}}\log(|x-y|)d\mu_0(x)d\mu_0(y)\Big|\leq\int_{\rho_i}^{\rho_{i+1}}d\mu_0(x)\int_{\rho_i}^{\rho_{i+1}}|\log(|x-y|)|d\mu_0(y) \nonumber\\
&& \leq C\sqrt{R_0(\bm{\phi})} k^{-1}\log(k)\mu_0([\rho_i,\rho_{i+1}])\leq C\sqrt{R_0(\bm{\phi})}k^{-2}\log(k).
\end{eqnarray}

By (\ref{Eq12.22}), (\ref{Eq12.23}), and (\ref{Eq12.24}), noting (\ref{Eq5.18.5}), we have
\begin{equation}\label{Eq12.27}
    \sum_{i=m_0'+1}^{n_0} |\mathscr{P}_{i,i}|\leq C\sqrt{R_0(\bm{\phi})}k^{-2}\log(k)\cdot n_0\leq CR_0(\bm{\phi})^2 k^{-1}\log(k).
\end{equation}

By (\ref{Eq12.26}), (\ref{Eq12.25}), and (\ref{Eq12.27}), we conclude that
\begin{equation}\label{Eq12.37}
    \Big|\int_{\mathbb{R}^2\backslash \Delta}\log(|x-y|)d\mathscr{W}_2(x)d\mathscr{W}_2(y)\Big|\leq CR_0(\bm{\phi})^2 k^{-1}\log(k).
\end{equation}

\paragraph{Step 3}

In this step, we bound $\int_{\mathbb{R}^2\backslash \Delta}\log(|x-y|)d\mathscr{W}_1(x)d\mathscr{W}_2(y)$. Note that
\begin{equation}\label{Eq12.31}
  \int_{\mathbb{R}^2\backslash \Delta}\log(|x-y|)d\mathscr{W}_1(x)d\mathscr{W}_2(y)=\sum_{i=0}^{m_0'}\sum_{j=m_0'+1}^{n_0}\mathscr{P}_{i,j}.
\end{equation}

\subparagraph{Sub-step 3.1}

In this sub-step, we bound $\mathscr{P}_{i,j}$ for any $i\in [m_0']$ and any $j\in [m_0'+1,n_0]\cap\mathbb{Z}$. In the following, we fix any $i\in [m_0']$ and $j\in [m_0'+1,n_0]\cap\mathbb{Z}$.

By (\ref{Eq5.43}) and (\ref{Eq9.6}), $\rho_{m_0'+1}-(\rho_{m_0'}+n^{-1})\geq ck^{-1}-n^{-1}\geq ck^{-1}$. Below we consider any $x\in [0,\rho_{m_0'}+n^{-1}]$. We define
\begin{eqnarray}
    h_j(x)&:=&k^{-1}\log(|\rho_j+t_j-x|)-\int_{\rho_j}^{\rho_{j+1}}\log(|y-x|)d\mu_0(y)\nonumber\\
    &=&\int_{\rho_j}^{\rho_{j+1}}(\log(|\rho_j+t_j-x|)-\log(|y-x|))d\mu_0(y),
\end{eqnarray}
where we use (\ref{Eq5.39}) for the second equality. For any $y\in [\rho_j,\rho_{j+1}]$, by (\ref{Eq5.20}), 
\begin{eqnarray*}
  &&  \min\{\rho_j+t_j-x,y-x\}\geq \rho_j-x\geq \rho_{j}-(\rho_{m_0'}+n^{-1})  \nonumber\\
  &\geq&  \rho_j-\rho_{m_0'+1}+ck^{-1}
    =\sum_{l=m_0'+1}^{j-1}(\rho_{l+1}-\rho_l)+ck^{-1}\nonumber\\
    &\geq& \pi k^{-1} R_0(\bm{\phi})^{-1\slash 2}(j-m_0'-1)+ck^{-1}\geq ck^{-1}R_0(\bm{\phi})^{-1\slash 2}(j-m_0'),
\end{eqnarray*}
where we use (\ref{Eq5.13}) for the last inequality. Hence by (\ref{Eq5.15.1}) and the inequality $\log(1+t)\leq t$ for any $t>-1$, we have
\begin{eqnarray}
&& \log(|\rho_j+t_j-x|)-\log(|y-x|)=\log\Big(1+\frac{\rho_j+t_j-y}{y-x}\Big)\leq \frac{\rho_j+t_j-y}{y-x} \nonumber\\
&\leq& \frac{n^{-1}}{ck^{-1}R_0(\bm{\phi})^{-1\slash 2}(j-m_0')}\leq \frac{C\sqrt{R_0(\bm{\phi})}}{j-m_0'},
\end{eqnarray}
\begin{eqnarray}
&& \log(|y-x|)-\log(|\rho_j+t_j-x|)=\log\Big(1+\frac{y-(\rho_j+t_j)}{\rho_j+t_j-x}\Big)\leq \frac{y-(\rho_j+t_j)}{\rho_j+t_j-x}\nonumber\\
&\leq& \frac{\rho_{j+1}-\rho_j}{\rho_j+t_j-x}\leq \frac{C k^{-1}}{k^{-1}R_0(\bm{\phi})^{-1\slash 2}(j-m_0')}\leq \frac{C\sqrt{R_0(\bm{\phi})}}{j-m_0'}.
\end{eqnarray}
Hence by (\ref{Eq5.39}), for any $x\in [0,\rho_{m_0'}+n^{-1}]$, we have
\begin{equation}\label{Eq12.28}
|h_j(x)|\leq \frac{C\sqrt{R_0(\bm{\phi})}}{j-m_0'}\cdot \mu_0([\rho_j,\rho_{j+1}])= \frac{C\sqrt{R_0(\bm{\phi})}}{k(j-m_0')}. 
\end{equation}
By (\ref{Eq5.17}), (\ref{Eq9.3}), and (\ref{rhh}), $\mu_0([\rho_{i-1},\rho_i])=\mu_0([0,r_0])\slash (m_0'+1)\leq Ck^{-1}$. Hence by (\ref{Eq12.28}), we have
\begin{eqnarray}\label{Eq12.32}
&& |\mathscr{P}_{i,j}|= \Big|k^{-1}h_j(\rho_i+t_i)-\int_{\rho_{i-1}}^{\rho_i}h_j(x)d\mu_0(x)\Big| \nonumber\\
&\leq& k^{-1}|h_j(\rho_i+t_i)|+\int_{\rho_{i-1}}^{\rho_i}|h_j(x)|d\mu_0(x)\nonumber\\
&\leq& \frac{C\sqrt{R_0(\bm{\phi})}}{k(j-m_0')}\cdot (k^{-1}+\mu_0([\rho_{i-1},\rho_i]))\leq \frac{C\sqrt{R_0(\bm{\phi})}}{k^2(j-m_0')}.
\end{eqnarray}

\subparagraph{Sub-step 3.2}

In this sub-step, we bound $\mathscr{P}_{0,,m_0'+1}$. Note that $\rho_{m_0'+1}=r_0$. For any $x\in [\rho_{m_0'},r_0)$, by (\ref{Eq9.3}) and (\ref{Eq9.6}),
\begin{equation*}
   0\leq t_{m_0'+1}< \rho_{m_0'+1}+t_{m_0'+1}-x\leq r_0-\rho_{m_0'}+n^{-1}\leq Ck^{-1}<1.
\end{equation*}
Hence
\begin{eqnarray}\label{Eq12.29}
&& \Big|\int_{\rho_{m_0'}}^{r_0}\log(|\rho_{m_0'+1}+t_{m_0'+1}-x|)d\mu_0(x)\Big|=\int_{\rho_{m_0'}}^{r_0}(-\log(|r_0+t_{m_0'+1}-x|))d\mu_0(x)\nonumber\\
&\leq& C\int_{\rho_{m_0'}}^{r_0}(-\log(|r_0+t_{m_0'+1}-x|))dx= C\int_{t_{m_0'+1}}^{r_0-\rho_{m_0'}+t_{m_0'+1}} (-\log(|s|))ds\nonumber\\
&\leq& C\int_0^{Ck^{-1}}(-\log(|s|))ds=C(-Ck^{-1}\log(Ck^{-1})+Ck^{-1})\leq Ck^{-1}\log(k). 
\end{eqnarray}

For any $x\in [\rho_{m_0'},r_0]$ and $y\in (\rho_{m_0'+1},\rho_{m_0'+2}]$, by (\ref{Eq9.3}), (\ref{Eq5.15.1}), and (\ref{Eq9.6}),
\begin{equation*}
    0<y-x\leq \rho_{m_0'+2}-\rho_{m_0'}= (\rho_{m_0'+2}-\rho_{m_0'+1})+(\rho_{m_0'+1}-\rho_{m_0'})\leq Ck^{-1}<1.
\end{equation*}
By (\ref{Eq5.15.1}), $\rho_{m_0'+2}= r_0+(\rho_{m_0'+2}-\rho_{m_0'+1})\leq 10+2\pi k^{-1}\leq 12$. Hence
\begin{eqnarray}
&&\Big|\int_{\rho_{m_0'}}^{r_0}\int_{\rho_{m_0'+1}}^{\rho_{m_0'+2}}\log(|x-y|)d\mu_0(x)d\mu_0(y)\Big|\nonumber\\
&=&\int_{\rho_{m_0'}}^{r_0}\int_{\rho_{m_0'+1}}^{\rho_{m_0'+2}}(-\log(y-x))d\mu_0(x)d\mu_0(y)\nonumber\\
&\leq& C\int_{\rho_{m_0'}}^{r_0}\int_{\rho_{m_0'+1}}^{\rho_{m_0'+2}}(-\log(y-x))dxdy.
\end{eqnarray}
For any $y\in [\rho_{m_0'+1},\rho_{m_0'+2}]=[r_0,\rho_{m_0'+2}]$, we have
\begin{eqnarray}
&& \int_{\rho_{m_0'}}^{r_0}(-\log(y-x))dx=\int_{y-r_0}^{y-\rho_{m_0'}}(-\log(s))ds\leq \int_0^{Ck^{-1}}(-\log(s))ds\nonumber\\
&& = -Ck^{-1}\log(C k^{-1})+Ck^{-1}\leq Ck^{-1}\log(k).
\end{eqnarray}
Hence
\begin{eqnarray}\label{Eq12.30}
&&  \Big|\int_{\rho_{m_0'}}^{r_0}\int_{\rho_{m_0'+1}}^{\rho_{m_0'+2}}\log(|x-y|)d\mu_0(x)d\mu_0(y)\Big|\nonumber \\
&\leq& Ck^{-1}\log(k) (\rho_{m_0'+2}-\rho_{m_0'+1})\leq Ck^{-2}\log(k).
\end{eqnarray}

By (\ref{Eq12.29}) and (\ref{Eq12.30}), we have
\begin{eqnarray}\label{Eq12.33}
    |\mathscr{P}_{0,m_0'+1}|&=&\Big|-k^{-1}\int_{\rho_{m_0'}}^{r_0}\log(|\rho_{m_0'+1}+t_{m_0'+1}-x|)d\mu_0(x)\nonumber\\
    &&\quad +\int_{\rho_{m_0'}}^{r_0}\int_{\rho_{m_0'+1}}^{\rho_{m_0'+2}}\log(|x-y|)d\mu_0(x)d\mu_0(y)\Big|\nonumber\\
    &\leq& Ck^{-2}\log(k).
\end{eqnarray}

\subparagraph{Sub-step 3.3}

In this sub-step, we bound $\mathscr{P}_{0,i}$ for any $i\in [m_0'+2,n_0]\cap\mathbb{Z}$. In the following, we fix any $i\in [m_0'+2,n_0]\cap\mathbb{Z}$.

For any $x\in [\rho_{m_0'},r_0]$ and $y\in [\rho_i,\rho_{i+1}]$, by (\ref{Eq5.20}), we have
\begin{equation*}
    \min\{\rho_i+t_i-x,y-x\}\geq \rho_i-r_0=\sum_{l=m_0'+1}^{i-1}(\rho_{l+1}-\rho_l)\geq \frac{c(i-m_0'-1)}{k\sqrt{R_0(\bm{\phi})}}>0.
\end{equation*}
Hence by (\ref{Eq5.15.1}) and the inequality $\log(1+t)\leq t$ for any $t>-1$, we have
\begin{eqnarray}
&& \log(|y-x|)-\log(|\rho_i+t_i-x|)=\log\Big(1+\frac{y-(\rho_i+t_i)}{\rho_i+t_i-x}\Big)\leq \frac{y-(\rho_i+t_i)}{\rho_i+t_i-x} \nonumber\\
&\leq& \frac{\rho_{i+1}-\rho_i}{\rho_i+t_i-x}\leq \frac{C k^{-1}}{k^{-1} R_0(\bm{\phi})^{-1\slash 2}(i-m_0'-1)}=\frac{C\sqrt{R_0(\bm{\phi})}}{i-m_0'-1},
\end{eqnarray}
\begin{eqnarray}
&& \log(|\rho_i+t_i-x|)-\log(|y-x|)=\log\Big(1+\frac{\rho_i+t_i-y}{y-x}\Big)\leq \frac{\rho_i+t_i-y}{y-x}\nonumber\\
&\leq& \frac{n^{-1}}{ck^{-1} R_0(\bm{\phi})^{-1\slash 2}(i-m_0'-1)}\leq \frac{C\sqrt{R_0(\bm{\phi})}}{i-m_0'-1}.
\end{eqnarray}
By (\ref{Eq5.17}), (\ref{Eq9.3}), and (\ref{rhh}), $\mu_0([\rho_{m_0'},\rho_{m_0'+1}])=\mu_0([0,r_0])\slash (m_0'+1)\leq Ck^{-1}$. Hence by (\ref{Eq5.39}), we have
\begin{eqnarray}\label{Eq12.34}
&& |\mathscr{P}_{0,i}|=\Big|\int_{\rho_{m_0'}}^{r_0}\int_{\rho_i}^{\rho_{i+1}}(\log(|x-y|)-\log(|\rho_i+t_i-x|))d\mu_0(x)d\mu_0(y)\Big| \nonumber\\
&\leq& \int_{\rho_{m_0'}}^{r_0}\int_{\rho_i}^{\rho_{i+1}}|\log(|x-y|)-\log(|\rho_i+t_i-x|)|d\mu_0(x)d\mu_0(y)\nonumber\\
&\leq& \frac{C\sqrt{R_0(\bm{\phi})}}{i-m_0'-1}\cdot \mu_0([\rho_{m_0'},\rho_{m_0'+1}])\mu_0([\rho_i,\rho_{i+1}])\leq \frac{C\sqrt{R_0(\bm{\phi})}}{k^2(i-m_0'-1)}.
\end{eqnarray}

By (\ref{Eq5.11}) and (\ref{Eq5.18.5}), we have
\begin{equation}
    \sum_{i=m_0'+1}^{n_0}\frac{1}{i-m_0'}\leq \sum_{l=1}^{n_0}\frac{1}{l}\leq 1+\int_{1}^{n_0}\frac{1}{x}dx=1+\log(n_0)\leq C\log(k).
\end{equation}
Hence by (\ref{Eq12.31}), (\ref{Eq12.32}), (\ref{Eq12.33}), and (\ref{Eq12.34}), noting (\ref{Eq5.11}) and (\ref{Eq5.18.5}), we have 
\begin{eqnarray}\label{Eq12.38}
&&\Big|\int_{\mathbb{R}^2\backslash \Delta}\log(|x-y|)d\mathscr{W}_1(x)d\mathscr{W}_2(y)\Big|\nonumber\\
&\leq& n_0\sum_{j=m_0'+1}^{n_0} \frac{C\sqrt{R_0(\bm{\phi})}}{k^2(j-m_0')}+Ck^{-2}\log(k)+\sum_{i=m_0'+2}^{n_0}\frac{C\sqrt{R_0(\bm{\phi})}}{k^2(i-m_0'-1)}\nonumber\\
&\leq& C n_0 k^{-2}\sqrt{R_0(\bm{\phi})} \sum_{i=m_0'+1}^{n_0}\frac{1}{i-m_0'}+Ck^{-2}\log(k)\nonumber\\
&\leq& C R_0(\bm{\phi})^2 k^{-1}\log(k).
\end{eqnarray}

\bigskip

By (\ref{Eq12.35}), (\ref{Eq12.36}), (\ref{Eq12.37}), and (\ref{Eq12.38}), noting (\ref{Eq5.11}), we have
\begin{equation}
    \Big|-\int_{\mathbb{R}^2\backslash\Delta}\log(|x-y|)d\tilde{\Upsilon}_{1;n_0,\mathbf{t}}(x)d\tilde{\Upsilon}_{1;n_0,\mathbf{t}}(y)\Big|\leq C\sqrt{\frac{\epsilon}{\log(R_0(\bm{\phi}))\log\log(R_0(\bm{\phi}))}}.
\end{equation}
We recall from the argument above (\ref{Eq12.39}) that $\tilde{\xi}(x)=0$ for any $x\in [0,R_0(\bm{\phi})]$. Hence by (\ref{Eq5.25}) and (\ref{Eq8.1}), we have 
\begin{equation}
    \int \tilde{\xi}(x)d\tilde{\Upsilon}_{1;n_0,\mathbf{t}}(x)=\frac{1}{k}\sum_{i=1}^{n_0}\tilde{\xi}(\rho_i+t_i)-\int_0^{R_0(\bm{\phi})}\tilde{\xi}(x)d\mu_0(x)=0.
\end{equation}
Therefore, by (\ref{Eq5.11}), we conclude that
\begin{eqnarray}
  &&  \Big|-\int_{\mathbb{R}^2\backslash\Delta}\log(|x-y|)d\tilde{\Upsilon}_{1;n_0,\mathbf{t}}(x)d\tilde{\Upsilon}_{1;n_0,\mathbf{t}}(y)+2\int \tilde{\xi}(x)d\tilde{\Upsilon}_{1;n_0,\mathbf{t}}(x)\Big|\nonumber\\
  && \leq C\sqrt{\frac{\epsilon}{\log(R_0(\bm{\phi}))\log\log(R_0(\bm{\phi}))}}\leq C\sqrt{\epsilon}.
\end{eqnarray}

\end{proof}

\subsection{Local large deviation upper bound}\label{local_upp}

In this subsection, we establish the local large deviation upper bound in Proposition \ref{local_up_p} below. We recall relevant notations and definitions from Definitions \ref{Defn1.1}-\ref{Defn1.3}.

\begin{proposition}\label{local_up_p}
For any $\mu\in\mathcal{X}_R$, we have
\begin{equation}
    \limsup_{\delta\rightarrow 0^{+}}\limsup_{k\rightarrow\infty}\frac{1}{k^2}\log(\mathbb{P}(d_R(\nu_{k;R},\mu)\leq \delta))\leq -\frac{\beta}{2} I_R(\mu).
\end{equation}
\end{proposition}
\begin{proof}

Throughout the proof, we denote by $C_{M,\epsilon}$ a positive constant that only depends on $\beta,M,\epsilon$. The value of $C_{M,\epsilon}$ may change from line to line.

We fix any $R\geq 10$ and $\mu\in\mathcal{X}_R$. We recall the setup stated at the beginning of this section, and let $\Lambda_0,\Lambda_0'$ be the constants that appear in Proposition \ref{P4.4} and (\ref{Eq5.19.1}). By Propositions \ref{P4.4}-\ref{P5.1}, (\ref{Eq12.40}), and the union bound, as $k\geq K_0(M,\epsilon,\delta)$ is sufficiently large (depending on $\beta,M,\epsilon,\delta$), we have
\begin{eqnarray}\label{Eq4.14.16}
&& \mathbb{P}(d_R(\nu_{k;R},\mu)\leq \delta)\nonumber\\
&\leq& \mathbb{P}(\mathcal{B}_0(\Lambda_0)^c)+\mathbb{P}(\mathscr{E}^c)+\mathbb{P}(\mathcal{B}_0(\Lambda_0)\cap \mathscr{E}\cap \{d_R(\nu_{k;R},\mu)\leq \delta\}) \nonumber\\
&\leq & C\exp(-cMk^2)+\sum_{\bm{\phi}\in\Phi} \mathbb{P}(\mathcal{B}(\bm{\phi};\Lambda_0)\cap \mathscr{E} \cap \{d_R(\nu_{k;R},\mu)\leq \delta\}).
\end{eqnarray}
Throughout the rest of the proof, we assume that $M\geq R+10$. By (\ref{Eq5.13}), for any $\bm{\phi}\in\Phi$, we have $R_0(\bm{\phi})\geq M\geq R+10$. 

In the following, we consider an arbitrary $\bm{\phi}\in\Phi$. Let $\mathcal{N}_0(\bm{\phi})$ be the set of $n_0\in \{0\}\cup[n]$ such that (\ref{Eq5.15}) holds. By Definition \ref{Defn5.1} and (\ref{Eq5.14}), we have
\begin{equation}
    \mathcal{B}(\bm{\phi};\Lambda_0) \subseteq \bigcup_{n_0\in\mathcal{N}_0(\bm{\phi})} \mathcal{D}(\bm{\phi},n_0).
\end{equation}
Hence by (\ref{Eq5.19.1}) and the union bound, we have
\begin{eqnarray}\label{Eq4.14.15}
&& \mathbb{P}(\mathcal{B}(\bm{\phi};\Lambda_0)\cap \mathscr{E} \cap \{d_R(\nu_{k;R},\mu)\leq \delta\}) \nonumber\\
&\leq& \mathbb{P}(\mathcal{C}(\bm{\phi};\Lambda_0')^c)+\mathbb{P}(\mathcal{B}(\bm{\phi};\Lambda_0)\cap \mathcal{C}(\bm{\phi};\Lambda_0')\cap \mathscr{E}\cap \{d_R(\nu_{k;R},\mu)\leq \delta\})\nonumber\\
&\leq& \sum_{n_0\in\mathcal{N}_0(\bm{\phi})}\mathbb{P}(\mathcal{B}(\bm{\phi};\Lambda_0)\cap\mathcal{D}(\bm{\phi},n_0)\cap \mathcal{C}(\bm{\phi};\Lambda_0')\cap \mathscr{E}\cap \{d_R(\nu_{k;R},\mu)\leq \delta\})\nonumber\\
&& +C\exp(-Mk^2).
\end{eqnarray}

In the following, we consider an arbitrary $n_0\in\mathcal{N}_0(\bm{\phi})$. By Proposition \ref{P5.5} and the union bound, we have
\begin{eqnarray}\label{Eq4.14.10}
&& \mathbb{P}(\mathcal{B}(\bm{\phi};\Lambda_0)\cap\mathcal{D}(\bm{\phi},n_0)\cap \mathcal{C}(\bm{\phi};\Lambda_0')\cap \mathscr{E}\cap \{d_R(\nu_{k;R},\mu)\leq \delta\})\nonumber\\
&\leq& \mathbb{P}(\mathcal{B}(\bm{\phi};\Lambda_0)\cap\mathcal{D}(\bm{\phi},n_0)\cap \mathcal{C}(\bm{\phi};\Lambda_0')\cap \mathscr{G}_R\cap \mathscr{E}\cap \{d_R(\nu_{k;R},\mu)\leq \delta\})\nonumber\\
&&+ \mathbb{P}(\mathcal{B}(\bm{\phi};\Lambda_0)\cap\mathcal{D}(\bm{\phi},n_0)\cap \mathcal{C}(\bm{\phi};\Lambda_0')\cap(\mathscr{G}_R)^c)\nonumber\\
&\leq& \mathbb{P}(\mathcal{B}(\bm{\phi};\Lambda_0)\cap \mathcal{D}(\bm{\phi},n_0)\cap\mathcal{C}(\bm{\phi};\Lambda_0')\cap \mathscr{G}_R \cap \{d_R(\mu_{n,k;R},\mu)\leq 2\delta\})\nonumber\\
&& + \mathbb{P}(\mathcal{B}(\bm{\phi};\Lambda_0)\cap\mathcal{D}(\bm{\phi},n_0)\cap \mathcal{C}(\bm{\phi};\Lambda_0')\cap(\mathscr{G}_R)^c).
\end{eqnarray}


Recall Definition \ref{Defn5.3}, and let $\omega_0$ be the Borel measure on $\mathbb{R}$ such that for any $A\in\mathcal{B}_{\mathbb{R}}$, $\omega_0(A)=\int_A \rho(x)dx$. By \cite[Theorem 2.12]{DE} (see also \cite[Section 4.5]{AGZ}), the joint density of $(n^{-1\slash 2}\lambda_1^{(n)},n^{-1\slash 2}\lambda_2^{(n)},\cdots,n^{-1\slash 2}\lambda_n^{(n)})$ is given by
\begin{eqnarray}
    &&\rho_n^{(\beta)}(x_1,x_2,\cdots,x_n)\propto \mathbbm{1}_{x_1>x_2>\cdots>x_n} \prod_{1\leq i<j\leq n}|x_j-x_i|^{\beta} \prod_{i=1}^n e^{-\beta n x_i^2\slash 4}\nonumber\\
    &\propto& \mathbbm{1}_{x_1>x_2>\cdots>x_n}\exp\Big(\beta\sum_{1\leq i<j\leq n}\log(|x_j-x_i|)-\frac{\beta n}{4}\sum_{i=1}^n x_i^2\Big)\nonumber\\
    &\propto& \exp\Big(\frac{\beta n^2}{2}\Big(\int_{\mathbb{R}^2\backslash\Delta}\log(|x-y|)d\big(n^{-1}\sum_{i=1}^n \delta_{x_i}-\omega_0\big)(x)d\big(n^{-1}\sum_{i=1}^n \delta_{x_i}-\omega_0\big)(y)\nonumber\\
    && \quad\quad\quad -2 \int \xi(x)d\big(n^{-1}\sum_{i=1}^n \delta_{x_i}-\omega_0\big)(x)\Big)\Big) \mathbbm{1}_{x_1>x_2>\cdots>x_n}.
\end{eqnarray}
As $b_i=-k^{-2\slash 3}n^{1\slash 6}(\lambda_i^{(n)}-2\sqrt{n})$ for every $i\in [n]$, the joint density of $(b_1,b_2,\cdots,b_n)$ is given by
\begin{eqnarray}\label{Eq14.1}
    && \tilde{\rho}_{n,k}^{(\beta)}(b_1,b_2,\cdots,b_n) \nonumber\\
    &\propto& \exp\Big(\frac{\beta k^2}{2}\Big(\int_{\mathbb{R}^2\backslash\Delta} \log(|x-y|)d\mu_{n,k}(x)d\mu_{n,k}(y)-2\int \tilde{\xi}(x)d\mu_{n,k}(x) \Big)\Big)\nonumber\\
    &&\times
    \mathbbm{1}_{b_1<b_2<\cdots<b_n}.
\end{eqnarray}
We define $\mathscr{B}_0:=\{(b_1,b_2,\cdots,b_n)\in\mathbb{R}^n: b_1<b_2<\cdots<b_n\}$.

Let
\begin{equation}\label{Eq4.14.13}
   \mathcal{G}_1:= \mathcal{B}(\bm{\phi};\Lambda_0)\cap \mathcal{D}(\bm{\phi},n_0)\cap\mathcal{C}(\bm{\phi};\Lambda_0')\cap \mathscr{G}_R \cap \{d_R(\mu_{n,k;R},\mu)\leq 2\delta\},
\end{equation}
\begin{equation}\label{Eq4.14.14}
    \mathcal{G}_2:=  \mathcal{B}(\bm{\phi};\Lambda_0)\cap\mathcal{D}(\bm{\phi},n_0)\cap \mathcal{C}(\bm{\phi};\Lambda_0')\cap(\mathscr{G}_R)^c.
\end{equation}
Let $\mathcal{G}_1'$ be the set of $(b_1,\cdots,b_n)\in \mathscr{B}_0$ such that the event $\mathcal{G}_1$ holds and $\mathcal{G}_2'$ the set of $(b_1,\cdots,b_n)\in\mathscr{B}_0$ such that the event $\mathcal{G}_2$ holds. We also let $\mathcal{G}_1''$ be the set of $(b_1,\cdots,b_n)\in\mathscr{B}_0$ such that the event
\begin{equation*}
    \mathcal{B}(\bm{\phi};\Lambda_0)  \cap\mathcal{C}(\bm{\phi};\Lambda_0')\cap \mathscr{G}_R \cap \{d_R(\mu_{n,k;R},\mu)\leq 2\delta\}
\end{equation*}
holds, and let $\mathcal{G}_2''$ be the set of $(b_1,\cdots,b_n)\in\mathscr{B}_0$ such that the event 
\begin{equation*}
    \mathcal{B}(\bm{\phi};\Lambda_0) \cap \mathcal{C}(\bm{\phi};\Lambda_0')\cap(\mathscr{G}_R)^c
\end{equation*}
holds. 

Let $\tilde{\mathscr{X}}$ be the set of compactly supported finite signed Borel measures $\mu$ on $\mathbb{R}$ such that $\mu+\mu_0$ is a positive measure and $\mu(\mathbb{R})=0$. For any $\nu\in \tilde{\mathscr{X}}$, we define
\begin{equation}
    J_0(\nu):=-\int_{\mathbb{R}^2\backslash\Delta} \log(|x-y|)d\nu(x)d\nu(y)+2\int \tilde{\xi}(x)d\nu(x).
\end{equation}
For every $i\in \{1,2\}$, by (\ref{Eq14.1}),
\begin{eqnarray}\label{Eq14.7}
\mathbb{P}(\mathcal{G}_i)=\frac{\int_{\mathcal{G}_i'}\exp(-\beta k^2 J_0(\mu_{n,k})\slash 2)db_1\cdots db_n}{\int_{\mathscr{B}_0}\exp(-\beta k^2 J_0(\mu_{n,k})\slash 2)db_1\cdots db_n}.
\end{eqnarray}
Recall the definition of $\tilde{\mu}_{n_0,\mathbf{t}}$ from (\ref{Eq8.2}). We have
\begin{eqnarray}\label{Eq14.4}
&& \int_{\mathcal{G}_i'}\exp(-\beta k^2 J_0(\mu_{n,k})\slash 2)db_1\cdots db_n\nonumber\\
&=& \sum_{m=0}^{n-n_0} \int_{b_1<\cdots<b_m<-R_0(\bm{\phi})} db_1\cdots db_m \int_{R_0(\bm{\phi})<b_{n_0+m+1}<\cdots<b_n}db_{n_0+m+1}\cdots db_n\nonumber\\
&&\quad\quad \int_{-R_0(\bm{\phi})\leq b_{m+1}<\cdots<b_{n_0+m}\leq R_0(\bm{\phi})} \exp(-\beta k^2 J_0(\mu_{n,k})\slash 2) \mathbbm{1}_{\mathcal{G}_i''}(b_1,\cdots,b_n)\nonumber\\
&&\quad\quad\quad\quad\quad\quad\quad\quad\quad\quad\quad\quad\quad\quad\quad\quad\quad db_{m+1}\cdots db_{n_0+m},
\end{eqnarray}
\begin{eqnarray}\label{Eq14.5}
&& \int_{\mathscr{B}_0}\exp(-\beta k^2 J_0(\mu_{n,k})\slash 2)db_1\cdots db_n\nonumber\\
&\geq & \sum_{m=0}^{n-n_0} \int_{b_1<\cdots<b_m<-R_0(\bm{\phi})} db_1\cdots db_m \int_{R_0(\bm{\phi})<b_{n_0+m+1}<\cdots<b_n}db_{n_0+m+1}\cdots db_n\nonumber\\
&&\quad\quad \int_{\mathbf{t}=(t_1,\cdots,t_{n_0}) \in \mathscr{T}_{n_0}} \exp(-\beta k^2 J_0(\tilde{\mu}_{n_0,\mathbf{t}})\slash 2) dt_1\cdots dt_{n_0}. 
\end{eqnarray}

In the following, we fix any $m\in \{0\}\cup [n-n_0]$, $b_1<\cdots<b_m<-R_0(\bm{\phi})$, and $R_0(\bm{\phi})<b_{n_0+m+1}<\cdots<b_n$. Recall (\ref{Eq5.21.2}), (\ref{Eq8.1}), and (\ref{Eq14.2}). We have
\begin{eqnarray}
    &&J_0(\mu_{n,k})=-\int_{\mathbb{R}^2\backslash\Delta}\log(|x-y|)d\Upsilon_1(x)d\Upsilon_1(y)-2\int_{\mathbb{R}^2\backslash\Delta}\log(|x-y|)d\Upsilon_1(x)d\Upsilon_2(y)\nonumber\\
    &&\quad -\int_{\mathbb{R}^2\backslash\Delta}\log(|x-y|)d\Upsilon_2(x)d\Upsilon_2(y)+2\int \tilde{\xi}(x)d\Upsilon_1(x)+2\int \tilde{\xi}(x)d\Upsilon_2(x).
\end{eqnarray}
For any $\mathbf{t}=(t_1,t_2,\cdots,t_{n_0})\in \mathscr{T}_{n_0}$, we have
\begin{eqnarray}
     &&J_0(\tilde{\mu}_{n_0,\mathbf{t}})=-\int_{\mathbb{R}^2\backslash\Delta}\log(|x-y|)d\tilde{\Upsilon}_{1;n_0,\mathbf{t}}(x)d\tilde{\Upsilon}_{1;n_0,\mathbf{t}}(y)\nonumber\\
     &&-2\int_{\mathbb{R}^2\backslash\Delta}\log(|x-y|)d\tilde{\Upsilon}_{1;n_0,\mathbf{t}}(x)d\tilde{\Upsilon}_{2;n_0,\mathbf{t}}(y)-\int_{\mathbb{R}^2\backslash\Delta}\log(|x-y|)d\tilde{\Upsilon}_{2;n_0,\mathbf{t}}(x)d\tilde{\Upsilon}_{2;n_0,\mathbf{t}}(y)\nonumber\\
     &&+2\int \tilde{\xi}(x)d\tilde{\Upsilon}_{1;n_0,\mathbf{t}}(x)+2\int \tilde{\xi}(x)d\tilde{\Upsilon}_{2;n_0,\mathbf{t}}(x).
\end{eqnarray}
As $\tilde{\Upsilon}_{2;n_0,\mathbf{t}}=\Upsilon_2$, we have
\begin{eqnarray}\label{Eq14.3}
&& J_0(\mu_{n,k})-J_0(\tilde{\mu}_{n_0,\mathbf{t}})=-\int_{\mathbb{R}^2\backslash\Delta}\log(|x-y|)d\Upsilon_1(x)d\Upsilon_1(y)+2\int \tilde{\xi}(x)d\Upsilon_1(x) \nonumber\\
&&+\int_{\mathbb{R}^2\backslash\Delta}\log(|x-y|)d\tilde{\Upsilon}_{1;n_0,\mathbf{t}}(x)d\tilde{\Upsilon}_{1;n_0,\mathbf{t}}(y)-2\int \tilde{\xi}(x)d\tilde{\Upsilon}_{1;n_0,\mathbf{t}}(x)\nonumber\\
&&-2\int_{\mathbb{R}^2\backslash\Delta}\log(|x-y|)d\Upsilon_1(x)d\Upsilon_2(y)+2\int_{\mathbb{R}^2\backslash\Delta}\log(|x-y|)d\tilde{\Upsilon}_{1;n_0,\mathbf{t}}(x)d\tilde{\Upsilon}_{2;n_0,\mathbf{t}}(y).\nonumber\\
&& 
\end{eqnarray}

Consider any choice of $-R_0(\bm{\phi})\leq b_{m+1}<\cdots<b_{n_0+m}\leq R_0(\bm{\phi})$. If $(b_1,\cdots,b_n)\notin \mathcal{G}_1''$, we have $\exp(-\beta k^2 J_0(\mu_{n,k})\slash 2) \mathbbm{1}_{\mathcal{G}_1''}(b_1,\cdots,b_n)=0$. Otherwise, for any $\mathbf{t}=(t_1,t_2,\cdots,t_{n_0})\in \mathscr{T}_{n_0}$, by (\ref{Eq14.3}) and Propositions \ref{P5.3}, \ref{P5.6}, and \ref{P5.7}, we have
\begin{equation}
    J_0(\mu_{n,k})-J_0(\tilde{\mu}_{n_0,\mathbf{t}})\geq \Big(1+\frac{C_1}{\log(n)}\Big)^{-1}I_R(\mu,3\delta)-C(M\sqrt{\epsilon}+M^{-1}),
\end{equation}
where $C,C_1$ are positive constants that only depend on $\beta$. Hence
\begin{eqnarray}
&& \int_{-R_0(\bm{\phi})\leq b_{m+1}<\cdots<b_{n_0+m}\leq R_0(\bm{\phi})} \exp(-\beta k^2 J_0(\mu_{n,k})\slash 2) \mathbbm{1}_{\mathcal{G}_1''}(b_1,\cdots,b_n)\nonumber\\
&& \quad\quad\quad\quad\quad\quad\quad\quad\quad\quad\quad\quad\quad\quad\quad db_{m+1}\cdots db_{n_0+m} \nonumber\\
&\leq& \exp\Big(-\frac{\beta k^2}{2}\Big(\Big(1+\frac{C_1}{\log(n)}\Big)^{-1}I_R(\mu,3\delta)-C(M\sqrt{\epsilon}+M^{-1})\Big)\Big)\nonumber\\
&& \times \inf_{\mathbf{t}=(t_1,t_2,\cdots,t_{n_0})\in \mathscr{T}_{n_0}}\Big\{\exp(-\beta k^2 J_0(\tilde{\mu}_{n_0,\mathbf{t}})\slash 2)\Big\} (2R_0(\bm{\phi}))^{n_0}\nonumber\\
&\leq& \exp\Big(-\frac{\beta k^2}{2}\Big(\Big(1+\frac{C_1}{\log(n)}\Big)^{-1}I_R(\mu,3\delta)-C(M\sqrt{\epsilon}+M^{-1})\Big)\Big)(2R_0(\bm{\phi}))^{n_0} n^{n_0} n_0!\nonumber\\
&& \times \int_{\mathbf{t}=(t_1,\cdots,t_{n_0}) \in \mathscr{T}_{n_0}} \exp(-\beta k^2 J_0(\tilde{\mu}_{n_0,\mathbf{t}})\slash 2) dt_1\cdots dt_{n_0}.
\end{eqnarray}
Hence by (\ref{Eq14.4}) and (\ref{Eq14.5}),
\begin{eqnarray}\label{Eq14.8}
  &&  \int_{\mathcal{G}_1'}\exp(-\beta k^2 J_0(\mu_{n,k})\slash 2)db_1\cdots db_n\nonumber\\
  &\leq& \exp\Big(-\frac{\beta k^2}{2}\Big(\Big(1+\frac{C_1}{\log(n)}\Big)^{-1}I_R(\mu,3\delta)-C(M\sqrt{\epsilon}+M^{-1})\Big)\Big)\nonumber\\
  &&\times (2R_0(\bm{\phi})n)^{n_0}  n_0! \int_{\mathscr{B}_0}\exp(-\beta k^2 J_0(\mu_{n,k})\slash 2)db_1\cdots db_n.
\end{eqnarray}
By (\ref{Eq5.11}) and (\ref{Eq5.18.5}), we have $(2R_0(\bm{\phi})n)^{n_0}  n_0!\leq (2^{T+1} n^2)^{n_0}\leq (2^{T+1} n^2)^{C 2^T k}$. Hence by (\ref{Eq14.7}) and (\ref{Eq14.8}), as $T$ only depends on $M,\epsilon$ and $k\geq K(M,\epsilon,\delta)$ is sufficiently large (depending on $\beta,M,\epsilon,\delta$), we have
\begin{eqnarray}\label{Eq4.14.11}
   &&\mathbb{P}(\mathcal{G}_1)\leq \exp\Big(-\frac{\beta k^2}{2}\Big(\Big(1+\frac{C_1}{\log(n)}\Big)^{-1}I_R(\mu,3\delta)-C(M\sqrt{\epsilon}+M^{-1})\Big)\nonumber\\
   &&\quad\quad\quad\quad\quad\quad +C 2^Tk \log(2^{T+1} n^2)\Big)\nonumber\\
   &\leq& \exp\Big(-\frac{\beta k^2}{2}\Big(\Big(1+\frac{C_1}{\log(n)}\Big)^{-1}I_R(\mu,3\delta)-C(M\sqrt{\epsilon}+M^{-1})\Big)\Big).
\end{eqnarray}

Similarly, by Propositions \ref{P5.3}, \ref{P5.4}, and \ref{P5.7}, we have
\begin{equation}\label{Eq4.14.12}
    \mathbb{P}(\mathcal{G}_2)\leq \exp(-ck^2\sqrt{\log(n)}+C_{M,\epsilon} k^2)\leq \exp(-ck^2\sqrt{\log(n)}).
\end{equation}

By (\ref{Eq4.14.10}), (\ref{Eq4.14.11}), and (\ref{Eq4.14.12}), noting (\ref{Eq4.14.13}) and (\ref{Eq4.14.14}), we obtain that
\begin{eqnarray}
&& \mathbb{P}(\mathcal{B}(\bm{\phi};\Lambda_0)\cap\mathcal{D}(\bm{\phi},n_0)\cap \mathcal{C}(\bm{\phi};\Lambda_0')\cap \mathscr{E}\cap \{d_R(\nu_{k;R},\mu)\leq \delta\})\nonumber\\
&\leq& \exp\Big(-\frac{\beta k^2}{2}\Big(\Big(1+\frac{C_1}{\log(n)}\Big)^{-1}I_R(\mu,3\delta)-C(M\sqrt{\epsilon}+M^{-1})\Big)\Big)\nonumber\\
&& +\exp(-ck^2\sqrt{\log(n)}).
\end{eqnarray}
Hence by (\ref{Eq4.14.15}), as $|\mathcal{N}_0(\bm{\phi})|\leq n+1\leq 2n$, we have
\begin{eqnarray}\label{Eq4.14.17}
&& \mathbb{P}(\mathcal{B}(\bm{\phi};\Lambda_0)\cap \mathscr{E} \cap \{d_R(\nu_{k;R},\mu)\leq \delta\})\nonumber\\
&\leq& \exp\Big(-\frac{\beta k^2}{2}\Big(\Big(1+\frac{C_1}{\log(n)}\Big)^{-1}I_R(\mu,3\delta)-C(M\sqrt{\epsilon}+M^{-1})\Big)\Big) \nonumber\\
&& +\exp(-ck^2\sqrt{\log(n)})+C \exp(-Mk^2).
\end{eqnarray}

Recall the definition of $\Phi$ below (\ref{Eq2.14}). As $T$ only depends on $M,\epsilon$, we have
\begin{equation}
    |\Phi|\leq \mathcal{I}_0 \cdot  2^{2 W_{\mathcal{I}_0+1}+1} \cdot 2^{2 W_{\mathcal{I}_0+1}+1}\leq 8 \lceil M^2\slash \epsilon\rceil T^4 \leq C_{M,\epsilon}.
\end{equation}
Hence by (\ref{Eq4.14.16}) and (\ref{Eq4.14.17}), we have
\begin{eqnarray}\label{Eq4.14.18}
&& \mathbb{P}(d_R(\nu_{k;R},\mu)\leq \delta) \nonumber\\
&\leq& C_{M,\epsilon} \exp\Big(-\frac{\beta k^2}{2}\Big(\Big(1+\frac{C_1}{\log(n)}\Big)^{-1}I_R(\mu,3\delta)-C(M\sqrt{\epsilon}+M^{-1})\Big)\Big)\nonumber\\
&&+C_{M,\epsilon} \exp(-ck^2\sqrt{\log(n)})+C_{M,\epsilon}\exp(-c Mk^2).
\end{eqnarray}

We take $n=k^{40000}$ and $\epsilon=M^{-4}$. Recall (\ref{Eq1.2.1}). In (\ref{Eq4.14.18}), first taking $k\rightarrow\infty$, and then taking $\delta\rightarrow 0^{+}$, we obtain that
\begin{eqnarray}
&& \limsup_{\delta\rightarrow 0^{+}}\limsup_{k\rightarrow\infty}\frac{1}{k^2}\log(\mathbb{P}(d_R(\nu_{k;R},\mu)\leq \delta)) \nonumber\\
&\leq& \max\Big\{-cM,-\frac{\beta}{2}I_R(\mu)+C(M\sqrt{\epsilon}+M^{-1})\Big\}\nonumber\\
&\leq& \max\Big\{-cM,-\frac{\beta}{2}I_R(\mu)+CM^{-1}\Big\}.
\end{eqnarray}
Letting $M\rightarrow\infty$, we obtain that
\begin{equation}
    \limsup_{\delta\rightarrow 0^{+}}\limsup_{k\rightarrow\infty}\frac{1}{k^2}\log(\mathbb{P}(d_R(\nu_{k;R},\mu)\leq \delta))\leq -\frac{\beta}{2} I_R(\mu).
\end{equation}

\end{proof}

\subsection{Local large deviation lower bound}\label{local_low}

In this subsection, we establish the local large deviation lower bound in Proposition \ref{local_low_p} below. Again, we recall relevant notations and definitions from Definitions \ref{Defn1.1}-\ref{Defn1.3}.

\begin{proposition}\label{local_low_p}
For any $\mu\in\mathcal{X}_R$, we have
\begin{equation}
    \liminf_{\delta\rightarrow 0^{+}}\liminf_{k\rightarrow\infty}\frac{1}{k^2}\log(\mathbb{P}(d_R(\nu_{k;R},\mu)\leq \delta))\geq -\frac{\beta}{2} I_R(\mu).
\end{equation}
\end{proposition}
\begin{proof}

Consider any $\delta>0$ and any $\mu'\in\mathscr{X}$ such that $d_R(\mu'_R,\mu)\leq \delta\slash 2$. Following the proof of \cite[Theorem 1.1(a)]{Zho} (which can be adapted to general $\beta\in\mathbb{N}^{*}$), we obtain that 
\begin{eqnarray*}
   && \liminf_{k\rightarrow\infty}\frac{1}{k^2}\log(\mathbb{P}(d_R(\nu_{k;R},\mu_R')\leq \delta\slash 2))\nonumber\\
   &\geq& \liminf_{\delta'\rightarrow 0^{+}} \liminf_{k\rightarrow\infty}\frac{1}{k^2}\log(\mathbb{P}(d_R(\nu_{k;R},\mu_R')\leq \delta'))\geq -\frac{\beta}{2} \mathscr{I}(\mu').
\end{eqnarray*}
As $\{d_R(\nu_{k;R},\mu)\leq \delta\}\supseteq \{d_R(\nu_{k;R},\mu_R')\leq \delta\slash 2\}$, we have
\begin{equation*}
    \liminf_{k\rightarrow\infty}\frac{1}{k^2}\log(\mathbb{P}(d_R(\nu_{k;R},\mu)\leq \delta))\geq -\frac{\beta}{2} \mathscr{I}(\mu').
\end{equation*}
Hence for any $\delta>0$, 
\begin{equation*}
    \liminf_{k\rightarrow\infty}\frac{1}{k^2}\log(\mathbb{P}(d_R(\nu_{k;R},\mu)\leq \delta))\geq -\frac{\beta}{2} \inf_{\substack{\mu'\in\mathscr{X}:\\ d_R(\mu'_R,\mu)\leq \delta\slash 2}}\mathscr{I}(\mu')=-\frac{\beta}{2} I_R(\mu,\delta\slash 2).
\end{equation*}
Taking $\delta\rightarrow 0^{+}$, we conclude that
\begin{equation}
    \liminf_{\delta\rightarrow 0^{+}}\liminf_{k\rightarrow\infty}\frac{1}{k^2}\log(\mathbb{P}(d_R(\nu_{k;R},\mu)\leq \delta))\geq -\frac{\beta}{2} I_R(\mu).
\end{equation}
\end{proof}

\bibliographystyle{acm}
\bibliography{Airy.bib}

\end{document}